\documentclass{amsart}

	\RequirePackage[OT1]{fontenc}

	\usepackage{amsmath,amssymb,bbm}
	\usepackage{amsthm}
	\usepackage{latexsym}
	\usepackage{graphicx}
	\usepackage{color,wrapfig}
	\usepackage{subfigure}
	\usepackage{multirow}
	\usepackage{bm,url,comment}
	\usepackage{mathtools,mathrsfs}
	\usepackage{cool}
	\usepackage{bm}
	\usepackage{array}

	\usepackage[neverdecrease]{paralist}

	\usepackage{booktabs}

	\newcommand{\NN}{\mathbb{N}}

	\newcommand{\RR}{\mathbb{R}}
	\newcommand{\CC}{\mathbb{C}}
	\newcommand{\EE}{\mathbb{E}}
	\newcommand{\PP}{\mathbb{P}}

	\newcommand{\B}[1]{\mathbf{#1}}

	\newcommand{\cov}[2]{\operatorname{Cov}({#1}, {#2})}

	\newcommand{\conj}[1]{\overline{#1}}

	\theoremstyle{definition}
	\newtheorem{theorem}{Theorem}[section]
	\newtheorem{defn}[theorem]{Definition}

	\newtheorem{lemma}[theorem]{Lemma}
	\newtheorem{corro}[theorem]{Corollary}
	\newtheorem{thm}{Theorem}[section]

	\theoremstyle{remark}
	\newtheorem*{remark}{Remark}
	
	\newtheorem{assumption}[thm]{Assumption}

	\usepackage{hyperref}
	\hypersetup{
	    colorlinks,
	    citecolor=black,
	    filecolor=black,
	    linkcolor=blue,
	    urlcolor=black
	}

\begin{document}

	\title[Asymptotic analysis of SST]{Asymptotic analysis of synchrosqueezing transform -- toward statistical inference with nonlinear-type time-frequency analysis}
	
	\author{Matt Sourisseau}

	\address{Department of Mathematics, University of Toronto, Toronto, ON,  Canada}

		\author{Hau-tieng Wu}

	\address{Department of Mathematics and Department of Statistical Science, Duke University, Durham, NC, USA} 

		\author{Zhou Zhou}

	\address{Department of Statistics, University of Toronto, Toronto, ON,  Canada}

		\maketitle

		\begin{abstract}
			We provide a statistical analysis of a tool in nonlinear-type time-frequency analysis, the synchrosqueezing transform (SST), for both the null and non-null cases. The intricate nonlinear interaction of different quantities in SST is quantified by carefully analyzing relevant multivariate complex Gaussian random variables. Specifically, we provide the quotient distribution of dependent and improper complex Gaussian random variables. Then, a central limit theorem result for SST is established. {As an example}, we provide a block bootstrap scheme based on the established SST theory to test if a given time series contains oscillatory components. 
		\end{abstract}

\section{Introduction}
	\label{sec:Introduction}

	Time series contain dynamical information of a system under observation, and their ubiquity is well-known~\cite{Hallin:1978}. A key task in understanding and forecasting such a system is to quantify the dynamics of an associated time series according to a chosen model, a task made challenging by the fact that often the system is nonstationary.
	Although there is no universal consensus on how to model and analyze time series extracted from nonstationary systems, two common schools of thought are those of time series analysis~\cite{Hamilton:1994,Brockwell_Davis:2002,Fan_Yao:2005} and time-frequency (TF) analysis \cite{Daubechies:1992,Flandrin:1999}. Roughly stated, the main difference between these two paradigms is the assumptions they make on the underlying random process modeling a time series.

	In classical time series analysis, this random process is typically assumed to have zero first-order statistics. The focus is then on analyzing the second-order statistics, mainly for the purpose of forecasting. Seasonality of a time series (that is, an oscillatory pattern of known periodicity in its mean) is modeled separately or included in the covariance structure~\cite{Brockwell_Davis:2002}.
	When a time series is modeled as a sum of a parametric periodic mean function and a stationary noise sequence, there is a small body of statistics literature on methods and algorithms to estimate its periodicity when unknown.  
	The common ground with TF analysis originates in investigating ``local spectral behavior'' \cite{Priestley1967}, a generalization of the idea of using the spectrum to capture local behavior. See Section \ref{Section More Simulation} for more literature review. 
	This direction has a long history, beginning with the consideration of analytic model \cite{Gabor:1946,Picinbono:1997} and more recently progressing to the {\em adaptive harmonic model} (AHM) modeling time-varying frequency and amplitude \cite{Daubechies_Lu_Wu:2011,Chen_Cheng_Wu:2014}, or {\em adaptive non-harmonic model} (ANHM) further modeling  nonsinusoidal, time-varying oscillatory patterns~\cite{Wu:2013,lin2016waveshape}{. Unlike the classical time series analysis, in this direction the oscillation is modeled in the first-order statistics.} 

	In the TF approach, available algorithms are roughly classified into linear-type, bilinear-type and nonlinear-type.
	The synchrosqueezing transform (SST) and its variations, a family of nonlinear-type TF tool, were developed based on AHM/ANHM in the past decade. 
	SST can be viewed as a special case of the reassignment technique pioneered in~\cite{Kodera-76} and further explored in~\cite{Auger_Flandrin:1995}. SST encodes the spirit of empirical mode decomposition~\cite{Huang_Shen_Long_Wu_Shih_Zheng_Yen_Tung_Liu:1998}.
	Broadly, SST nonlinearly modifies the TF representation (TFR), and hence the spectrogram, derived from the short-time Fourier transform (STFT) by utilizing the {\em phase} information in STFT so that the TFR is sharper. 
	 Recall that spectrogram comes from dividing a time series into short segments by a chosen window (taper), and evaluate the tapered periodogram at each moment. As a result, the phase information included in STFT is lost. {The key feature that distinguishes SST from traditional spectral analysis is how the phase information is used to sharpen the TFR.}

	  {From the application perspective, SST has been  applied to handle diverse signal processing challenges since its development. Typical applications include (1) estimate the time-varying frequency and amplitude, (2) obtain the non-sinusoidal oscillatory pattern, (3) decompose the constitutional oscillatory components and their phase functions from a noisy observation, (4) determine if there is an oscillatory component and when it exists, and many others. See \cite{WU20208} for a recent review article for its successful scientific applications in medicine.  
	 Due to its flexibility, several variations of SST} have been proposed. For example, taking the S-transform~\cite{Huang_Zhang_Zhao_Sun:2015} or wave packets~\cite{Yang:2014} into account, considering higher order phase information~\cite{Oberlin_Meignen_Perrier:2015}, combining the cepstrum tool \cite{lin2016waveshape}, and applying multi-taper techniques~\cite{Xiao_Flandrin:2007,Daubechies_Wang_Wu:2016}. 
	 
	 {From the theoretical perspective, recently the theoretical analysis of SST under the AHM or ANHM when noise does not exist has been well-established. See \cite{Daubechies_Lu_Wu:2011,Chen_Cheng_Wu:2014} for example. 
	However, when noise (or any stochastic process) is present, the exploration has been limited to asymptotic expansion~\cite{Brevdo_Fuckar_Thakur_Wu:2013,Chen_Cheng_Wu:2014,YANG2018526} or only part of the algorithm~\cite{chassande1998statistics}. 
	 To our knowledge, its statistical property, particularly the asymptotic distribution, even in the null case (that is, no oscillatory signal), is still missing.} 
	%

	See Figure~\ref{fig:PPG} for an example of SST when applied to a noisy photoplethysmography (PPG) signal. 
	The PPG signal is non-invasive, cost-effective, and widely used in healthcare environment \cite{alian2014photoplethysmography}. However, {noise is inevitable in the clinic environment, which might downgrade the reliability of the signal. In this example, the second half of the signal is of low quality in the sense that the cardiac oscillation cannot be visualized. 
	In the associated TFR determined by SST, there is a curious ``texture'' structure (indicated by blue arrows) in the second half, which comes from the low-quality signal, while there is a dominant curve (indicated by red arrows) in the first half, which encodes the time-varying heart rate information. If we could understand the asymptotic distribution of SST under various situations, we could further utilize information encoded in the signal}. 

	{Motivated by the wide application of SST and the missing asymptotic analysis of SST (e.g., for a systematic study of various applications of SST) from the statistical perspective, the} main focus of this work is providing a statistical analysis of SST toward the statistical inference purpose. Specifically, we write down the distribution associated with SST of a stationary colored Gaussian random process, in both null and non-null cases. {We apply the developed result to design} a local bootstrapping algorithm for statistical inference for testing the existence of an oscillatory component, and its theoretical justification is also provided. {This algorithm is applied to determine the signal quality of the PPG signal shown in Figure \ref{fig:PPG}.}

	The first technical challenge encountered along the way is dealing with improper multivariable complex random variables and their ratio distributions. While proper (or, ``circular'') complex random variables have been widely discussed in signal processing literature~\cite{Baxley-10}, their improper counterparts and corresponding ratios have been mostly ignored, except for~\cite{Picinbono:1997,Schreier-10}. In our case, impropriety arises naturally from the phase information encoded in the STFT, and handling its effect on the quotient structure forms the first part of the paper, and is of its own interest for other applications.
	
	The second technical challenge is handling the nonlinear reassignment of STFT coefficients according to the reassignment rule. This nonlinear reassignment is challenged by handling the nonlinear change of variable by the $\CC\RR$-calculus computation~\cite{Kreutz-Delgado-09} and approximating the integration of confluent hypergeometric function that naturally pops out when we evaluate moments of SST, particularly in the non-null case. The analysis is complicated in the low frequency region due to the degeneracy of the covariance structure.
	This step has an interesting interpretation and helps us connect time series, TF analysis and other topics; the big picture is that the reassignment rule has a natural interpretation within the kernel regression framework of time series analysis, and can be understood in the framework of diffusion geometry in the manifold learning setup. 
	
	The third technical challenge is handling the dependence structure when we show the central limit theorem (CLT) of SST. The main technique here is exploiting the $M$-dependence, and the associated critical quantity is the ``effective sampling rate'' --  once we find a proper $M$-dependent surrogate of the original random process associated with SST, if the ``effective sampling rate'' is correctly specified, we show that in {\em both} null and non-null cases, SST of a stationary colored Gaussian random process follows a complex normal distribution. With the above results, we could establish a theoretical justification of a local bootstrapping algorithm for the statistical inference.

	\begin{figure}[h!]
		\centering
		\includegraphics[trim=0 0 0 0, clip,width=.9\textwidth]{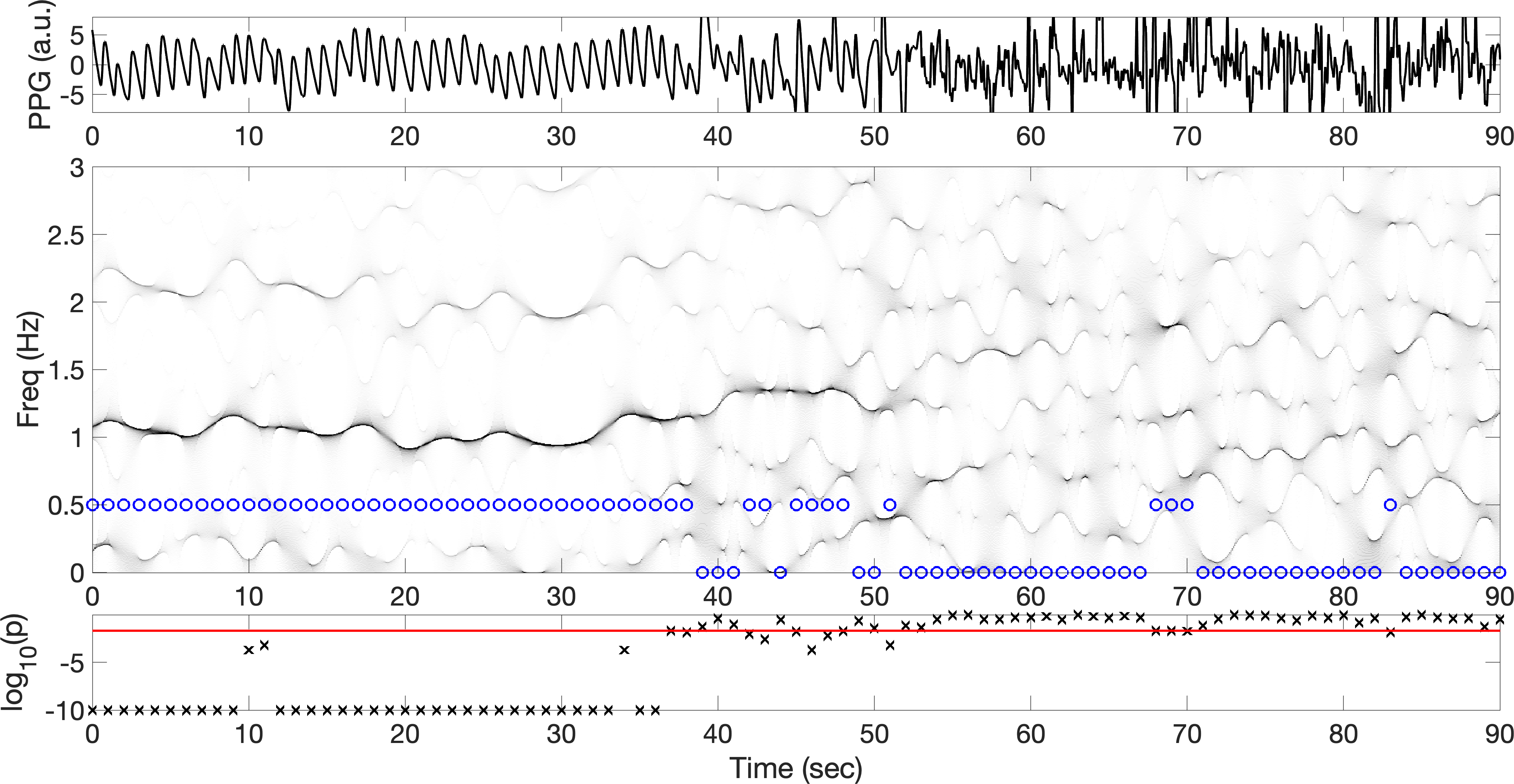}
		\caption{An application of the synchrosqueezing transform (SST) to the photoplethysmogram (PPG) signal. 
		Top: the PPG signal, where the second half of the signal is of low quality. 
		Middle: the time-frequency representation determined by the STFT-based SST of the PPG signal. The dominate curve around 1Hz indicates the instantaneous frequency (or instantaneous heart rate). The blue diamonds comes from the proposed bootstrapping algorithm {with the false discovery rate set to $0.05$ (Section \ref{Section:Dection of oscillatory signal}), where those at 0.5  (resp. 0) indicate the null hypothesis is (resp. cannot be) rejected, and hence the high (resp. low) quality.}
		{Bottom: the log 10 of p-values are plotted as black crosses with the threshold for the false discovery rate plotted as the red line.} 
		}\label{fig:PPG}
	\end{figure}

	The paper is organized in the following way. In Section~\ref{sec:SST summary}, we summarize the STFT-based SST. In Section~\ref{sec:multivariate complex Gaussian RVs}, we handle quotients of two complex Gaussian random variables. These results are of interest aside from their use in the sequel.
		The mathematical setup for the SST analysis is given in Section \ref{sec: SST setup}. Section~\ref{sec: SST statistics} includes the asymptotic analysis of SST. 
		A local bootstrapping algorithm for an application of SST to the oscillatory component detection problem and its theoretical validation are shown in Section \ref{Section:Dection of oscillatory signal}.
		 A numerical simulation is shown in Section \ref{sec:Numerics}.
		We conclude the paper with a discussion in Section \ref{sec:Discussion and conclusion}.
		More literature review, numerical example, and all proofs are relegated to the supplementary material. 
		In this paper, we use the following asymptotic notations. For two set of $\{a_u\}_{u\in U},\,\{b_u\}_{u\in U}\subset \mathbb{R}_+$ indexed by a set $U$, $a_u\asymp b_u$ means that there exist constants $0<c_1\leq c_2<\infty$ so that $c_1b_u\leq a_u\leq c_2 b_u$ for all $u\in U$, $a_u=O(b_u)$ means that there exists a constant $c$ so that $a_u\leq cb_u$ for all $u\in U$. For $a,b\in \mathbb{R}$, $a\vee b$ means $\max\{a,b\}$.


\section{A summary of the SST algorithm}
	\label{sec:SST summary}
	
	Take a Schwartz function $h$.
	For a tempered distribution $f$, the STFT of $f$ associated with the window function $h$ is defined by the equation
	$V_f^{(h)}(t,\eta)
		\vcentcolon=
		f(h_{t,\eta})$,
	where $t\in\RR$ is the time and $\eta >0$ is the frequency and 
	$
		h_{t,\eta}(s)
		\vcentcolon=
		h(s-t)e^{- 2\pi i \eta (s-t)}
	$. 
	We mention that this is a modification of the ordinary STFT by the phase modulation $e^{2\pi i\eta t}$, and we choose to work with it to simplify the upcoming heavy notation. 
	When $f$ is represented by a function, we may abuse notation in the usual way and write
	\begin{equation}\label{eq: STFT on functions}
		V_f^{(h)}(t,\eta)
		=
		\int_{-\infty}^{\infty} 
			f(s)h(s-t)e^{- 2\pi i \eta (s- t)}
		\,ds,
	\end{equation}
	Commonly, the window function $h$ is chosen to be a Gaussian function with mean $0$ and bandwidth $\sigma>0$. 
	Given the above, the {\em STFT-based synchrosqueezing transform} (SST) of 
	$f$ with the modified window function $h$ with {\em resolution} $\alpha > 0$ is defined to be
	\begin{equation}\label{eq: SST}
		S_f^{(h,\alpha)}(t,\xi) 
		\vcentcolon= 
		\int 
			V_f^{(h)}(t,\eta)
			\, g_\alpha\big(\xi-\Omega^{(h)}_f(t,\eta)\big)
		\, d\eta,
	\end{equation}
	where the \emph{reassignment rule} $\Omega^{(h)}_f(t,\eta)$ is defined by
	\begin{equation}\label{eq: STFT-based reassignment rule}
		\Omega^{(h)}_f(t,\eta)
		\vcentcolon= 
		\left\{
		\begin{array}{ll}
			\frac{1}{2\pi i} 
			\frac{\partial_t V_f^{(h)}(t,\eta)}{ V_f^{(h)}(t,\eta)} 
			& \text{if } V_f^{(h)}(t,\eta)\neq 0\\ 
			-\infty 
			& \text{otherwise}
		\end{array}\right.
	\end{equation}
	and $g_\alpha:\mathbb{C}\to \mathbb{R}$ is an approximate $\delta$-distribution when restricted on $\mathbb{R}$ when $\alpha$ is sufficiently small. For concreteness, we will take ${g}_\alpha(z) = \frac{1}{\sqrt{\pi\alpha}} e^{-|z|^2/\alpha}$, which has the $L^1$ norm $1$. 
	Notice that the nonlinearity of SST over signals arises from the dependence of equation~\eqref{eq: SST} on the reassignment rule, which provides information about the instantaneous frequency of the signal (as made precise in~\cite{Daubechies_Lu_Wu:2011,Wu-11}). $\alpha$ is interpreted as the {\em resolution} of SST in the frequency axis. Numerically, SST is implemented by a direct discretization of \eqref{eq: STFT on functions}, \eqref{eq: SST}, and \eqref{eq: STFT-based reassignment rule}. We will come back to this part when we discuss the proposed bootstrapping algorithm in Section \ref{Section:Dection of oscillatory signal}.
 {We should mention a commonly confusing point regarding SST. At the first glance, since a differentiation is taken when we define the reassignment rule, SST is unstable when noise exists. This is not the case since $\partial_t V_f^{(h)}(t,\eta)=-V_f^{(h')}(t,\eta)+2\pi i \eta V_f^{(h)}(t,\eta)$; that is, the differentiation operator is equivalent to evaluating another STFT with the differentiation of $h$ as the window and summing $2\pi i \eta V_f^{(h)}(t,\eta)$. Based on this fact, the stability of SST has been established in \cite{Brevdo_Fuckar_Thakur_Wu:2013,Chen_Cheng_Wu:2014}.}

	To better appreciate which kind of information is utilized in SST, rewrite $V_f^{(h)}(t,\eta)=|V_f^{(h)}(t,\eta)|e^{i2\pi P_f^{(h)}(t,\eta)}$, where $P_f^{(h)}(t,\eta)$ is the ``phase'' of the complex value $V_f^{(h)}(t,\eta)$. By properly choosing the branch for $\log$, we have the relationship $\Omega_f^{(h)}(t,\eta)=\frac{1}{2\pi i}\partial_t \log(V_f^{(h)}(t,\eta))=\frac{1}{2\pi i}\partial_t \log |V_f^{(h)}(t,\eta)| + \partial_t P_f^{(h)}(t,\eta)$. Suppose the magnitude $|V_f^{(h)}(t,\eta)|$ changes slowly, we see that the reassignment rule encodes the phase information.

	The goal of this paper is to initiate the study of the distribution of $S_{f+\Phi}^{(h,\alpha)}(t,\xi)$ for $t\in\RR$ and $\xi>0$, where $f$ is a deterministic tempered distribution and $\Phi$ is a generalized random process (GRP); see Section \ref{sec:app: GRP summary} for a summary of GRP. In this case, to understand the statistics of $S_{f+\Phi}^{(h,\alpha)}(t,\xi)$, we are led to consider the distribution of the ratio of the random variables
	\begin{align}
		V_{f+\Phi}^{(h)}(t,\eta) 
			&=
			f(h_{t,\eta})
			+ \Phi(h_{t,\eta})
			\label{eq:general V_f+Phi}\\
		\partial_t V_{f+\Phi}^{(h)}(t,\eta)
			&= 
			\partial_t f(h_{t,\eta})
			+
			\Phi(\partial_t h_{t,\eta})
			\label{eq:general d_t V_f+Phi}.
	\end{align}
	We work under the assumptions that the noise $\Phi$ is mean-zero. Under this assumption, we have $\mu = [\mu_1\,\,\,\mu_2]^\top$ so that 
	\begin{equation}\label{eq:mu vector}
		\mu 
		\vcentcolon= 
		\EE\big[
			V_{f+\Phi}^{(h)}(t,\eta)\ \ 
			\partial_t V_{f+\Phi}^{(h)}(t,\eta)
			\big]^\top=\begin{bmatrix}
			f(h_{t,\eta})& 
			\partial_t f(h_{t,\eta})
			\end{bmatrix}^\top\,.
	\end{equation}
	By a direct expansion and the fact that $\partial_t[h_{t,\eta}(s)]=-(h')_{t,\eta}(s)+i2\pi\eta h_{t,\eta}(s)$, by definition, the reassignment rule takes the form
	\begin{equation}\label{eq:general reassignment rule}
		\Omega_{f+\Phi}^{(h)}(t,\eta) 
		=
		\frac{1}{2\pi i} 
		\left(
			\frac{\mu_2 + 2\pi i \eta \,\Phi(h_{t,\eta}) - \Phi((h')_{t,\eta})}
			{\mu_1 + \Phi(h_{t,\eta})}
		\right)\,.
	\end{equation}
	Note that $(h')_{t,\eta}(s) = h'(s-t)e^{-2\pi i \eta(s-t)}$. If the noise is such that $\Phi(h_{t,\eta})$ is zero only on a set of measure zero, we may add and subtract $2\pi i \eta \mu_1$ in the numerator of~\eqref{eq:general reassignment rule} to obtain the almost-sure equality
	\begin{align}\label{eq:general reassignment rule a.s. expression}
		\Omega_{f+\Phi}^{(h)}(t,\eta) 
		&\buildrel \text{a.s} \over =
		\,\eta - \frac{1}{2\pi i} \left(
		\frac{2\pi i \eta \mu_1 -\mu_2 + \Phi((h')_{t,\eta})}
		{\mu_1 + \Phi(h_{t,\eta})}
		\right).
	\end{align}
	
	We will treat the \emph{null case} when $f(t) = 0$, and the \emph{non-null case} when $f(t)$ is not identically zero.	%
	The analysis depends on understanding the random variables at hand, specifically $\frac{2\pi i \eta \mu_1 -\mu_2 + \Phi((h')_{t,\eta})}
		{\mu_1 + \Phi(h_{t,\eta})}$, which we address now.

\section{Complex Gaussians and their quotients}\label{sec:multivariate complex Gaussian RVs}

	Suppose the complex random vector $\B{Z}\in \CC^n$ can be written in the form $\B{Z}= \B{X}+i\B{Y}$ for some $n$-dim real-valued random vectors $\B{X}$ and $\B{Y}$. The density of $\B{Z}$ is then defined to be the density of $[\B{X}^\top\,\,\,\B{Y}^\top]^\top \in \RR^{2n}$; that is, $f_{\B{Z}}(x+iy) \coloneqq f_{\B{X},\B{Y}}(x,y)$. 

	\begin{defn}[Complex Gaussian distribution \cite{Schreier-10}]
	\label{def: complex gaussian}
		Let $\mu \in \CC^n$. Suppose $\Gamma, C \in \CC^{n\times n}$ are Hermitian positive-definite and complex symmetric, respectively, and the Hermitian matrix $\conj{\Gamma} - C^* \Gamma^{-1} C$ is positive definite. We write $\B{Z} \sim \CC N_n(\mu, \Gamma, C)$ and say $\B{Z} \in \CC^n$ follows a \emph{complex Gaussian} distribution with mean $\mu$, covariance $\Gamma$, and pseudocovariance $C$ if
		\begin{equation}
		\label{eq: complex gaussian density}
			f_{\B{Z}}(z) =
			\pi^{-n} (\det{\underline{\Sigma}})^{-1/2}
			e^{-\frac{1}{2}(\underline{z}-\underline{\mu})^* \underline{\Sigma}^{-1} (\underline{z}-\underline{\mu})},
		\end{equation}
		where $z\in \CC^n$, $\underline{z}\vcentcolon=\begin{bmatrix}z\\\conj{z}\end{bmatrix}$, and 
		$\underline{\Sigma} 
			\vcentcolon= \begin{bmatrix}
			\Gamma & C\\
			\overline{C} & \overline{\Gamma}
			\end{bmatrix}
		$ is the {\em augmented covariance matrix}.
		$\B{Z}$ is said to be \emph{proper} if $C=\B{0}$, and \emph{improper} otherwise. 
	\end{defn}

	Note that while \emph{real} Gaussian vectors are completely characterized by their mean and covariance, complex Gaussian vectors are characterized by their mean and {\em augmented covariance}, as is clearly seen by the structure of the matrix $\underline{\Sigma}$. If a complex Gaussian vector has uncorrelated components (that is, diagonal $\Gamma$), it does not necessarily follow that these components are independent, as $C$ may be nonzero.  
	When $\B{Z}$ is proper, commutativity of matrix inversion with conjugation and positive-definiteness of $\Gamma$ gives us
	\begin{equation}
	f_{\B{Z}}(z) = 
	\pi^{-n} (\det{\Gamma})^{-1}
	e^{-(z-\mu)^* \Gamma^{-1} (z-\mu)}.
	\end{equation}
	If $\Gamma$ is also diagonal, the real and imaginary parts of $\B{Z}$ are seen by direct calculation to be independent $N(\Re{(\mu)},\Gamma/2)$ and $N(\Im{(\mu)},\Gamma/2)$ random variables, respectively. 

	Note that positive-definiteness of ${\underline{\Sigma}}$ guarantees the invertibility of $\Gamma$, and hence by~\cite[Proposition 2.8.3]{Bernstein-09} we have 
	$\det{\underline{\Sigma}} = \det{\Gamma} \det{P}$,
	where the so-called {\em Schur complement},
	$P \vcentcolon= \conj{\Gamma}- \conj{C} \Gamma^{-1} C$ is also invertible.
	By block matrix inversion~\cite[(2.8.16), (2.8.18) and (2.8.20)]{Bernstein-09}, we have  
	$\underline{\Sigma}^{-1} 
		= 
		\begin{bmatrix}
			\conj{P^{-1}} & -\conj{P^{-1}}\conj{R} \\
			-R^\top \conj{P^{-1}} & P^{-1}
		\end{bmatrix}$,
	 where
	$R \coloneqq \conj{C} \Gamma^{-1}$. Observe that since $\underline{\Sigma}$ is positive definite, $P^{-1}$ is too.
	Moreover, 
	 by a direct calculation, $\conj{P^{-1}}\conj{R}$ is symmetric. 

	\subsection{Complex Gaussian quotient density and moments}
	\label{ssec:Complex Gaussian quotient density}

		If the complex random vector $(Z_1,Z_2) \sim \CC N_2(0,\Gamma,0)$, then the density of $Z_2/Z_1$ has a simple closed-form determined in \cite{Baxley-10}. 
		To extend this result to the most general case, we recall from \cite[p.~285, (10.2.8)]{Lebedev-72}, \cite[p.~290, (10.5.2)]{Lebedev-72} and \cite[(4)]{Pham-Gia-06} that the Hermite function $H_\nu$ of order $\nu<0$ is an analytic function of $z\in \CC$ and satisfies the identities
		\begin{align}
			\label{eq:Hermite function of negative order integral representation}
			H_{\nu}(z) 
				&= \frac{1}{\Gamma(-\nu)} \int_0^\infty t^{-(\nu+1)} e^{-t^2 -2tz} \,dt\,\\
			\label{eq: H_nu  sum}
			H_{\nu}(-z) + H_{\nu}(z) 
				&= \frac{2^{\nu+1}\sqrt{\pi}}{\Gamma((1-\nu)/2)} \Hypergeometric{1}{1}{\frac{-\nu}{2}}{\frac{1}{2}}{z^2}\\
			\label{eq: H_nu difference}
			H_{\nu}(-z) - H_{\nu}(z) 
				&= \frac{2^{\nu+2}\sqrt{\pi}z}{\Gamma(-\nu/2)} 
				\Hypergeometric{1}{1}{\frac{-\nu+1}{2}}{\frac{3}{2}}{z^2}\,,
		\end{align}
		where $\Hypergeometric{1}{1}{}{}{}$ is the confluent hypergeometric function \cite[p.~239, (9.1.4)]{Lebedev-72}.

		With these in mind, we have the following new result:

		\begin{thm}[Complex Gaussian quotient density]
			\label{thm: Quotient density}
			If $\begin{bmatrix}Z_1&Z_2\end{bmatrix}^\top \sim \CC N_2(\mu,\Gamma,C)$, then for any $q \in \CC$ the density of the random variable $Q = Z_2/Z_1$ is given by
			\begin{align}
				\label{eq: Gaussian quotient density}
				f_Q(q) 
				&= 
					\frac{e^{-\frac{1}{2}\underline{\mu}^*\underline{\Sigma}^{-1} \underline{\mu} }}
					{\pi^2 \sqrt{\det{\Gamma}\det{P}}}
					\int_0^{\pi}
					\Hypergeometric{1}{1}{2}{\frac{1}{2}}{\frac{
					B_\mu(\theta,q)^2
					}{
					A(\theta,q)
					}}
					\frac{1}{
					A(\theta, q)^2
					}
					 \,d\theta,
			\end{align}
			where 
			\begin{align}
				\label{eq:A}
				A(\theta, q) 
					&= \B{q}^* \overline{P^{-1}} \B{q} - \Re{\big(e^{2i\theta} \B{q}^\top R^\top \overline{P^{-1}} \B{q} \big)}\\
				B_\mu(\theta,q)
					&= \Re{\big(e^{i\theta} (\mu^* - \mu^\top R^\top) \conj{P^{-1}} \B{q} \big)}
					\label{eq: zeta_2}
			\end{align}
			and $\B{q}:=[1\,\,\,q]^\top\in\mathbb{C}^2$.
			In the special case when $C=0$ and $\mu = 0$, we write $Q^\circ$ instead of $Q$, and we have
			$f_{Q^\circ}(q) 
				= 
				\frac{1}{\pi \det{\Gamma}}
				\frac{1}{(\B{q}^* \Gamma^{-1} \B{q})^2}$.
					\end{thm}

		The proof of Theorem \ref{thm: Quotient density} is given in Appendix~\ref{Proof Theorem 2.1}. The symbols $A$ and $B_\mu$ originating from the existence of the impropriety are introduced to keep the formula \eqref{eq: Gaussian quotient density} and the following analysis succinct. Notice that $B_\mu$ vanishes when the mean $\mu$ is zero. Moreover, $f_{Q^\circ}$ does not blow up at $0$ because $\|\B{q}\|^2  = 1+|q|^2$ for all $q$ and decays like $\|\B{q}\|^{-4}$ as $|q|\to \infty$.

		\begin{thm} 
			\label{thm: quotient mean bound proposition}
			Let $[Z_1\,\,\,Z_2]^\top \sim \CC N_2(\mu,\Gamma,C)$ and $Q = Z_2/Z_1$. Then 
			\begin{enumerate}[(i)]
				\item $\EE |Q|^\beta$ and $\EE Q^\beta$ are finite when $0\leq \beta<2$, and infinite when $\beta=2$.

				\item When $C=0$ and $\mu=0$, we have
				$\EE Q^\circ = \Gamma_{21}/\Gamma_{11}$. 
				\item When $C=0$, $\Gamma$ is diagonal, and $\mu=(\mu_1,0)\in \mathbb{C}^2$, we have $\EE Q = 0$. 
			\end{enumerate}
			
		\end{thm}

		The proof of Theorem \ref{thm: quotient mean bound proposition} is given in~\ref{Proof Proposition 2.3}. 
		Note that when $\Gamma$ is not diagonal, $Z_1$ and $Z_2$ are dependent. In the SST application we soon consider, $\Gamma$ will automatically be diagonal, and for most applications, a whitening process can achieve this condition.

				\begin{remark}
		We mention a point of potential confusion from the literature. It is well-known that the quotient of two \emph{real} Gaussian random variables (independent or not) has a Cauchy tail, and its density is given explicitly in \cite[Theorem 2]{Pham-Gia-06}. One might expect a parallel statement that ``a quotient of complex Gaussians should have a complex Cauchy distribution''. However, this is not the case: a distribution by the name of complex Cauchy has already appeared in the literature~\cite[p.~46, (2.80) with $n=1$]{Schreier-10}, its density being given by
		$
		f(z) = \frac{1}{\pi \sqrt{\det{\underline{S}}} }
		\left( 1+(\underline{z}-\underline{\mu})^*\underline{S}^{-1}(\underline{z}-\underline{\mu})\right)^{-3/2}
		$
		for some location parameter $\mu$ and scatter (or dispersion) matrix $\underline{S}$. This does not coincide with the distribution of the quotient of two complex Gaussians; for example, the former does not have a mean~\cite{Schreier-10}, whereas the latter does. 
		\end{remark}
		
\section{Mathematical model for SST analysis}\label{sec: SST setup}
	
	The following mathematical model is considered to analyze SST. 
	We follow the ideas in~\cite{Gelfand-64,Koralov-07} and introduce a complex version of Gaussian noise. Then, we study each step of the algorithm.
	\begin{defn}[Stationary complex Gaussian noise]\label{Definition:stationary complex Gaussian}
		A stationary GRP $\Phi$ is \emph{complex Gaussian} if for any finite collection $\psi_1,\ldots,\psi_n\in \mathcal S$ we have 
		$(\Phi(\psi_1),\ldots,\Phi(\psi_n)) \sim \CC N_n(0,\Gamma,C)$,
		where
		$\Gamma_{i,j} = \mathbb{E}[
			(\Phi(\psi_i)-\mathbb{E}[\Phi(\psi_i)])
			\overline{(\Phi(\psi_j)-\mathbb{E}[\Phi(\psi_j)])}
			]$
		and
		$C_{i,j} =\mathbb{E}[
			(\Phi(\psi_i)-\mathbb{E}[\Phi(\psi_i)])
			(\Phi(\psi_j)-\mathbb{E}[\Phi(\psi_j)])
			]$
		for all $1\le i,j\le n$. Denote the spectral measure associated with $\Phi$ by $d \vartheta(\xi)$, where $\int (1+|\xi|)^{-2l}d\vartheta(\xi)<\infty$ for some $l>0$; see Section~\ref{sec:app: GRP summary} for more information.
	\end{defn}

	
	\begin{assumption}\label{assump:noise part}
		 For the stationary GRP $\Phi$, we assume that the spectral measure associated with $\Phi$ is absolutely continuous related to the Lebesgue measure and there exists a smooth function $p(\xi)$, the spectral density of $\Phi$, so that $d\vartheta(\xi)=p(\xi)d\xi$ by the Radon-Nikodym theorem. Further, we assume that $p(\xi)\asymp (1+|\xi|)^\varrho$ for $\xi\in \mathbb{R}$, where $\varrho<2l-1$. When $\varrho=0$, the GRP is white; otherwise it is colored.
	\end{assumption}
	
		A concrete example of the colored GRP is the continuous autoregressive and moving average (CARMA) random process \cite[(2.16)]{Brockwell2001}. The CARMA(1,0) process is defined as a stationary solution of the first-order stochastic differential equation $(D+a)X(t)=bDW(t)$, where $a>0$, $b\neq 0$, $t\geq 0$, $D$ is the differentiation with respect to $t$ in the proper sense, and $\{W(t)\}$ is the standard Brownian motion. We have $\mathbb{E}X(t)=0$ and the power spectrum function is $p(\xi)=\frac{1}{2\pi}\frac{b^2}{|i\xi+a|^2}$; that is, $\varrho=-2$. 

	To study the non-null case, we focus on the oscillatory signal we have interest. In practice, the frequency and amplitude of an oscillatory signal both vary with time \cite{Daubechies_Lu_Wu:2011, Chen_Cheng_Wu:2014}. The AHM is commonly applied to model such oscillatory signals, where the frequency and amplitude are assumed to change slowly relative to its time-varying frequency. Under the slowly varying assumption, a function satisfying the AHM can be well-approximated {\em locally} by a single harmonic component \cite{Daubechies_Lu_Wu:2011,Chen_Cheng_Wu:2014}. In light of this, we assume from now on the following:
	\begin{assumption}\label{assump:nonnull signal}
		 Consider $Y=f+\Phi$, where $\Phi$ is a stationary GRP and
	$f(t) = A e^{2\pi i (\xi_0t+\phi_0)}$ is the oscillatory signal  for some fixed frequency $\xi_0>0$, phase shift $\phi_0\in[0,1)$, and amplitude $A\ge 0$. We refer to the situation when $A=0$ as the {\em null case}, otherwise the {\em non-null case}.
	\end{assumption}

	Since working with a more general kernel will not provide more insight to understanding SST but the notation will become highly intense, we make the following assumption in the following analysis.
	
	\begin{assumption}\label{assump:gaussian window}
		The window is $h(x)=(2\pi)^{-1/2} e^{-x^2/2}$. 
	\end{assumption}

\section{Statistical analysis of SST}\label{sec: SST statistics}

	From now on, unless otherwise described, we always assume Assumptions \ref{assump:noise part}, \ref{assump:nonnull signal} and \ref{assump:gaussian window} hold.

\subsection{The statistical behavior of  STFT}
	{We start from establishing the statistical behavior of STFT.
	\begin{thm}\label{thm: V_Phi and V_Phi' second-order stats}
		Suppose Assumptions~\ref{assump:noise part}, \ref{assump:nonnull signal} and \ref{assump:gaussian window} hold. For any $t\in \mathbb{R}$ and $\eta,\eta'>0$, then $\begin{bmatrix}\Phi(h_{t,\eta})\ \ \Phi(h_{t,\eta'})\end{bmatrix}^\top\sim \CC N_2(0,\Gamma_{\eta,\eta'},C_{\eta,\eta'})$, where
		\begin{align*}
		 \Gamma_{\eta,\eta'}&\,=
		 \begin{bmatrix} 
		 \gamma_0(\eta,\eta)  & e^{-\pi^2(\eta'-\eta)^2}\gamma_0(\eta,\eta') \\ e^{-\pi^2(\eta'-\eta)^2}\gamma_0(\eta',\eta)  & \gamma_0(\eta',\eta')
		  \end{bmatrix},\ \
		 C_{\eta,\eta'}=e^{-\pi^2(\eta'+\eta)^2}\begin{bmatrix} 
		 \nu_0(\eta,\eta)  
		 & \nu_0(\eta,\eta') \\  
		 \nu_0(\eta',\eta)
		 & \nu_0(\eta',\eta')\end{bmatrix}\,,
		 \end{align*}
			 $\gamma_0(\eta,\eta') \vcentcolon= \int  e^{-4\pi^2(\xi+\frac{\eta+\eta'}{2})^2}\,d\vartheta(\xi)$ and
		$\nu_0(\eta,\eta') \vcentcolon=\int e^{-4\pi^2(\xi+\frac{\eta-\eta'}{2})^2}\,d\vartheta(\xi)$. In other words, we have
		\begin{align*}
			\cov{V_{f+\Phi}^{(h)}(t,\eta)}{V_{f+\Phi}^{(h)}(t,\eta')} 
			= &\,
			e^{-\pi^2(\eta'-\eta)^2}\gamma_0(\eta,\eta')\\
			\cov{V_{f+\Phi}^{(h)}(t,\eta)}{\overline{V_{f+\Phi}^{(h)}(t,\eta')}} 
			= &\,
			e^{-\pi^2(\eta'+\eta)^2}\nu_0(\eta,\eta')\,.
		\end{align*}
	\end{thm}
	In other words, at a fixed time $t$, the STFT coefficient at each frequency $\eta$ is a complex normal distribution, and the dependence of two coefficients of two frequencies, $\eta$ and $\eta'$, decay exponentially fast when $|\eta-\eta'|$ increases. We refer readers to \cite{yang2020spectral} for more discussion of statistical inference via STFT. Next, we prepare results to study SST.}
	By equations~\eqref{eq:general V_f+Phi} and~\eqref{eq:general d_t V_f+Phi}, 
	we investigate the noise structure
	$\B{W}_{t,\eta} 
		\vcentcolon= 
		\begin{bmatrix}\Phi(h_{t,\eta})\ \ \Phi((h')_{t,\eta})\end{bmatrix}^\top$
	The second-order statistics of $\B{W}_{t,\eta}$ are computed in the following lemma.

	\begin{lemma}\label{lem: V_Phi and d_t V_Phi second-order stats}
		For any $t\in \mathbb{R}$ and $\eta>0$, $\B{W}_{t,\eta}\sim \CC N_2(0,\Gamma_\eta,C_\eta)$, where
		\begin{align}
			\Gamma_\eta 
			&= 
			\begin{bmatrix}
				\gamma_0(\eta) & -2\pi i\gamma_1(\eta)\\
				2\pi i\gamma_1(\eta) & 4\pi^2\gamma_2(\eta)
			\end{bmatrix}
			\label{even/odd covariance of W}
			\,\,\mbox{and}\,\,
			C_\eta
			:=
			e^{-4\pi^2\eta^2}\begin{bmatrix}
				\gamma_0(0)  & 2\pi i \eta\gamma_0(0)\\
				 2\pi i \eta\gamma_0(0) & 4\pi^2[\gamma_2(0)-\eta^2\gamma_0(0)]
			\end{bmatrix},\nonumber
		\end{align} 
			and $\gamma_k(s) 
			:= 
			\int_{-\infty}^\infty 
				e^{-4\pi^2(\xi+s)^2}(\xi+s)^kd\vartheta(\xi)$ for $s\in\mathbb{R}$ and $k\ge 0$. 
	\end{lemma}
	See~Section \ref{section:Technical lemmas for the SST analysis} for a proof. 
	Clearly, $C_\eta \to \Gamma_\eta$ when $\eta\to 0$. Therefore, the eigenvalues of the augmented covariance matrix becomes more degenerate when $\eta\to 0$. On the other hand, $C_\eta\to 0$ when $\eta\to \infty$. 
	Note that when the noise is white, that is, $d\vartheta(\xi)=d\xi$, 
		the formula are simplified as 
			$\Gamma_\eta = 
				\frac{1}{2\sqrt{\pi}}
				\begin{bmatrix}
					1 & 0 \\
					0 & 1/2
				\end{bmatrix}$ and
			$C_\eta = 
				\frac{e^{-4\pi^2\eta^2}}{2\sqrt{\pi}}
				\begin{bmatrix}
					1 & 2\pi i \eta \\
					2\pi i \eta & 1/2 -4\pi^2\eta^2
				\end{bmatrix}$. Due to the relationship between different moments, in general $\Gamma_\eta$ is not degenerate with a lower bound for eigenvalues, like the white noise case (see Lemma \ref{Lemma: c of mu_xi0eta when eta small}).

	\subsection{The statistical behavior of the reassignment}

	Note that \eqref{eq:general V_f+Phi}, \eqref{eq:general d_t V_f+Phi} and~\eqref{eq:mu vector} 
	are reduced to
	\begin{align}
		V_{f+\Phi}^{(h)}(t,\eta) 
		&= 
		f(t) \hat{h}(\eta-\xi_0) 
		+ \Phi(h_{t,\eta}),\nonumber\\
		\partial_t V_{f+\Phi}^{(h)}(t,\eta)  
		&= 
		f'(t) \hat{h}(\eta-\xi_0) 
		+ 2\pi i \eta \,\Phi(h_{t,\eta}) 
		- \Phi((h')_{t,\eta}),\\
		\mu 
		&= 
		\big[ f(t)\hat{h}(\eta-\xi_0)\ \ f'(t)\hat{h}(\eta-\xi_0)\big]^\top=\big[ \mu_1\ \ 2\pi i\xi_0\mu_1\big]^\top\,,\nonumber
	\end{align}
	where $\hat{h}$ is the Fourier transform of $h$, $\mu_1=f(t)\hat{h}(\eta-\xi_0)$, and
	\eqref{eq:general reassignment rule a.s. expression} becomes{\small
	\begin{equation}\label{eq:nonnull reassignment}
		\Omega_{f+\Phi}^{(h)}(t,\eta) 
		\buildrel \text{a.s} \over =
		\,\eta - \frac{1}{2\pi i} \left(
		\frac{2\pi i (\eta-\xi_0) \mu_1 + \Phi((h')_{t,\eta})}
		{\mu_1 + \Phi(h_{t,\eta})}
		\right)=:\eta-\frac{1}{2\pi i}Q_{f+\Phi}^{(h)}(t,\eta)\,.
	\end{equation}}
	Since $\Omega_{f+\Phi}^{(h)}(t,\eta)$ and $Q_{f+\Phi}^{(h)}(t,\eta)$ are linearly related, we only need to study one of them. As a quotient of two complex Gaussian random variables, the behavior of $Q_{f+\Phi}^{(h)}(t,\eta)$ could be immediately understood from Section \ref{ssec:Complex Gaussian quotient density} when the mean is $\mu=\begin{bmatrix}\mu_1 & 2\pi i(\eta-\xi_0)\mu_1 \end{bmatrix}^\top$. For example, the variance of $Q_{f+\Phi}^{(h)}(t,\eta)$ does not exist.
	In the special case when $\eta=\xi_0$, by \eqref{eq:nonnull reassignment}, we have
	\begin{equation}
		\Omega_{f+\Phi}^{(h)}(t,\xi_0) 
		\buildrel \text{a.s.} \over= 
		\xi_0 - \frac{1}{2\pi i}
		\left( \frac{\Phi((h')_{t,\xi_0})}
		{f(t) + \Phi(h_{t,\xi_0})}\right)
	\end{equation}
	since $\hat{h}(0)=1$. When $\xi_0$ is sufficiently large so that the pseudocovariance is small by Lemma \ref{lem: V_Phi and d_t V_Phi second-order stats} and the noise is white so that $\Gamma_\eta$ is diagonal, by Theorem \ref{thm: quotient mean bound proposition}(iii), we know that $\frac{\Phi((h')_{t,\xi_0})}
		{f(t) + \Phi(h_{t,\xi_0})}$ has a mean bounded by $Ce^{-4\pi^2\xi_0^2}\xi_0^{\max\{-2\varrho,\varrho\}}$ for some $C>0$. This says that when $\eta=\xi_0$ and the noise is white, the reassignment rule gives accurate frequency information. See Section \ref{OS section Proof of Proposition prop:Q convergence to proper case} for an argument. 

%

	
	Since $h$ is Gaussian, the covariance between $\Phi(h_{t,\eta})$ and $\Phi(h_{t,\eta'})$ decays exponentially when $|\eta-\eta'|$ increases. Intuitively, the covariance between $Q_{f+\Phi}^{(h)}(t,\eta)$ and $Q_{f+\Phi}^{(h)}(t,\eta')$ should also be small when $|\eta-\eta'|$ is large, but the ratio structure might obfuscate the speed of decay. The following theorem shows that this intuition is true. See Section~\ref{Section: proof of M dependent related arguments} for the proof.
	\begin{thm}[Ratio covariance] \label{thm: Q covariance} Suppose Assumptions~\ref{assump:noise part}, \ref{assump:nonnull signal} and \ref{assump:gaussian window} hold.
		For distinct 
		$\eta, \eta'>0$, as $|\eta-\eta'| \to \infty$ we have
		\begin{equation}
			\cov{Q_{f+\Phi}^{(h)}(t,\eta)}{Q_{f+\Phi}^{(h)}(t,\eta')} 
			= 
			O((\eta+\eta')^2e^{-(\eta -\eta')^2})\,.
		\end{equation}
	\end{thm}
	%

	\subsection{Preparation for the SST distribution}

		Following \eqref{eq: SST}, define the complex random vector 
		$
			\B{Z}_{\alpha,\xi,\eta} 
			\vcentcolon= 
			\begin{bmatrix}Y_{f+\Phi}^{(h,\alpha,\xi)}(t,\eta)& \Omega_{f+\Phi}^{(h)}(t,\eta)\end{bmatrix}^\top$, where
		\begin{equation} \label{Yt,eta) definition}
			Y_{f+\Phi}^{(h,\alpha,\xi)}(t,\eta) 
			\vcentcolon=
			V_{f+\Phi}^{(h)}(t,\eta)\frac{1}
				{\sqrt{\pi\alpha}} 
			\exp{\Big(
				-\frac{1}{\alpha} 
				\left|
					\xi-\Omega_{f+\Phi}^{(h)}(t,\eta)
				\right|^2
			\Big)}\,,
		\end{equation}
		$\eta>0$, $\xi>0$, and $\alpha>0$. Since $\Omega_{f+\Phi}^{(h)}(t,\eta)$ is not defined when $V_{f+\Phi}^{(h)}(t,\eta)=0$, $\B{Z}_{\alpha,\xi,\eta}$ is defined on $\CC^2\setminus\{(0,z_2):z_2\in\CC\}$. 
		Naturally, the distribution of $Y_{f+\Phi}^{(h,\alpha,\xi)}(t,\eta)$ is the marginal distribution of $\B{Z}_{\alpha,\xi,\eta}$. However, to precisely write down the distribution of $\B{Z}_{\alpha,\xi,\eta}$ and hence $Y_{f+\Phi}^{(h,\alpha,\xi)}(t,\eta)$ is not simple. We need to carry out a careful change of variable argument with the $\CC\RR$-calculus computation~\cite{Kreutz-Delgado-09} for this purpose, which is of its own independent interest and is summarized in Section \ref{Section:Collection of Useful Lemmas}.
						
		The seemingly complicated random variable $Y_{f+\Phi}^{(h,\alpha,\xi)}(t,\eta)$ turns out to have nice behavior -- 
		while the reassignment rule {$\Omega_{f+\Phi}^{(h)}(t,\eta)$} has a fat tail for every $\eta>0$ by Proposition~\ref{thm: quotient mean bound proposition} and \eqref{eq:nonnull reassignment}, 
		after being composed with a Gaussian function this tail is ``tamed''. To simplify the heavy notation, when there is no danger of confusion, we suppress $t$, $f+\Phi$ and $\xi$ and emphasize the ``bandwidth'' $\alpha$ and $\eta$ by denoting 
		\begin{equation}
		{Y}_{\alpha,\xi,\eta}:=Y_{f+\Phi}^{(h,\alpha,\xi)}(t,\eta),\,\, \Omega_\eta:=\Omega_{f+\Phi}^{(h)}(t,\eta)\,\,\mbox{ and }V_\eta:=V_{f+\Phi}^{(h)}(t,\eta).
		\end{equation}
		We state below that ${Y}_{\alpha,\xi,\eta}$ has finite moments of all orders, one for null and one for non-null case, relegating the proof to Section~\ref{Proof of Proposition joint density Y and omega}. 

		\begin{thm}[Absolute moments and moments of ${Y}_{\alpha,\xi,\eta}$, null case]\label{Proposition joint density Y and omega}
		Suppose Assumptions~\ref{assump:noise part}, \ref{assump:nonnull signal} and \ref{assump:gaussian window} hold, $A=0$ and 
			$k >0$.  For any $\eta,\xi,\alpha>0$, the $k$-th (absolute) moment of ${Y}_{\alpha,\xi,\eta}$ is finite. Moreover,
			\begin{enumerate}
			
			\item (absolute moments)
			For $\eta\geq1$ and $\xi>0$, when $\alpha$ is sufficiently small, we have
			$\EE{|{Y}_{\alpha,\xi,\eta}|^k}
				\asymp \alpha^{-k/2+1}$,
				where 
				 the implied constant depends on $\varrho$, $k$ and $\frac{\eta^{k\varrho/2}}{(1+4\pi^2|\eta-\xi|^2)^{(k+4)/2}}$; 
				 for $\eta<1$ and $\xi>\sqrt{\eta}$, when $\alpha$ is sufficiently small, we have
			\begin{align}
				c_1\alpha^{-k/2+1} \leq \EE{|{Y}_{\alpha,\xi,\eta}|^k}
				\leq  c_2 \Big(\frac{1}{\sqrt{\alpha\eta}}e^{-\frac{k\xi^2}{2\alpha}}\vee 1\Big)\alpha^{-k/2+1}\,,
				\end{align}
				where 
				 $c_1>0$ depends on $\varrho$, $k$ and $\frac{\eta^{3k+8}}{(1+4\pi^2|\eta-\xi|^2)^{(k+4)/2}}$, and $c_2>0$ depends on $\varrho$ and $k$.
				 
				 \item  (moments)
			When $k$ is odd, for any $\eta,\xi>0$, 
			$
			\mathbb{E}{Y}_{\alpha,\xi,\eta}^k=0$.
			When $k$ is even, for $\eta\geq1$, when $\alpha$ is sufficiently small, we have $
				|\EE{Y}_{\alpha,\xi,\eta}^k|
				=O( \alpha^{-k/2+1})$, 
				where 
				 the implied constant depends on $\varrho$, $k$ and $\frac{\eta^{k\varrho/2}e^{-4\pi^2\eta^2}}{(1+4\pi^2|\eta-\xi|^2)^{(k+4)/2}}$; 
				 for $\eta<1$, we have a simple bound
			$|\EE {Y}_{\alpha,\xi,\eta}^k|\leq  \EE |{Y}_{\alpha,\xi,\eta}|^k$.
				\end{enumerate}
				
				\end{thm}

In the theorem, we split the (absolute) moment evaluate into two cases, one is for $\eta\geq 1$ and one is for $\eta<1$, particularly when $\eta$ is close to $0$. In fact, when $\eta$ is close to $0$, $\underline\Sigma^{(h)}_\eta$ is close to being degenerate, and controlling the (absolute) moments depends on controlling the degeneracy, which needs a different approach compared with that when $\eta\geq 1$. Moreover, the spectral property of color noise is reflected in the (absolute) moment bound, and it interacts with $\xi$, the frequency we want to detect by SST. A similar fact holds for the non-null case, where the signal plays another role in the analysis. Note that the results for the non-null case is not optimal, while it is sufficient for our purpose.

		\begin{thm}[Absolute moments and moments of ${Y}_{\alpha,\xi,\eta}$, non-null case]\label{Proposition joint density Y and omega non-null}
			Suppose Assumptions~\ref{assump:noise part}, \ref{assump:nonnull signal} and \ref{assump:gaussian window} hold, $A>0$ and $k \in \NN$. 
			For any $\eta,\xi,\alpha>0$, the $k$-th (absolute) moment  of ${Y}_{\alpha,\xi,\eta}$ is finite. Moreover,
			\begin{enumerate}
			\item (absolute moment) 
			for $\eta\geq 1$, $|\eta-\xi_0|\geq 1/2$ and $\xi>0$, when $\alpha$ is sufficiently small, we have
			\begin{align}
				\EE{|Y_{\alpha,\xi,\eta}|^k}
				\asymp \alpha^{-k/2+1} \,,
				\end{align}
				where 
				 the implied constant depends on $\varrho$, $k$ and $\frac{\eta^{k\varrho/2}}{(1+4\pi^2|\eta-\xi|^2)^{(k+4)/2}}$; 
				 when $|\eta-\xi_0|< 1/2$ and $\xi>0$, when $\alpha$ is sufficiently small, we have
			\begin{align}
				c_1 \alpha^{-k/2+1}  \leq \EE{|Y_{\alpha,\xi,\eta}|^k}
				\leq  c_2\alpha^{-k/2+1} \,,
				\end{align}
				where $c_1$ depends on $\varrho$, $k$ and $e^{-A^2(1+4\pi^2\xi_0)^2\xi_0^{-\rho}}$ and
				 $c_2$ depends on $\varrho$, $k$ and $A^{k+3}\xi_0^{(2-\rho)(k+3)/2}$;
				  for $\eta<1$ (particularly when $\eta$ is close to $0$), 
				 when $\alpha$ is sufficiently small so that $\alpha<\eta$, we have
			\begin{align}
				c_1&\,e^{-CA^2\hat{\hbar}(\xi_0)^2(\eta^{-6}\vee\xi_0^2\eta^{-2})}\alpha^{-k/2+1} \leq  \EE{|Y_{\alpha,\xi,\eta}|^k}\leq  c_2  \max\Big\{\frac{1}{\sqrt{\alpha\eta}}e^{-\frac{k\xi^2}{2\alpha}},\, 1,\,\frac{1}{A\hat{\hbar}(\xi_0)\xi_0}\Big\} \alpha^{-k/2+1} \nonumber\,
				\end{align}
				where 
				 $c_1$ depends on $\varrho$, $k$ and $\frac{\eta^{3k+8}}{(1+4\pi^2|\eta-\xi|^2)^{(k+4)/2}}$, $C>0$ depends on $\varrho$, and $c_2$ depends on $\varrho$ and $k$.				  
				 
				 \item (moment) When $\eta\geq 1$ and $|\xi_0-\eta|\geq 1/2$, when $\alpha$ is sufficiently small, we have
			$
			|\mathbb{E}{Y}_{\alpha,\xi,\eta}^k|=O(e^{-4\pi^2(\xi_0-\eta)^2})$,
			where the implied constant depends on $\EE|{Y}_{\alpha,\xi,\eta}|^k$.
			When $\eta< 1$ or when $\eta\geq 1$ and $|\xi_0-\eta|< 1/2$, when $\alpha$ is sufficiently small, we have the trivial bound
			$|\EE{Y}_{\alpha,\xi,\eta}^k|
				\leq \EE|{Y}_{\alpha,\xi,\eta}|^k$.
				\end{enumerate}
				\end{thm}

				Note that ${Y}_{\alpha,\xi,\eta}$ is the product of two dependent random variables, $V_\eta$ and $g_\alpha( |\xi-{\Omega_\eta} |\big)$. By Lemma \ref{lem: V_Phi and d_t V_Phi second-order stats}, we know that the covariance of $V_\eta$ and $V_{\eta'}$ decays exponentially when $|\eta-\eta'|\to \infty$, and by Theorem~\ref{thm: Q covariance}, the same decay is true for the covariance of ${\Omega_\eta}$ and ${\Omega_{\eta'}}$. It is thus natural to expect that the covariance of ${Y}_{\alpha,\xi,\eta}$ and ${Y}_{\alpha,\xi,\eta'}$ also decays exponentially when $|\eta-\eta'|\to \infty$. Below, we show that despite the involved nonlinear transform, the same decay rate is also true for the covariance of ${Y}_{\alpha,\xi,\eta}$ and ${Y}_{\alpha,\xi,\eta'}$. 
		
		\begin{thm}
		\label{Theorem covariance of Yeta and Yetap}
			{Suppose Assumptions~\ref{assump:noise part}, \ref{assump:nonnull signal} and \ref{assump:gaussian window} hold. Fix $\xi>0$. For any $\alpha>0$}, $\text{Cov}\big({Y}_{\alpha,\xi,\eta},\,{Y}_{\alpha,\xi,\eta'}\big)$ and $\text{Cov}\big({Y}_{\alpha,\xi,\eta},\,\overline{{Y}_{\alpha,\xi,\eta'}}\big)$ are continuous over $\eta>0$ and $\eta'>0$. For $\eta,\eta'>0$ satisfying $|\eta-\eta'|\geq 1$, {when $\alpha$ is sufficiently small, 
			\begin{align}
				|\text{Cov}\big({Y}_{\alpha,\xi,\eta},\,{Y}_{\alpha,\xi,\eta'}\big)|\vee |\text{Cov}\big({Y}_{\alpha,\xi,\eta},\,\overline{{Y}_{\alpha,\xi,\eta'}}\big)|
				=
				O((\eta-\eta')^2e^{-\pi^2(\eta-\eta')^2}) \nonumber\,,
			\end{align}
			where the implied constants depend on $\varphi$ and $\varrho$}
		\end{thm}

	\subsection{Distribution of \texorpdfstring{$S_{f+\Phi}^{(h,\alpha)}(t,\xi)$}{SPhi}}

		With the above preparation, we may state the main result. 
		Without loss of generality, we focus on $t=0$. 
		To study the distribution of $S_{f+\Phi}^{(h,\alpha)}(0,\xi)$ for a given $\alpha>0$ and $\xi>0$, by viewing ${Y}_{\alpha,\xi,\eta}$ as a random process indexed by $\eta$, a natural approach is to discretize ${Y}_{\alpha,\xi,\eta}$, approximate $S_{f+\Phi}^{(h,\alpha)}(0,\xi)$ by a Riemann sum, and apply the CLT. 
		Below we consider the following discretization in $\eta$. 
		For each $l=1,\ldots,n$, denote $\eta_l\vcentcolon=l\Delta \eta$, where $\Delta \eta=n^{-1/2-\beta}$, and $\beta\geq0$ is to be determined in the proof. Also, denote $H=n\Delta \eta=n^{1/2-\beta}$. 
		To further simplify the notation, when there is no danger of confusion, denote 
		\begin{align}
		&{V_l\vcentcolon=V_{f+\Phi}(t,\eta_l)},\,\, {\Omega_l\vcentcolon=\Omega_{\eta_l}},\,\, Y_{\alpha,\xi,l}\vcentcolon=Y_{\alpha,\xi,\eta_l},\,\,\mathsf Y_{\alpha,\xi,l}\vcentcolon=\mathsf Y_{\alpha,\xi,\eta_l}.
		\end{align}  
		For $\xi>0$, we approximate $S_{f+\Phi}^{(h,\alpha)}(0,\xi)$ by the Riemann sum:
		\begin{equation}\label{sst discretization}
		S_{\alpha,\xi,n}\vcentcolon= \Delta \eta\sum_{l=1}^n  Y_{\alpha,\xi,l}\,.
		\end{equation}
		The asymptotic distribution of $S_{\alpha,\xi,n}$ when $n\to \infty$ represents the distribution of $S_{f+\Phi}^{(h,\alpha)}(0,\xi)$, which is related to integrating over a wider spectral range with a finer frequency resolution. 
		Note that the dependence structure of ${Y}_{\alpha,\xi,\eta}$ generates difficulty when we evaluate \eqref{sst discretization}, despite its exponential decay indicated in Theorem \ref{Theorem covariance of Yeta and Yetap}. 
		To handle this difficulty, the proof heavily depends on the $M$-dependent argument. Intuitively, the behavior of the $M$-dependent random process of ${Y}_{\alpha,\xi,\eta}$, denoted as $\mathsf{Y}_{\alpha,\xi,\eta}$, will be essentially the same as that of ${Y}_{\alpha,\xi,\eta}$ for large $M$, and we expect
		$
			\text{Var}({Y}_{\alpha,\xi,\eta})
			\approx
			\text{Var}(\mathsf {Y}_{\alpha,\xi,\eta})
			\approx 
			\text{Cov}({Y}_{\alpha,\xi,\eta},\mathsf{Y}_{\alpha,\xi,\eta})
		$
		 for any $(t,\eta)$. The following lemma quantifies this intuition. The main challenge toward this seemingly simple conclusion is the nonlinearity inherited from SST, which boils down to 
		 studying the relationship between the covariances of ${Y}_{\alpha,\xi,\eta}$ and $\mathsf {Y}_{\alpha,\xi,\eta}$. To the best of our knowledge, this kind of problem is less considered in the Gaussian approximation literature and there is no standard approach toward it. We provide a separate section elaborating this technique in Section \ref{Section: proof of M dependent related arguments}, which is the basis of the proof of Theorem \ref{thm:SSQ-CLT} shown in Section \ref{Proof SST related Section4}.

		\begin{thm} \label{thm:SSQ-CLT}
			{Suppose Assumptions~\ref{assump:noise part}, \ref{assump:nonnull signal} and \ref{assump:gaussian window} hold and assume $\varrho<5$.} Fix $\xi>0$, take a small $\delta>0$ and set $\beta=\frac{\delta}{4(2+\delta)}$ and $\Delta \eta=n^{-1/2-\beta}$.
			Also assume $\alpha=\alpha(n)$ so that $\alpha\to 0$ and $n\alpha\to 1$ when $n\to \infty$. We have $S_{\alpha,\xi,n}-\mathbb{E}S_{\alpha,\xi,n}\to \CC N_1(0,\nu,\wp)$ weakly when $n\to \infty$, where $\nu>0$ and $\wp\in \mathbb{C}$ are of order $(1+\xi)^\varrho$, where the implied constant depends on $\varrho$ when $A=0$ and on $A$, $\xi_0$ and $\varrho$ when $A>0$.
		\end{thm}

		Note that  
		the condition $n\alpha\to 1$ could be understood as the ``sufficient sampling'' condition.

	\section{An application -- oscillatory component detection via SST}\label{Section:Dection of oscillatory signal}

A critical question in practice is how to determine if a given time series contains an oscillatory component. This challenging problem has attracted lots of attention \cite{priestley1981spectral}, but so far there is no universally accepted solution, particularly when handling modern biomedical signals. 
Below, we propose a detection algorithm based on SST to handle this challenge. 

	\subsection{Discretization of SST}

Before introducing the algorithm, we detail the discretization of SST. First, we follow the setup in \cite{Chen_Cheng_Wu:2014} to discretize $\Phi$. 
\begin{assumption}\label{Assumption boosting}
Take a symmetric Schwartz function $\psi$ so that $\hat{\psi}(\xi)=1$ when $|\xi|\leq 1/4$ and $\hat{\psi}(\xi)=0$ when $|\xi|>1/2$. Set $X_j:=\Phi({\bar N}^{1/2}\psi(\bar N\cdot-j))$, where ${\bar N}^{1/2}\psi(\bar N\cdot-j)$ is of unit $L^1$ norm centered at $j/\bar N$, $\bar N>0$ is the sampling frequency, $j=-\lfloor\frac{N}{2}\rfloor,-\lfloor\frac{N}{2}\rfloor+1,\ldots,\lfloor\frac{N}{2}\rfloor-1,\lfloor\frac{N}{2}\rfloor$, and $N$ is the number of sampling points. Below, we assume $\bar N= N^{1/2}$. 
\end{assumption} 

Note that $\psi$ is the {\em measurement function}, which models the properties of the measurement equipment, and $N/\bar N>0$ is the recording length. Clearly, $\mathsf{X}:=[X_{-\lfloor\frac{N}{2}\rfloor},\ldots,X_{\lfloor\frac{N}{2}\rfloor}]$ is a discretization of the stationary GRP $\Phi$ so that $\mathsf{X}$ is a stationary Gaussian time series with mean $0$. For $l\in\{-\lfloor\frac{N}{2}\rfloor,-\lfloor\frac{N}{2}\rfloor+1,\ldots,\lfloor\frac{N}{2}\rfloor-1,\lfloor\frac{N}{2}\rfloor\}$, $\mathsf{X}(l)$ is the discretization at time $l/\bar N$ and $\mathsf{X}(l)$ has a Gaussian distribution with the standard deviation $\sigma_N:=(\int|\hat{\psi}(\xi/\sqrt{N})|^2p(\xi)d\xi)^{1/2}$, which increases when $N$ increases. Specifically, when $\varrho\geq 0$, the high frequency noise is not negligible. Thus, when $N$ gets larger, the ``measurement period'' is shorter, high frequency noise dominates, and the measured value is more uncertain. 

To discretize the analysis of $\Phi$ by SST, we need the following discretization. Without loss of generality, since $\Phi$ is stationary, in the following analysis we fix to $t=0$.  First, the STFT of $\Phi$ in \eqref{eq: STFT on functions} is discretized by
$
\mathbf V_{\eta,N}
		:=
		\frac{1}{\bar N}\sum_{j=-\lfloor N/2\rfloor}^{ \lfloor N/2\rfloor} X_j h\Big(\frac{j}{\bar N}\Big)e^{- 2\pi i \frac{\eta j}{\bar N}}$,
	where $\eta>0$. Note that $\mathbf V_{\eta,N}$ is a periodic function of $\eta$ with period $\bar N$, and $\mathbf V_{\eta,N}$ is the numerical implementation of $V_\Phi^{(h)}(0,\eta)=\Phi(h_{0,\eta})$. For the reassignment rule, the $\Phi((h')_{0,\eta})$ in \eqref{eq:general reassignment rule a.s. expression} can be implemented in the same way with $h$ replaced by $h'$. By a direct expansion, we have 
	\begin{equation}\label{Discretization Vf}
	\mathbf V_{\eta,N}=\Phi\left(\frac{1}{\bar N}\sum_{j=-\lfloor N/2\rfloor}^{ \lfloor N/2\rfloor} {\bar N}\psi(\bar N\cdot-j) h\Big(\frac{j}{\bar N}\Big)e^{- 2\pi i \frac{\eta j}{\bar N}}\right).
	\end{equation}
	Notice that $\frac{1}{\bar N}\sum_{j=-\lfloor N/2\rfloor}^{ \lfloor N/2\rfloor} {\bar N}\psi(\bar Ny-j) h\Big(\frac{j}{\bar N}\Big)e^{- 2\pi i \frac{\eta j}{\bar N}}$ could be viewed as an approximation of $h(y)e^{-i2\pi \eta y}$. Thus, if we define 
	$
	\mathbf h_{\eta,N}(\cdot):=\frac{1}{\bar N}\sum_{j=-\lfloor N/2\rfloor}^{ \lfloor N/2\rfloor} {\bar N}\psi(\bar N\cdot-j) h\Big(\frac{j}{\bar N}\Big)e^{- 2\pi i \frac{\eta j}{\bar N}}$, 
	which is a Schwartz function, $\mathbf V_{\eta,N}=\Phi(\mathbf h_{\eta,N})$. Similarly, we can define 
	$
	\mathbf h'_{\eta,N}(\cdot):=\frac{1}{\bar N}\sum_{j=-\lfloor N/2\rfloor}^{ \lfloor N/2\rfloor} {\bar N}\psi (\bar N\cdot-j) h'\Big(\frac{j}{\bar N}\Big)e^{- 2\pi i \frac{\eta j}{\bar N}}$, 
	and have $\mathbf V'_{\eta,N}:=\Phi(\mathbf h'_{\eta,N})$, which is the numerical implementation of $\Phi((h')_{0,\eta})$. 
	The reassignment rule can be implemented following \eqref{eq:general reassignment rule a.s. expression} by a direct division. For $\xi>0$, the integrand of SST is discretized in the same way as \eqref{sst discretization}, which is denoted by $\{\mathbf Y_{\alpha,\xi,l,N}\}_{l=1}^{n}$. Denote 
	$
	\mathbf S_{\alpha, \xi,n,N}\vcentcolon= \Delta \eta\sum_{l=1}^{n}  \mathbf Y_{\alpha,\xi,l,N}$.
The following corollary states the intuition that when the sampling rate $\bar N$ is high, the discretization of SST approaches the continuous version of SST. The proof is postponed to Section \ref{Section proof for section 5}.

\begin{corro}\label{Discretization corollary}
Assume Assumptions \ref{assump:noise part}, \ref{assump:nonnull signal}, \ref{assump:gaussian window} and \ref{Assumption boosting} hold. 
Adapt notations used in Theorem \ref{thm:SSQ-CLT}. 
For $\xi>0$ and a sufficiently large $n$, we have ${\mathbf S}_{\alpha,\xi,n,N}\to S_{\alpha,\xi,n}$ 
in probability when $N\to \infty$. 
\end{corro}

\subsection{Fisher-SST statistic}

Assume Assumptions~\ref{assump:noise part}, \ref{assump:nonnull signal} and \ref{assump:gaussian window} hold.
To test if an oscillatory component exists at a given time point, we propose to use the {\em maximal magnitude} of SST coefficients at the associated properly chosen frequency grid points as a test statistic under the null hypothesis $H_0:f(t)=0$ against the alternative $H_1:f(t)\neq 0$. 
Consider the following procedure.
Take a recorded time series $\boldsymbol{x}=[x_{-\lfloor\frac{ N}{2}\rfloor},\ldots,x_{\lfloor\frac{ N}{2}\rfloor}]\in\mathbb{R}^N$ following the discretization scheme in Assumption \ref{Assumption boosting}; that is, the sampling period is $1/ \sqrt{N}$. Note that $\boldsymbol{x}=\mathsf X$ under $H_0$. 
To simplify the notation, we fix to time $0$. For a given $N$, we could sample the frequency axis up to $\sqrt{N}/2$ Hz according to the Nyquist-Shannon theorem. Take $n$ to be the discretization of the SST integration in \eqref{sst discretization} and suppose $N$ is sufficiently larger than $n$. 
Then choose a uniform grid 
$
G:= (\xi_1,\xi_2,\ldots,\xi_r)\subset (0, \sqrt{N}/2)
$ 
on the frequency axis, where $r\in \mathbb{N}$ is chosen to be $\lfloor\sqrt{N}/\log(N)\rfloor$ so that $\xi_j=j\sqrt{N}/\lfloor\sqrt{N}/\log(N)\rfloor$ for $j=1,\ldots, r$. 
The {\em Fisher-SST statistic} is defined as 
\[
\mathfrak{m}_{n,N}:=\max_{\xi_j\in G} \big|\mathbf S_{\alpha,\xi_j,n,N}\big|\,.
\]
For a preassigned $a\in [0,1]$, let $T_a$ be the $(100\times a)$-th percentile of $\mathfrak{m}_{n,N}$ under the null.
We reject the null hypothesis if the Fisher-SST statistic of $\boldsymbol{x}$ exceeds $T_a$; that is, we detect a sufficiently strong oscillatory component compared with the noise around time $0$. 

{To determine when an oscillatory component exists over a period, we repeat the above steps and evaluate Fisher-SST statistics over a set of chosen timestamps. 
Note that SST on the TF domain has a dependent structure. When two consecutive time stamps are sufficiently separated, the SST coefficients would be approximately independent. 
Indeed, for the chosen kernel in Assumption \ref{assump:gaussian window}, $h(c)$ is numerically zero for a sufficiently large constant $c>0$, say, $c=10$. Therefore, numerically the SST coefficients are independent if two timestamps are separated by $2c$. The proof comes from the Plancheral theorem and the fact that the noise is Gaussian, and we omit details. Thus, we propose to choose a uniform grid $(t_1,t_2,\ldots,t_{r'})\subset [- \sqrt{N}/2,\sqrt{N}/2]$ on the time axis, where $r'\in \mathbb{N}$ is chosen so that $t_i-t_{i-1}=\log (N)$.
Then, for each timestamp $t_i$, evaluate the Fisher-SST statistic and obtain the p-value $p_i$. 
Since we will run this test for $r'=\lfloor\sqrt{N}/\log N\rfloor$ times and these tests are numerically independent, 
we recommend to consider the false discovery rate control \cite{benjamini1995controlling} to handle the multiple testing issue in the following way. Rank obtained p-values $p_1,\ldots,p_{r'}$ as $p_{(1)}\leq p_{(2)}\leq \ldots\leq p_{(r')}$. Take $q>0$ to be the false discovery rate and set $k^*$ to be the largest integer so that $p_{(k^*)}\leq k^*q/r'$. At time $t_i$, the null hypothesis is rejected if $p_i\leq p_{(k^*)}$. See \cite{benjamini1995controlling} for details. The result shows when an oscillatory component exists over a period.}

\subsection{Bootstrapping}\label{section local bootstrapping}
{Although we know that the discretization effect on $\mathbf S_{\alpha,\xi_j,n,N}$ disappears asymptotically by Corollary \ref{Discretization corollary}, the distribution of $\lim_{N\to \infty}\mathfrak{m}_{n,N}$ is not known.} 
We thus propose a {{\em bootstrapping}} algorithm to approximate the Fisher-SST statistic. 
Decompose $\boldsymbol{x}$ into the possibly existing oscillatory component $\boldsymbol{y}$ by the reconstruction formula provided in \cite{Daubechies_Lu_Wu:2011,Chen_Cheng_Wu:2014} and the noise part $\boldsymbol{n}$. 
%
%
Based on the stationary assumption in Assumption \ref{assump:noise part}, we could estimate the covariance structure of the noise by applying the banding covariance approximation approach \cite[Section 3]{xiao2012covariance}. 
Denote the estimated covariance as $\hat{\Sigma}$, and generate $m\in \mathbb{N}$, {say, 10,000,} pseudo-observed Gaussian noises $\boldsymbol{n}^{*(l)}\in\mathbb{R}^{N}$, where $l=1,\ldots,m$, with mean $0$ and the covariance structure $\hat\Sigma$.
For each $l$ and a timestamp, apply SST to $\boldsymbol{n}^{*(l)}$ and all grid points in $G$, denoted as $\big\{{\mathbf S}^{*(l)}_{\alpha,\xi_j,n,N}|\,\xi_j\in G\big\}$. 
Define 
\[
{\mathfrak{m}}^{*(l)}_{n,N}:=\max_{\xi_j\in G} \big|{\mathbf S}^{*(l)}_{\alpha,\xi_j,n,N}\big|\,.
\] 
The distribution of $\mathfrak{m}_{n,N}$ can be approximated by the empirical distribution of $\big\{{\mathfrak{m}}^{*(l)}_{n,N}\big\}_{l=1}^m$.
	The justification of the proposed bootstrapping algorithm is given in the following theorem, whose proof is postponed to Section \ref{Section proof for section 5}.

\begin{thm}\label{Bootstrapping corollary}
Assume Assumptions \ref{assump:noise part}, \ref{assump:nonnull signal}, \ref{assump:gaussian window} and \ref{Assumption boosting} hold and $G\subset (0, \sqrt{N}/2)$ is a uniform grid with $|G|=\lfloor\sqrt{N}/\log(N)\rfloor$. There exists a probability space associated with the bootstrapping, called $(\Omega, \mathsf F, \mathbb P)$, where we could construct a sequence of i.i.d. random variables $\big\{{\mathfrak{m}}^{*(l)}_{n,N}\big\}_{l=1}^m$ that follows the same distribution as that of $\mathfrak m^*_{n,N}$ and $\mathfrak m^{\#}_{n,N}$ which has the same distribution as that of $\mathfrak{m}_{n,N}$.
Then we have $\mathfrak{m}^*_{n,N}-\mathfrak{m}^{\#}_{n,N}\to 0$ in probability when $N\to \infty$. 
\end{thm}

\section{Numerical Results}
	\label{sec:Numerics}
	The Matlab code of SST is available in \url{http://hautiengwu.wordpress.com/}. {More numerical results of the developed theorems could be found in Section \ref{Appendix more numerical evidence}.}
	We demonstrate the proposed oscillatory signal detection algorithm and show the rejection rate with a { realistic simulated signal. 
	We consider the smoothed Brownian path realizations to model an oscillation with a slowly varying amplitude and frequency \cite{Daubechies_Wang_Wu:2016}. Suppose $W$ is the standard Brownian motion defined on $[0,\infty)$. A smoothed Brownian motion with the bandwidth $B>0$ is defined as $W_B := W\star K_B$, where $K_B$ is the Gaussian function with the bandwidth $B>0$ and $\star$ denotes the convolution operator. Given $T>0$ and parameters $\zeta_1,\dots,\zeta_6>0$, we then define a family of random processes on $[0,T]$ by
$\Psi_{[\zeta_1,\dots,\zeta_6]}(x):= \zeta_1 + \zeta_2 x+ \zeta_3\frac{W_{\zeta_4}(x)}{\|W_{\zeta_4}\|_{L^{\infty}[0,T]}} + \zeta_5\int_0^x \frac{W_{\zeta_6}(s)}{\|W_{\zeta_6}\|_{L^{\infty}[0,T]}}d{s}$.
	Consider $f(t)=Aa(t)\cos(2\pi\phi(t))+\Phi$ over $[0,T]$, where $A\geq 0$ controls the strength of the oscillation, $a(t)$ is a realization of $\Psi_{[1,0,0.5,5,0,0]}(t)$, $\phi(t)$ is a realization of $\Psi_{[0,8,0,0,2,4]}(t)+t^{1.1}/15$, and $\Phi$ follows an autoregressive and moving average (ARMA) process with a proper normalization so that the standard deviation is $1$ at each timestamp, where the ARMA process is determined by the auto-regression polynomial $a(z) = 0.5z + 1$ and the moving averaging polynomial $b(z) = -0.5z + 1$, with the innovation process taken as independent and identically distributed Gaussian random variables. Set $n=2,048$ and $\beta=0.05$, and realize $f$ with the sampling rate $64$Hz and sample $4,096$ points from $f$ over $[0,64]$ s. Take grids of the frequency and time axes as the above and and construct Fisher-SST at 32s. Note that the amplitude and frequency of the oscillation $a(t)\cos(2\pi\phi(t))$ both change slowly, and hence locally the signal oscillates like a harmonic function. 
	%
	%
	For a comparison, we also consider STFT and the chirplet path pursuit algorithm \cite{candes2008detecting}. Based on Theorem \ref{thm: V_Phi and V_Phi' second-order stats}, we consider STFT and define a similar statistic, called the Fisher-STFT statistic, by $\max_{\xi_j\in G} \big|\mathbf V_{\xi_j,N}\big|$. The bootstrapping to estimate the null distribution is carried out in the same way shown in Section \ref{section local bootstrapping}, except the step of decomposing the noise out of the noisy signal. Specifically, we follow the common practice to subtract the oscillation associated with the maximal peak determined by the periodogram from the noisy signal, and view the remaining component as the noise. 
	For Fisher-SST and Fisher-STFT, we repeat the bootstrapping for $10,000$ times to determine the threshold, where we take $0.05$ as our significance level.
	For the chirplet path pursuit, we use the ChirpLab v1.1 package provided by the authors of \cite{candes2008detecting}, where we use the best path statistic with the cubic polynomial to fit the amplitude, run Monte Carlo simulation for $10,000$ times, and take $0.05$ as the significance level.
	We realize $f$ for $1,000$ times with $A=0.12(k-1)$, where $k=1,\ldots,10$, and plot the simulated rejection rate in Figure \ref{Fig:RejctionRate}. We see that the Fisher-SST has a higher rejection rate compared with the Fisher-STFT and chirplet path pursuit, and the chirplet path pursuit performs better than the Fisher-STFT when the signal is strong. It is expected that chirplet path pursuit performs better than Fisher-STFT since a multiscale scheme and the chirp information is captured in the chirplet path pursuit algorithm. We shall mention that chirplet coefficients used in the chirplet path pursuit, which comes from the inner product of the signal and $\exp(i2\pi(at^2/2+bt))$ for a range of $a$ and $b$, decay at the rate $(a-a_0)^{-1/2}$ when $a_0$ is the true chirp, $b$ is fixed and $|a|\to\infty$. This slow decay could be understood as an uncertainty principle \cite{chen2021disentangling}. So, while chirplet could help capture an oscillatory component with a chirp, its performance might be impacted.}

\begin{figure}[bht!]
  \centering 
  \includegraphics[width=0.5\textwidth]{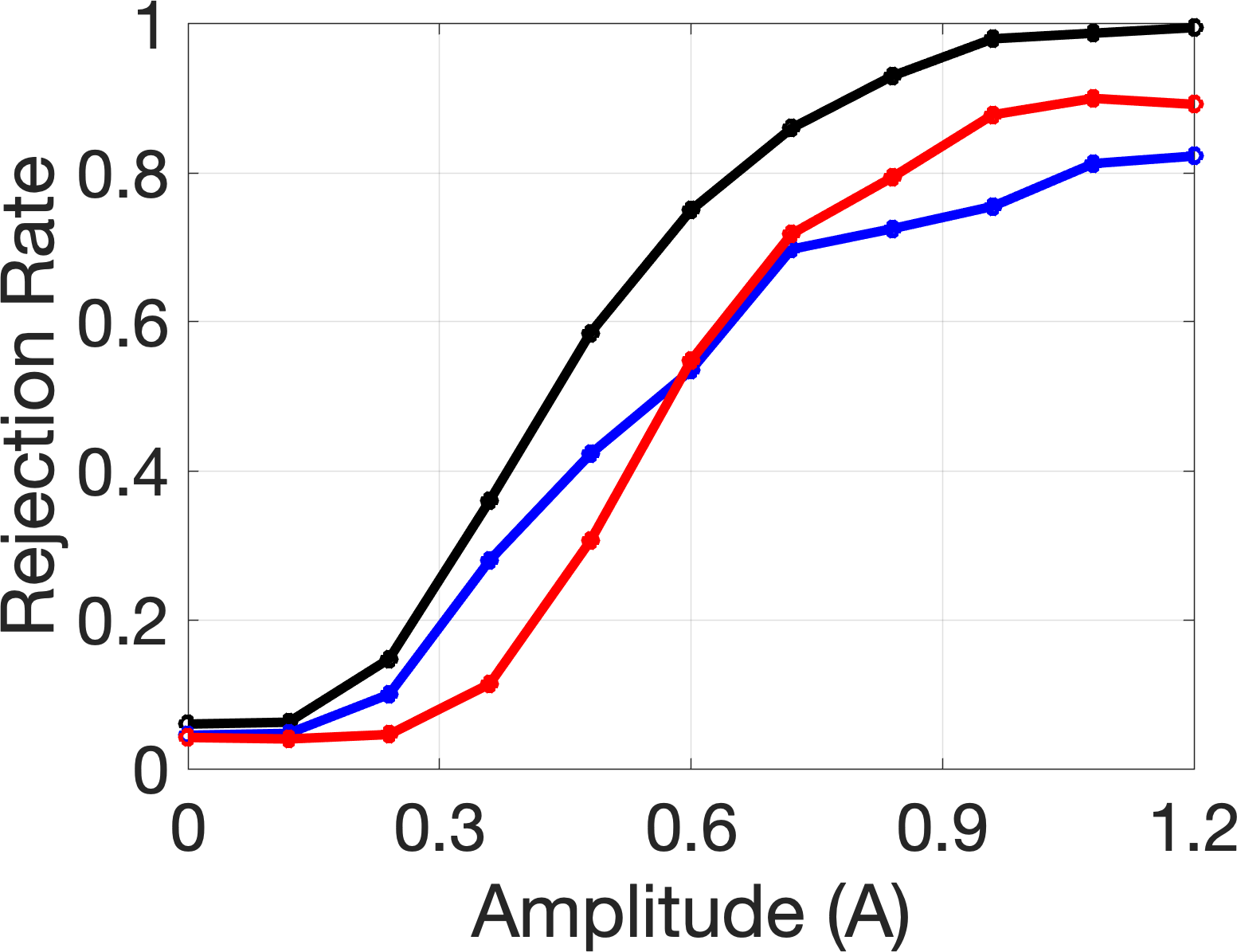}
\caption{The rejection rate of the proposed oscillatory signal detection scheme over a series of simulated signals, with the signal amplitude $A$ ranging from $0$ to $1.2$. {The black, blue and red curves are based on the Fisher-SST, the Fisher-STFT, and the chirplet path pursuit respectively.}}
\label{Fig:RejctionRate}
\end{figure}

Finally, we come back to the PPG signal shown in Figure \ref{fig:PPG}. 
{A critical biomedical signal processing step, particularly for long-term monitoring in clinics, is determining when the signal quality is trustworthy so that the obtained information is usable for decision making. This step is usually referred to signal quality assessment. Signal quality assessment is in general challenging, and the strategy depends on the clinical problem. When the heart rate and its variability are the concern, we care if a PPG signal oscillates properly and reflects how the heart beats, so that the TFR of a high-quality PPG encodes the time-varying heart rate as a curve with distinguishable intensity. 
See Figure \ref{fig:PPG} for an example, where the PPG signal in the first 50 second is labeled as high quality and the remaining signal is labeled as low quality. It is visually obvious to see an oscillation from the 0th to 40th second, which are cardiac cycles. However, after the 50th second, the signal looks chaotic and it is not clear if it provides any useful cardiac information. The signal between the 40th and 50th second is also oscillatory, but the pattern is slightly distorted compared with that before the 40th second. 
This visual inspection suggests that the signal quality over each segment of predetermined length could be quantified by the confidence of oscillation detection. Note that
the frequency and amplitude might change slowly, so it is reasonable to assume that locally the signal oscillates with fixed amplitude and frequency, and the proposed bootstrapping algorithm and the Fisher-SST statistic could be applied.}
In this signal, the sampling rate is 100Hz and the signal length is $100$ sec. The window satisfies Assumption \ref{assump:gaussian window}, $\Delta \eta=1/1201$, and the significance level is $0.05$. We evaluate the Fisher-SST statistic every 1 second by setting the grid size to be $0.5$ in $G$. {By setting the desired false discovery rate to be $0.05$}, the rejection of the null hypothesis, marked as blue diamonds, coincides with the visually identifiable cardiac oscillations. {Thus, the first 50-seconds segment are of high quality with some type I errors, which coincides with the expert's annotation.} 
The potential of designing a signal quality assessment algorithm based on the proposed algorithm will be further explored in our future work.

\section{Discussion and Conclusion}
	\label{sec:Discussion and conclusion}
	
	We provide a theoretical support for the nonlinear-type TF analysis algorithm, SST, {that forms a foundation for future statistical inference studies}. In particular, we extend the existing quotient distribution of proper complex normal random variables to the improper case, and quantify the asymptotic distribution of SST at a given frequency entry. While there are a multitude of available nonlinear-TF analysis algorithms, to the best of our knowledge, this is the first work providing an extensive quantification of the asymptotic distribution. This result sets the stage for further analysis of SST and other nonlinear-type TF analysis algorithms. 
	In particular, we provide several analytic tools to handle the main challenges when studying a nonlinear-type TF analysis algorithm. Specifically, in order to handle the nonlinearity involved in SST, a careful change of variables, an approximation scheme for the confluent hypergeometric function, and a construction of the associated $M$-dependent random process are given. 
	Observe that one major challenge in nonlinear frequency domain analysis lies in the lack of systematic dependence measures, such as strong mixing conditions \cite{rosenblatt1956central} and physical dependence measures \cite{wu2005nonlinear}. In this article, we adopted a highly nontrivial $M$-dependent approximation scheme in the frequency domain and successfully combined it with the nonlinear kernel regression technique in time series analysis to derive the asymptotic distribution of the STFT-based SST, and construct a local bootstrap inference procedure with theoretical supports.

	\subsection{Relationship with kernel regression}

		The approximation  \eqref{sst discretization} is related to the kernel regression perspective (See for instance Chapter 6 of \cite{Fan_Yao:2005}) of time series analysis.
		Suppose we {\em were} able to model $\{({\Omega_l,\,V_l})\}_{l=1}^n$ as a dataset sampled from a random vector $[X \,\,\, Y]^\top\in \CC^2$ so that $Y$ and $X$ were related by 
		$Y=F(X)+\mathfrak{N}$,
		where $\mathfrak{N}$ is random noise satisfying $\mathbb{E}[\mathfrak{N}|X=\xi]=0$, and the response ${V_l}$ and predictor ${\Omega_l}$ are related by the ``regression function'' $F$. If we further imagine ${f_{\Omega}(\xi)}$ to model the ``density of ${\Omega_l}$ at $\xi$'', then the kernel regression $S_{\xi,n}/{f_{\Omega}(\xi)}$ estimates that regression function at a fixed $\xi>0$; that is,
		it gives the conditional expectation of $Y$ given $X=\xi$ so that 
		$F(\xi)=\mathbb{E}[Y|X=\xi]$.
		In our case, this model is not correct, but still we have $S_{\xi,n}/{f_{\Omega}(\xi)}\to\mathbb{E}[Y|X=\xi]$ as $n\to \infty$.
		Adapting this kernel regression perspective, intuitively if we view ${V_l}$ as a ``noisy'' version of some regression function over ${\Omega_l}$, with the ``clean'' regression function providing the ``best'' TF representation, then the kernel regression helps recover this representation. When the signal is only noise, we expect $F$ to be zero and $S_{\xi,n}$ to converge to $0$. However, while this intuition helps us better understand how SST works, the structure of the regression function is not easy to directly identify in the non-null case.

	\subsection{Continuous wavelet transform based SST}

		The same analysis can be mimicked in the continuous wavelet transform (CWT) setup, but a significant simplification occurs regarding the pseudocovariance. In particular, let $\psi \in \mathcal{S}$ and
		for $a, b >0$, define $\Psi_{a,b}(t) = a^{-1/2}\psi((t-b)/a)$. Then the CWT of a tempered distribution $f$ takes the form
		$C_f^{(\psi)}(a,b)
			=
			f(\Psi_{a,b})$.
		Analogously to~\eqref{eq:general V_f+Phi} and~\eqref{eq:general d_t V_f+Phi}, we have
		$C_{f+\Phi}^{(\psi)}(a,b) 
				=
				f(\Psi_{a,b})
				+ \Phi( \Psi_{a,b} )\,, 			
				\partial_t V_{f+\Phi}^{(h)}(t,\eta)
				= 
				\partial_t f(\Psi_{a,b})
				+ \Phi( \Psi_{a,b}')$.
		To simplify the discussion, suppose $\Phi$ is white, and 
		the covariance becomes 
		$	\Gamma 
			= 
			\frac{1}{a^2}
			\begin{bmatrix}
				a^2 \int |\hat{\psi}(\xi)|^2 
				\, d\xi
				& -2\pi i a \int \xi |\hat{\psi}(a\xi)|^2
				\,d\xi\\
				2\pi i a \int \xi |\hat{\psi}(a\xi)|^2
				\,d\xi 
				& 4\pi^2 \int \xi^2 |\hat{\psi}(\xi)|^2 \, d\xi
			\end{bmatrix}
		$. For the pseudocovariance, 
		if we further assume that $\hat{\psi}$ is analytic with $\hat{\psi}$ real and $\text{supp}(\hat{\psi}) \subset (0,\infty)$, then the pseudocovariance matrix is manifestly zero. 
		Hence, $\Gamma$ and $C$ trivially commute and are thus simultaneously diagonalizable, so there is a basis of $\CC^2$ where the components of $[C_\Phi^{(\psi)}\,\,\,\partial_t C_\Phi^{(\psi)}]^\top$ are uncorrelated \emph{and} have a zero pseudocovariance. The reassignment rule is thus made as a quotient of \emph{independent} random variables, and the relevant nonlinear transform depending on the complex gaussian quotients simplify significantly when $C=0$.
		Note that the main technical challenge in analyzing STFT-based SST is handling pseudocovariance and this challenge is not encountered in the CWT-based SST. As a result, the proof is similar to that of the STFT-based SST shown in this paper, and we omit the details.

	\subsection{Future work}
	
	We remark that while the bounds in the proof are sufficient for our purpose, they might not be optimal, particularly when $\eta\to 0$. We need a different approach to handle the degeneracy of the covariance structure for a finer analysis.
	In addition to SST, there are many other nonlinear-type TF analysis algorithms, for example, reassignment \cite{Auger_Flandrin:1995}, concentration of frequency and time \cite{Daubechies_Wang_Wu:2016}, synchrosqueezed wave packet transform \cite{Yang:2014}, synchrosqueezing S-transform \cite{Huang_Zhang_Zhao_Sun:2015}, second-order SST \cite{Oberlin_Meignen_Perrier:2015}, and bilinear TF analysis tools like Cohen and Affine classes \cite{Flandrin:1999}. The current work sheds light on constructing a systematic approach to the study of statistical properties of those algorithms. 
	{As mentioned in Introduction, SST and these TF analysis tools have been widely applied in the signal processing society. Inspired by these applications, there are many important statistical inference problems remain open. For example, how to generalize the proposed oscillatory component detection algorithm to the case when multiple oscillatory components exist? How to establish the inference procedure for estimating instantaneous frequency, amplitude modulation and phase function and decomposing the noisy signal into its constitutional components? How to handle the nonstationary and/or non-Gaussian noise (e.g. the piecewise locally stationary \cite{zhou2013}) or study the statistical structure of a random process? How to detect the number of oscillatory components present inside a noisy signal or at which times such components exist? To answer these questions, we need to fully understand the distribution of SST on the TF domain (e.g. for various $\xi$ and $t$) so that an inference can be carried out on the TFR level. Note that} the oscillatory signal detection algorithm proposed in Section \ref{Section:Dection of oscillatory signal} is a special inference example. {More discussion can be found in Section \ref{Section More Simulation} of the supplement}.
	%

\section{Acknowledgements}

	The authors acknowledge Professors Almut Burchard and Mary Pugh for the fruitful discussion. {They thank the authors of \cite{candes2008detecting} for providing the ChirpLab 1.1 code.} They also thank the associate editor and the anonymous reviewers for their valuable and constructive feedbacks and comments.

\bibliographystyle{abbrv}
\bibliography{bib-10-09.bib}

	\setcounter{page}{1}
	\setcounter{equation}{0}
	\setcounter{figure}{0}
	\renewcommand{\thepage}{SI.\arabic{page}}
	\renewcommand{\thesection}{SI.\arabic{section}}
	\renewcommand{\theequation}{SI.\arabic{equation}}
	\renewcommand{\thelemma}{SI.\arabic{lemma}}
	\renewcommand{\theprop}{SI.\arabic{proposition}}
	\renewcommand{\thetable}{SI.\arabic{table}}
	\renewcommand{\thefigure}{SI.\arabic{figure}}

	%

\appendix

\section{More literature review and further discussion}\label{Section More Simulation}

In a pioneering work on statistical inference of unknown periodicity, Fisher~\cite{Fisher1929} proposed a maximum periodogram test, which was later investigated and extended~\cite{Hannan1961,LewisFieller1979,Priestley1981, Chiu1989,LinLiu2009}. Parametric and non-paramtric approaches have found applications in modelling the light curves of variable stars with chirp behavior in their frequencies~\cite{Genton_Hall:2007} and possible oscillatory patterns~\cite{Oh_Nychka_Brown_Charbonneau:2004,Bickel_Kleijn_Rice:2008}, and a more general combination of all of the above has also been considered~\cite{DeLivera_Alysha_Hyndman_Snyder:2011}. However, it seems that in time series literature, little attention has been paid to the possibility of complex seasonality with time-varying amplitude and frequency.
	To capture nonstationarity, a random process can be modelled as locally stationary~\cite{Silverman1957,Dahlhaus:1997}, piecewise locally stationary~\cite{zhou2013heteroscedasticity,Zhou2014}, or satisfying a time-varying autoregressive (AR) model~\cite{Hallin:1978}, for example. It is interesting that one of major driving forces in the recent surge in nonstationary time series analysis lies in modelling such series via various evolutionary spectral decompositions of the underlying covariance structure, which shares a similar flavour to that of TF analysis. That is, one seeks to model general classes of nonstationary time series by various time-varying Fourier or wavelet representations of the covariance. See for instance~\cite{Priestley1965} and~\cite{Dahlhaus:1997} for an evolutionary Fourier decomposition or Cramer representation approach, and~\cite{Nason2000} for a method based on time-varying wavelet spectra. Among others, \cite{Adak1998} and~\cite{Ombao2001} provide contributions in estimating algorithms.


		\subsection{Relationship with nonlocal mean}

		Another interpretation of the approximation \eqref{sst discretization} for SST at time $0$ and frequency $\xi$
		is the {\em nonlocal mean} \cite{buades2005non}. To understand this relationship, we view ${\Omega_l}$
		as a {\em frequency feature} designed from the STFT coefficient at frequency $\eta_l$. The main purpose of this feature is in designing a {\em metric} comparing frequency information $V_l$ and $\xi$; that is, 
		$d(\xi,V_l):=\big|\xi-{\Omega_l} \big|$. 
		Clearly, the closer $\xi$ and $V_l$ are in the sense of $d(\xi,V_l)$, the more weight will be given to the ``mean'' process in \eqref{sst discretization}.
		The mean is obviously kernel weighted with the bandwidth $\alpha$, but since those $V_l$'s close to $\xi$ might not be contiguous, the mean is nonlocal. In other words, SST is functioning like a nonlocal mean on a metric space with the metric depending on the phase information of the signal. 

	\subsection{Multitaper SST}

		Consider two windows as an example,
		$h_0(x) 
			=
				\pi^{-1/4} e^{-x^2/2}\,, 
			h_1(x) 
			=
				\sqrt{2}\pi^{-1/4} x e^{-x^2/2}$,
				the first two Hermite functions, orthonormal with respect to the standard $L^2$ inner product. For the complex Gaussian vector $\big[V_{f+\Phi}^{(h_0)}(t,\eta)\ V_{f+\Phi}^{(h_1)}(t,\eta)\big]^\top$, when $\Phi$ is white we have 
		$
			C
			=
			\Gamma 
			= 
			\begin{bmatrix}
				1 & \sqrt{2}/2\\
				\sqrt{2}/2 & 3/8
			\end{bmatrix}
		$ 
		since $\hat{h}_0$ and $\hat{h}_1$ are both real-valued functions and the expressions appearing in the pseudocovariance calculation reduce to precisely the same calculations of the covariance calculation. Hence $C$ and $\Gamma$ can be clearly simultaneously diagonalized. This ensures that the complex Gaussian vector $\big[V_{f+\Phi}^{(h_0)}(t,\eta)\ V_{f+\Phi}^{(h_1)}(t,\eta)\big]^\top$ has independent components. A similar argument show the independence of $\big[\partial_tV_{f+\Phi}^{(h_0)}\ \partial_tV_{f+\Phi}^{(h_1)}\big]^\top$, and hence the independence of $S_{f+\Phi}^{(h_0,\alpha)}(t,\xi)$ and $S_{f+\Phi}^{(h_1,\alpha)}(t,\xi)$. This  provides a theoretical justification of applying the multitaper approach proposed in \cite{Xiao_Flandrin:2007}. 

	\section{Some useful lemmas} \label{Section:Collection of Useful Lemmas}

		If $g:\mathbb{C}^n\to \mathbb{C}^n$ is a bijective real-analytic mapping, we may find the complex Jacobian of $g^{-1}$ by a direct $\CC\RR$-calculus computation~\cite{Kreutz-Delgado-09}, which we expound below.
				
		\begin{lemma}[Jacobian for ratio distribution]
			\label{lem: Jacobian for ratio distribution}
			Let $g: (\CC\backslash \{0\})\times \CC \to (\CC\backslash \{0\})\times \CC$ be defined by $g_Q(z_1,z_2) = (z_1, z_2/z_1)$. Then the complex Jacobian determinant of $g_Q^{-1}(z,q)$ is $|z|^2$.
		\end{lemma}
		
		\begin{proof}
			Writing $g_Q^{-1}(z,q) =  (g^{-1}_{Q,1}(z,q), g^{-1}_{Q,2}(z,q)) = (z,zq)$, by~\cite{Kreutz-Delgado-09} we compute the complex Jacobian matrix at $(z,q)$ to be
			\begin{equation}\label{Expansion complex Jacobian matrix Proof1}
				\begin{bmatrix}
					\partial g^{-1}_{Q,1}/\partial z & \partial g^{-1}_{Q,1}/\partial q & \partial g^{-1}_{Q,1}/\partial \bar{z} & \partial g^{-1}_{Q,1}/\partial \bar{q} \\
					\partial g^{-1}_{Q,2}/\partial z & \partial g^{-1}_{Q,2}/\partial q & \partial g^{-1}_{Q,2}/\partial \bar{z} & \partial g^{-1}_{Q,2}/\partial \bar{q} \\
					\partial \overline{g^{-1}_{Q,1}} /\partial z & \partial \overline{g^{-1}_{Q,1}} /\partial q & \partial \overline{g^{-1}_{Q,1}}/\partial \bar{z} & \partial \overline{g^{-1}_{Q,1}}/ \partial \bar{q} \\
					\partial \overline{g^{-1}_{Q,2}}/ \partial z & \partial \overline{g^{-1}_{Q,2}}/ \partial q & \partial \overline{g^{-1}_{Q,2}}/\partial \bar{z} & \partial \overline{g^{-1}_{Q,2}}/ \partial \bar{q}
				\end{bmatrix} 
				=
				\begin{bmatrix}
					1 & q & 0 & 0 \\
					0 & z & 0 & 0\\
					0 & 0 & 1 & \bar{q}\\
					0 & 0 & 0 & \bar{z}
				\end{bmatrix}.
			\end{equation}
			Taking determinants, we obtain the desired result.
		\end{proof}
		\begin{remark}
			In this case it happens that the usual relationship $\det{J_\RR} = |\det{J_\CC}|^2$ between real and complex Jacobians, $J_\RR$ and $J_\CC$ (see~\cite{Cross-08}), still holds. However, since $g_Q$ and its inverse are not holomorphic, we cannot apply this usual relationship.
		\end{remark}
			
		\begin{lemma}[Jacobian for synchrosqueezing integrand]
			\label{Lemma:ChangeVariableJacobianC2}
			Fix {$\eta>0$,} $\alpha>0$ and $\xi>0$. Consider a complex change of variables $g_Y:\CC^2\setminus\{(0,z_2):z_2\in\CC\}\to \CC^2\setminus\{(0,z_2):z_2\in\CC\}$ defined by
			\begin{equation}
				g_Y(z_1,z_2) 
				= \left( 
				\frac{z_1}{\sqrt{\pi\alpha}} 
				e^{-\frac{1}{\alpha} 
				\left|\xi{-\eta+}\frac{z_2}{2\pi i z_1}\right|^2},
				 \frac{z_2}{2\pi i z_1}
				 \right)\,.
			\end{equation}
			By defining 
			\begin{equation}
			E_{\alpha,\xi}(\omega) \vcentcolon=\sqrt{\pi\alpha} e^{ |\xi -\omega| ^2/\alpha}\,, \label{Definition:Ealphaxi}
			\end{equation}
			the Jacobian of $g^{-1}_Y(y,\omega)$ is 
			\begin{equation}
				\det{J(y,\omega)} 
				= 
				4\pi^2 |y|^2E^4_{\alpha,\xi}(\omega)\,.
			\end{equation}
		\end{lemma}
		
		\begin{proof}
			By a direct verification, $g_Y$ is a bijective mapping from $\CC^2\setminus\{(0,z_2):z_2\in\CC\}$ to itself, with inverse $g_Y^{-1}(y,\omega)=(g^{-1}_{Y,1}(y,\omega),g^{-1}_{Y,2}(y,\omega))$ defined by
			\begin{equation}
				\begin{aligned}
					{z_1}&=g^{-1}_{Y,1}(y,\omega) 
					= 
					\pi\sqrt\alpha
					e^{| \xi -\omega|^2/\alpha}y\label{Definition:g1invg2inv}\\
					{z_2}&=g^{-1}_{Y,2}(y,\omega) 
					= 
					2\pi^2 i \sqrt\alpha
					e^{| \xi -\omega|^2/\alpha}{(\eta-\omega)} y \,.
				\end{aligned}
			\end{equation}
			The complex Jacobian has the same form as \eqref{Expansion complex Jacobian matrix Proof1}, but with $g_{Q,1}$ and $g_{Q,2}$ replaced by $g_{Y,1}$ and $g_{Y,2}$ respectively. After expansion, we obtain 
			\begin{align}
				J(y,\omega)
				=
				E_{\alpha,\xi}(\omega)M_{\alpha,\xi}(y,\omega)\,,
			\end{align}
			where $M_{\alpha,\xi}(y,\omega)$ is 
			\begin{equation}\label{Definition:Malphaxi}
				\begin{bmatrix}
				1 & c_1 y(\bar{\omega}-\xi) 
				& 0 & c_1 y(\omega-\xi)\\
				 c_2{(\eta-\omega)} &   c_2y(-1+c_1{(\eta-\omega)}(\bar{\omega}-\xi)) & 0 & c_1 c_2 y {(\eta-\omega)} (\omega-\xi)\\
				0 &c_1 \bar{y}(\bar{\omega}-\xi) & 1 & c_1 \bar{y}(\omega-\xi)\\
				0 &  -c_1 c_2 \bar{y}{(\eta-\bar{\omega})}(\bar{\omega}-\xi) &   -c_2{(\eta-\bar{\omega})} &
				 -c_2\bar{y}(-1 +c_1  {(\eta-\bar{\omega})}(\omega-\xi))  
				\end{bmatrix}\,,
			\end{equation} 
			$c_1\vcentcolon=1/\alpha$ and $c_2=2\pi i$.
			We directly compute
			$\det{M_{\alpha,\xi}(y,\omega)} 
				= 4\pi^2 |y|^2$,
			which does not depend on $\alpha$, $\xi$ and $\omega$. We then have the conclusion that
			$\det{J(y,\omega)} 
				= 
				4\pi^2 |y|^2E^4_{\alpha,\xi}(\omega)$.
			
		\end{proof}
		\begin{remark}
		Note that $g_Y$ is not a holomorphic function of its second argument, but a real-analytic diffeomorphism from the underlying set $\RR^4\setminus\{(0,0,x_3,x_4): x_3,x_4 \in \RR\}$ to itself under the usual identification of $z_1$ and $z_2$ with $x_1+ix_2$ and $x_3+ix_4$, respectively. 
		\end{remark}
		
		When we evaluate the covariance structure in Theorem \ref{Theorem covariance of Yeta and Yetap} and other perturbation arguments, like the $M$-dependent approximation in Lemma \ref{Main Lemma 2 for Kernel SST}, we need to find the joint density of random variables $Y_{f+\Phi}^{(h,\alpha,\xi)}(t,\eta)$ and $Y_{f+\Phi}^{(h,\alpha,\xi)}(t,\eta')$ defined in \eqref{Yt,eta) definition}. To achieve this goal, we need the following lemma whose proof follows the same line as Lemma \ref{Lemma:ChangeVariableJacobianC2} but with a more tedious calculation. For the sake of completeness, we provide details below.

		\begin{lemma}\label{Lemma:ChangeVariableJacobianC4}
			Fix $\eta_1,\eta_2>0$, $\alpha>0$ and $\xi_1,\xi_2>0$. Consider a complex change of variables $g_{Y_1Y_2}:\CC^4\setminus\{z_1z_3=0\}\to \CC^4\setminus\{z_1z_3=0\}$ defined by
			\begin{equation*}
				g_{Y_1Y_2}(z_1,z_2,z_3,z_4) 
				= 
				\left( 
				\frac{z_1}{\sqrt{\pi\alpha}}  
				e^{-\frac{1}{\alpha} 
					\left|
						\xi_1-{\eta_1+}\frac{z_2}{2\pi i z_1}
					\right|^2
					},
				\frac{z_2}{2\pi i z_1},
				\frac{z_3}{\sqrt{\pi\alpha}}  
				e^{-\frac{1}{\alpha} 
					\left|
						\xi_2-{\eta_2+}\frac{z_4}{2\pi i z_3}
					\right|^2
					},
				\frac{z_4}{2\pi i z_3}
				\right)\,.
			\end{equation*}
			The complex Jacobian determinant of $g_{Y_1Y_2}$ is 
			\begin{equation}
				\det{J(y_1,\omega_1,y_2,\omega_2)} 
				= 
				16\pi^4 |y_1|^2 |y_2|^2
				E^4_{\alpha,\xi_1}(\omega_1)
				E^4_{\alpha,\xi_2}(\omega_2)\,,
			\end{equation}
			where $E_{\alpha,\xi_i}$ is defined in \eqref{Definition:Ealphaxi}.
		\end{lemma}
		
		\begin{proof}
			We see that $g_{Y_1Y_2}$ is bijective from $\CC^4\setminus\{z_1z_3=0\}$ to itself with inverse given by 
			\begin{equation}
				g_{Y_1Y_2}^{-1}(y_1,\omega_1,y_2,\omega_2)
				=
				(
				g^{-1}_{Y_1,1}(y_1,\omega_1),
				g^{-1}_{Y_1,2}(y_1,\omega_1),
				g^{-1}_{Y_2,1}(y_2,\omega_2),
				g^{-1}_{Y_2,2}(y_2,\omega_2)
				)\,, \nonumber
			\end{equation}
			where $g_{Y_j,1}$ and $g_{Y_j,2}$, $j=1,2$, are defined similarly to \eqref{Definition:g1invg2inv}.
			The complex Jacobian of $g_{Y_1Y_2}^{-1}$ is an $8\times 8$ block-diagonal matrix whose main diagonal blocks are $4\times 4$ matrices of the form $E_{\alpha,\xi_j}(\omega_j) M_{\alpha,\xi_j}(y_j,\omega_j)$ for $j=1,2$, where $M_{\alpha,\xi}$ is defined in~\eqref{Definition:Malphaxi}; that is the complex Jacobian of $g_{Y_1Y_2}^{-1}$ becomes  
			\begin{align}
				J(y_1,\omega_1,y_2,\omega_2) 
				=
				\begin{bmatrix}
				E_{\alpha,\xi_1}(\omega_1)M_{\alpha,\xi_1}(y_1,\omega_1) & 0\\ 0 & E_{\alpha,\xi_2}(\omega_2)M_{\alpha,\xi_2}(y_2,\omega_2) 
				\end{bmatrix},\nonumber
			\end{align}
			and so we have
			$
				\det{J(y_1,\omega_1,y_2,\omega_2)} 
				= 
				16\pi^4 |y_1|^2|y_2|^2
				E^4_{\alpha,\xi_1}(\omega_1)
				E^4_{\alpha,\xi_2}(\omega_2)
			$.
		\end{proof}

		Below, we quantify the key quantity we encounter when we analyze SST, the confluent hypergeometric function.
		\begin{lemma}\label{Lemma: 1F1 bound}
		For $k\in \mathbb{N}$, there exist $m\in(0,1)$ so that
		\begin{equation}
			m\max\{1,\,e^{x^2}x^{k+3}\}\leq \Hypergeometric{1}{1}{\frac{k}{2}+2}{\frac{1}{2}}{x^2}\leq m^{-1}\max\{1,\,e^{x^2}x^{k+3}\}
			\end{equation}
			when $x\in \mathbb{R}^+$.
		\end{lemma}
		
		\begin{proof}
		By \cite[p.323 (13.2.13)]{NIST}, we have 
			\begin{equation}
			\Hypergeometric{1}{1}{\frac{k}{2}+2}{\frac{1}{2}}{x^2}=1+O(x^2)
			\end{equation} 
			when $x\to 0$ for all $k\in\mathbb{N}$. We also have the following asymptotical approximation of the confluent hypergeometric function \cite[p.328 (13.7(i))]{NIST}:
			\begin{equation}
			\Hypergeometric{1}{1}{\frac{k}{2}+2}{\frac{1}{2}}{x^2}=e^{x^2}x^{k+3}(1+O(x^{-2}))
			\end{equation}
			when $x\to\infty$ for any $k\in \NN$. 
			Moreover, by \cite[p.331 (13.9.2)]{NIST}, 
			\[
			\Hypergeometric{1}{1}{\frac{k}{2}+2}{\frac{1}{2}}{x^2}>0
			\] 
			for all $x\geq 0$. Hence, by the control of $\Hypergeometric{1}{1}{\frac{k}{2}+2}{\frac{1}{2}}{x^2}$ when $x\to 0$ and $x\to \infty$ and the smoothness of $\Hypergeometric{1}{1}{\frac{k}{2}+2}{\frac{1}{2}}{x^2}$, 
			we see that 
			\[
			m\leq\frac{\Hypergeometric{1}{1}{\frac{k}{2}+2}{\frac{1}{2}}{x^2}}{\max\{1,\,e^{x^2}x^{k+3}\}}\leq m^{-1}
			\] 
			for all $x>0$ for some universal constant $0<m<1$. This concludes the proof. 
			\end{proof}

\section{Proofs for Section~3  about complex Gaussian quotient}

	\subsection{Proof of Theorem~\ref{thm: Quotient density}} \label{Proof Theorem 2.1}

		\begin{proof}[Proof of Theorem \ref{thm: Quotient density}]
			Let $
			\underline{\Sigma} 
			= 
			\begin{bmatrix}
				\Gamma & C\\
				\conj{C} & \conj{\Gamma}
			\end{bmatrix} 
			\in \CC^{4\times 4}
			$, which is invertible by the assumption of nondegeneracy.
			The joint density of $(Z_1,Q)$, where $Q = Z_2/Z_1$, is evaluated from \eqref{eq: complex gaussian density} and Lemma \ref{lem: Jacobian for ratio distribution} by changing variables via $g_Q$:
			\begin{equation} \label{Z_1, Q joint density}
				f_{Z_1,Q}(z,q) 
				= 
				|z|^2 f_\B{Z}(g_Q^{-1}(z,q))
				= 
				\frac{|z|^2 e^{-\frac{1}{2} \underline{\mu}^* \underline{\Sigma}^{-1} \underline{\mu}} }{\pi^2 \sqrt{\det{\underline{\Sigma}}}}
				e^{-\frac{1}{2} \underline{\B{g}}^* \underline{\Sigma}^{-1} \underline{\B{g}} + 
				\Re{(\underline{\mu}^* \underline{\Sigma}^{-1} \underline{\B{g}}) }},
			\end{equation}
			$\B{g} \vcentcolon= z \B{q}\in \CC^2$. The density of $Q$ is then obtained by evaluating the marginal distribution of $(Z_1,Q)$. Integrating over $z$ in the polar form $z=re^{i
			\theta}$, where $r>0$ and $\theta\in [0,2\pi)$,
			gives us
			\begin{equation}
				f_Q(q) 
				= 
				\frac{ e^{-\frac{1}{2} \underline{\mu}^* \underline{\Sigma}^{-1} \underline{\mu}} }{\pi^2 \sqrt{\det{\underline{\Sigma}}}}
				\int_0^{2\pi}\!\!\! \int_0^\infty r^3
				e^{-\frac{1}{2} r^2 \underline{\B{e}}^* \underline{\Sigma}^{-1} \underline{\B{e}} + 
				r \Re{(\underline{\mu}^* \underline{\Sigma}^{-1} \underline{\B{e}})}}\,dr\,d\theta\,,
			\end{equation}
			where 
			\[
			\B{e} \vcentcolon= e^{i\theta} \B{q}.
			\] 
			 Using the fact that $\conj{C} = C^*$ by complex symmetry, it follows that $\det{\underline{\Sigma}} = \det{\Gamma} \det{ P}$. Also recall by block matrix inversion:
			\begin{equation}
				\underline{\Sigma}^{-1} 
				= 
				\begin{bmatrix}
					\conj{ P^{-1}} & - P^{-1}\conj{ R} \\
					- R^\top \conj{ P^{-1}} &  P^{-1}
				\end{bmatrix}.
			\end{equation}
			Due to the appearance of conjugation in $\underline{\B{e}}$, the phase $e^{i\theta}$ plays an essential role, which leads to the following quantities:
			\begin{align}
				\label{Expansion Sigma zetai Lemma1}
				\underline{\B{e}}^* \underline{\Sigma}^{-1} \underline{\B{e}} 
				&=2 A(\theta,q)\in\RR\\
				\underline{\mu}^* \underline{\Sigma}^{-1} \underline{\B{e}}
				&= 2 B_\mu(\theta, q)\in\RR\,.\nonumber
			\end{align}
			Recall that due to the nondegeneracy of $\underline{\Sigma}$, $\det{\underline{\Sigma}} = \det{\Gamma} \det{ P}$ implies that $ P$ and $\Gamma$ are both invertible. 
			Also, note that $ P$ is Hermitian, so $\B{q}^* \overline{ P^{-1}} \B{q}$ is real. As a result, $\B{q}^* \overline{ P^{-1}} \B{q}$ and $\B{q}^\top { P^{-1}}\conj{\B{q}}$ in the expansion of $\underline{\B{e}}^* \underline{\Sigma}^{-1} \underline{\B{e}}$ are equivalent, and this fact leads to the first equality in \eqref{Expansion Sigma zetai Lemma1}.

			By hypothesis, $ \underline{\Sigma}^{-1} $ is a positive-definite Hermitian, so $A(\theta,q)>0$ and the substitution $r = t/\sqrt{A(\theta,q)}$ is permissible. We then recognize the Hermite function $H_{-4}$ with the help of equation~\eqref{eq:Hermite function of negative order integral representation}:
			\begin{align*}
				f_{Q}(q) 
				&= 
				\frac{ e^{-\frac{1}{2}\underline{\mu}^* \underline{\Sigma}^{-1} \underline{\mu}}}{\pi^2 \sqrt{\det{\Gamma} \det{ P}}}
				\int_0^{2\pi} 
				\frac{1}{A(\theta,q)^2}
				\int_0^\infty
				 t^3
				\exp{\left(
				   -t^2 + 2 t \frac{B_\mu(\theta,q)}
				   {\sqrt{A(\theta,q)}}
				   \right)
				}
				\,dt\,d\theta\\
				&=\frac{ e^{-\frac{1}{2}\underline{\mu}^* \underline{\Sigma}^{-1} \underline{\mu}} }{\pi^2 \sqrt{\det{\Gamma} \det{ P}}}
				\int_0^{2\pi}
				\frac{6}{A(\theta,q)^2}
				H_{-4}\left( \frac{-B_\mu(\theta,q)}
				   {\sqrt{A(\theta,q)}}
				   \right)
				\,d\theta.
			\end{align*}
			Finally, observe that $A(\theta+\pi, q) = A(\theta, q)$ and $B_\mu(\theta+\pi,q) = -B_\mu(\theta,q)$, which lets us break up the domain of integration and apply~\eqref{eq: H_nu difference} to obtain equation~\eqref{eq: Gaussian quotient density}. 

			When $\mu = 0$, the hypergeometric function in the integrand of~\eqref{eq: Gaussian quotient density} reduces to $\Hypergeometric{1}{1}{2}{1/2}{0} = 1/12$. When $C=0$, the function $A(\theta,q)$ becomes $\B{q}^* \Gamma^{-1} \B{q}$. As a result,  when $\mu=0$ and $C=0$, we see that $f_{Q^\circ}(z) = (\pi \det{\Gamma})^{-1} (\B{q}^* \Gamma^{-1} \B{q} )^{-2}$. 
		\end{proof}

	\subsection{Proof of Theorem \ref{thm: quotient mean bound proposition}}\label{Proof Proposition 2.3}

		We need the following two lemmas to finish the proof. Recall that we have $f_{Q^\circ}(q) = (\pi \det{\Gamma} )^{-1} (\B{q}^* \Gamma^{-1} \B{q})^{-2}$.
		
		\begin{lemma}\label{Proof quotient Lemma 1 symmetry}
			Follow the notation in Theorem \ref{thm: Quotient density}. In general $f_{Q^\circ}$ satisfies $f_{Q^\circ}(q) = f_{Q^\circ}(\conj{q})$. Moreover, when $\Gamma$ is diagonal, $f_{Q^\circ}(q) = f_{Q^\circ}(-q)$.
		\end{lemma}

		\begin{proof}
			By a direct expansion, we have the symmetry $f_{Q^\circ}(q) = f_{Q^\circ}(\conj{q})$. Denote $g_{ij}$ to be the $(i,j)$-th entry of $\Gamma^{-1}$ for $i,j = 1,2$. For $\B{q}=(1,q)\in \CC^2$, where $q\in \CC$, we have 
			\begin{align*}
				\B{q}^* \Gamma^{-1} \B{q} 
				= g_{11} 
				+ (g_{12}q+g_{21} \conj{q})
				+ g_{22} |q|^2
				= g_{11} 
				+ 2 \Re{(g_{12}q)}
				+ g_{22} |q|^2,
			\end{align*}
			and if $\Gamma^{-1}$ is diagonal, then $\Re{(g_{12}q)}=0$. Thus, when $\Gamma^{-1}$ is diagonal, we obtain another symmetry $f_{Q^\circ}(q) = f_{Q^\circ}(-q)$. 
		\end{proof}

		\begin{lemma}\label{lem: circular density squeezing lemma}
			Suppose $Q = Z_2/Z_1$, where $\B{Z} = (Z_1,Z_2) \sim \CC N_2(\mu,\Gamma,C)$. Let $b := \|\Gamma^{-1/2} \mu \|$. Denote $\mathsf M \vcentcolon=  P^{-1/2} C^* \, \Gamma^{-1/2}$ and
				$\mathsf N \vcentcolon= \conj{\Gamma^{-1/2} C  P^{-1}} \Gamma^{1/2}$. We have the following lower bound control of $f_Q$:
			\begin{equation}
			\label{quotient density squeezing inequality}
				\frac{12 e^{-(1+\|\mathsf M\|^2+\|\mathsf N\|) b^2} H_{-4}(\sqrt{1+\|\mathsf M\|^2+\|\mathsf N\|} b)}{(1+\|\mathsf M\|^2+\|\mathsf N\|)^2 \sqrt{\det{\Gamma^{-1} P}}}  f_{Q^\circ}(q)
				\le f_Q(q)\,.
			\end{equation}
		\end{lemma}

		\begin{proof}
			The bound comes from a straightforward expansion. By a direct expansion of~\eqref{eq: complex gaussian density}, we have
			\begin{align*}
					f_{\B{Z}}(z) 
						&\,=
						\frac{1}{\pi^{2} \sqrt{\det{\underline{\Sigma}}}}
						e^{-[(\underline{z}-\underline{\mu})^* \conj{ P}^{-1} (\underline{z}-\underline{\mu})- \Re(\underline{z}-\underline{\mu})^\top { R}^{\top}\conj{ P^{-1}} (\underline{z}-\underline{\mu}) ]}\,,
				\end{align*}	
			where $z\in\CC^2$. Let $u = \Gamma^{-1/2} (z - \mu)$. Note that by the definition of ${P}$, we have $\Gamma^{1/2}\conj{ P^{-1}}\Gamma^{1/2}=I+\mathsf M^*\mathsf M$. Thus, we have 
			\begin{align*}
				-(z-\mu)^*\conj{ P^{-1}}(z-\mu) + \Re{[(z-\mu)^\top  R^\top \conj{ P^{-1}} (z-\mu)]}
				= -\|u\|^2 - \|\mathsf M u\|^2 + \Re{(u^\top \mathsf N u)}.
			\end{align*}
			Combining Cauchy-Schwarz with the inequality $-|w| \le \Re{(w)} \le |w|$ and submultiplicativity of matrix norms, we have
			\begin{equation*}
				-(\|\mathsf M\|^2+\|\mathsf N\|)\|u\|^2
				\le -\|\mathsf M u\|^2 + \Re{(u^\top \mathsf N u)} .
			\end{equation*}
			Adding $-\|u\|^2$ throughout, exponentiating, and adjusting constants, this gives us
			\begin{equation*}
				c e^{-(1+\|\mathsf M\|^2+\|\mathsf N\|) \|u\|^2}
				\le f_{\B{Z}}(z)\,,			\end{equation*}
			where $c = (\pi^{2} \sqrt{ \det{\Gamma P}} )^{-1}$. On the other hand, by the triangle inequality we also have
			\begin{equation*}
				\|u\|^2 
				\le (\|\Gamma^{-1/2} z\| + \|\Gamma^{-1/2}\mu \|)^2\,,
			\end{equation*}
			so we may apply this to the previous string of inequalities to obtain
			\begin{equation*}
				c e^{-(1+\|\mathsf M\|^2+\|\mathsf N\|) (\|\Gamma^{-1/2} z\| + \|\Gamma^{-1/2}\mu \|)^2}
				\le f_{\B{Z}}(z)\,.
			\end{equation*}
			We may now change variables as in the proof of Theorem~\ref{thm: Quotient density} and integrate to get
			\begin{equation*}
				c \int_0^{2\pi} \int_0^\infty r^3 e^{-(1+\|\mathsf M\|^2+\|\mathsf N\|) (ar + b)^2} \,dr\,d\theta
				\le f_Q(q)\,,
				\end{equation*}
			where for brevity we have set $a = \|\Gamma^{-1/2} \B{e}\|$ (with $\B{e}$ as in the proof of Theorem~\ref{thm: Quotient density}) and $b = \|\Gamma^{-1/2} \mu \|$. By expanding out the exponent and performing a substitution for $r$, this gives us
			\begin{equation*}
				6c \int_0^{2\pi} \frac{e^{-(1+\|\mathsf M\|^2+\|\mathsf N\|) b^2}}{(1+\|\mathsf M\|^2+\|\mathsf N\|)^2 a^4} H_{-4}(\sqrt{1+\|\mathsf M\|^2+\|\mathsf N\|} b) \,d\theta
				\le f_Q(q)\,.
				\end{equation*}
			Finally, since $a^2 = \B{e}^* \Gamma^{-1} \B{e} = \B{q}^* \Gamma^{-1} \B{q}$ does not depend on $\theta$, this reduces to the desired inequality.
		\end{proof}

			We may now prove the desired theorem.

			\begin{proof}[Proof of Theorem~\ref{thm: quotient mean bound proposition}]
				 We will finish the proof by considering various situations.\newline\newline
			\underline{\textbf{When $C=0$\,:}}
				We start with expanding the density 
				\[
				f_{Q^\circ}(q) 
				= \frac{1}{\pi \det{\Gamma}(\B{q}^* \Gamma^{-1} \B{q})^{-2}} 
				= \frac{1}{\pi \det{\Gamma} (g_{11} + 2 \Re{(g_{12} q)} + g_{22}|q|^2)^2},
				\]
				where $g_{ij}$ denotes the $(i,j)$-th entry of $\Gamma^{-1}$ for $i,j = 1,2$. Since $\Gamma$ is assumed to be Hermitian and positive definite, we have $\det{\Gamma^{-1}} = g_{11} g_{22} - |g_{12}|^2 > 0$. It follows that neither $g_{11}$ nor $g_{22}$ is zero and that $g_{11}g_{22} > |g_{12}|^2$.  
				Moreover, since $\Gamma^{-1}$ is a Hermitian, positive-definite matrix, we are assured that $g_{11}, g_{22} > 0$. 
				Writing $g_{12} = |g_{12}|e^{i\tau}$ for some $\tau \in [0,2\pi)$, and converting the above density to polar form by setting $q = re^{i\theta}$, we notice that 
				\[
				2\Re{(g_{12} q)} = 2\Re{(|g_{12}|r e^{i(\tau+\theta)})}
				= 2|g_{12}|r \cos{(\tau+\theta)}.
				\]
				Since $\B{q}^* \Gamma^{-1} \B{q}$ is positive, we know $g_{11} + 2|g_{12}|r \cos{\sigma} + g_{22}r^2>0$.

				With the above preparation, we now show the first claim. 
				For $\beta\geq 0$, in the polar coordinate we have 
				\begin{align*}
					\EE[ |Q^\circ|^{\beta} ]
					= \frac{1}{\pi \det{\Gamma} }
					\int_0^{2\pi}\!\!\!  \int_0^\infty \frac{r^{1+\beta}}{(g_{11} + 2|g_{12}|r \cos{(\theta + \tau)} + g_{22}r^2)^2} \,dr\, d\theta.
				\end{align*}
				By the reverse triangle inequality, we have 
				\begin{align}
					\EE[ |Q^\circ|^\beta ]
					&\le 
					\frac{1}{\pi \det{\Gamma} }
					\int_0^{2\pi}\!\!\!  \int_0^\infty 
					\frac{r^{1+\beta}}
					{(g_{11} - 2|g_{12}|r + g_{22}r^2)
					^2} \,dr\, d\theta\nonumber\\
					&=
					\frac{2}{\det{\Gamma}} 
					\int_0^\infty 
					\frac{r^{1+\beta}}{(g_{11}-2|g_{12}|r + g_{22} r^2)^2} \,dr\,.
				\end{align}
				Note that since $|g_{12}|<g_{11}^{1/2}g_{22}^{1/2}$ due to the positive definiteness of $\Gamma$, we have $g_{11}-2|g_{12}|r + g_{22} r^2>(rg_{22}^{1/2}-g_{11}^{1/2})^2\geq 0$. Since $g_{22}>0$, 
				we have $\EE |Q^\circ|^\beta < \infty$ when $\beta<2$. Since $|\EE (Q^\circ)^\beta|\leq \EE[ |Q^\circ|^\beta]$, we conclude that $|\EE (Q^\circ)^\beta|<\infty$ when $0\leq \beta<2$.
				When $\beta=2$, a similar argument gives us
				\begin{equation*}
					\EE [(Q^\circ)^2] 
					= \frac{1}{\pi \det{\Gamma} }
					\int_0^{2\pi} \!\!\! \int_0^\infty \frac{r^3e^{i2\theta}}{(g_{11}+ 2|g_{12}|r \cos{(\theta + \tau)} + g_{22}r^2)^2} \,dr d\theta,
				\end{equation*}
				but the $r$ integral diverges by the $p$-test.

				For the second claim, applying Fubini's theorem, we change the order of integration and make the substitution $\sigma = \theta + \tau$ to obtain
				\begin{align*}
					\EE Q^\circ &=\frac{1}{\pi \det{\Gamma} }
					\int_0^{2\pi} \!\!\! \int_0^\infty \frac{r^{2}e^{i\theta}}{(g_{11} + 2|g_{12}|r \cos{(\theta + \tau)} + g_{22}r^2)^2} \,dr\, d\theta\\
					&= \frac{1}{\pi \det{\Gamma} }
					\int_0^\infty \!\!\! \int_{\tau}^{2\pi+\tau} \frac{r^2 e^{i(\sigma-\tau)}}{(g_{11} + 2|g_{12}|r \cos{\sigma} + g_{22}r^2)^2} \,d\sigma dr\,.
				\end{align*}
				Now, the integrand is a periodic function of $\sigma$ that is being integrated over its full period, so the result is the same if we integrate over any other interval of length $2\pi$. We choose
				$(0,2\pi)$ for this purpose to obtain
				\begin{equation}\label{E Q circ intermediate}
					\EE Q^\circ  
					= 
					\frac{e^{-i \tau}}{\pi \det{\Gamma} }
					\int_0^\infty r^2 \int_{0}^{2\pi} \frac{e^{i\sigma}}{(g_{11} + 2|g_{12}|r \cos{\sigma} + g_{22}r^2)^2} \,d\sigma dr.
				\end{equation}
				We now digress for a moment to point out that the contour integral
				\begin{equation} \label{contour int 1}
					-4 i \int_\gamma \frac{z^2}{(b z^2 + 2a z + b)^2} \, dz,
				\end{equation}
				where $\gamma$ is the unit circle oriented counterclockwise, may be reparametrized by letting $z = e^{i\sigma}$. This then implies that $2 \cos{\sigma} = z+z^{-1}$ by Euler's identity, and if we also let $a = g_{11} + g_{22}r^2$ and $b = 2|g_{12}|r$, our reparametrized contour integral is precisely the innermost integral in~\eqref{E Q circ intermediate}. So it suffices to determine the value of~\eqref{contour int 1}, which we do now.

				If $b=0$, expression~\eqref{contour int 1} is seen to be the integral of a function with a removeable singularity at the origin over a smooth closed contour; by Cauchy's theorem, this integral is then zero.

				Otherwise, $b>0$. In this case the function $p(z) = z^2 + 2(a/b) z + 1$ satisfies $p(0) \neq 0$ and $p(z) = z^2 p(1/z)$, which shows that if $z_0$ is a root of $p(z)$, then so is $1/z_0$. Moreover, $p(\bar{z}) = \conj{p(z)}$, because $\Gamma^{-1}$ is Hermitian so its diagonal entries are real. Hence for some real $x$ in the unit disk, the residue theorem gives us
				\begin{equation}
					\int_\gamma \frac{z^2}{(b z^2 + 2a z + b)^2} \, dz
					=
					\frac{1}{b^2}\int_\gamma \frac{z^2}{(z-x)^2(z-1/x)^2} \, dz
					=
					\frac{-4\pi i x^3}{b^2 (x^2-1)^3}\,.\nonumber
				\end{equation}
				Now, the roots of $z^2 + 2(a/b) z + 1$ occur at 
				\begin{equation*}
					z_\pm 
					=
					-\frac{a}{b} \pm \sqrt{\frac{a^2}{b^2} - 1},
				\end{equation*}
				and the arithmetic-geometric mean inequality implies that
				\begin{equation*}
					\frac{a}{b} 
					= 
					\frac{g_{11}+g_{22}r^2}{2|g_{12}|r} 
					\ge \frac{2\sqrt{g_{11}g_{22}r^2}}{2|g_{12}|r} 
					= 
					\frac{\sqrt{g_{11}g_{22}}}{|g_{12}|} > 1\,,
				\end{equation*}
				where the final inequality follows from the fact that $\det{\Gamma^{-1}}>0$. So $z_+$ is the root in the unit disk, and plugging this in for $x$ above we have
				\begin{equation}
					\int_\gamma \frac{z^2}{(b z^2 + 2a z + b)^2} \, dz
					=
					\frac{-\pi b i }{2 (a^2-b^2)^{3/2}}.
				\end{equation}
				Consequently, equation~\eqref{E Q circ intermediate} reduces to
				\begin{align}
					\EE Q^\circ 
					&= \frac{-4 e^{-i \tau}|g_{12}|}{ \det{\Gamma} }
					\int_0^\infty \frac{r^3}{\left((g_{11}+g_{22}r^2)^2-4|g_{12}|r^2\right)^{3/2}} dr\nonumber\\
					&= \frac{-4 g_{21} }{ g_{22} \det{\Gamma} }
					\frac{1}{(4g_{11}g_{22}-4|g_{12}|^2)}
					= \frac{-g_{21} }{ g_{22} \det{\Gamma} \det{\Gamma^{-1}} }= \frac{-g_{21} }{ g_{22} }\,, \nonumber
				\end{align}
				where the $r$ integral can be done with the substitution of $u = r^2$ and the obvious trigonometric substitutions arising later.
								
				For the third claim, note that since $C=0$, we have ${R}=0$ and ${P}=\bar{\Gamma}$. Thus, 
				\begin{equation*}
				A(\theta,q)=A(q)=\textbf{q}^*\Gamma^{-1}\textbf{q} =g_{11}+2\Re[g_{12}q]+g_{22}|q|^2=g_{11}+g_{22}|q|^2
				\end{equation*}
				since $\Gamma$ is diagonal. Clearly, we have the symmetry that $A(q)=A(-q)$ under this condition. On the other hand,
				$B(\theta,q)=\Re[e^{i\theta}\mu^*\Gamma^{-1}\textbf{q}]$. By a direct expansion, since $\mu=(\mu_1,0)$, we have 
				\begin{equation}
				\mu^*\Gamma^{-1}\textbf{q}=\bar{\mu}_1(g_{11}+g_{12}q)=g_{11}\bar{\mu}\,,
				\end{equation}
				where the last equality comes from the assumption that $\Gamma$ is diagonal. Thus, $B(\theta,q)=B(\theta)=\Re[e^{i\theta}g_{11}\bar{\mu}_1]$. As a result, $f_Q(q)$ is reduced to
				\begin{align}\label{Equation of f_Q when C=0}
				f_Q(q) 
				&= 
					\frac{e^{-\frac{1}{2}\underline{\mu}^*\underline{\Sigma}^{-1} \underline{\mu} }}
					{\pi^2 \sqrt{\det{\underline{\Sigma}}}}
					\int_0^{\pi}
					\Hypergeometric{1}{1}{2}{\frac{1}{2}}{\frac{
					B_\mu(\theta)^2
					}{
					A(q)
					}}
					 \,d\theta\frac{1}{
					A(q)^2
					}\,,
			\end{align}
			and hence
			\begin{align*}
			\mathbb{E}Q=\frac{e^{-\frac{1}{2}\underline{\mu}^*\underline{\Sigma}^{-1} \underline{\mu} }}
					{\pi^2 \sqrt{\det{\underline{\Sigma}}}}
					\int_0^{\pi} \Big[\int_{\mathbb{C}}
					\Hypergeometric{1}{1}{2}{\frac{1}{2}}{\frac{
					B_\mu(\theta)^2
					}{
					A(q)
					}}
					\frac{q}{
					A(q)^2
					}
					 \,dq \Big] d\theta=0
			\end{align*}
			by the symmetry of $A(q)$. We thus finish the claim.
			\newline\newline
			\underline{\textbf{When $C\neq 0$\,:}} First, since $\EE |Q^\circ|^2$ diverges, by the lower bound of $f_Q$ by $f_{Q^\circ}$ shown in Lemma~\ref{lem: circular density squeezing lemma}, we know that $\EE |Q|^2$ and $\EE (Q^2)$ also blow up. When $\beta<2$, note by Lemma \ref{Lemma: 1F1 bound}, we have
			\begin{equation}
			\Hypergeometric{1}{1}{2}{\frac{1}{2}}{\frac{
					B_\mu(\theta,q)^2
					}{
					A(\theta,q)
					}}
					\leq m^{-1}\left\{\Big(\frac{B_\mu(\theta,q)^2}{A(\theta,q)}\Big)^{3/2}e^{\frac{B_\mu(\theta,q)^2}{A(\theta,q)}},1\right\}\,.
			\end{equation}
			Moreover, by the Cauchy-Schwartz inequality, we have $\frac{B_\mu(\theta,q)^2}{A(\theta,q)}\leq \underline{\mu}^*\underline{\Sigma}^{-1} \underline{\mu}$. Hence, we have
			\begin{equation}
			e^{\frac{B_\mu(\theta,q)^2}{A(\theta,q)}-\frac{1}{2}\underline{\mu}^*\underline{\Sigma}^{-1} \underline{\mu} }\leq 1.
			\end{equation}
			Denote the eigenvalues of $\underline{\Sigma}$ to be $\lambda_1\geq \lambda_2\geq \lambda_3\geq \lambda_4>0$ due to the positive definite assumption. We thus have $\underline{\mu}^*\underline{\Sigma}^{-1} \underline{\mu}\geq \lambda_1^{-1}\|\underline{\mu}\|^2$, and hence $e^{-\frac{1}{2}\underline{\mu}^*\underline{\Sigma}^{-1} \underline{\mu} }\leq e^{-\frac{1}{2}\lambda_1^{-1}\|\underline{\mu}\|^2}$. As a result, we have
			\begin{align}
			\EE |Q|^\beta\leq \frac{e^{-\frac{1}{2}\lambda_1^{-1}\|\underline{\mu}\|^2}}{\pi^2\lambda_4^2}\int_{\mathbb{C}}\int_0^\pi\frac{|q|^\beta}{A(\theta,q)^2}d\theta dq
			\end{align}
			By \eqref{Expansion Sigma zetai Lemma1}
			we have $A(\theta,q)=\frac{1}{2}\underline{\B{e}}^* \underline{\Sigma}^{-1} \underline{\B{e}}\geq \frac{1}{2}\|\underline{\B{e}}\|^2\lambda_1^{-1}$, which leads to 
			\begin{equation}
			\EE |Q|^\beta\leq e^{-\frac{1}{2}\lambda_1^{-1}\|\underline{\mu}\|^2}\frac{2\lambda_1^2}{\pi\lambda_4^2}\int_{\mathbb{C}}\frac{|q|^\beta}{(1+|q|^2)^2} dq<\infty\,,
			\end{equation}
	where we use the fact that $\|\underline{\B{e}}\|^2=2	\|\B{q}\|^2=2(1+|q|^2)$ and $\beta<2$. We thus finish the claim.
			
			\end{proof}

\section{Generalized Random Process} \label{sec:app: GRP summary}

	An ordinary random process is a family $(X_t)_{t\in T}$ of random variables defined on a common probability space $(\Omega,\mathcal{F},\PP)$ and taking values in a common measurable space; see for example~\cite{Koralov-07}.
	A generalized random process (GRP) is the extension of this idea to the distribution setting. In particular, let $\mathcal{M}(\Omega,\mathcal{X})$ denote the collection of random variables defined on 
	$(\Omega,\mathcal{F},\PP)$ and taking values in the measurable space $\mathcal{X}$. Then a linear function	
	$\Phi : \mathcal{S} \to \mathcal{M}(\Omega,\mathcal{X})$ that is
	continuous in the sense of finite-dimensional distributions is called a generalized random process (GRP); see~\cite{Gelfand-64}.
	Such a process is said to be wide-sense stationary (WSS) if there exists a function $m:\mathcal{S} \to \CC$ and positive-definite function 
	$B_\Phi:\mathcal{S}\times \mathcal{S} \to \CC$ satisfying 
	\begin{align*}
		m(\phi) 
		&= 
		\EE \Phi(\phi)\\
		B_\Phi(\phi,\psi)
		&=
		\mathbb{E}[
			(\Phi(\phi)-m(\phi))
			\overline{(\Phi(\psi)-m(\psi))}
			]
	\end{align*}

	If $\Phi$ is a wide-sense stationary generalized random process (GRP) with mean $m(\phi) = \EE \Phi(\phi)$, the covariance functional $B_\Phi$ of $\Phi$ is defined by
	\begin{equation}\label{covariance functional definition}
		B_\Phi(\phi,\psi)
		\vcentcolon= 
		\mathbb{E}[
			(\Phi(\phi)-m(\phi))
			\overline{(\Phi(\psi)-m(\psi))}
			] 
		=
		\cov{\Phi(\phi)}{\Phi(\psi)}
	\end{equation}
	for any test functions $\phi,\psi\in\mathcal{S}$. 
	For use later, we introduce the pseudocovariance as
	\begin{equation}\label{pseudocovariance functional definition}
		P_\Phi(\phi,\psi)
		\vcentcolon= 
		\mathbb{E}[(\Phi(\phi)-m(\phi)){(\Phi(\psi)-m(\psi))}] 
		= 
		B_\Phi(\phi,\overline{\psi})\,.
	\end{equation}
	Here, we use the linearity of $\Phi$ and the fact that $\Phi(\psi)-m(\psi)=\overline{\Phi(\overline{\psi})-m(\overline{\psi})}$.
	Notice that because of the complex conjugation, $B_\Phi$ is a sesquilinear form on test functions, and this form is Hermitian. 

	A GRP is called stationary if for any test functions $\phi_1, \ldots,\phi_n$, and any $h\in \CC$, the random vectors 
	$(\Phi(\phi_1 \circ \tau_h),\ldots,\Phi(\phi_n \circ \tau_h))$
	and
	$(\Phi(\phi_1),\ldots,\Phi(\phi_n))$ have the same distribution, where $\phi_k\circ\tau_h(t) = \phi_k(t+h)$ is translation by $h$.
	It turns out \cite{Gelfand-64} that for any stationary GRP $\Phi$, there exists a functional $B_0$ with the property that
	\begin{equation}
		B_\Phi(\phi,\psi)
		=
		B_0(\phi * \psi^\star),
	\end{equation}
	where $*$ denotes convolution and $\psi^\star(t) \vcentcolon= \overline{\psi(-t)}$. In fact, $B_0$ is the Fourier transform of a unique positive tempered measure $\mu$, which lets us write
	\begin{equation} \label{covariance functional integral form}
		B_\Phi(\phi,\psi)
		= 
		\int_{-\infty}^\infty
			\widehat{\phi}(\xi)
			\overline{\widehat{\psi}(\xi)}
		\,d \mu(\xi)
	\end{equation}
	for any $\phi,\psi \in \mathcal{S}$. We call $d \mu$ the \emph{power spectrum} of the GRP $\Phi$. If $d\mu$ is absolutely continuous with related to the Lebesgue measure so that $d\mu(\xi)=f(\xi)d\xi$ for some non-negative function $f$, when there is no danger of confusion, we also call $f$ the power spectrum of $\Phi$.
	For ease of notation in what follows, we define the associated \emph{pseudocovariance functional} $P_\Phi$ by
	\begin{equation}
		\label{pseudocovariance functional definition and integral form}
		P_\Phi(\phi,\psi) 
		= 
		B_\Phi(\phi,\bar{\psi}) 
		= 
		\int_{-\infty}^\infty 
			\widehat{\phi}(\xi)\widehat{\psi}(-\xi)
		\,d \mu(\xi).
	\end{equation}
	Note that the above framework could be reduced to the case where $\Phi$ is an ordinary random process with covariance function 
	$C(s) = \mathbb{E}[\Phi(t)\overline{\Phi(t+s)}]$, but we proceed below using the more general framework.

	An ordinary random process may be thought of as the measurements of some quantity at a sequence of instants. On the other hand, a GRP $\Phi$ describes a random process that cannot be measured precisely at each instant, so that $\Phi(\psi)$ is a random variable describing the measurement of some quantity when it is measured by an instrument that is characterized by the {\it measurement function} $\psi$, commonly taken to be a Schwartz function. This provides a more general framework that takes into account the inability to measure physical quantities instantaneously.

\section{Summary of Notation in the SST analysis}\label{Section Notation table}

	We systematically use the following notations, where $\Phi$ is a stationary random process satisfying Assumption \ref{Definition:stationary complex Gaussian}, the signal satisfies Assumption \ref{assump:nonnull signal}, and $h$ is a Schwartz function. 
	In the proof, the condition of $h$ will be made clear from one to another, and we will make it clear when it is expressed in the superscript.
	Take $\eta>0$. Define
	\begin{equation}
			\mu^{(h)}_{\xi_0,\eta} \vcentcolon=A\hat{h}(\xi_0-\eta)e^{i2\pi \xi_0 t}\begin{bmatrix}1&i2\pi\xi_0\end{bmatrix}^\top\,. \label{Definition of muhxi0eta}
			\end{equation}
			Denote $\underline{\Sigma}^{(h)}_\eta$ to be the augmented covariance matrix of the complex random vector 
			\[
			\B{W}^{(h)}_\eta := \begin{bmatrix}\Phi(h_{t,\eta})&\Phi((h')_{t,\eta})\end{bmatrix}^\top
			\] 
			and $\lambda^{(h)}_{\eta,1}\geq \lambda^{(h)}_{\eta,2}\geq \lambda^{(h)}_{\eta,3}\geq\lambda^{(h)}_{\eta,4}$ to be its eigenvalues. For $\theta\in [0,2\pi)$ and $\omega\in \mathbb{C}$, define
	\begin{align}
			c^{(h)}_\eta &\vcentcolon= 4(\det{\underline{\Sigma}^{(h)}_\eta})^{-1/2}\nonumber\\
			A^{(h)}_{\eta}(\theta,\omega)&\vcentcolon=\frac{1}{2}\big(\underline{e^{i\theta}\boldsymbol{\omega}}\big)^*\underline{\Sigma}_\eta^{(h)-1}\underline{e^{i\theta}\boldsymbol{\omega}}\label{Expansion:ABC}\\
			B^{(h)}_{\xi_0,\eta}(\theta,\omega)&\vcentcolon=\frac{1}{2}\Re\Big[\left(\underline{\mu^{(h)}_{\xi_0,\eta}}\right)^* \underline{\Sigma}_\eta^{(h)-1}\underline{e^{i\theta}\boldsymbol{\omega}}\Big]\nonumber\\
			C^{(h)}_{\xi_0,\eta}&\vcentcolon=\frac{1}{2}\left(\underline{\mu^{(h)}_{\xi_0,\eta}}\right)^* \underline{\Sigma}_\eta^{(h)-1} \underline{\mu^{(h)}_{\xi_0,\eta}}\,,\nonumber
			\end{align}
			where \begin{equation}
			\boldsymbol{\omega}\vcentcolon=\begin{bmatrix}1& 2\pi i {(\eta-\omega)}\end{bmatrix}^\top\,.
			\end{equation} 
			Since $\underline{\Sigma}^{(h)}_\eta$ is positive definite and $\|\boldsymbol{\omega}\|^2=1+4\pi^2|\eta-\omega|^2\geq 1$, for all $\omega\in \mathbb{C}$ we have 
		\begin{align}
			1/\lambda_{\eta,1}^{(h)}\leq \|\boldsymbol{\omega}\|^2/\lambda_{\eta,1}^{(h)}&\leq A_{\eta}(\theta,\omega)\leq \|\boldsymbol{\omega}\|^2/\lambda_{\eta,4}^{(h)}\,. \label{Upper bound of Aeta2}
		\end{align}
		Similarly, we know that $C^{(h)}_{\xi_0,\eta}>0$
		in the non-null case.
		 This bound gives us a rough idea of $A^{(h)}_\eta(\theta,\omega)$ and is a reasonably good bound when $\eta$ is not close to zero. When $\eta$ is close to zero, as will be shown in Lemma \ref{Lemma: c of mu_xi0eta when eta small}, due to the degeneracy of $\underline{\Sigma}^{(h)}_\eta$, $\lambda^{(h)}_{\eta,4}$ is closer to zero and the upper bound is bad. A more precise bound described in Lemma \ref{Lemma: c of mu_xi0eta when eta small} is needed later when we control the moments. 
		 Note that the definition of $A^{(h)}_\eta$ and $B^{(h)}_{\xi_0,\eta}$ here mirrors that of \eqref{eq:A} in Theorem \ref{thm: Quotient density}, since by a direct expansion we have
			\begin{align}
			A^{(h)}_{\eta}(\theta,\omega)&=\boldsymbol{\omega}^* \conj{{P}^{-1}} \boldsymbol{\omega} - \Re \Big(e^{2i\theta}\boldsymbol{\omega}^\top {R}^\top\conj{ P^{-1}} \boldsymbol{\omega}\Big)\,,\label{Expansion Atheta omega in a different form}\\
			B^{(h)}_{\xi_0,\eta}(\theta,\omega)&=
			\Re\Big[ e^{i\theta}\Big(\mu_{\xi_0,\eta}^{(h)*}-\mu_{\xi_0,\eta}^{(h)\top} R^\top\Big)\conj{ P^{-1}}\boldsymbol{\omega} \Big]\,.\nonumber
			\end{align}
			Here, recall that the expansion of $B^{(h)}_{\xi_0,\eta}(\theta,\omega)$ depends on the fact that $\conj{ P^{-1}}\conj{ R}$ is symmetric.

			When we compare two windows or two frequencies, we need the following notations.
	Fix $t\in\mathbb{R}$, for $\eta,\eta'>0$ and two Schwartz functions $h$ and $\hbar$ as windows, denote
	\begin{align}
				\mu^{(h, \hbar)}_{\xi_0,\eta,\eta'}=e^{i2\pi \xi_0 t}\begin{bmatrix}\hat{h}(\xi_0-\eta)&i2\pi\xi_0\hat{h}(\xi_0-\eta)&\hat{\hbar}(\xi_0-\eta')&i2\pi\xi_0\hat{\hbar}(\xi_0-\eta')\end{bmatrix}^\top\label{eq: special case mu}
			\end{align}
			and
		\[
				\B{W}^{(h, \hbar)}_{\eta,\eta'} = \begin{bmatrix}\Phi(h_{t,\eta})&\Phi((h')_{t,\eta})&\Phi(\hbar_{t,\eta'})&\Phi((\hbar')_{t,\eta'})\end{bmatrix}^\top\in \mathbb{C}^4
			\] 
			and denote the associated augmented covariance matrix as $\underline{\Sigma}_{\eta,\eta'}^{(h, \hbar)}$. 

		When we study the noise structure and carry out the perturbation arguments, we need the following notations. Take two even, bounded and smooth functions $\varphi_1,\varphi_2$. For $k=0,1,\ldots$, 
		define the following real-valued functions on $\mathbb{R}\times\mathbb{R}$:
		\begin{align}
		\gamma^{[\varphi_1,\varphi_2]}_k(\eta,\eta') &\vcentcolon= \int \left(\xi+\frac{\eta+\eta'}{2}\right)^k e^{-4\pi^2(\xi+\frac{\eta+\eta'}{2})^2}\varphi_1(\eta+\xi) \varphi_2(\eta'+\xi) \,d\vartheta(\xi)\nonumber\\
		\nu^{[\varphi_1,\varphi_2]}_k(\eta,\eta') &\vcentcolon=\int \left(\xi+\frac{\eta-\eta'}{2}\right)^k e^{-4\pi^2(\xi+\frac{\eta-\eta'}{2})^2}\varphi_1(\eta+\xi)\varphi_2(\eta'-\xi) \,d\vartheta(\xi)\nonumber
		\end{align}
		for $k=0,1,2,\ldots$. 
		In the special case when $\eta=\eta'$, we use the following simplified notations:
		\begin{align}\label{Definition:Gamma function l}
		\gamma^{[\varphi_1,\varphi_2]}_k(\eta)&:=\int (\xi+\eta)^ke^{-4\pi^2(\xi+\eta)^2}\varphi_1(\eta+\xi)\varphi_2(\eta+\xi)d\vartheta(\xi)\\
		\nu^{[\varphi_1,\varphi_2]}_k(\eta)&:=\int \xi^k e^{-4\pi^2\xi^2}\varphi_1(\eta+\xi)\varphi_2(\eta-\xi) \,d\vartheta(\xi)\nonumber
		\end{align}
		Note the following facts summarized as a lemma, which follow immediately from the definition and we omit the proof.
		
		\begin{lemma}\label{Basic relationship between gamma and nu over windows}
		Following the above notation, we have
		\begin{enumerate}
		\item When $\varphi_1=\varphi_2$, $\varphi_1(\eta+\xi)\varphi_2(\eta-\xi)$ is an even function of $\xi$, so $\nu^{[\varphi_1,\varphi_1]}_k(\eta)=0$ when $k$ is odd for any $\eta\in \mathbb{R}$. Also, $\nu^{[1,1]}_k(\eta)=\gamma^{[1,1]}_k(0)$ for any $\eta$. Furthermore, when $d\vartheta(\xi)=d\xi$, we have $\gamma^{[1,1]}_k(\eta)=\gamma^{[1,1]}_k(0)$ for any $\eta$.
		
		\item $\gamma_k^{[\varphi_1,\varphi_2]}(\eta)=\gamma_k^{[\varphi_2,\varphi_1]}(\eta)$ for any $k>0$ and $\eta>0$.
		
		\item $\nu_k^{[\varphi_1,\varphi_2]}(\eta)=(-1)^k\nu_k^{[\varphi_2,\varphi_1]}(\eta)$  for any $k>0$ and $\eta>0$.

		\item $\nu^{[\varphi_1,\varphi_1]}_k(0)=\gamma^{[\varphi_1,\varphi_1]}_k(0)$ for all $k$.
		
		\item When $\varphi_1\varphi_2=0$, $\gamma_l(0)=0$ for all $l$. 		\end{enumerate}
		\end{lemma}
		
		Below, $\varphi_1$ may be chosen to be $1$ and $\varphi_2$ is chosen to be a ``good'' function; for example, in the $M$-approximation, we choose $\varphi_2$ to be a smooth, bounded and symmetric function that is $1$ when $|t|\leq M$ and $0$ when $|t|\geq 2M$.

	\section{Technical lemmas for the SST analysis}\label{section:Technical lemmas for the SST analysis}
	We need the following lemma when we control the covariance between different $\eta$'s. Lemma~\ref{lem: V_Phi and d_t V_Phi second-order stats} is a direct consequence of Lemma~\ref{Appendex:C and P of W for Cov}. Indeed, when $\varphi_1=\varphi_2=1$, we obtain Lemma~\ref{lem: V_Phi and d_t V_Phi second-order stats}.

	\begin{lemma}\label{Appendex:C and P of W for Cov}
		Suppose $h$ satisfies Assumption \ref{assump:gaussian window}, $\varphi_1,\varphi_2$ are two even, bounded and smooth functions and $h_1,h_2$ are Schwartz functions satisfying $\hat{h}_i=\hat{h}\varphi_i$, $i=1,2$. For $\eta,\eta'>0$, the covariance and the pseudocovariance matrices of $\B{W}^{(h_1,h_2)}_{\eta,\eta'}$ are
		 \[
		 \begin{bmatrix} \Gamma^{[\varphi_1,\varphi_1]}_{\eta,\eta} & \Gamma^{[\varphi_1,\varphi_2]}_{\eta,\eta'} \\
		 \Gamma^{[\varphi_2,\varphi_1]}_{\eta',\eta} &\Gamma^{[\varphi_2,\varphi_2]}_{\eta',\eta'}\end{bmatrix}\in \mathbb{C}^{4\times 4}
		 \quad
		 \mbox{and}
		  \quad
		 \begin{bmatrix} C^{[\varphi_1,\varphi_1]}_{\eta,\eta} & C^{[\varphi_1,\varphi_2]}_{\eta,\eta'} \\
		 C^{[\varphi_2,\varphi_1]}_{\eta',\eta} & C^{[\varphi_2,\varphi_2]}_{\eta',\eta'}\end{bmatrix}\in \mathbb{C}^{4\times 4}
		 \]
		 respectively, where for $k,l=1,2$
		 \begin{align*}
		 &\Gamma^{[\varphi_k,\varphi_l]}_{\eta,\eta'}=e^{-\pi^2(\eta'-\eta)^2}\\
		 &\quad\times\begin{bmatrix} \gamma^{[\varphi_k,\varphi_l]}_0(\eta,\eta')  & -2\pi i[\gamma^{[\varphi_k,\varphi_l]}_1(\eta,\eta')+\frac{\eta-\eta'}{2}\gamma^{[\varphi_k,\varphi_l]}_0(\eta,\eta')] \\  2\pi i[\gamma^{[\varphi_k,\varphi_l]}_1(\eta,\eta')+\frac{\eta-\eta'}{2}\gamma^{[\varphi_k,\varphi_l]}_0(\eta,\eta')] & 4\pi^2 [\gamma^{[\varphi_k,\varphi_l]}_2(\eta,\eta')-\frac{(\eta-\eta')^2}{4}\gamma^{[\varphi_k,\varphi_l]}_0(\eta,\eta')] \end{bmatrix}\,,
		 \end{align*}
		 and
		 \begin{align*}
		 &C^{[\varphi_k,\varphi_l]}_{\eta,\eta'}=e^{-\pi^2(\eta'+\eta)^2}\\
		 &\quad\times\begin{bmatrix} 
		 \nu^{[\varphi_k,\varphi_l]}_0(\eta,\eta')  
		 & 2\pi i[-\nu^{[\varphi_k,\varphi_l]}_1(\eta,\eta')+\frac{\eta+\eta'}{2}\nu^{[\varphi_k,\varphi_l]}_0(\eta,\eta')] \\  
		 2\pi i[\nu^{[\varphi_k,\varphi_l]}_1(\eta,\eta')+\frac{\eta+\eta'}{2}\nu^{[\varphi_k,\varphi_l]}_0(\eta,\eta')] 
		 & 4\pi^2 [\nu^{[\varphi_k,\varphi_l]}_2(\eta,\eta')-\frac{(\eta+\eta')^2}{4}\nu^{[\varphi_k,\varphi_l]}_0(\eta,\eta')] \end{bmatrix}\,.
		 \end{align*}
	\end{lemma}
	
	\begin{proof}
		It is a straightforward calculation  
		by plugging $h(x)=\frac{1}{\sqrt{2\pi}}e^{-x^2/2}$, whose Fourier transform is $\hat{h}(\xi)=e^{-2\pi^2\xi^2}$, into \eqref{covariance functional integral form} and \eqref{pseudocovariance functional definition and integral form}, where $\phi$ and $\psi$ are replaced by either $(h_1)_{t,\eta}$, $(h_1')_{t,\eta}$, $(h_2)_{t,\eta'}$ or $(h_2')_{t,\eta'}$. By noting that for $i=1,2$, 
		\begin{align*}
			\mathcal{F}((h_i)_{t,\eta})(\xi) 
			&= 
			\hat{h}(\eta+\xi) \varphi_i(\eta+\xi) e^{-2\pi i \xi t}\\
			\mathcal{F}((h_i')_{t,\eta})(\xi)
			&= 
			2\pi i (\eta+\xi) \hat{h}(\eta+\xi)\varphi_i(\eta+\xi) e^{-2\pi i \xi t},
		\end{align*} 
		and using the decompositions like $(\xi+\eta)^2+(\xi-\eta')^2=2(\xi+(\eta-\eta')/2)^2+(\eta+\eta')^2/2$,
		we finish the proof.
		
	\end{proof}

		Note that $C^{[\varphi_1,\varphi_2]}_{\eta,\eta'}$ is the transpose of $C^{[\varphi_2,\varphi_1]}_{\eta,\eta'}$ since $\nu^{[\varphi_1,\varphi_2]}_1(\eta,\eta')=-\nu^{[\varphi_2,\varphi_1]}_1(\eta,\eta')$ by Lemma \ref{Basic relationship between gamma and nu over windows}.

	The following lemma is used when we handle the degeneracy of the covariance matrix when $\eta\to0$.

	\begin{lemma}\label{lem: V_Phi and d_t V_Phi second-order stats each entries}
		Follow the same notations used in Lemma \ref{Appendex:C and P of W for Cov}.
		When $\eta\to 0$, $\gamma^{[\varphi_1,\varphi_2]}_l(\eta)$ defined in \eqref{Definition:Gamma function l} satisfies
		\begin{align*}
		\gamma^{[\varphi_1,\varphi_2]}_0(\eta)&\,=\gamma^{[\varphi_1,\varphi_2]}_0(0)+Q^{[\varphi_1,\varphi_2]}_0\eta^2+O(\eta^4)\\
		\gamma^{[\varphi_1,\varphi_2]}_1(\eta)&\,=Q^{[\varphi_1,\varphi_2]}_1\eta+O(\eta^3)\\
		\gamma^{[\varphi_1,\varphi_2]}_2(\eta)&\,=\gamma^{[\varphi_1,\varphi_2]}_2(0)+Q^{[\varphi_1,\varphi_2]}_2\eta^2+O(\eta^4)\,,
		\end{align*}
		where
		\begin{align*}
		Q^{[\varphi_1,\varphi_2]}_0&\,=4\pi^2\left(8\pi^2\gamma^{[\varphi_1,\varphi_2]}_2(0)-\gamma^{[\varphi_1,\varphi_2]}_0(0)\right.\\
		&\quad\left.+\int e^{-4\pi^2\xi^2}\left[-8\pi^2\xi(\varphi_1\varphi_2)'(\xi)+\frac{1}{2}(\varphi_1\varphi_2)''(\xi)\right]d\vartheta(\xi)\right)\\
		Q^{[\varphi_1,\varphi_2]}_1&\,=-8\pi^2\gamma^{[\varphi_1,\varphi_2]}_2(0)+\gamma^{[\varphi_1,\varphi_2]}_0(0)+\int e^{-4\pi^2\xi^2}\xi(\varphi_1\varphi_2)'(\xi)d\vartheta(\xi)\\
		Q^{[\varphi_1,\varphi_2]}_2&\,=32\pi^4\gamma^{[\varphi_1,\varphi_2]}_4(0)-20\pi^2\gamma^{[\varphi_1,\varphi_2]}_2(0)+\gamma^{[\varphi_1,\varphi_2]}_0(0)\\
		&\quad+ \int e^{-4\pi^2\xi^2}\left[-8\pi^2\xi^3(\varphi_1\varphi_2)'(\xi)+\frac{1}{2}\xi^2(\varphi_1\varphi_2)''(\xi)+2\xi(\varphi_1\varphi_2)'(\xi)\right] d\vartheta(\xi)\,.
		\end{align*}
		On the other hand, when $\eta\to 0$, $\nu^{[\varphi_1,\varphi_2]}_l(\eta)$ defined in \eqref{Definition:Gamma function l} satisfies
		\begin{align*}
		\nu^{[\varphi_1,\varphi_2]}_0(\eta)&\,=\gamma^{[\varphi_1,\varphi_2]}_0(0)+R^{[\varphi_1,\varphi_2]}_0\eta^2+O(\eta^4)\\
		\nu^{[\varphi_1,\varphi_2]}_1(\eta)&\,=R^{[\varphi_1,\varphi_2]}_1\eta+O(\eta^3)\\
		\nu^{[\varphi_1,\varphi_2]}_2(\eta)&\,=\gamma^{[\varphi_1,\varphi_2]}_2(0)+R^{[\varphi_1,\varphi_2]}_2\eta^2+O(\eta^4)\,,
		\end{align*}
		where
		\begin{align*}
		R^{[\varphi_1,\varphi_2]}_0&\,=\int e^{-4\pi^2\xi^2}\left[\phi_1'\phi_2-\phi_1\phi_2'+\frac{1}{2}(\phi_1''\phi_2+\phi_1\phi_2'')\right](\xi)d\vartheta(\xi)\\
		R^{[\varphi_1,\varphi_2]}_1&\,=\int e^{-4\pi^2\xi^2}\xi(\varphi'_1\varphi_2+\varphi_1\varphi_2')(\xi)d\vartheta(\xi)\\
		R^{[\varphi_1,\varphi_2]}_2&\,=\int e^{-4\pi^2\xi^2}\left[\phi_1'\phi_2-\phi_1\phi_2'+\frac{1}{2}(\phi_1''\phi_2+\phi_1\phi_2'')\right](\xi)d\vartheta(\xi)\,.
		\end{align*}
	\end{lemma}

	\begin{proof}
		The lemma follows from a straightforward calculation by Taylor's expansion and the boundedness and smoothness assumptions of $\varphi_1$ and $\varphi_2$. 
		By plugging Taylor's expansion of 
		$e^{-4\pi^2(\xi+\eta)^2}$ and $\varphi_1(\xi+\eta)\varphi_2(\xi+\eta)$ at $\eta=0$ into	 
		$\gamma^{[\varphi_1,\varphi_2]}_l(\eta)=\int e^{-4\pi^2(\xi+\eta)^2}(\xi+\eta)^l\varphi_1(\xi+\eta)\varphi_2(\xi+\eta)d\vartheta(\xi)$ and using the symmetry of $\varphi_1$, $\varphi_2$, and $d\vartheta(\xi)=p(\xi)d\xi$ to cancel the odd order terms when $\eta=0$, we obtain the first claim. 
		For example, since 
		\[
		e^{-4\pi^2(\xi+\eta)^2}=e^{-4\pi^2\xi^2}-8\pi^2\eta \xi e^{-4\pi^2\xi^2}-4\pi^2\eta^2e^{-4\pi^2\xi^2}+32\pi^4\eta^2\xi^2e^{-4\pi^2\xi^2}+O(\eta^3)
		\]
		when $\varphi_1=\varphi_2=1$ and $\eta$ is close to $0$,
		we have
		\begin{align*}
		\gamma^{[1,1]}_0(\eta)&\,=\int e^{-4\pi^2(\xi+\eta)^2}d\vartheta(\xi)\\
		&\,=\int [e^{-4\pi^2\xi^2}-8\pi^2\eta \xi e^{-4\pi^2\xi^2}-4\pi^2\eta^2e^{-4\pi^2\xi^2}+32\pi^4\eta^2\xi^2e^{-4\pi^2\xi^2}+O(\eta^3)]d\vartheta(\xi)\\
		&\,=\gamma_0(0)+4\pi^2[-\gamma_0(0)+2\gamma_2(0)]\eta^2+O(\eta^4)
		\end{align*} 
		since $\int \xi e^{-4\pi^2\xi^2} d\vartheta(\xi)=0$. The proof for $\nu_k^{[\varphi_1,\varphi_2]}(\eta)$ is the same, while we use the fact that $\nu^{[\varphi_1,\varphi_1]}_k(0)=\gamma^{[\varphi_1,\varphi_1]}_k(0)$ for all $k$.
		\end{proof}
		
		\begin{lemma}\label{lem: V_Phi and d_t V_Phi second-order stats each entries part2}
		Follow the same notations used in Lemma \ref{Appendex:C and P of W for Cov} and assume $\varphi_1\varphi_2\neq 0$.
		For $|\eta|\geq 1$, we have
		\begin{align*}
		\frac{|\gamma^{[\varphi_1,\varphi_2]}_1(\eta)|}{\int e^{-4\pi^2(\xi+\eta)^2}|\xi+\eta|\varphi_1(\xi+\eta)\varphi_2(\xi+\eta)d\vartheta(\xi)}\leq  1-\mathsf d_m\,,
		\end{align*}
		where 
		\[
		\mathsf{d}_m:=\min_{\eta\geq 1}\frac{2\min\{\int_0^\infty e^{-4\pi^2\xi^2}\xi \varphi_1(\xi)\varphi_2(\xi)p(\xi+\eta)d\xi,\,\int_0^\infty e^{-4\pi^2\xi^2}\xi \varphi_1(\xi)\varphi_2(\xi)p(\xi-\eta)d\xi\}}{\int e^{-4\pi^2\xi^2}|\xi| \varphi_1(\xi)\varphi_2(\xi) p(\xi-\eta)d\xi}
		>0\,,
		\] 
		where $\mathsf{d}_m$ is bounded above by $1$ and depends on $\varrho$, $\varphi_1$ and $\varphi_2$.
	\end{lemma}
	\begin{remark}
	Note that $\mathsf{d}_m=1$ when $\varrho=0$; that is, when the noise is white. When the noise is white, all notations are simplified.
	\end{remark}
	\begin{proof}
	Note that $\gamma^{[\varphi_1,\varphi_2]}_1(\eta)=\int e^{-4\pi^2\xi^2}\xi \varphi_1(\xi)\varphi_2(\xi)p(\xi-\eta)d\xi$ by a change of variable, so by symmetry we have
		\begin{align*}
		\gamma^{[\varphi_1,\varphi_2]}_1(\eta)
		&\,=\int_0^\infty e^{-4\pi^2\xi^2}\xi\varphi_1(\xi)\varphi_2(\xi)(p(\xi-\eta)-p(\xi+\eta))d\xi\\
		&\,=\int_0^\infty e^{-4\pi^2\xi^2}\xi \varphi_1(\xi)\varphi_2(\xi)p(\xi-\eta)d\xi-\int_0^\infty e^{-4\pi^2\xi^2}\xi \varphi_1(\xi)\varphi_2(\xi)p(\xi+\eta)d\xi\,.
		\end{align*}
		On the other hand, 
		\begin{align*}
		&\int e^{-4\pi^2\xi^2}|\xi| \varphi_1(\xi)\varphi_2(\xi)p(\xi-\eta)d\xi\\
		=&\,\int_0^\infty e^{-4\pi^2\xi^2}|\xi| \varphi_1(\xi)\varphi_2(\xi) (p(\xi-\eta)+p(\xi+\eta))d\xi\\
		=&\,\int_0^\infty e^{-4\pi^2\xi^2} \xi \varphi_1(\xi)\varphi_2(\xi) p(\xi-\eta)d\xi+\int_0^\infty e^{-4\pi^2\xi^2}\xi \varphi_1(\xi)\varphi_2(\xi)p(\xi+\eta)d\xi\,.
		\end{align*}
		We thus have 
		\[
		0\leq \frac{|\gamma^{[\varphi_1,\varphi_2]}_1(\eta)|}{\int e^{-4\pi^2\xi^2}|\xi| \varphi_1(\xi)\varphi_2(\xi)p(\xi-\eta)d\xi}=1-\mathsf d(\eta)\,,
		\] 
		where 
		\[
		\mathsf{d}(\eta)=\frac{2\min\{\int_0^\infty e^{-4\pi^2\xi^2}\xi \varphi_1(\xi)\varphi_2(\xi)p(\xi+\eta)d\xi,\,\int_0^\infty e^{-4\pi^2\xi^2}\xi \varphi_1(\xi)\varphi_2(\xi)p(\xi-\eta)d\xi\}}{\int e^{-4\pi^2\xi^2}|\xi| \varphi_1(\xi)\varphi_2(\xi) p(\xi-\eta)d\xi}\,.
		\]
		By setting $\mathsf{d}_m:=\min_{\eta\geq1} \mathsf{d}(\eta)>0$, we obtain the claim. 
		
	\end{proof}

	We need the following key Lemma summarizing the spectral behavior of the augmented covariance matrix in different regimes of $\eta$. This behavior is critical when we handle those integrations involving confluent hypergeometric functions. 
	
	\begin{lemma}[Key lemma]\label{Lemma: c of mu_xi0eta when eta small}
			Suppose Assumptions~\ref{assump:noise part}, \ref{assump:nonnull signal} and \ref{assump:gaussian window} hold. Take an even, bounded and smooth function $\varphi$ and set a new window $\hbar$ so that $\hat{\hbar}=\hat{h}\varphi$. Then the following statements hold. 
			\begin{enumerate}
			\item When $\eta\geq 1$, $\lambda^{(\hbar)}_{\eta,1},\ldots,\lambda^{(\hbar)}_{\eta,4}\asymp \eta^\varrho$, where the implied constants depend on $\varphi$ and $\varrho$.

			\item When $0<\eta<1$, particularly when $\eta$ is close to $0$, we have $\lambda^{(\hbar)}_{\eta,1}\asymp 8\pi^2\gamma^{[\hbar,\hbar]}_2(0)$, $\lambda^{(\hbar)}_{\eta,2}\asymp 2\gamma^{[\hbar,\hbar]}_0(0)$ and 
			\begin{equation}
			\lambda^{(\hbar)}_{\eta,3}\asymp\eta^2,\,\, \lambda^{(\hbar)}_{\eta,4}\asymp \eta^6,\,\mbox{ and hence }\,\det\underline{\Sigma}^{(\hbar)}_\eta\asymp \eta^{8}\,,
			\end{equation} 
			where the implied constants depend on $\varphi$ and $\varrho$.

			\item For $u=e^{i\theta}\begin{bmatrix} 1& 2\pi i  w 
			\end{bmatrix}^\top\in \mathbb{C}^2$, where $\theta\in [0,2\pi)$ and $w=re^{i\phi}\in \mathbb{C}$, we have 
			\begin{align}
			\underline{u}^*\underline{ \Sigma}_\eta^{(\hbar)-1}\underline{u}\asymp&
			\frac{2\sin^2(\theta)}{\mathsf c_4}\eta^{-6}+\frac{8\pi^2r^2\cos^2(\theta+\phi)}{\mathsf c_3}\eta^{-2}+\frac{4\pi^2r^2\sin^2(\theta+\phi)}{\gamma^{[\varphi,\varphi]}_0(0)}+\frac{\cos^2(\theta)}{4\pi^2\gamma^{[\varphi,\varphi]}_2(0)}\,\nonumber
			\end{align}
			when $\eta$ is close to $0$ for constants $\mathsf{c}_3>0$ and $\mathsf c_4>0$ independent of $\eta$ but dependent on $\varrho$ and $\varphi$, where the implied constants depend on $\varphi$ and $\varrho$. 
			In particular, for $\mu^{(h)}_{\xi_0,\eta}$ defined in \eqref{Definition of muhxi0eta}, where $A>0$ and $\xi_0>0$ are fixed, we have
			\begin{align}
			\underline{\mu^{(\hbar)*}_{\xi_0,\eta}}\underline{\Sigma}_\eta^{(\hbar)-1} \underline{\mu^{(\hbar)}_{\xi_0,\eta}}\asymp \,A^2\hat{\hbar}(\xi_0-\eta)^2\Big[&\frac{2\sin^2(2\pi\xi_0t)}{\mathsf c_4 \eta^{6}}+\frac{8\pi^2\xi_0^2\cos^2(2\pi\xi_0t)}{\mathsf c_3 \eta^{2}}\Big] \label{Asymptotic muSigmainvmu or C} 
			\end{align}
			and
			\begin{align}
			\underline{\mu^{(\hbar)*}_{\xi_0,\eta}}\underline{ \Sigma}_\eta^{(\hbar)-1}\underline{u}\asymp&\,
			A\hat{\hbar}(\xi_0-\eta)\Big[\frac{2\sin(\theta)\sin(2\pi\xi_0t)}{\mathsf c_4 \eta^{6}}+\frac{8\pi^2r\xi_0\cos(\theta+\phi)\cos(2\pi\xi_0t)}{\mathsf c_3 \eta^{2}}\label{Asymptotic muSigmainvmu or B}\\
			&+\frac{4\pi^2r\xi_0\sin(\theta+\phi)\sin(2\pi\xi_0t)}{\gamma^{[\varphi,\varphi]}_0(0)}+\frac{\cos(\theta)\cos(2\pi\xi_0t)}{4\pi^2\gamma^{[\varphi,\varphi]}_2(0)}\Big]\nonumber
			\end{align}
			when $\eta$ is close to $0$, where the implied constants depend on $\varphi$ and $\varrho$. 
			\end{enumerate}
			\end{lemma}
			
			Note that the third point of this Lemma is the key tool we use to handle the degenerate $\underline{\Sigma}_\eta^{(\hbar)}$ when $\eta$ is close to $0$, particularly the expansion of $A_\eta(\theta,\omega)$, $B^2_{\xi_0,\eta}(\theta,\omega)$ and $C_{\xi_0,\eta}$.

			\begin{proof}
			To simplify the notation, since there is no danger of confusion, we omit the superscript ${}^{[\varphi,\varphi]}$ or ${}^{(\hbar)}$ describing the dependence on the window below.
			By Lemma \ref{Appendex:C and P of W for Cov}, we have the expansion of $\underline{\Sigma}_\eta
			= \begin{bmatrix}
			\Gamma_\eta & C_\eta\\
			\overline{C_\eta} & \overline{\Gamma_\eta}
			\end{bmatrix}
		$, where $\Gamma_\eta:=\Gamma^{[\varphi,\varphi]}_{\eta,\eta}$ and $C_\eta:=C^{[\varphi,\varphi]}_{\eta,\eta}$.

				For the first claim, by Assumption \ref{assump:nonnull signal} and a direct expansion, we have $\gamma_0(\eta)\asymp \eta^\varrho$ and $\gamma_2(\eta)\asymp \eta^\varrho$ when $|\eta|\geq 1$.  Since $\varrho$ is finite, $C_\eta$ decays exponentially fast when $\eta$ increases. Thus, $C_\eta$ is negligible compared with $\Gamma_\eta$, and the eigenvalues of $\underline{\Sigma}_\eta$ are mainly determined by $\Gamma_\eta$ by a perturbation argument. By a direct calculation, the eigenvalues of $\Gamma_\eta$ are $\frac{1}{2}[\gamma_0(\eta)+4\pi^2\gamma_2(\eta)\pm \sqrt{(\gamma_0(\eta)+4\pi^2\gamma_2(\eta))^2-16\pi^2(\gamma_0(\eta)\gamma_2(\eta)-\gamma_1(\eta)^2)}]$.
		By Lemma \ref{lem: V_Phi and d_t V_Phi second-order stats each entries part2}, 
		\begin{align*}
		\frac{|\gamma_1(\eta)|^2}{\gamma_0(\eta)\gamma_2(\eta)}&\,=\frac{(\int e^{-4\pi^2(\xi+\eta)^2}|\xi+\eta|d\vartheta(\xi))^2}{\gamma_0(\eta)\gamma_2(\eta)}\left(\frac{|\gamma_1(\eta)|}{\int e^{-4\pi^2(\xi+\eta)^2}|\xi+\eta|d\vartheta(\xi)}\right)^2 \\
		&\,\leq \left(\frac{|\gamma_1(\eta)|}{\int e^{-4\pi^2(\xi+\eta)^2}|\xi+\eta|d\vartheta(\xi)}\right)^2 \leq (1-\mathsf d_m)^2\,, 
		\end{align*}
		where the first inequality comes from the Cauchy-Schwartz, and the second inequality comes from Lemma \ref{lem: V_Phi and d_t V_Phi second-order stats each entries}. Since $\mathsf{d}_m>0$ and is independent of $|\eta|\geq1$, hence $\gamma_0(\eta)\gamma_2(\eta)-\gamma_1(\eta)^2\asymp \gamma_0(\eta)\gamma_2(\eta)$ when $|\eta|\geq 1$. As a result, 
		\[
		\gamma_0(\eta)\gamma_2(\eta)-\gamma_1(\eta)^2\geq \frac{4\mathsf d_m}{(1+\mathsf d_m)^2}\gamma_0(\eta)\gamma_2(\eta)\,, 
		\]
		and hence 
		\begin{align*}
		&(\gamma_0(\eta)+4\pi^2\gamma_2(\eta))^2-16\pi^2(\gamma_0(\eta)\gamma_2(\eta)-\gamma_1(\eta)^2)\\
		\leq&\, \gamma_0(\eta)^2+8\pi^2\Big(1-\frac{8\mathsf d_m}{(1+\mathsf d_m)^2}\Big)\gamma_0(\eta)\gamma_2(\eta)+16\pi^4\gamma_2(\eta)^2\,.
		\end{align*}
		By combining this fact and the fact that $\gamma_0(\eta)\asymp \gamma_2(\eta)$ when $\eta\geq \eta_0$, the eigenvalues of $\Gamma_\eta$, $\lambda_{\eta,1},\ldots,\lambda_{\eta,4}$, are of the order $\eta^\varrho$ when $\eta\geq \eta_0$.
			
			For the second claim, we first show the result when $\eta$ is close to $0$. Denote the eigendecomposition of $\underline{ \Sigma}_\eta$ as $UDU^*$, where $U\in U(4)$ and $D$ is a diagonal matrix. To quantify $\lambda_{\eta,3}$ and $\lambda_{\eta,4}$ when $\eta$ is small, by Lemma \ref{lem: V_Phi and d_t V_Phi second-order stats each entries},
			we have
			\begin{align*}
			\Gamma_\eta 
			=&\, 
			\begin{bmatrix}
				\gamma_0(0) & 0\\
				0 & 4\pi^2\gamma_2(0)
			\end{bmatrix}
			+\begin{bmatrix}
				Q_0\eta^2 & -2\pi i Q_1\eta\\
				2\pi iQ_1\eta & 4\pi^2 Q_2\eta^2 
			\end{bmatrix}+\begin{bmatrix}
				O(\eta^4) & O(\eta^3)\\
				O(\eta^3) & O(\eta^4) 
			\end{bmatrix}
			\end{align*}
			and
			\begin{align*}
			C_\eta
			=&\,\begin{bmatrix}
				\gamma_0(0)  & 0\\
				 0 & 4\pi^2\gamma_2(0)
			\end{bmatrix}+\begin{bmatrix}
				(R_0-4\pi^2\gamma_0(0))\eta^2  & 2\pi i \gamma_0(0)\eta\\
				 2\pi i \gamma_0(0) \eta& 4\pi^2(R_2-4\pi^2\gamma_2(0))\eta^2
			\end{bmatrix}+\begin{bmatrix}
				O(\eta^4) & O(\eta^3)\\
				O(\eta^3) & O(\eta^4) 
			\end{bmatrix}
		\end{align*} 
			Therefore, $U$ and $D$ are well approximated by
			\[
			U_0=\frac{1}{\sqrt{2}}\begin{bmatrix}1 & 0 & 0 & 1\\ 0 & 1 & 1 & 0\\ -1 & 0 & 0 & 1 \\ 0 & -1 & 1 & 0\end{bmatrix}\mbox{ and }
			D_0=\text{diag}\Big(\mathsf c_4\eta^6,\mathsf c_3\eta^2,2\gamma_0(0),8\pi^2\gamma_2(0)\Big)
			\]
			up to a negligible higher order error, where $\mathsf c_3>0$ and $\mathsf c_4>0$ are constants depending on $\varphi$ and $\varrho$ via $\gamma_i$ and $\nu_i$. As a result, $\lambda_{\eta,3}$ and $\lambda_{\eta,4}$ approach zero at the rate of $\eta^2$ and $\eta^6$ when $\eta\to 0$ respectively, and hence $\det\underline{\Sigma}_\eta\to 0$ at the rate of $\eta^8$ as well. Since eigenvalues continuously depend on $\eta$ and we have had a control when $\eta=1$, we finish the claim by the compactness argument.

			For the third claim, we use the fact that $\underline{ \Sigma}_\eta$ can be well approximated by $U_0D_0U_0^*$ via approximating the eigenstructure. 
			By a direct expansion, we have
			\begin{equation}
			U_0^*\underline{u}=\begin{bmatrix}\sqrt{2}i\sin(\theta) & 2\sqrt{2}\pi ir\cos(\theta+\phi)& -2\sqrt{2}\pi r \sin(\theta+\phi) & \sqrt{2}\cos(\theta)\end{bmatrix}^\top\nonumber
			\end{equation}
			and
			\begin{equation*}
			U_0^*\underline{\mu_{\xi_0,\eta}}=\begin{bmatrix}\sqrt{2}i\sin(2\pi\xi_0t) & 2\sqrt{2}\pi i\xi_0\cos(2\pi\xi_0t)& -2\sqrt{2}\pi \xi_0 \sin(2\pi\xi_0t) & \sqrt{2}\cos(2\pi\xi_0t)\end{bmatrix}^\top\,.\nonumber
			\end{equation*}
			As a result, 
			$\underline{u}^*\underline{ \Sigma}_\eta^{-1}\underline{u}$ is approximated by 
			\begin{align*}
			\underline{u}^*U_0D_0^{-1}U_0^*\underline{u}=\frac{2\sin^2(\theta)}{\mathsf c_4}\eta^{-6}+\frac{8\pi^2r^2\cos^2(\theta+\phi)}{\mathsf c_3}\eta^{-2}+\frac{4\pi^2r^2\sin^2(\theta+\phi)}{\gamma_0(0)}+\frac{\cos^2(\theta)}{4\pi^2\gamma_2(0)}
						\end{align*}
			up to a negligible error. 
		The argument for  $\underline{\mu_{\xi_0,\eta}^*}\underline{ \Sigma}_\eta^{-1} \underline{\mu_{\xi_0,\eta}}$ and $\underline{\mu_{\xi_0,\eta}^*}\underline{ \Sigma}_\eta^{-1} \underline{u}$ follow the same line by setting $\theta=(2\pi \xi_0t \textup{ mod }2\pi)$ and $\omega=\eta-\xi_0$. We thus conclude the proof.
		\end{proof}
	
	\begin{remark}
	Note that for $\underline{\mu_{\xi_0,\eta}^{*}}\underline{ \Sigma}_\eta^{-1} \underline{\mu_{\xi_0,\eta}}$, since $\xi_0$ is positive, the phase $\phi=0$. Thus, when the $\eta^{-6}$ term is zero, the $\eta^{-2}$ term is not zero.   
	If we rewrite $r^2\sin^2(\theta+\phi)=|\omega|^2(1-\cos(2\theta+2\phi))=|\omega|^2-\Re(e^{i2\theta}\omega^2)$, we can further observe the interaction between $\theta$ and $\omega$.
	\end{remark}
	
	The next key lemma is about a simplification of quantities $A^{(\hbar)}_{\eta}(\theta,\omega)$, $B^{(\hbar)}_{\xi_0,\eta}(\theta,\omega)$ and $C^{(\hbar)}_{\xi_0,\eta}$ defined in \eqref{Expansion:ABC}. Note that it only holds when $\eta$ is large. When $\eta$ is small, particularly
	
	\begin{lemma}\label{Control ABC bound over different regions}
	Consider quantities $A^{(\hbar)}_{\eta}(\theta,\omega)$, $B^{(\hbar)}_{\xi_0,\eta}(\theta,\omega)$ and $C^{(\hbar)}_{\xi_0,\eta}$ defined in \eqref{Expansion:ABC}, where we take window to be $\hat{\hbar}=\hat{h}\varphi$ with an even, bounded and smooth function $\varphi$. For $\theta\in [0,\pi]$, $\omega\in \mathbb{C}$, and $\eta\geq1$, we have
	\begin{align*}
			{\mathsf c}^{-1}\boldsymbol{\omega}^* \Gamma_\eta^{-1} \boldsymbol{\omega} &\,\leq A^{(\hbar)}_{\eta}(\theta,\omega)\leq \mathsf c\boldsymbol{\omega}^* \Gamma_\eta^{-1} \boldsymbol{\omega}\\
			{\mathsf c}^{-1}\Re\big[ e^{i\theta}\mu_{\xi_0,\eta}^{(\hbar)*}\Gamma_\eta^{-1}\boldsymbol{\omega} \big] &\, \leq B^{(\hbar)}_{\xi_0,\eta}(\theta,\omega)\leq
			\mathsf c\Re\big[ e^{i\theta}\mu_{\xi_0,\eta}^{(\hbar)*}\Gamma_\eta^{-1}\boldsymbol{\omega} \big]\\
			{\mathsf c}^{-1} \mu^{(\hbar)*}_{\xi_0,\eta}\Gamma_\eta^{-1} \mu^{(h)*}_{\xi_0,\eta} &\,\leq  C^{(\hbar)}_{\xi_0,\eta}\leq \mathsf c \mu^{(\hbar)*}_{\xi_0,\eta}\Gamma_\eta^{-1} \mu^{(\hbar)}_{\xi_0,\eta}
			\end{align*}
			for some $\mathsf c>1$ depending on $\varphi$ and $\rho$. 
	\end{lemma}
	\begin{proof}
	To simplify the notation, when there is no danger of confusion, we omit the superscript ${}^{(\hbar)}$ describing the dependence on the window below.
	 Consider \eqref{Expansion Atheta omega in a different form}; that is, $A_{\eta}(\theta,\omega)=\boldsymbol{\omega}^* \conj{{P}^{-1}} \boldsymbol{\omega} - \Re \big(e^{2i\theta}\boldsymbol{\omega}^\top {R}^\top\conj{ P^{-1}} \boldsymbol{\omega}\big)$. Also note that $C^{(h)}_{\xi_0,\eta}$ is a special $A_{\eta}(\theta,\omega)$ when $\theta=0$ and $\omega=\eta-\xi_0$. First, by Lemma \ref{Appendex:C and P of W for Cov}, $e^{4\pi^2 \eta^2}\conj{C}_\eta$ and $\Gamma_\eta$ are of the same order when $\eta$ is bounded, like when $\eta\leq 10$. On the other hand, $\|R\|$ decays to $0$ exponentially when $\eta$ increases. Thus, $R = \conj{C}_\eta \Gamma_\eta^{-1}$ is negligible in that we have $\|R\|=O(e^{-4\pi^2\eta^2})$ when $\eta\geq 1$, where the implied constant depends on $\phi$ and $\varrho$. Since $\boldsymbol{\omega}^\top {R}^\top\conj{ P^{-1}} \boldsymbol{\omega}=(\overline{R\boldsymbol{\omega}})^*\conj{ P^{-1}} \boldsymbol{\omega}$, we use the Cauchy-Schwartz inequality to get 
	 \[
	 |\Re \big(e^{2i\theta}\boldsymbol{\omega}^\top {R}^\top\conj{ P^{-1}} \boldsymbol{\omega}\big)|\leq |(R\boldsymbol{\omega})^* \conj{ P^{-1}} (R\boldsymbol{\omega})|^{1/2}|\boldsymbol{\omega}^* \conj{ P^{-1}} \boldsymbol{\omega}|^{1/2}\leq \|R\|\boldsymbol{\omega}^* \conj{{P}^{-1}} \boldsymbol{\omega}\,.
	 \] 
	 The control of $\boldsymbol{\omega}^* \conj{{P}^{-1}} \boldsymbol{\omega}$ by $\boldsymbol{\omega}^* \Gamma_\eta^{-1} \boldsymbol{\omega}$ is similar. Indeed, by a direct block matrix inversion, we have
	 \[
	 \conj{{P}^{-1}}=\Gamma_\eta^{-1}+\Gamma_\eta^{-1}C_\eta(\bar{\Gamma}_\eta-C_\eta^*\Gamma_\eta^{-1}{C}_\eta)^{-1}C_\eta^*\Gamma_\eta^{-1}\,.
	 \]
	 By the same argument as above, $\Gamma_\eta^{-1}{C}_\eta$  exponentially when $\eta$ increases, so $C_\eta^*\Gamma_\eta^{-1}\overline{C}_\eta$ is small compared with $\Gamma_\eta^{-1}$. By the inversion approximation 
	 \begin{equation}\label{Proof the inversion approximation expansion}
	 (\bar{\Gamma}_\eta-C_\eta^*\Gamma_\eta^{-1}{C}_\eta)^{-1}=\bar{\Gamma}^{-1}_\eta+\bar{\Gamma}^{-1}_\eta C_\eta^*\Gamma_\eta^{-1}{C}_\eta  \bar{\Gamma}^{-1}_\eta +O(\|\bar{\Gamma}^{-1}_\eta C_\eta^*\Gamma_\eta^{-1}{C}_\eta\|^2 )
	 \end{equation}
	 and the fact that $\bar{\Gamma}^{-1}_\eta C_\eta^*\Gamma_\eta^{-1}{C}_\eta  \bar{\Gamma}^{-1}_\eta={R}^\top\Gamma_\eta^{-1}\overline{R}$, we have 
	 \[
	 |\boldsymbol{\omega}^*(\bar{\Gamma}_\eta -C_\eta^*\Gamma_\eta^{-1}{C}_\eta)^{-1}\boldsymbol{\omega} -\boldsymbol{\omega}^* \bar{\Gamma}^{-1}_\eta \boldsymbol{\omega}|\leq 2\|R\|^2 \boldsymbol{\omega}^* {\Gamma}^{-1}_\eta \boldsymbol{\omega}\,.
	 \]
	 Thus,
	 \begin{align*}
	 &|\boldsymbol{\omega}^* \conj{{P}^{-1}} \boldsymbol{\omega}-\boldsymbol{\omega}^* \Gamma_\eta^{-1} \boldsymbol{\omega}|
	 = |\boldsymbol{\omega}^*\Gamma_\eta^{-1}C_\eta(\bar{\Gamma}_\eta-C_\eta^*\Gamma_\eta^{-1}{C}_\eta)^{-1}C_\eta^*\Gamma_\eta^{-1}\boldsymbol{\omega}|\\
	 \leq & \,|\boldsymbol{\omega}^*\Gamma_\eta^{-1}C_\eta \overline{\Gamma_\eta^{-1}}C_\eta^*\Gamma_\eta^{-1}\boldsymbol{\omega}|+2\|R\|^2|\boldsymbol{\omega}^*\Gamma_\eta^{-1}C_\eta {\Gamma_\eta^{-1}}C_\eta^*\Gamma_\eta^{-1}\boldsymbol{\omega}| \\
	 \leq &\, \|R\|^2\boldsymbol{\omega}^* \conj{{\Gamma}_\eta^{-1}} \boldsymbol{\omega}+2\|R\|^4\boldsymbol{\omega}^* {{\Gamma}_\eta^{-1}} \boldsymbol{\omega}
	 \end{align*}
	 since $C_\eta^*=\overline{C}_\eta$. By putting all together, we have
	 \begin{align*}
	 |A_{\eta}(\theta,\omega)-\boldsymbol{\omega}^* {{\Gamma}_\eta^{-1}} \boldsymbol{\omega}|&\,\leq  (\|R\|^2+\|R\|^3)\boldsymbol{\omega}^* \conj{{\Gamma}_\eta^{-1}} \boldsymbol{\omega}+(\|R\|+2\|R\|^4+2\|R\|^5)\boldsymbol{\omega}^* {{\Gamma}_\eta^{-1}} \boldsymbol{\omega}\\
	 &\,\leq 2\|R\| \boldsymbol{\omega}^* {{\Gamma}_\eta^{-1}} \boldsymbol{\omega}+2\|R\|^2\boldsymbol{\omega}^* \conj{{\Gamma}_\eta^{-1}} \boldsymbol{\omega}\,.
	 \end{align*}
	 Since $\|R\|^2\boldsymbol{\omega}^* \conj{{\Gamma}_\eta^{-1}} \boldsymbol{\omega}$ is negligible compared with $\|R\| \boldsymbol{\omega}^* {{\Gamma}_\eta^{-1}} \boldsymbol{\omega}$ and $2\|R\|<1$, we obtain the proof.
	 The proof of $B_{\xi_0,\eta}(\theta,\omega)$ is similar.
	 \end{proof}
	 
	 The next two lemmas are about the eigenstructure perturbation when the window is perturbed. 

\begin{lemma}[$A^{(\hbar)}_{\eta}(\theta,\omega)$, $B^{(\hbar)}_{\xi_0,\eta}(\theta,\omega)$ and $C^{(\hbar)}_{\xi_0,\eta}$ deformation caused by window perturbation]\label{Lemma ABC perturbation after window perturbation}
			Follow the same notations used in Lemma \ref{Appendex:C and P of W for Cov}. Fix $\eta>0$.
			Assume two symmetric, bounded and smooth functions, $\varphi_1\neq \varphi_2$, and satisfy 
			$\varphi_1\varphi_2\neq 0$, 
			$\max_{l,k=1,2}|\gamma^{[\varphi_1,\varphi_1]}_i(\eta)-\gamma^{[\varphi_l,\varphi_k]}_i(\eta)|\leq\epsilon$ and 			$\max_{l,k=1,2}|\nu^{[\varphi_1,\varphi_1]}_i(\eta)-\nu^{[\varphi_l,\varphi_k]}_i(\eta)|\leq\epsilon$ for some small $\epsilon>0$ for $i=0,1,2$. 

			\end{lemma}
	 
	 \begin{proof}
	 For $i=1,2$, denote
	 \begin{align*}
			A^{(h_i)}_{\eta}(\theta,\omega)&=\boldsymbol{\omega}^* \conj{{P}_i^{-1}} \boldsymbol{\omega} - \Re \Big(e^{2i\theta}\boldsymbol{\omega}^\top {R}_i^\top\conj{ P_i^{-1}} \boldsymbol{\omega}\Big)\\
			B^{(h_i)}_{\xi_0,\eta}(\theta,\omega)&=
			\Re\Big[ e^{i\theta}\Big(\mu_{\xi_0,\eta}^{(h_i)*}-\mu_{\xi_0,\eta}^{(h_i)\top} R_i^\top\Big)\conj{ P_i^{-1}}\boldsymbol{\omega} \Big]\nonumber\,,
			\end{align*}
			where $P_i = \conj{\Gamma}_\eta^{[\varphi_i,\varphi_i]}- \conj{C_\eta^{[\varphi_i,\varphi_i]}} \Gamma_\eta^{[\varphi_i,\varphi_i]-1} C_\eta^{[\varphi_i,\varphi_i]}$ and $R_i= \conj{C}^{[\varphi_i,\varphi_i]}_\eta \Gamma_\eta^{[\varphi_i,\varphi_i]-1}$. To simplify the notation, we denote $\Gamma_i:={\Gamma}_\eta^{[\varphi_i,\varphi_i]}$ and $C_i:=C_\eta^{[\varphi_i,\varphi_i]}$ below.
	 Again, since $C^{(h_i)}_{\xi_0,\eta}$ is a special case of $A^{(h_i)}_{\eta}(\theta,\omega)$, its results follow directly from those for $A^{(h_i)}_{\eta}(\theta,\omega)$.
	 When $\eta\geq 1$, the proof is based on the same approximation technique used in Lemma \ref{Control ABC bound over different regions}, particularly the inversion formula \eqref{Proof the inversion approximation expansion}. 
	 \end{proof}
	 
	The following lemma will be used to control the precision matrix perturbation.
	
	\begin{lemma}[Covariance deformation caused by window perturbation]\label{Lemma covariance eigenstructure after perturbation}
			Follow the same notations used in Lemma \ref{Appendex:C and P of W for Cov}. Fix $\eta>0$.
			Assume two symmetric, bounded and smooth functions, $\varphi_1\neq \varphi_2$, and satisfy 
			$\varphi_1\varphi_2\neq 0$, 
			$\max_{l,k=1,2}|\gamma^{[\varphi_1,\varphi_1]}_i(\eta)-\gamma^{[\varphi_l,\varphi_k]}_i(\eta)|\leq\epsilon$ and 
						$\max_{l,k=1,2}|\nu^{[\varphi_1,\varphi_1]}_i(\eta)-\nu^{[\varphi_l,\varphi_k]}_i(\eta)|\leq\epsilon$  
			for some small $\epsilon>0$ for $i=0,1,2$. 
			Denote the eigendecomposition of $\underline{\Sigma}_\eta^{(h_1)}=U_2D_2U_2^*$, where $U_2\in U(4)$ and $D_2\in \mathbb{R}^{4\times 4}$ is diagonal. 
			Then, when $\epsilon$ is sufficiently small, the eigendecomposition of $\underline{\Sigma}_{\eta,\eta}^{(h_1, h_2)}={U}_4 D_4{U}_4^*$, where $U_4\in U(4)$ and $D_4$ is diagonal, satisfies
			\begin{equation}
				D_4 =\begin{bmatrix}
				0 & 0 \\ 0 & 2D_2 
				\end{bmatrix}+\begin{bmatrix}
				\Lambda_1'& 0 \\ 0 & \Lambda_2'	
				\end{bmatrix},\nonumber
			\end{equation}
			where $\Lambda_1'$ and $\Lambda_2'$ are of order $O(\epsilon)$ with the implied constants depending on $\eta$, $\varphi_i$ and $\varrho$, and 
			\begin{equation}\label{expansion:U4}
				U_4=\frac{1}{\sqrt{2}}\begin{bmatrix}
				U_2(\Gamma_1+\Gamma_1 B-\Gamma_2 C^*) & U_2(\Gamma_2+\Gamma_1 C+\Gamma_2 D)\\
				-U_2(\Gamma_1+\Gamma_1 B+\Gamma_2C^*) & U_2(\Gamma_2-\Gamma_1C+\Gamma_2 D)
				\end{bmatrix}\,.\nonumber
						\end{equation}
			where $\Gamma_1$ and $\Gamma_2$ can be found in \eqref{Expansion Gamma2 and Lambda2'} and $B,C$ and $D$ are of order $O(\epsilon)$ with the implied constants depending on $\eta$, $\varphi_i$ and $\varrho$.
		\end{lemma}

		\begin{proof}
		To ease the intense notation, we use the superscript $[i,j]$ to replace $[\varphi_i,\varphi_j]$ in this proof. We study the relationship between $\underline{\Sigma}_{\eta,\eta}^{(h_1,h_2)}\in \mathbb{C}^{8\times 8}$ and $\underline{\Sigma}_{\eta}^{(h_1)}\in \mathbb{C}^{4\times 4}$ by exploring the perturbed eigenstructure of $\underline{\Sigma}_{\eta,\eta}^{(h_1,h_2)}$. Denote
			\begin{equation}
				{\Sigma}:=\begin{bmatrix}I\\I\end{bmatrix} \underline{\Sigma}_\eta^{(h_1)}\begin{bmatrix}I&I\end{bmatrix}\,,
			\end{equation}
			where $I$ is a $4\times 4$ identity matrix.
			Then ${\Sigma}$ and $\underline{\Sigma}_{\eta,\eta}^{(h_1, h_2)}$ have a simple relationship:
			\begin{equation}\label{definition of P permutation first time}
				\mathsf P\underline{\Sigma}_{\eta,\eta}^{(h_1,h_2)} \mathsf P^\top={\Sigma}+E_4\,,
			\end{equation}
			where $\mathsf P$ is a permutation matrix mapping $[x_1,x_2,x_3,x_4,x_5,x_6,x_7,x_8]^\top\in \mathbb{C}^8$ to $[x_1,x_2,x_5,x_6,x_3,x_4,x_7,x_8]^\top\in \mathbb{C}^8$, and 
			\begin{equation}\label{Definition E4 error matrix for window perturbation}
				E_4:=\begin{bmatrix}
				0&  E_2\\
				{E}^*_2 & \breve E_2
				\end{bmatrix}\,,
			\end{equation}
			where 
			\[
			 E_2=\begin{bmatrix} \Gamma^{[1,1]}_{\eta,\eta}-\Gamma^{[1,2]}_{\eta,\eta}  & C^{[1,1]}_{\eta,\eta}-C^{[1,2]}_{\eta,\eta}  \\
		 \overline{C^{[1,1]}_{\eta,\eta}}-\overline{C^{[1,2]}_{\eta,\eta}} & \overline{\Gamma^{[1,1]}_{\eta,\eta}}-\overline{\Gamma^{[1,2]}_{\eta,\eta}} 
		 \end{bmatrix}
			\]
			and
			\[
			 \breve E_2=\begin{bmatrix} \Gamma^{[1,1]}_{\eta,\eta}-\Gamma^{[2,2]}_{\eta,\eta}  & C^{[1,1]}_{\eta,\eta}-C^{[2,2]}_{\eta,\eta}  \\
		 \overline{C^{[1,1]}_{\eta,\eta}}-\overline{C^{[2,2]}_{\eta,\eta}} & \overline{\Gamma^{[1,1]}_{\eta,\eta}}-\overline{\Gamma^{[2,2]}_{\eta,\eta}} 
		 \end{bmatrix}\,.
			\]
			Here, $E_2$ contains entries of $\gamma^{[1,1]}_i(\eta)-\gamma^{[1,2]}_i(\eta)$ and $\nu^{[1,1]}_i(\eta)-\nu^{[1,2]}_i(\eta)$ and $\breve E_2$ contains entries of $\gamma^{[1,1]}_i(\eta)-\gamma^{[2,2]}_i(\eta)$ and $\nu^{[1,1]}_i(\eta)-\nu^{[2,2]}_i(\eta)$, so that they are of order $O(\epsilon)$ by assumption. 

			Now we evaluate the perturbation bound. Clearly, due to the non-zero pseudocovariance, $\underline{\Sigma}_\eta^{(h_1)}$ has four distinct eigenvalues, and ${\Sigma}$ is of rank $4$ and has 4 distinct non-zero eigenvalues. A (non-unique) eigendecomposition of ${\Sigma}$ naturally becomes $ U  D U^*$,
			where 
			\begin{equation}
				 U:=\frac{1}{\sqrt{2}}\begin{bmatrix}
				U_2 & U_2 \\ -U_2 & U_2
				\end{bmatrix}\in U(8),\quad D:=\begin{bmatrix}
				0 & 0 \\ 0 & 2D_2 
				\end{bmatrix}\in \mathbb{R}^{8\times 8}.
			\end{equation} 
			Based on the eigendecomposition of ${\Sigma}$, we apply the perturbation calculation of the eigensystem problem \cite{VanDerAn_TerMorsche:2007} to approximate the eigenvalues and eigenvectors of the eigendecomposition 
			\begin{align}
			\underline{\Sigma}_{\eta,\eta}^{(h_1, h_2)}={U}_4 D_4{U}_4^*. 
			\end{align}
			
			To bound the eigenvalue derivative of $\underline{\Sigma}_{\eta,\eta}^{(h_1, h_2)}$ from that of ${\Sigma}$ when $\epsilon$ is sufficiently small, we apply \cite[(3.7d)]{VanDerAn_TerMorsche:2007}. When $\epsilon$ is sufficiently small, after a direct expansion, we get the first order approximation of eigenvalues of $\underline{\Sigma}_{\eta,\eta}^{(h_1,h_2)}$; that is, 
			\begin{equation}
				D_4 = D+\begin{bmatrix}
				\Lambda_1'& 0 \\ 0 & \Lambda_2'	
				\end{bmatrix},
			\end{equation}
			where $\Lambda_1'$ and $\Lambda_2'$ are diagonal with entries of order $\epsilon$. Indeed, $\Lambda_1'$ and $\Lambda_2'$ can be approximated by solving the following eigenvalue problems \cite[(3.7a) and (3.7d)]{VanDerAn_TerMorsche:2007}
			\begin{align}
			\frac{1}{\sqrt{2}}\begin{bmatrix}
			U_2^* & -U_2^*
			\end{bmatrix}E_4\frac{1}{\sqrt{2}}\begin{bmatrix}
			U_2\\ -U_2
			\end{bmatrix}\Gamma_1=\frac{1}{2}U_2^*(\breve E_2-2\Re E_2) U_2\Gamma_1=\Gamma_1\Lambda_1'\label{Expansion Gamma2 and Lambda2'}\\
			\frac{1}{\sqrt{2}}\begin{bmatrix}
			U_2^* & U_2^*
			\end{bmatrix}E_4\frac{1}{\sqrt{2}}\begin{bmatrix}
			U_2\\ U_2
			\end{bmatrix}\Gamma_2=\frac{1}{2}U_2^*(\breve E_2+2\Re E_2) U_2\Gamma_2=\Gamma_2\Lambda_2'\nonumber
			\end{align}
			for some $\Gamma_1\in U(4)$ and $\Gamma_2\in U(4)$.
			Therefore, $\Lambda_1'$ and $\Lambda_2'$ are eigenvalues  of $\breve E_2-2\Re E_2$ and $\breve E_2+2\Re E_2$ respectively. 
			When $\eta$ is sufficiently large, the peudocovariance part becomes exponentially small. Thus, $E_2$ and $\breve E_2$, and hence $\breve E_2-2E_2$ and $\breve E_2+2E_2$, get closer to a diagonal block matrix. As a result, $\Lambda_1'$ and $\Lambda_2'$ are of order $\epsilon$. 
			
			When $\eta\to 0$, we need to
			take a closer look at $\breve E_2-2\Re E_2$ as a $2\times 2$ block matrix. By a direct expansion, the $(1,1)$-th block is $-\overline{\Gamma^{[1,1]}_{\eta,\eta}}+2\Re \Gamma^{[1,2]}_{\eta,\eta} - \Gamma^{[2,2]}_{\eta,\eta}$, which reads like, where we omit the dependence on $\eta$ for the $\gamma$ and $\nu$ terms,
			\[
			-\begin{bmatrix}
			\gamma^{[1,1]}_0+\gamma^{[2,2]}_0 -2\gamma^{[1,2]}_0 & -2\pi i(\gamma^{[1,1]}_0-\gamma^{[2,2]}_0) \\
			2\pi i(\gamma^{[1,1]}_0-\gamma^{[2,2]}_0) & 4\pi^2(\gamma^{[1,1]}_2+\gamma^{[2,2]}_2-2\gamma^{[1,2]}_2)
			\end{bmatrix}\,,
			\]
			 the $(1,2)$-th block is $-\overline{C^{[1,1]}_{\eta,\eta}}+2\Re C^{[1,2]}_{\eta,\eta} - C^{[2,2]}_{\eta,\eta}$, which reads like
			 \[
			-e^{-4\pi^2\eta^2}\begin{bmatrix}
			\nu^{[1,1]}_0+\nu^{[2,2]}_0-2\nu^{[1,2]}_0 
			& 2\pi i(\nu^{[1,1]}_1-\nu^{[2,2]}_1-\eta(\nu^{[1,1]}_0-\nu^{[2,2]}_0)) \\
			 2\pi i(-\nu^{[1,1]}_1+\nu^{[2,2]}_1-\eta(\nu^{[1,1]}_0-\nu^{[2,2]}_0)) & 4\pi^2(\nu^{[1,1]}_2+\nu^{[2,2]}_2-2\nu^{[1,2]}_2)
			\end{bmatrix}\,,
			\]
			and the $(2,1)$-th and $(2,2)$-th entries are the complex conjugation of the $(1,2)$-th and $(1,1)$-th entries respectively.
			Clearly, the pseudocovariance part gets closer to the covariance part when $\eta\to 0$. By an argument similar to that for the second part of Lemma \ref{Lemma: c of mu_xi0eta when eta small}, we obtain the claim.

			To bound the derivation of eigenvectors of $\underline{\Sigma}_{\eta,\eta}^{(h_1, h_2)}$ from those of ${\Sigma}$ when $\epsilon$ is sufficiently small, we apply \cite[(2.3),(2.5),(3.9a),(3.9b),(3.11),(3.5)]{VanDerAn_TerMorsche:2007}.  
			To this end, we need to handle the non-unique eigendecomposition since the zero eigenvalue of ${\Sigma}$ has a multiplicity exceeding $1$. Suppose $U_4$ is perturbed from an eigenvector matrix $U$ of ${\Sigma}$, and $\tilde U$ and $ U$ are related by $\Gamma_4$; that is, $\tilde U=U\Gamma_4$ \cite[(2.5)]{VanDerAn_TerMorsche:2007}, where  
			\begin{equation}
				\Gamma_4=\begin{bmatrix}
				\Gamma_1 &0 \\ 0& \Gamma_2
				\end{bmatrix}\in U(8)\,\label{calculation of Gamma 2 and Gamma4}
			\end{equation}
			is calculated directly by \cite[(3.7)]{VanDerAn_TerMorsche:2007} or \eqref{Expansion Gamma2 and Lambda2'}.
			Hence, the form of $\tilde U$ is confirmed to be
			\begin{equation}
				\tilde U:=\frac{1}{\sqrt{2}}\begin{bmatrix}
				U_2\Gamma_1 & U_2\Gamma_2 \\ -U_2\Gamma_1 & U_2\Gamma_2
				\end{bmatrix}\in U(4)\,.
			\end{equation}
			The eigenvector derivative, denoted as $U'$, comes from evaluating $\tilde U^{-1}U'$ \cite[(2.3)]{VanDerAn_TerMorsche:2007}, which can be evaluated by carrying out \cite[(3.9a),(3.9b),(3.11),(3.5)]{VanDerAn_TerMorsche:2007} sequentially. As a result, we have
			\begin{equation}
				U_4
				=
				\tilde U+\tilde U
				\begin{bmatrix}
					B & C \\ 
					\tilde C & D 
				\end{bmatrix},
			\end{equation}
			where $B,C,\tilde C$ and $D$ are of order $\epsilon$. 
			For example, by \cite[(3.9b)]{VanDerAn_TerMorsche:2007},
			\begin{equation}\label{Expansion tilde C}
			\tilde C=-(2D_2)^{-1}\Gamma_1^* U_2^*\breve E_2 U_2\Gamma_2=-C^*\,,
			\end{equation}
			which is of order $\epsilon$.
			Note that while by \cite{VanDerAn_TerMorsche:2007} we can write done a more precise error term, the bound found here is sufficient since we only care about the case $\epsilon\to 0$.
			Thus, we have
			\begin{equation}
				U_4=\frac{1}{\sqrt{2}}\begin{bmatrix}
				U_2(\Gamma_1+\Gamma_1 B-\Gamma_2 C^*) & U_2(\Gamma_2+\Gamma_1 C+\Gamma_2 D)\\
				-U_2(\Gamma_1+\Gamma_1 B+\Gamma_2C^*) & U_2(\Gamma_2-\Gamma_1C+\Gamma_2 D)
				\end{bmatrix}\,.\nonumber
			\end{equation}
		\end{proof}
	
	\begin{remark}
	Note that for $\Lambda_1'$, the error $\breve E_2-2E_2$ depends on $(\varphi_1-\varphi_2)^2$ and for $\Lambda_1'$, the error $\breve E_2+2E_2$ depends on $4[\varphi_1^2-(\frac{\varphi_1+\varphi_2}{2})^2]$ since, for example, 
			\[
			 \gamma^{[\varphi_2,\varphi_2]}_l(\eta)+\gamma^{[\varphi_1,\varphi_1]}_l(\eta)-2\gamma^{[\varphi_1,\varphi_2]}_l(\eta)=\int e^{-4\pi^2(\xi+\eta)^2}(\xi+\eta)^l(\varphi_1(\xi+\eta)-\varphi_2(\xi+\eta))^2p(\xi)d\xi\,. 
			\]
			This means that the error depends on how two windows differ.
		Next, recall that when $\eta$ is small, by Lemma \ref{Lemma: c of mu_xi0eta when eta small}, the smallest eigenvalues in $D_2$ is of order $\eta^6$, which comes from the degeneracy of the augmented covariance matrix. Thus, the error matrix $E_4$ does have further structures that could be further explored. However, for our application, the provided bound is sufficient. 
		\end{remark}

	For a given Schwartz window function $\hbar$, we need the following Lemma evaluating the density of $\B{Z}^{(\hbar)}_{\alpha,\xi_0,\eta} \vcentcolon= \begin{bmatrix}Y^{(\hbar)}_{\alpha,\eta}& {\Omega}^{(\hbar)}_{\eta}\end{bmatrix}^\top$ when we study the moments of $Y_{f+\Phi}^{(\hbar,\alpha,\xi)}(t,\eta)$.

		\begin{lemma}\label{lemma:Z density function}
			Suppose Assumptions~\ref{assump:noise part} and \ref{assump:nonnull signal} hold and $\hbar$ is the given Schwartz window function. For $\eta>0$, $\alpha>0$, and $\xi>0$, the density function of $\B{Z}^{(\hbar)}_{\alpha,\xi_0,\eta}\in\mathbb{C}^2$ is 
			\begin{align}
			f_{\B{Z}^{(\hbar)}_{\alpha,\xi_0,\eta}}(y,\omega)=\frac{4 |y|^2E^4_{\alpha,\xi}(\omega)}{\sqrt{\det{\underline{\Sigma}_\eta^{(\hbar)}}}}
			\exp\Big[-\frac{1}{2}\big(\underline{g_Y^{-1}(y,\omega)}-\underline{\mu^{(\hbar)}_{\xi_0,\eta}}\big)^* \underline{\Sigma}_\eta^{(\hbar)-1} \big(\underline{g_Y^{-1}(y,\omega)}-\underline{\mu^{(\hbar)}_{\xi_0,\eta}}\big)\Big]\,,
			\label{expansion Z distribution}
			\end{align}
			where $g_Y$ is defined in Lemma \ref{Lemma:ChangeVariableJacobianC2}
			and
			$E_{\alpha,\xi}(\omega)$ is defined in \eqref{Definition:Ealphaxi}.
		\end{lemma}

		\begin{proof}
			To reduce notation load, we omit the dependence on $\hbar$ in the superscript in the proof.
			Consider the complex change of variables $g_Y:\CC^2\setminus\{(0,z_2):z_2\in\CC\}\to \CC^2\setminus\{(0,z_2):z_2\in\CC\}$ discussed in Lemma \ref{Lemma:ChangeVariableJacobianC2} to build $\B{Z}_{\alpha,\xi_0,\eta}$ from the Gaussian random vector $\B{U}_{\xi_0,\eta}\vcentcolon=\mu_{\xi_0,\eta}+\B{W}_\eta$, where $\B{W}_\eta = \begin{bmatrix}\Phi(\hbar_{t,\eta})&\Phi((\hbar')_{t,\eta})\end{bmatrix}^\top$.  
			By recalling the density function of $\B{W}_\eta$, we know that the density function of $\B{U}_{\xi_0,\eta}$ satisfies 
			\begin{align}\label{distribution of V}
			f_{\B{U}_{\xi_0,\eta}}(z) &\,=
			\frac{1}{\pi^{2} \sqrt{\det{\underline{\Sigma}_\eta}}}
			\exp\Big[-\frac{1}{2}\big(\underline{z}-\underline{\mu_{\xi_0,\eta}}\big)^* \underline{\Sigma}_\eta^{-1} \big(\underline{z}-\underline{\mu_{\xi_0,\eta}}\big)\Big]\,,
			\end{align}
			where $\underline{\Sigma}_\eta=\begin{bmatrix}\Gamma_\eta & C_\eta \\ \conj{C_\eta} & \conj{\Gamma_\eta}\end{bmatrix}$ is the augmented covariance matrix and $z\in \mathbb{C}^2$. 
			Therefore, with $|\det{J(y,\omega)}|$ evaluated in Lemma \ref{Lemma:ChangeVariableJacobianC2}, we have the claim that
			\begin{equation}
			f_{\B{Z}_{\alpha,\xi_0,\eta}}(y,\omega)=|\det{J(y,\omega)}|f_{\B{U}_{\xi_0,\eta}}(g_Y^{-1}(y,\omega))\,.
			\end{equation}

		\end{proof}

	The following Lemma is another expression of the moments of $Y_{f+\Phi}^{(h,\alpha,\xi)}(t,\eta)$ that is convenient for the upcoming perturbation argument. Note that it is different from the expression used in \eqref{var eq intermediate}.

		\begin{lemma}\label{Lemma:ChangeOfVariableVariance}
			Suppose Assumptions~\ref{assump:noise part} and \ref{assump:nonnull signal} hold. Suppose $\hbar$ is a given Schwartz window function. For $k\in\mathbb{N}$, $\eta>0$, $\xi>0$ and $\alpha>0$, we have
			\begin{align}
						\mathbb E \left[Y_{f+\Phi}^{(\hbar,\alpha,\xi)}(t,\eta)\right]^k=\frac{1}{\pi^2\sqrt{\text{det}\underline\Sigma^{(\hbar)}_\eta}}\iint e^{-\frac{1}{2}\underline{\boldsymbol{z}}^* \underline{\Sigma}_\eta^{(\hbar)-1}\underline{\boldsymbol{z}}} \left(\frac{p+\mu_1}{E_{\alpha,\xi}\Big(\frac{1}{2\pi i}
			\frac{q+\mu_2}{p+\mu_1}\Big)}\right)^k\textup dp \textup dq\,\nonumber,\\
			\mathbb E \left|Y_{f+\Phi}^{(\hbar,\alpha,\xi)}(t,\eta)\right|^k=\frac{1}{\pi^2\sqrt{\text{det}\underline\Sigma^{(\hbar)}_\eta}}\iint e^{-\frac{1}{2}\underline{\boldsymbol{z}}^* \underline{\Sigma}_\eta^{(\hbar)-1}\underline{\boldsymbol{z}}} \left|\frac{p+\mu_1}{E_{\alpha,\xi}\Big(\frac{1}{2\pi i}
			\frac{q+\mu_2}{p+\mu_1}\Big)}\right|^k\textup dp \textup dq\,\nonumber,
			\end{align}
			where $\mu_1=A\hat{\hbar}(\xi_0-\eta)e^{i2\pi \xi_0 t}$, $\mu_2=i2\pi \mu_1$, $\boldsymbol{z}=\begin{bmatrix}p& q\end{bmatrix}^\top\in \mathbb{C}^2$ and the integrand domain is $p\in \mathbb{C}\backslash\{-\mu_1\}$ and $q\in \mathbb{C}$. 
		\end{lemma}
		\begin{proof}
			To reduce notation load, we omit the dependence on $\hbar$ in the superscript in the proof.
			The proof is by a straightforward change of variable. Define $\B{U}_{\xi_0,\eta}:=\mu_{\xi_0,\eta}+\B{W}_\eta$. 			Then, by Lemma \ref{Lemma:ChangeVariableJacobianC2}, we have
			\begin{align}
			\mathbb{E}\left[Y_{f+\Phi}^{(h,\alpha,\xi)}(t,\eta)\right]^k=4\pi^2\int_{\mathbb{C}\backslash\{0\}}x^k  \int_{\mathbb{C}} |x|^2 E^4_{\alpha,\xi}(\omega) f_{\B{U}_{\xi_0,\eta}}(g_Y^{-1}(x,\omega))\textup d\omega \textup dx \,,\nonumber
			\end{align}
			where $g_Y$ is defined in \eqref{Definition:g1invg2inv}. Set
			\begin{align}
			p:=xE_{\alpha,\xi}(\omega)-\mu_1,\,\,\,q:=2\pi i \omega x E_{\alpha,\xi}(\omega)-\mu_2\,,
			\end{align}
			which is equivalent to $\omega=\frac{q+\mu_2}{2\pi i(p+\mu_1)}$ and $x=\frac{p+\mu_1}{E_{\alpha,\xi}\big(\frac{q+\mu_2}{2\pi i(p+\mu_1)}\big)}$. Here, note that since $x$ is not defined at $0$ and $E_{\alpha,\xi}(\omega)$ is nonzero, $p\in \mathbb{C}\backslash\{-\mu_1\}$ and $q\in \mathbb{C}$.
			By denoting $\boldsymbol{z}=\begin{bmatrix}p& q\end{bmatrix}^\top$, we have the claim, since the Jacobian of changing $(x,\omega)$ to $(p,q)$ is $(2\pi)^{-2} |x|^{-2} E^{-4}_{\alpha,\xi}(\omega)$.  The absolute moments follows by the same argument immediately.
		\end{proof}

	\section{Argument for the reassignment rule}\label{OS section Proof of Proposition prop:Q convergence to proper case}

We will omit the subscript ${}^{(h)}$ to simplify the notation. When $\eta=\xi_0$, we have $\mu=\begin{bmatrix}f(t)& 0\end{bmatrix}$. Denote $Q_{\xi_0}:=Q_{f+\Phi}(t,\xi_0)$.  When $\xi_0$ is sufficiently large so that the pseudo-covariance of $\mathbf{W}_{\xi_0}$ is negligible, to evaluate the mean of $Q_{f+\Phi}(t,\xi_0)$, we could first replace the augmented covariance matrix $\underline{\Sigma}_{\xi_0}$ by 
\[
\underline{\Upsilon}_{\xi_0}=\begin{bmatrix} \Gamma_{\xi_0} & 0 \\ 0 & \overline{\Gamma}_{\xi_0} \end{bmatrix}
\]
and denote the resulting random variable as $Z_{\xi_0}$. By a direct calculation,  the density function of $Z_{\xi_0}$ is 
				\begin{align}
				f_{Z_{\xi_0}}(q) 
				&= 
					\frac{e^{-\frac{1}{2}\underline{\mu}^*\underline{\Upsilon}_{\xi_0}^{-1} \underline{\mu} }}
					{\pi^2 \sqrt{\det{\underline{\Upsilon}_{\xi_0}}}}
					\int_0^{\pi}
					\Hypergeometric{1}{1}{2}{\frac{1}{2}}{\frac{
					\mathfrak B(\theta,q)^2
					}{
					\mathfrak A(q)
					}}
					 \,d\theta\frac{1}{
					\mathfrak A(q)^2
					}\,,
			\end{align}
			where
			\begin{equation*}
				\mathfrak A(q)=\textbf{q}^*\Gamma_{\xi_0}^{-1}\textbf{q} \quad\mbox{and}\quad\mathfrak B(\theta,q)=\Re[e^{i\theta}\mu^*\Gamma_{\xi_0}^{-1}\textbf{q}]\,.
				\end{equation*}
				Note that while in general $\mathfrak A$ does not have the symmetry $\mathfrak A(-q)=\mathfrak A(q)$, since the off-diagonal entries of $\Gamma_{\xi_0}$ is imaginary, it has the symmetry $\mathfrak A(-\overline{q})=\mathfrak A(q)$. Indeed, by a direct expansion, $\mathfrak A(q)=\det(\Gamma_{\xi_0})^{-1}[\gamma_2(\xi_0)+\gamma_0(\xi_0)|q|^2-4\pi \gamma_1(\xi_0)\Im q]$ and $\Im q=\Im(-\overline{q})$.

				Furthermore, when the noise is white, we know $\gamma_1(\xi_0)=0$, so $\Gamma_{\xi_0}$ is diagonal.
In this case,
			$\mathfrak A(q)$ further has the symmetry that $\mathfrak A(q)=\mathfrak A(-q)$ (that is, Theorem \ref{thm: quotient mean bound proposition} (iii)), and $B_\mu(\theta,q)$ does not depend on $q$. Thus, by the same derivation for \eqref{Equation of f_Q when C=0}, we have
\begin{align*}
			\mathbb{E}Z_{\xi_0}=\frac{e^{-\frac{1}{2}\underline{\mu}^*\underline{\Upsilon}_{\xi_0}^{-1} \underline{\mu} }}
					{\pi^2 \sqrt{\det{\underline{\Upsilon}_{\xi_0}}}}
					\int_0^{\pi} \Big[\int_{\mathbb{C}}
					\Hypergeometric{1}{1}{2}{\frac{1}{2}}{\frac{
					\mathfrak B_\mu(\theta,q)^2
					}{
					\mathfrak A(q)
					}}
					\frac{q}{
					\mathfrak A(q)^2
					}
					 \,dq \Big] d\theta=0\,.
			\end{align*}
			On the other hand, the expectation of $Q_{\xi_0}$ is
			\begin{align*}
				\mathbb{E}Q_{\xi_0}=\frac{e^{-\frac{1}{2}\underline{\mu}^*\underline{\Sigma}_{\xi_0}^{-1} \underline{\mu} }}
					{\pi^2 \sqrt{\det{\underline{\Sigma}_{\xi_0}}}}
					\int_0^{\pi}\int_{\mathbb{C}}
					\Hypergeometric{1}{1}{2}{\frac{1}{2}}{\frac{
					B_{\xi_0,\xi_0}(\theta,q)^2
					}{
					A_{\xi_0}(\theta,q)
					}}
					\frac{q}{
					A_{\xi_0}(\theta, q)^2
					}dq
					 \,d\theta\,,
			\end{align*}
Since the pseudo-covariance of $\mathbf{W}_{\xi_0}$ is negligible, we could apply the perturbation argument to obtain the claim. Indeed, by definition, since $\nu_i(\eta)\asymp \eta^{\rho}$ for $i=1,2,3$ when $\eta$ is large, the pseudo-covariance $C_{\xi_0}$ satisfies $\|C_{\xi_0}\|=O(e^{-4\pi^2\xi_0^2}\xi_0^{\varrho})$, which is exponentially small when $\xi_0$ is large. Thus, by a direct expansion, we have 
\[
\overline{C}_{\xi_0}\Gamma_{\xi_0}^{-1}C_{\xi_0}=O(e^{-8\pi^2\xi_0^2}\xi_0^{-\varrho})
\]
since $\|\Gamma_{\xi_0}\|\asymp \xi_0^{\varrho}$, which leads to the control of inverse of $P_{\xi_0}:=\overline{\Gamma}_{\xi_0}-\overline{C}_{\xi_0}\Gamma_{\xi_0}^{-1}C_{\xi_0}$. 
That is, 
\[
\mathbf{q}^*\overline{P}^{-1}_{\xi_0}\mathbf{q}=\mathbf{q}^*{\Gamma}^{-1}_{\xi_0}\mathbf{q}+O(e^{-8\pi^2\xi_0^2}\xi_0^{-\varrho}\|\mathbf{q}\|^2)\,, 
\]
where the fact $\|\Gamma_{\xi_0}\|\asymp \xi_0^{\varrho}$ is again used (see the proof of Lemma \ref{Control ABC bound over different regions} for details). 
The other term in $A_{\xi_0}(\theta, q)$ that involves 
\[
R_{\xi_0}:=\overline{C}_{\xi_0}\Gamma_{\xi_0}^{-1}=O(e^{-4\pi^2\xi_0^2}\xi_0^{-\varrho})
\] 
is controlled in the same way so that 
\[
\|\Re(e^{2\pi \theta}\mathbf{q}R_{\xi_0}^\top\overline{P}_{\xi_0}^{-1}\mathbf{q})\|= O(e^{-4\pi^2\xi_0^2}\xi_0^{-2\varrho}\|\mathbf{q}\|^2)\,.
\] 
By combining the above bounds together, we have
\[
|A_{\xi_0}(\theta, q)-\mathfrak A(q)|=O(e^{-4\pi^2\xi_0^2}\xi_0^{-2\varrho}\|\mathbf{q}\|^2)\,.
\] 
By a similar argument that we omit, we have 
\[
|B_{\xi_0,\xi_0}(\theta, q)-\mathfrak B(q)|=O(e^{-4\pi^2\xi_0^2}\xi_0^{-2\varrho}A\|\mathbf{q}\|)\,.
\] 
Since $\|\underline{\Sigma}_{\xi_0}^{-1}-\underline{\Upsilon}^{-1}_{\xi_0}\|=O(e^{-4\pi^2\xi_0^2}\xi_0^{-2\varrho})$ and all eigenvalues of $\underline{\Sigma}_{\xi_0}$ are of order $\xi_0^{\rho}$ by Lemma \ref{Lemma: c of mu_xi0eta when eta small}, by Weyl's inequality for the eigenvalue perturbation, we have 
\[
|(\det{\underline{\Sigma}_{\xi_0}})^{-1/2}-(\det{\underline{\Upsilon}_{\xi_0}})^{-1/2}|=O(e^{-4\pi^2\xi_0^2}\xi_0^{\varrho})\,.
\] 
By another direct bound, we have 
\[
|\underline{\mu}^*\underline{\Sigma}_{\xi_0}^{-1} \underline{\mu} - \underline{\mu}^*\underline{\Upsilon}_{\xi_0}^{-1} \underline{\mu}|=O(e^{-4\pi^2\xi_0^2}\xi_0^{-2\varrho}A^2)\,.
\] 
By plugging these bounds, we obtain the claim.

	\section{Properties of 
	 the integrand in the synchrosqueezing transform}\label{Proof of Proposition joint density Y and omega}

		In this section, we prove three theorems, Theorem \ref{Proposition joint density Y and omega gen}, \ref{Proposition joint density Y and omega gen222}, \ref{Proposition joint density Y and omega gen333} and \ref{Proposition joint density Y and omega gen444}, which in combination proves Theorem \ref{Proposition joint density Y and omega}. In Theorem \ref{Proposition joint density Y and omega gen}, \ref{Proposition joint density Y and omega gen222}, \ref{Proposition joint density Y and omega gen333} and \ref{Proposition joint density Y and omega gen444}, we show a more general result in the sense that we use a window $\hbar$ that is a perturbation of $h$ satisfying Assumption \ref{assump:gaussian window}. 
		In this section, to reduce the heavy notational load, we omit the dependence on $\hbar$ in the superscript unless mentioned otherwise, and denote
		\begin{equation*}
		{Y}_{\alpha,\xi,\eta}:=Y_{f+\Phi}^{(\hbar,\alpha,\xi)}(t,\eta),\,\, \Omega_\eta:=\Omega_{f+\Phi}^{(\hbar)}(t,\eta)\,\,\mbox{ and }V_\eta:=V_{f+\Phi}^{(\hbar)}(t,\eta).
		\end{equation*}
		We first have the following lemma bounding the absolute moments of ${Y}_{\alpha,\xi,\eta}$.

\begin{lemma}\label{LEMMA joint density Y and omega gen}
			Suppose Assumptions~\ref{assump:noise part}, \ref{assump:nonnull signal} and \ref{assump:gaussian window} hold. Assume $\hbar$ is a smooth function satisfying $\hat{\hbar}=\hat{h}\varphi$ for a symmetric, smooth and bounded function $\varphi$. 
			Fix $\xi>0$ and $k > 0$. We have 
			\begin{align}
			\EE |Y_{\alpha,\xi,\eta}|^k=& c_\eta e^{-C_{\xi_0,\eta}}\frac{\sqrt{\pi}\Gamma(k+4)}{2^{k+3}\Gamma(\frac{k+5}{2})} \int_{\mathbb{C}} 
			\frac{1}{E^k_{\alpha,\xi}(\omega)} 
			\int_0^{\pi} \frac{1}{A^{2+k/2}_\eta(\theta,\omega)}
			\label{var eq intermediate}\\
			&\qquad\times \Hypergeometric{1}{1}{\frac{k}{2}+2}{\frac{1}{2}}{\frac{B^2_{\xi_0,\eta}(\theta,\omega)}{A_\eta(\theta,\omega)}}
			\,d\theta\,d\omega\,, \nonumber
			\end{align}
			Moreover,
			we have
			the following bound
\begin{align}
				&m Kc^{(\hbar)}_\eta \int_{\mathbb{C}} \int_0^{\pi}
				\frac{G^{(\hbar)}(\theta,\omega)}{E^k_{\alpha,\xi}(\omega)A^{(\hbar)2+k/2}_\eta(\theta,\omega)}\,d\theta\,d\omega \label{Proof:Bound k abs moments of Y part1} \\
				 \leq &\, \EE{|{Y}_{\alpha,\xi,\eta}|^k}
				 \leq  m^{-1}K c^{(\hbar)}_\eta  \int_{\mathbb{C}}\int_0^{\pi}
				\frac{G^{(\hbar)}(\theta,\omega)}{E^k_{\alpha,\xi}(\omega)A^{(\hbar)2+k/2}_\eta(\theta,\omega)}\,d\theta\,d\omega\,,\nonumber
			\end{align}
			where $m$ comes from Lemma \ref{Lemma: 1F1 bound}, 
			\begin{align*}
				K 
				&\vcentcolon=  
				\frac{\sqrt{\pi}\Gamma(k+4)}{2^{k+3}\Gamma(\frac{k+5}{2})}\\
				G^{(\hbar)}(\theta,\omega)&\vcentcolon=\max\left\{\left(\frac{B^{(\hbar)2}_{\xi_0,\eta}(\theta,\omega)}{A^{(\hbar)}_\eta(\theta,\omega)}\right)^{(k+3)/2}\exp\left\{\frac{B^{(\hbar)2}_{\xi_0,\eta}(\theta,\omega)}{A^{(\hbar)}_\eta(\theta,\omega)}-C^{(\hbar)}_{\xi_0,\eta}\right\},\,e^{-C^{(\hbar)}_{\xi_0,\eta}} \right\}\,.
			\end{align*}
				\end{lemma}
		\begin{proof}
			The density function of $Y_{\alpha,\xi,\eta}$ comes from the standard step of marginalizing ${\Omega}_{f+\Phi}(t,\eta)$ from $f_{\B{Z}_{\alpha,\xi_0,\eta}}(y,\omega)$ shown in Lemma \ref{lemma:Z density function}. 
			First, denote 
			\begin{equation}
			\boldsymbol{\omega}\vcentcolon=\begin{bmatrix}1& 2\pi i {(\eta-\omega)}\end{bmatrix}^\top
			\end{equation} 
			so that
			we have $g_Y^{-1}(y,\omega)=yE_{\alpha,\xi}(\omega)\boldsymbol{\omega}$. {Note that $\boldsymbol{\omega}$ depends on $\eta$, but we use it to simplify the notation.} Therefore, if we write $y=re^{i\theta}$  in the polar coordinate for $r>0$ and $\theta\in [0,2\pi)$, the density function $f_{\B{Z}_{\alpha,\xi_0,\eta}}(y,\omega)$ shown in \eqref{expansion Z distribution} becomes
			\begin{align}
			f_{\B{Z}_{\alpha,\xi_0,\eta}}(re^{i\theta},\omega)=c_\eta r^2E^4_{\alpha,\xi}(\omega)
			\exp[-r^2E^2_{\alpha,\xi}(\omega)A_{\eta}(\theta,\omega)+2rE_{\alpha,\xi}(\omega)B_{\xi_0,\eta}(\theta,\omega)-C_{\xi_0,\eta}]\,.\nonumber
			\end{align}
			For any $k\in \NN$, the $k$-th absolute moment satisfies
			\begin{equation}\label{expansion Y kth moment eq1}
			\EE |Y_{\alpha,\xi,\eta}|^k
			= \int_0^\infty r^k\int_0^{2\pi}\!\!\! \int_{\mathbb{C}}
			 f_{\B{Z}_{\alpha,\xi_0,\eta}}(re^{i\theta}, \omega) \, d\omega\,r\,d\theta\,dr\,.
			\end{equation}
			Since the integrand is nonnegative, Tonelli's theorem allows us to reorder the integration and have
			\begin{align}
			\EE |Y_{\alpha,\xi,\eta}|^k
			=&\, c_\eta e^{-C_{\xi_0,\eta}}\int_{\mathbb{C}}  E^4_{\alpha,\xi}(\omega)\int_0^{2\pi} \!\!\!\int_0^\infty
			r^{k+3} \nonumber\\
			&\quad\times e^{-r^2E^2_{\alpha,\xi}(\omega) A_\eta(\theta,\omega)  + 2rE_{\alpha,\xi}(\omega) B_{\xi_0,\eta}(\theta,\omega)}
			\,dr\,d\theta\,d\omega\,.\label{Expansion absolute moment of order k}
			\end{align}
			Here, the term involving $E_{\alpha,\xi}(\omega)$ defined in \eqref{Definition:Ealphaxi} comes from the Jacobian associated with the change of variable. We now change it back to a more trackable form.
			Note that $E_{\alpha,\xi}(\omega)>0$, and by \eqref{Upper bound of Aeta2}, $A_\eta(\theta,\omega)>0$. 
			By changing variables by letting 
			\begin{equation}
			r = \frac{t}{E_{\alpha,\xi}(\omega)\sqrt{A_\eta(\theta,\omega)}}, 
			\end{equation}
			we have
			 \begin{align}
			\EE |Y_{\alpha,\xi,\eta}|^k
			= &\,c_\eta e^{-C_{\xi_0,\eta}}\int_{\mathbb{C}} E^4_{\alpha,\xi}(\omega)
			\int_0^{2\pi} 
			\frac{1}{E^{k+4}_{\alpha,\xi}(\omega) A^{k/2+2}_\eta(\theta,\omega)}\nonumber\\
			&\qquad \times\Big[\int_0^\infty t^{k+3} e^{-t^2 + \frac{2B_{\xi_0,\eta}(\theta,\omega)}{\sqrt{A_\eta(\theta,\omega)}}t}
			\,dt\Big]\,d\theta\,d\omega\,,
			\end{align}
			which can be converted to the confluent hypergeometric function \eqref{eq:Hermite function of negative order integral representation} via
			\begin{align}
			\EE |Y_{\alpha,\xi,\eta}|^k=&\, c_\eta e^{-C_{\xi_0,\eta}} \int_{\mathbb{C}} 
			\frac{1}{E^k_{\alpha,\xi}(\omega)}
			\int_0^{2\pi} 
			\frac{\Gamma(k+4)}{A^{k/2+2}_\eta(\theta,\omega)}
			\nonumber\\
			&\qquad\times H_{-k-4}\left(\frac{-B_{\xi_0,\eta}(\theta,\omega)}{\sqrt{A_\eta(\theta,\omega)}}\right)
			\,d\theta\,d\omega\nonumber\\
			=&\, c_\eta e^{-C_{\xi_0,\eta}}\frac{\sqrt{\pi}\Gamma(k+4)}{2^{k+3}\Gamma(\frac{k+5}{2})} \int_{\mathbb{C}} 
			\frac{1}{E^k_{\alpha,\xi}(\omega)} 
			\int_0^{\pi} \frac{1}{A^{2+k/2}_\eta(\theta,\omega)}
			\nonumber\\
			&\qquad\times \Hypergeometric{1}{1}{\frac{k}{2}+2}{\frac{1}{2}}{\frac{B^2_{\xi_0,\eta}(\theta,\omega)}{A_\eta(\theta,\omega)}}
			\,d\theta\,d\omega\,, \nonumber
			\end{align}
			where the second equality holds since $A_\eta(\theta+\pi,\omega)=A_\eta(\theta,\omega)$, and $B_{\xi_0,\eta}(\theta+\pi,\omega)=-B_{\xi_0,\eta}(\theta,\omega)$. We thus get \eqref{var eq intermediate}.
			To bound $\EE |Y_{\alpha,\xi,\eta}|^k$, we apply Lemma \ref{Lemma: 1F1 bound}. 
			Note that $\frac{B^2_{\xi_0,\eta}(\theta,\omega)}{ A_\eta(\theta,\omega)}$ is smooth on $\theta$ and $\omega$, and $\frac{B^2_{\xi_0,\eta}(\theta,\omega)}{A_\eta(\theta,\omega)}\geq 0$ since $A_\eta(\theta,\omega)>0$ and $B_{\xi_0,\eta}(\theta,\omega)\in \mathbb{R}$. 
			By Lemma \ref{Lemma: 1F1 bound}, the above culminates in the inequality
			\begin{align}
			&m\max\left\{\Big(\frac{B^2_{\xi_0,\eta}(\theta,\omega)}{A_\eta(\theta,\omega)}\Big)^{(k+3)/2}e^{\frac{B^2_{\xi_0,\eta}(\theta,\omega)}{A_\eta(\theta,\omega)}},\,1\right\}\nonumber\\
			\le&\, \Hypergeometric{1}{1}{\frac{k}{2}+2}{\frac{1}{2}}{\frac{B^2_{\xi_0,\eta}(\theta,\omega)}{A_\eta(\theta,\omega)}}\label{Theorem Bound Y moments General bound}\\
			\le&\, m^{-1}\max\left\{\Big(\frac{B^2_{\xi_0,\eta}(\theta,\omega)}{A_\eta(\theta,\omega)}\Big)^{(k+3)/2}e^{\frac{B^2_{\xi_0,\eta}(\theta,\omega)}{A_\eta(\theta,\omega)}},\,1\right\}\,\nonumber.
			\end{align}
			Then, plug this bound to~\eqref{var eq intermediate} to obtain the claimed bound
			\begin{align}
				&m Kc_\eta \int_{\mathbb{C}} \int_0^{\pi}
				\frac{G(\theta,\omega)}{E^k_{\alpha,\xi}(\omega)A^{2+k/2}_\eta(\theta,\omega)}\,d\theta\,d\omega \nonumber \\
				 \leq &\, \EE{|Y_{\alpha,\xi,\eta}|^k}
				 \leq  m^{-1}K c_\eta  \int_{\mathbb{C}}\int_0^{\pi}
				\frac{G(\theta,\omega)}{E^k_{\alpha,\xi}(\omega)A^{2+k/2}_\eta(\theta,\omega)}\,d\theta\,d\omega\,.\nonumber
			\end{align}
			
			\end{proof}
			
			We start from describing the absolute moments of $Y_{\alpha,\xi,\eta}$ in the null case; that is, when $f=0$.
			\begin{thm}[Absolute moments of $Y_{\alpha,\xi,\eta}$, null case]\label{Proposition joint density Y and omega gen}
			Suppose Assumptions~\ref{assump:noise part}, \ref{assump:nonnull signal} and \ref{assump:gaussian window} hold and $A=0$. Assume $\hbar$ is a smooth function satisfying $\hat{\hbar}=\hat{h}\varphi$ for a symmetric, smooth and bounded function $\varphi$.
			Fix $k >0$. 
			For any $\eta>0$ and $\alpha>0$, the $k$-th absolute moment of ${Y}_{\alpha,\xi,\eta}$ is finite. Specifically,
			for $\eta\geq1$ and $\xi>0$, when $\alpha$ is sufficiently small, we have
			\begin{align}
				\EE{|{Y}_{\alpha,\xi,\eta}|^k}
				\asymp \alpha^{-k/2+1} \,,
				\end{align}
				where 
				 the implied constant depends on $\varphi$, $\varrho$, $k$ and $\frac{\eta^{k\varrho/2}}{(1+4\pi^2|\eta-\xi|^2)^{(k+4)/2}}$; 
				 for $\eta<1$ (particularly when $\eta$ is close to $0$), when $\alpha$ is sufficiently small, we have
			\begin{align}
				c_1\alpha^{-k/2+1} \leq \EE{|{Y}_{\alpha,\xi,\eta}|^k}
				\leq  c_2 \max\Big\{\frac{1}{\sqrt{\alpha\eta}}e^{-\frac{k\xi^2}{2\alpha}},\, 1\Big\}\alpha^{-k/2+1}\,,
				\end{align}
				where 
				 $c_1$ depends on $\varphi$, $\varrho$, $k$ and $\frac{\eta^{3k+8}}{(1+4\pi^2|\eta-\xi|^2)^{(k+4)/2}}$, and $c_2$ depends on $\varphi$, $\varrho$ and $k$.
				\end{thm}

\begin{remark}
This bound shows that the variance of $Y_{\alpha,\xi,\eta}$ is finite and bounded from below since $\EE{|Y_{\alpha,\xi,\eta}|}\asymp \alpha^{1/2}$ and the variance bound is independent of $\xi$.
When $\eta\to 0$ and $\alpha$ is sufficiently small, the upper bound of $\EE{|Y_{\alpha,\xi,\eta}|^k}$ might not be sharp since it comes from several rough bounds for the sake of controlling the degeneracy of the augmented covariance. However, this bound is sufficient for us to derive the desired asymptotic analysis.  Also, note that $(\alpha\eta)^{-1/2}e^{-k(\xi/\sqrt{2\alpha})^2}\leq 1$ when $\alpha$ is sufficiently small for any $\xi>\sqrt{\eta}$; for example, when $\alpha<\eta^2$. Moreover, the lower bound when $\eta$ is close to zero is not sharp since we used a rough bound, but it is sufficient for our purpose.
\end{remark}

			\begin{proof}
			We follow the notation and convention used in the proof for Lemma \ref{LEMMA joint density Y and omega gen}. To get the proof, we control the integration in the upper and lower bounds in \eqref{Proof:Bound k abs moments of Y part1}. 
			In this case $C_{\xi_0,\eta}=0$ and $B_{\xi_0,\eta}(\theta,\omega)=0$, so we have
			\[
				G(\theta,\omega)=e^{-C_{\xi_0,\eta}}=1\,.
			\]
			Note that in the null case, $m=1$ in \eqref{Theorem Bound Y moments General bound}. By the definition of $E_{\alpha,\xi}(\omega)$, we have
			\begin{align}
			&\int_{\mathbb{C}}\int_{0}^\pi
			\frac{1}{E^k_{\alpha,\xi}(\omega)} \frac{1}{A^{(k+4)/2}_\eta(\theta,\omega)}\,d\theta\,d\omega\label{Proof:Bound k abs moments of Y part2}\\
			=&\,\frac{1}{\pi^{k/2}\alpha^{k/2-1}}\int_{\mathbb{C}}\Big[ \int_0^\pi
			\frac{1}{A^{(k+4)/2}_\eta(\theta,\omega)} \,d\theta\Big]\,\frac{1}{\alpha}e^{-k|\omega-\xi|^2/ \alpha}d\omega\,,
			\nonumber
			\end{align}
			where to simplify the notation we further denote
			\[
			c_\eta'':=\int_{\mathbb{C}}\Big[ \int_0^\pi
			\frac{1}{A^{(k+4)/2}_\eta(\theta,\omega)} \,d\theta\Big]\,\frac{1}{\alpha}e^{-k|\omega-\xi|^2/ \alpha}d\omega\,.
			\]
			\underline{{\bf When $\eta\geq 1$:}}
			By \eqref{Upper bound of Aeta2}, when $\eta\geq 1$, we have a trivial bound
			\begin{equation}
			 \frac{\pi \lambda_{\eta,4}^{(k+4)/2}}{C(1+4\pi^2|\eta-\omega|^2)^{(k+4)/2}}\leq \int_0^\pi
			\frac{1}{A^{(k+4)/2}_\eta(\theta,\omega)} \,d\theta\leq \frac{2\pi \lambda_{\eta,1}^{(k+4)/2}}{(1+4\pi^2|\eta-\omega|^2)^{(k+4)/2}}\,,\label{upper and lower bounds of A-k+4/2 integral}
			\end{equation}
			where we use Lemma \ref{Control ABC bound over different regions} 
			On the other hand, by Lemma \ref{Lemma: c of mu_xi0eta when eta small}, $\lambda_{\eta,1}\asymp \lambda_{\eta,4}$ when $\eta\geq 1$.
			Thus, when $\eta\geq 1$, the problem is simply reduced to control  
			\begin{align}
			\int_{\mathbb{C}} \frac{1}{(1+4\pi^2|\eta-\omega|^2)^{(k+4)/2}} \,\frac{1}{\alpha}e^{-k|\omega-\xi|^2/ \alpha}d\omega\,,\label{Proof Y moments Approximation of Identity of A}
			\end{align}
			which is finite due to the exponential decay of the Gaussian function. When $\alpha\to 0$, 
			its dependence on $\alpha$ could be directly evaluated by an approximation of identity of $\frac{1}{(1+4\pi^2|\eta-\omega|^2)^{(k+4)/2}}$ at $\xi$ when $\alpha$ is small with an error of order $\alpha$ since the Hessian of $\frac{1}{(1+4\pi^2|\eta-\omega|^2)^{(k+4)/2}}$ is uniformly bounded. As a result, when $\alpha$ is sufficiently small and $\eta>\eta_0$, 
			\begin{equation}
			c_\eta''\asymp \frac{\eta^{(k+4)\varrho/2}}{(1+4\pi^2|\eta-\xi|^2)^{(k+4)/2}}\,.
			\end{equation} 
			 since $\lambda_{\eta,1}\asymp \lambda_{\eta,4}\asymp \eta^\varrho$. Finally, note that $c_\eta\asymp \eta^{-2\varrho}$. So we have the implied constant depending on $\frac{\eta^{k\varrho/2}}{(1+4\pi^2|\eta-\xi|^2)^{(k+4)/2}}$ since $(\eta^\varrho)^{(k+4)/2}\eta^{-2\varrho}=\eta^{k\varrho/2}$.
			\medskip			
			
			\underline{{\bf When $\eta<1$, particularly when $\eta$ is small:}}
			When $\eta$ approaches $0$, the control is different. For the lower bound, we can simply use the lower bound shown in \eqref{upper and lower bounds of A-k+4/2 integral}, so by Lemma \ref{Lemma: c of mu_xi0eta when eta small}, the lower bound is of the order $\lambda_{\eta,4}^{(k+4)/2}c_\eta=(\eta^6)^{(k+4)/2}\eta^{-4}=\eta^{3k+8}$. 
			
			However,  
			the upper bound in \eqref{upper and lower bounds of A-k+4/2 integral} is not sharp enough for our purpose, since $c_\eta$ in \eqref{Proof:Bound k abs moments of Y part1} is of order $\eta^{-4}$ and $\lambda_{\eta,4}^{(k+4)/2}$ is of order $1$ by Lemma \ref{Lemma: c of mu_xi0eta when eta small}. We thus need to carefully take the structure of $A_{\eta}(\theta,\omega)$ into account to control $\eta^{-4}$ caused by the degeneracy of $\underline{\Sigma}_\eta$.
			We will split the integral domain of $\theta$ in \eqref{Proof:Bound k abs moments of Y part2} into three pieces:
			\begin{align}
			[0,\pi)=I_1\cup I_2\cup I_3\,,
			\end{align}
			where 
			\begin{align}
			I_1=&\,(\eta^{2},\pi-\eta^{2})\,,\quad
			I_2=(\eta^{5/2},\eta^{2}]\cup [\pi-\eta^{2},\pi-\eta^{5/2})\,,\nonumber\\
			I_3=&\,[0,\eta^{5/2}]\cup [\pi-\eta^{5/2},\pi)\,.\nonumber
			\end{align}
			By Lemma \ref{Lemma: c of mu_xi0eta when eta small}, when $\theta\in I_1$,  $A_\eta(\theta,\omega)$ is lower bounded by $\frac{2\sin(\theta)^2}{\mathsf c_4}\eta^{-6}\geq \frac{\eta^{-2}}{\mathsf c_4}$.
			 {Therefore, over $I_1$, we have
			\begin{align}
			&\int_{\mathbb{C}} \int_{I_1}
			\frac{1}{A^{(k+4)/2}_\eta(\theta,\omega)} \,d\theta\,\frac{1}{\alpha}e^{-k|\omega-\xi|^2/ \alpha}d\omega\nonumber\\
			\leq\,&\int_{\mathbb{C}} \int_{I_1}
			(\mathsf{c}_4\eta^2)^{(k+4)/2} \,d\theta\,\frac{1}{\alpha}e^{-k|\omega-\xi|^2/ \alpha}d\omega\nonumber\\
			\leq\,&\mathsf{c}^{(k+4)/2}_4\eta^{k+4}\int_{\mathbb{C}} 
			\frac{1}{\alpha}e^{-k|\omega-\xi|^2/ \alpha}d\omega= \frac{\pi\mathsf{c}^{(k+4)/2}_4}{k}\eta^{k+4}\,,
			\end{align}
			where we use $\frac{1}{\alpha}\int_{\mathbb{C}} 
			e^{-k|\omega-\xi|^2/ \alpha}d\omega=\frac{\pi}{k}$. The integral over $I_2$ is an easy one. Note that when $\theta\in I_2$, we have a loose lower bound
			\[
			A_\eta(\theta,\omega)\geq\sin^2(\theta)\eta^{-6}\,, 
			\]
			which ranges from values of order $\eta^{-1}$ to values of order $\eta^{-2}$.  As a result,			
			Thus, we have 
			\begin{align}
			&\int_{\mathbb{C}} \int_{I_2}
			\frac{1}{A^{(k+4)/2}_\eta(\theta,\omega)} \,d\theta\,\frac{1}{\alpha}e^{-k|\omega-\xi|^2/ \alpha}d\omega\nonumber\\
			\leq\,&\int_{\mathbb{C}} \int_{I_2}
			\frac{\eta^{3(k+4)}}{\sin(\theta)^{k+4}} \,d\theta\,\frac{1}{\alpha}e^{-k|\omega-\xi|^2/ \alpha}d\omega\nonumber
			\leq\eta^{3(k+4)}\int_{\mathbb{C}}\int_{I_2}
			\frac{2^{k+4}}{\theta^{k+4}} d\theta\frac{1}{\alpha}e^{-k|\omega-\xi|^2/ \alpha}d\omega\nonumber\\
			\leq\,&\eta^{3(k+4)}\times \frac{2^{k+4}}{k+3}\eta^{-5(k+3)/2}\times \frac{\pi}{k}=\frac{\pi 2^{k+4}}{k(k+3)}\eta^{(k+9)/2}\,,\nonumber
			\end{align}
			where in the second inequality we use the fact that $\sin(\theta)\geq \theta/2$ when $\theta$ is sufficiently small.}
			Finally, we handle $I_3$. In this regime, we need to further divide $\mathbb{C}$. First, note that over this regime, we loss the lower bound control of $\sin(\theta)^2\eta^{-6}$. Instead, we note that by rewriting $\eta-\omega=re^{i\phi}$, we have 
			\begin{align*}
			A_\eta(\theta,\omega)
			\leq C_1[\theta^2\eta^{-6}+ r^2\cos^2(\theta+\phi)\eta^{-2}+r^2\sin^2(\theta+\phi)+1]
			\end{align*}
			for some constant $C_1$. However, to control the integral over $I_3$, we need to control the lower bound of $A_\eta(\theta,\omega)$. 
			Split the integral domain of $\phi$ into
			\begin{align*}
			J_1&\,=[\pi/2-\eta^{1/2},\pi/2+\eta^{1/2}]\cup [3\pi/2-\eta^{1/2},3\pi/2+\eta^{1/2}],\\ 
			J_2&\,=[0,2\pi]\backslash J_1, 
			\end{align*}
			and rewrite the integral in the polar coordinate form:
			\begin{align}
			&\int_{\mathbb{C}} \int_{I_3}
			\frac{1}{A^{(k+4)/2}_\eta(\theta,\omega)} \,d\theta\,\frac{1}{\alpha}e^{-k|\omega-\xi|^2/ \alpha}d\omega\nonumber\\
			=\,&\int_{0}^\infty\int_{J_1\cup J_2} \int_{I_3}
			\frac{1}{A^{(k+4)/2}_\eta(\theta,\eta-re^{i\phi})} \,d\theta\,\frac{1}{\alpha}e^{-k|\xi-\eta+re^{i\phi}|^2/ \alpha}d\phi rdr\,.\nonumber
			\end{align}
			Since $\theta\in I_3$, over $J_1$, $\theta+\phi$ differs from $\pi/2$ or $3\pi/2$ by maximally $2\eta^{1/2}$, $\sin(\theta+\phi)\geq 1/2$, and hence 
			\begin{align*}
			A_\eta(\theta,\omega)&\,\geq C_2[r^2(\cos(\theta+\phi)^2\eta^{-2}+1)+1]
			\end{align*}
			for some constant $C_2>0$.
						On the other hand, since $\phi\in J_1$, $\theta\in I_3$ and $\xi$ is much larger than $\eta$, $|\xi-\eta+re^{i\phi}|^2 \geq \xi^2/2+r^2/2$ and hence 
			\[
			e^{-k|\xi-\eta+re^{i\phi}|^2/ \alpha}\leq e^{-k\xi^2/2\alpha}e^{-kr^2/ 2\alpha}\,. 
			\]
			Thus, we have
			\begin{align}
			&\int_{0}^\infty\int_{J_1} \int_{I_3}
			\frac{1}{A^{(k+4)/2}_\eta(\theta,\eta-re^{i\phi})} \,d\theta\,\frac{1}{\alpha}e^{-k|\xi-\eta+re^{i\phi}|^2/ \alpha}d\phi rdr\nonumber\\
			\leq &\,e^{-k\xi^2/2\alpha}\int_{0}^\infty\int_{J_1} \int_{I_3}
			\frac{C_2^{-(k+4)/2}}{[r^2(\cos(\theta+\phi)^2\eta^{-2}+1)+1]^{(k+4)/2}} \,d\theta\,\frac{1}{\alpha}e^{-kr^2/ 2\alpha}d\phi rdr \nonumber
			\end{align}
			When $\phi\in J_{1,1}:=\phi\in [\pi/2-\eta^{3/2},\pi/2+\eta^{3/2}]\cup [3\pi/2-\eta^{3/2},3\pi/2+\eta^{3/2}]\subset J_1$, we use the following simple bound 
			\[
			r^2(\cos(\theta+\phi)^2\eta^{-2}+1)\geq r^2+1\,.
			\]
			We then have
			\begin{align*}
			&\int_{0}^\infty\int_{J_{1,1}} \int_{I_3}
			\frac{1}{[r^2(\cos(\theta+\phi)^2\eta^{-2}+1)+1]^{(k+4)/2}} \,d\theta\,\frac{1}{\alpha}e^{-kr^2/ 2\alpha}d\phi rdr\\
			\leq &\, \int_{0}^\infty \frac{r}{(r^2+1)^{(k+4)/2}}  \int_{J_1} \int_{I_3}
			\,d\theta d\phi \,\frac{1}{\alpha}e^{-kr^2/ 2\alpha} dr\\
			\leq &\, \eta^4\int_{0}^\infty \frac{r}{(r^2+1)^{(k+4)/2}}  \frac{1}{\alpha}e^{-kr^2/ 2\alpha} dr \leq C_3\eta^4
			\end{align*}
			for a constant $C_3>0$.

			When $\phi\in J_{1,2}:=J_1\backslash J_{1,1}$, we have $|\phi-\pi/2+\theta|\leq 2|\phi-\pi/2|$ when $\phi\in [\pi/2-\eta^{1/2},\pi/2+\eta^{1/2}]$ and $|\phi-3\pi/2+\theta|\leq 2|\phi-3\pi/2|$ when $\phi\in [3\pi/2-\eta^{1/2},3\pi/2+\eta^{1/2}]$, thus, we use the rough bound
			\[
			\cos(\theta+\phi)^2\eta^{-2}\geq \phi^2\eta^{-2}/2\,.
			\]
			Thus, we have
			\begin{align}
			&\int_{0}^\infty\int_{J_{1,2}} \int_{I_3}
			\frac{1}{[r^2(\cos(\theta+\phi)^2\eta^{-2}+1)+1]^{(k+4)/2}} \,d\theta\,\frac{1}{\alpha}e^{-kr^2/ 2\alpha}d\phi rdr\nonumber\\
			\leq \,&\int_{0}^\infty\int_{J_{1,2}} \int_{I_3}
			\frac{2^{(k+4)/2}}{[r^2(\phi^2\eta^{-2}+1)]^{(k+4)/2}} \,d\theta\,\frac{1}{\alpha}e^{-kr^2/ 2\alpha}d\phi rdr\leq C_4 \frac{\eta^{7/2}}{\alpha^{1/2}}\nonumber
			\end{align}
			for some constant $C_4>0$, where the last bound comes from $\int_{0}^\infty \frac{1}{\alpha}e^{-kr^2/ \alpha} dr=\frac{\sqrt{\pi}}{2\sqrt{\alpha k}}$ and the following rough control. Indeed, by a change of variable $r\phi/\eta=y$, we obtain
			\begin{align*}
			&\int_{\eta^{3/2}}^{\eta^{1/2}} 
			\frac{1}{(r^2\phi^2\eta^{-2}/2+1)^{(k+4)/2}}\,d\phi\\
			=\,&\frac{\eta}{r}\int_{r\eta^{1/2}}^{r\eta^{-1/2}} 
			\frac{1}{(y^2/2+1)^{(k+4)/2}}\,dy\leq \frac{\eta}{r}\int_{0}^{\infty} 
			\frac{1}{(y^2/2+1)^{(k+4)/2}}\,dy=\frac{C'\eta}{r}
			\end{align*}
			for some constant $C'>0$.
			As a result, we have
			\begin{align*}
			&\int_{0}^\infty\int_{J_1} \int_{I_3}
			\frac{1}{A^{(k+4)/2}_\eta(\theta,\eta-re^{i\phi})} \,d\theta\,\frac{1}{\alpha}e^{-k|\xi-\eta+re^{i\phi}|^2/ \alpha}d\phi rdr\\
			\leq &\,C_5\alpha^{-1/2}e^{-k\xi^2/2\alpha}\eta^{7/2}\nonumber
			\end{align*}
			for some constant $C_5>0$. Note that if $\xi>\sqrt{\eta}$, when $\alpha$ is sufficiently small, $\alpha^{-1/2}e^{-k\xi^2/2\alpha}$ decays to $0$. 
			
			The integration over $J_2$ is different since over $J_2$, we lose all the above controls. We use the rough bound $A_\eta(\theta,\omega)\geq  C_6(r^2\cos(\phi)^2\eta^{-2}+1)$ for some constant $C_6>0$. This is because $\sin(\theta+\phi)$ is bounded by $1$ from above, $\cos(\theta+\phi)^2$ is bounded by $\eta^{-1}$ from below, and $\cos(\theta+\phi)\geq \cos(\phi)/2$ when $\phi\in J_2$ and $\theta\in I_3$.  Also, note that since $\eta$ is close to $0$ and $\xi\geq \sqrt{\eta}$ by assumption, we have $|\xi-\eta+re^{i\phi}|\geq |\xi+re^{i\phi}|/2$. 
			This leads to
			\begin{align}
			&\int_{0}^\infty\int_{J_2} \int_{I_3}
			\frac{1}{A^{(k+4)/2}_\eta(\theta,\eta-re^{i\phi})} \,d\theta\,\frac{1}{\alpha}e^{-k|\xi-\eta+re^{i\phi}|^2/ \alpha}d\phi rdr\nonumber\\
			\leq \,&\frac{2\eta^{5/2}}{C_6^{(k+4)/2}}\int_{0}^\infty\int_{J_2} 
			\frac{1}{(r^2\cos(\phi)^2\eta^{-2}+1)^{(k+4)/2}}\,\frac{1}{\alpha}e^{-k|\xi+re^{i\phi}|^2/ 2\alpha}d\phi rdr\nonumber\\
			\leq \,&\frac{2\eta^{5/2}}{C_6^{(k+4)/2}}\int_{\mathbb{C}} 
			\frac{1}{(\eta^{-2}(z+\bar{z})^2+1)^{(k+4)/2}}\,\frac{1}{\alpha}e^{-k|\xi+z|^2/ \alpha} dz\nonumber\\
			\leq \,&\frac{4\eta^{5/2}}{C_6^{(k+4)/2}}\frac{1}{(4\eta^{-2}\xi^2+1)^{(k+4)/2}}
			\leq C_7\eta^{(k+9)/2} \nonumber
			\end{align}
			for some constant $C_7>0$, where the second inequality comes from the positivity of the integrand, the third inequality comes from the approximation of identify with the assumption that $\xi>\sqrt{\eta}$ when $\alpha$ is sufficiently small, and the last inequality comes from a direct bound.

			We conclude that when $\eta\to 0$, we have $c''_\eta\leq C_8(\alpha^{-1/2}e^{-k\xi^2/2\alpha}\eta^{7/2}\vee\eta^4)$ for all $k\geq 1$ for some $C_8>0$ independent of $\eta$. 
			As a result, in the null case we have the desired claim
			when $\alpha$ is sufficiently small.  

			\end{proof}

			Next, we discuss the absolute moments of $Y_{\alpha,\xi,\eta}$ in the non-null case; that is, when $f\neq 0$. 
			The key step in the proof is controlling $\frac{B^2_{\xi_0,\eta}(\theta,\omega)}{A_\eta(\theta,\omega)}-C_{\xi_0,\eta}$, particularly when $\eta$ is close to $0$. In this case, this term cannot be simply bound using the Cauchy-Schwartz inequality $B^2_{\xi_0,\eta}(\theta,\omega)\leq A_\eta(\theta,\omega)C_{\xi_0,\eta}$ and we need a sharper one. 
			The key observation in our setup is that we do not work with generic vectors when we evaluate $\frac{B^2_{\xi_0,\eta}(\theta,\omega)}{A_\eta(\theta,\omega)}-C_{\xi_0,\eta}$. Instead, vectors in $A_\eta(\theta,\omega)$, $B^2_{\xi_0,\eta}(\theta,\omega)$, and $C_{\xi_0,\eta}$ do have a specific low dimensional structure that is specified in the following Lemma.
			
			\begin{lemma} \label{Lemma:Theta eta control}
			Consider the map $\Theta_\eta:\mathbb{C}\times \mathbb{R}\to \mathbb{C}^2$ defined by
		\[
		\Theta_\eta:(\omega,\theta)\to e^{i\theta}\boldsymbol{\omega}=e^{i\theta}\begin{bmatrix}1& 2\pi i {(\eta-\omega)}\end{bmatrix}^\top\,.
		\] 
		Note that we have $\mu_{\xi_0,\eta}:=\Theta_\eta(\eta-\xi_0, 2\pi\xi_0t)$. Then, when $\Theta_\eta(\omega,\theta)\neq \mu_{\xi_0,\eta}$, $\Theta_\eta(\omega,\theta)$ is generically not parallel to $\mu_{\xi_0,\eta}$ locally around $(\eta-\xi_0, 2\pi\xi_0t)$. Similarly, the map $\underline{\Theta_\eta}:\mathbb{C}\times \mathbb{R}\to \mathbb{C}^4$ defined by
		\[
		\underline{\Theta_\eta}:(\omega,\theta)\to \underline{e^{i\theta}\boldsymbol{\omega}}
		\] 
		also satisfies the same property.
		\end{lemma}
		
		\begin{proof}
		Clearly, $\Theta_\eta$ is smooth and a periodic map on the $\theta$ variable with the periodicity $2\pi$. Also, $\Theta_\eta$ has a low dimensional range. Note that $\mu_{\xi_0,\eta}$ is a special case of $e^{i\theta}\boldsymbol{\omega}$ scaled by $A\hat{h}(\xi_0-\eta)$. Note that $\mu_{\xi_0,\eta}$ is not in the range of $\Theta_\eta$ unless $A\hat{h}(\xi_0-\eta)=1$ since the first coordinate of the range is of unit norm. 
		 By a direct calculation, we have
		\[
		\nabla\Theta_\eta|_{(\eta-\xi_0, 2\pi\xi_0t)}\begin{bmatrix}\omega\\\theta\end{bmatrix}=e^{i2\pi\xi_0t}\begin{bmatrix}i\theta\\ -2\pi\xi_0\theta-2\pi i\omega \end{bmatrix}=i\theta u_{\xi_0,\eta}+e^{i2\pi\xi_0t}\begin{bmatrix}0\\ -2\pi i\omega \end{bmatrix}\,.
		\]
		Note that $\begin{bmatrix}\omega\\\theta\end{bmatrix}\in T_{(\eta-\xi_0, 2\pi\xi_0t)}(\mathbb{C}\times \mathbb{R})$, which plays a different role compared with $(\eta-\xi_0, 2\pi\xi_0t)\in \mathbb{C}\times \mathbb{R}$.
		Thus, $\Theta_\eta(\omega,\theta)$ is parallel to $\mu_{\xi_0,\eta}$ when $\omega=0$, which finishes the claim. 
		Note that the inner product structure on $\mathbb{C}^4$ of the range of $\Theta_\eta$ does not play a role.
		\end{proof}
		
		Next, we need the following statement regarding the Cauchy-Schwartz inequality.
		\begin{lemma}\label{Lemma: improved Cauchy-Schwartz}
		Take a positive definite matrix $\Sigma\in \mathbb{C}^{p\times p}$ for $p\in \mathbb{N}$ and $u,v\in \mathbb{C}^p$. We have
		\begin{equation*}
		\frac {[\Re(u^* {\Sigma} v)]^2}{v^* {\Sigma} v}-u^*{\Sigma} u\leq -z^*{\Sigma} z\,,
		\end{equation*} 
		where $z=u-P_vu$ and  $P_vu:=\frac { u^* {\Sigma} v}{v^* {\Sigma} v}v$. 
		\end{lemma}

		\begin{proof}
		We have 
		\[
		\frac { |u^* {\Sigma} v|^2}{v^* {\Sigma} v}-u^*{\Sigma} u=-z^*{\Sigma} u\,,
		\]
		which leads to 
		\begin{equation*}
		\frac {[\Re(u^* {\Sigma} v)]^2}{v^* {\Sigma} v}-u^*{\Sigma} u=-z^*{\Sigma} u-\frac {[\Im(u^* {\Sigma} v)]^2}{v^* {\Sigma} v}\leq -z^*{\Sigma} u=-z^*{\Sigma} z
		\end{equation*}   
		where the last equality holds since $u=P_vu+z$ and $P_vu$ and $z$ are perpendicular with related to ${\Sigma}$. 
		Note that in general, $z^*{\Sigma} u=0$ and $\Im(u^* {\Sigma} v)=0$ can happen. 
 We thus finish the proof.
		\end{proof}
			
			Note that in general we have
		\begin{align*}
		(u-v)^*{\Sigma}(u-v)=(z+(P_vu-v))^*{\Sigma}(z+(P_vu-v))\geq z^*{\Sigma} z
		\end{align*}
		since $z$ and $P_vu$ are perpendicular with related to ${\Sigma}$.

			With the above preparation, we have the following theorem about the absolute moments of $Y_{\alpha,\xi,\eta}$ in the non-null case. The proof depends on reducing the non-null case to the null case shown in Theorem \ref{Proposition joint density Y and omega gen}.
			
			\begin{thm}[Absolute moments of $Y_{\alpha,\xi,\eta}$, non-null case]\label{Proposition joint density Y and omega gen222}
			Suppose Assumptions~\ref{assump:noise part}, \ref{assump:nonnull signal} and \ref{assump:gaussian window} hold and $A>0$. Assume $\hbar$ is a smooth function satisfying $\hat{\hbar}=\hat{h}\varphi$ for a symmetric, smooth and bounded function $\varphi$.
			Fix $k \in \NN$. 
			For any $\eta>0$, $\xi>0$ and $\alpha>0$, the $k$-th absolute moment are finite. 
			Moreover,
			\begin{enumerate}
			\item for $\eta\geq 1$, $|\eta-\xi_0|\geq 1/2$ and $\xi>0$, when $\alpha$ is sufficiently small, we have
			\begin{align}
				\EE{|Y_{\alpha,\xi,\eta}|^k}
				\asymp \alpha^{-k/2+1} \,,
				\end{align}
				where 
				 the implied constant depends on $\varphi$, $\varrho$, $k$ and $\frac{\eta^{k\varrho/2}}{(1+4\pi^2|\eta-\xi|^2)^{(k+4)/2}}$; 
				 when $|\eta-\xi_0|< 1/2$ and $\xi>0$, when $\alpha$ is sufficiently small, we have
			\begin{align}
				c_1 \alpha^{-k/2+1}  \leq \EE{|Y_{\alpha,\xi,\eta}|^k}
				\leq  c_2\alpha^{-k/2+1} \,,
				\end{align}
				where $c_1$ depends on $\varphi$, $\varrho$, $k$ and $e^{-A^2(1+4\pi^2\xi_0)^2\xi_0^{-\rho}}$ and
				 $c_2$ depends on $\varphi$, $\varrho$, $k$ and $A^{k+3}\xi_0^{(2-\rho)(k+3)/2}$;
				 \item for $\eta<1$ (particularly when $\eta$ is close to $0$), 
				 when $\alpha$ is sufficiently small so that $\alpha<\eta$, we have
			\begin{align}
				c_1&\,e^{-CA^2\hat{\hbar}(\xi_0)^2\max\{\eta^{-6},\,\xi_0^2\eta^{-2}\}}\alpha^{-k/2+1} \leq  \EE{|Y_{\alpha,\xi,\eta}|^k}\nonumber\\
				&\qquad\leq  c_2  \max\Big\{\frac{1}{\sqrt{\alpha\eta}}e^{-\frac{k\xi^2}{2\alpha}},\, 1,\,\frac{1}{A\hat{\hbar}(\xi_0)\xi_0}\Big\} \alpha^{-k/2+1} \nonumber\,
				\end{align}
				where 
				 $c_1$ depends on $\varphi$, $\varrho$, $k$ and $\frac{\eta^{3k+8}}{(1+4\pi^2|\eta-\xi|^2)^{(k+4)/2}}$, $C>0$ depends on $\varrho$ and $\varphi$, and $c_2$ depends on $\varphi$, $\varrho$ and $k$.
				 \end{enumerate}
				\end{thm}
				
			\begin{proof}
			We follow the notation and convention used in the proof for Lemma \ref{LEMMA joint density Y and omega gen}. We need to control the impact of the signal, which involves the ``signal strength'' $A$, the frequency $\xi_0$, and the window effect $\hat{\hbar}(\xi_0-\eta)$ in $c_\eta$, $B_{\xi_0,\eta}(\theta,\omega)$ and $C_{\xi_0,\eta}$. Without loss of generality, below we assume that $\hat{\hbar}(\xi_0-\eta)\neq 0$, otherwise the proof is trivially the same as that for the null case. 
			\medskip
			
			\underline{{\bf When $\eta\geq 1$:}} 		
		Since all eigenvalues of $\underline{\Sigma}_\eta$ are of the same order and grow at the polynomial rate maximally by Lemma \ref{Lemma: c of mu_xi0eta when eta small} and $A\hat{h}(\xi_0-\eta)$ decays exponentially with the rate depending on $A$, $\xi_0$, $\varphi$ and $\varrho$, when $|\xi_0-\eta|$ grow, $C_{\xi_0,\eta}$  
		decays exponentially when $|\xi_0-\eta|$ increases.  		We then apply \eqref{var eq intermediate} in Lemma \ref{LEMMA joint density Y and omega gen} to control the absolute moment. By the Cauchy-Schwartz inequality, $\frac{B^2_{\xi_0,\eta}(\theta,\omega)}{A_\eta(\theta,\omega)}$ decays exponentially when $|\xi_0-\eta|$ increases. By combining this fact and the approximation that \cite[p.323 (13.2.13)]{NIST} 
			$\Hypergeometric{1}{1}{\frac{k}{2}+2}{\frac{1}{2}}{x^2}=1+O(x^2)$ when $x\to 0$ for all $k\in\mathbb{N}$, when $\eta\geq 1$ and $|\xi_0-\eta|\geq 1/2$, $e^{-C_{\xi_0,\eta}}\Hypergeometric{1}{1}{\frac{k}{2}+2}{\frac{1}{2}}{\frac{B^2_{\xi_0,\eta}(\theta,\omega)}{A_\eta(\theta,\omega)}}\asymp 1$, where the implied constant depends on $k$, $A$, $\xi_0$, $\varrho$ and $\varphi$. These approximations together imply that we can control $\EE |Y_{\alpha,\xi,\eta}|^k$ by its null parallel result.

		On the other hand, when $\eta$ approaches $\xi_0$, $A\hat{h}(\xi_0-\eta)$ approaches $A$ and eigenvalues of $\underline{\Sigma}_\eta$ are of order $\eta^\varrho$ since $\xi_0\geq 1$ by assumption. Precisely, when $|\xi_0-\eta|< 1/2$, $C_{\xi_0,\eta}\asymp A^2\xi_0^{2-\rho}$. In this case, the term $G(\theta,\omega)$ in \eqref{Proof:Bound k abs moments of Y part1} could be roughly bounded by $\max\{C_{\xi_0,\eta}^{(k+3)/2}, e^{-C_{\xi_0,\eta}}\}$ from above and $e^{-C_{\xi_0,\eta}}$ from below by the bound
			$\frac{B^2_{\xi_0,\eta}(\theta,\omega)}
					{A_\eta(\theta,\omega)}-C_{\xi_0,\eta}
				\leq 
				0
			$
		by the Cauchy-Schwartz.

		As a result, when $\eta\geq 1$, the finite absolute $k$-th moment in the non-null case follows from the same argument as that in the null case, with an extra constant depending on $A$ and $\xi_0$.
		
		\medskip
			
			\underline{{\bf When $\eta< 1$, particularly when $\eta$ is close to $0$:}} When $\eta$ is small, 
		 again we need to handle the degeneracy of $\underline{\Sigma}_\eta$. The lower bound is simple. 
		 Note that when $\eta$ is close to $0$, 
		 \[
		 C_{\xi_0,\eta}\leq CA^2\hat{\hbar}(\xi_0)^2\max\{\eta^{-6},\xi_0^2\eta^{-2}\}
		 \] 
		 for some $C>0$ according to \eqref{Asymptotic muSigmainvmu or C} in Lemma \ref{Lemma: c of mu_xi0eta when eta small} since $\xi_0$ is assumed to be finite. Thus, we can bound $e^{-C_{\xi_0,\eta}}$ from below.
		 With this bound, we could control the lower bound of $\EE{|Y_{\alpha,\xi,\eta}|^k}$ by controlling $G(\theta,\omega)$ from below by $e^{-C_{\xi_0,\eta}}$. As a result, by the same argument as that in the null case, we get a rough lower bound control of $\EE{|Y_{\alpha,\xi,\eta}|^k}$ with an extra constant $e^{-C_{\xi_0,\eta}}$.

		 Next, we control the upper bound of $\EE{|Y_{\alpha,\xi,\eta}|^k}$. 
		 Apply Lemma \ref{Lemma: improved Cauchy-Schwartz} with $\Sigma=\underline{\Sigma}^{-1}_\eta$, $v\in \underline{\Theta_\eta}(\mathbb{C}\times \mathbb{R})\in \mathbb{C}^4$ and $u=\underline{\mu_{\xi_0,\eta}}\in \mathbb{C}^4$. We have
		\begin{align*}
		\frac{B^2_{\xi_0,\eta}(\theta,\omega)}
					{A_\eta(\theta,\omega)}-C_{\xi_0,\eta}=\frac {[\Re(u^* \underline{\Sigma}^{-1}_\eta v)]^2}{v^* \underline{\Sigma}^{-1}_\eta v}-u^*\underline{\Sigma}^{-1}_\eta u\leq -z^*\underline{\Sigma}^{-1}_\eta z\,,
		\end{align*} 
		where $z=\underline{\mu_{\xi_0,\eta}}-P_v\underline{\mu_{\xi_0,\eta}}$ and $P_v\underline{\mu_{\xi_0,\eta}}:=\frac { \underline{\mu_{\xi_0,\eta}}^* \underline{\Sigma}^{-1}_\eta v}{v^* \underline{\Sigma}^{-1}_\eta v}v$ is the projection of $\underline{\mu_{\xi_0,\eta}}$ onto $v$ via the inner product structure $\underline{\Sigma}^{-1}_\eta$.
		By Lemma \ref{Lemma:Theta eta control}, generically we have $z^*\underline{\Sigma}^{-1}_\eta z\neq 0$. 
		To finish the proof, we need to further quantify when $z^*\underline{\Sigma}^{-1}_\eta z$ is away from $0$, and how. 
		By the third fact in Lemma \ref{Lemma: c of mu_xi0eta when eta small} and setting the second coordinate of $\Theta_\eta(\omega,\theta)$ as $re^{i\phi}$, we have
		\[
		z=\underline{\mu_{\xi_0,\eta}}-P_v\underline{\mu_{\xi_0,\eta}}\asymp A\hat{\hbar}(\xi_0-\eta)\Big[\underline{\Theta}_\eta(\eta-\xi_0, 2\pi\xi_0t)-\mathsf{D}_p(\xi_0, \omega, \theta, \eta)v\Big]\,,
		\]
		where $\mathsf{D}_p(\xi_0, \omega, \theta, \eta):=\frac{\mathsf{D}_d(\xi_0, \omega, \theta, \eta)}{\mathsf{D}_n(\omega,\theta,  \eta)}$,
		\begin{align*}
		\mathsf{D}_d(\xi_0, \omega, \theta, \eta):=&\,\frac{2\sin(\theta)\sin(2\pi\xi_0t)}{\mathsf c_4 \eta^{6}}+\frac{8\pi^2r\xi_0\cos(\theta+\phi)\cos(2\pi\xi_0t)}{\mathsf c_3 \eta^{2}}\\
		&+\frac{4\pi^2r\xi_0\sin(\theta+\phi)\sin(2\pi\xi_0t)}{\gamma^{[\varphi,\varphi]}_0(0)}+\frac{\cos(\theta)\cos(2\pi\xi_0t)}{4\pi^2\gamma^{[\varphi,\varphi]}_2(0)}
		\end{align*}
		and
		\[
		\mathsf{D}_n(\omega,\theta,  \eta):=\frac{2\sin^2(\theta)}{\mathsf c_4\eta^{6}}+\frac{8\pi^2r^2\cos^2(\theta+\phi)}{\mathsf c_3\eta^{2}}+\frac{4\pi^2r^2\sin^2(\theta+\phi)}{\gamma^{[\varphi,\varphi]}_0(0)}+\frac{\cos^2(\theta)}{4\pi^2\gamma^{[\varphi,\varphi]}_2(0)}\,.
		\]
		This seemingly complicated formula actually has a simple rule. First, since $A\hat{\hbar}(\xi_0-\eta)\asymp A\hat{\hbar}(\xi_0)$, which is a fixed constant, we only focus on the quantity inside the bracket. Thanks to the Cauchy-Schwartz inequality or by a direct checkup, we know that $|\mathsf{D}_p(\xi_0, \omega, \theta, \eta)|\leq 1$. Thus, $z$ will be small only if $\theta$ is close to $2\pi\xi_0t$ (mod $2\pi$), and $|\omega|$ is close to $\xi_0$. When $2\pi\xi_0t$ (mod $2\pi$) is small, the $\eta^{-2}$ term will dominate simultaneously in $\mathsf{D}_d(\xi_0, \omega, \theta, \eta)$ and $\mathsf{D}_n(\omega,\theta,  \eta)$ if $\phi$ is away from $0$; otherwise the constant term will dominate simultaneously. In other words, only over a small region will $z^*\underline{\Sigma}^{-1}_\eta z$ be ``small''.
		To be more specific, consider the following example. When $2\pi_0t$ is close to $0$ or $-\pi$, $C_{\xi_0,\eta}$ is of order $\eta^{-2}$. Further, when $\omega$ is close to $\xi_0$, we have $\mathsf{D}_d(\xi_0, \omega, \theta, \eta)$ is close to
		$\frac{8\pi^2\xi^2_0}{\mathsf c_3 \eta^{2}}+\frac{1}{4\pi^2\gamma^{[\varphi,\varphi]}_2(0)}$, and 
		$\mathsf{D}_n(\omega,\theta,  \eta)$ is close to $\frac{8\pi^2\xi_0^2}{\mathsf c_3\eta^{2}}+\frac{1}{4\pi^2\gamma^{[\varphi,\varphi]}_2(0)}$; that is, $\mathsf{D}_p(\xi_0, \omega, \theta, \eta)$ is close to $1$, and $u$ is close to $\underline{\mu_{\xi_0,\eta}}$, and hence $\underline{\Theta}_\eta(\eta-\xi_0, 2\pi\xi_0t)-\mathsf{D}_p(\xi_0, \omega, \theta, \eta)v$ is close to $0$. 
		
		To further quantify the above observation, by a further expansion, we have 
		\begin{align*}
		z^*\underline{\Sigma}^{-1}_\eta z\asymp&\,
			A^2\hat{\hbar}(\xi_0-\eta)^2\Big[
			\frac{2(\sin(2\pi\xi_0t)-\mathsf{D}_p(\xi_0, \omega, \theta, \eta)\sin(\theta))^2}{\mathsf c_4 \eta^{6}}\\
			&\qquad+\frac{8\pi^2(\xi_0\cos(2\pi\xi_0t)-r\mathsf{D}_p(\xi_0, \omega, \theta, \eta)\cos(\theta))^2}{\mathsf c_3 \eta^{2}}\\
			&\qquad+\frac{4\pi^2(\xi_0\sin(2\pi\xi_0t)-r\mathsf{D}_p(\xi_0, \omega, \theta, \eta)\sin(\theta))^2}{\gamma^{[\varphi,\varphi]}_0(0)}\\
			&\qquad+\frac{(\cos(2\pi\xi_0t)-\mathsf{D}_p(\xi_0, \omega, \theta, \eta)\cos(\theta))^2}{4\pi^2\gamma^{[\varphi,\varphi]}_2(0)}\Big]\,.\nonumber
		\end{align*}
		A direct control shows that 
		\[
		z^*\underline{\Sigma}^{-1}_\eta z\geq C_1\eta^2C_{\xi_0,\eta}
		\] 
		for some $C_1>0$  when $|\theta-2\pi\xi_0t\mbox{ (mod }2\pi)|> C_2\eta^2$ and $|\omega-\xi_0|> C_2\eta^2$ for some $C_2>0$, where $C_2>0$ will be chosen later.
		 Note that by the third fact in Lemma \ref{Lemma: c of mu_xi0eta when eta small}, we have 
		\begin{equation}\label{Bound: C_xi lower bound when eta->0}
		C_{\xi_0,\eta}\geq C_3 A^2\hat{\hbar}(\xi_0)^2\xi_0^2\eta^{-2}
		\end{equation}
		for some constant $C_3>0$. 
		If we bound $\frac{B^2_{\xi_0,\eta}(\theta,\omega)}
					{A_\eta(\theta,\omega)}$ in $G(\theta,\omega)$ trivially by $C_{\xi_0,\eta}$, when $C_2$ is chosen so that $C_3C_1 A^2\hat{\hbar}(\xi_0)^2\xi_0^2\geq 3k+9$, we trivially have 
					\begin{align}
					G(\theta,\omega)\leq C_4
					\end{align}
					for some constant $C_4>0$.
					Denote 
\begin{equation}
		\mathcal{N}:=\Big\{(\omega,\theta)\in \mathbb{C}\times [0,2\pi)\big|\,|\theta-2\pi\xi_0t\mbox{ (mod }2\pi)|\vee|\omega-\xi_0|\leq C_2\eta^2\Big\}\,,\nonumber
		\end{equation}
		which is a subset of $\mathbb{C}\times [0,2\pi)$.

		By the third fact in Lemma \ref{Lemma: c of mu_xi0eta when eta small}, we immediately have that over $\mathcal{N}$, despite the factor $A\hat{\hbar}(\xi_0-\eta)$, $A_{\eta}(\theta,\omega)$, $B_{\xi_0,\eta}(\theta,\omega)$ and $C_{\xi_0,\eta}$ are close.
		Thus, we have 
		\begin{equation}
		A_\eta(\theta,\omega)/2\leq \frac{B^2_{\xi_0,\eta}(\theta,\omega)}{A_\eta(\theta,\omega)}\leq  2A_\eta(\theta,\omega)\,,
		\end{equation}
		As a result, 
		\begin{align}
		G(\theta,\omega)\leq C_5 \max\{A_\eta(\theta,\omega)^{(k+3)/2}, \,e^{-C_{\xi_0,\eta}}\}=C_5A_\eta(\theta,\omega)^{(k+3)/2}\,,
		\end{align}
		for some constant $C_5>0$, where the last equality holds since $A_\eta(\theta,\omega)\geq \lambda_{\eta,1}^{-1}$, which is of order $\eta^{-\rho}$, and $e^{-C_{\xi_0,\eta}}\leq e^{-C_3 A^2\hat{\hbar}(\xi_0)^2\xi_0^2\eta^{-2}}$ by \eqref{Bound: C_xi lower bound when eta->0}, which is exponentially small when $\eta$ is close to $0$.

		As a result, we have
				\begin{align}
		&\int_{\mathbb{C}}\int_0^{\pi}
				\frac{G(\theta,\omega)}{E^k_{\alpha,\xi}(\omega)A^{2+k/2}_\eta(\theta,\omega)}\,d\theta\,d\omega\nonumber\\
				\leq &\,\int_{\mathcal{N}^c}
				\frac{C_4}{E^k_{\alpha,\xi}(\omega)A^{2+k/2}_\eta(\theta,\omega)}\,d\theta\,d\omega+\int_{\mathcal{N}}
				\frac{C_5}{E^k_{\alpha,\xi}(\omega)A^{1/2}_\eta(\theta,\omega)}\,d\theta\,d\omega\nonumber\\
				\leq &\,C_4\int_{\mathbb{C}}\int_0^{\pi}
				\frac{1}{E^k_{\alpha,\xi}(\omega)A^{2+k/2}_\eta(\theta,\omega)}\,d\theta\,d\omega+\int_{\mathcal{N}}
				\frac{C_5}{E^k_{\alpha,\xi}(\omega)A^{1/2}_\eta(\theta,\omega)}\,d\theta\,d\omega\nonumber\,,
		\end{align}
		where the first term can be bounded by the same way for the null case.
		We now control the second term involving $\mathcal{N}$. Note that over $\mathcal{N}$, $A_\eta(\theta,\omega)$ is close to $C_{\xi_0,\eta}$, which is bounded from below by $C_3 A^2\hat{\hbar}(\xi_0)^2\xi_0^2\eta^{-2}$, so we have
		\begin{align*}
	 \int_{\mathcal{N}}&
				\frac{C_5}{E^k_{\alpha,\xi}(\omega)A^{1/2}_\eta(\theta,\omega)}d\theta\,d\omega\leq 
				\frac{C_5\eta\alpha^{-k/2}}{\sqrt{C_3} A\hat{\hbar}(\xi_0)\xi_0} \int_{\mathcal{N}} e^{-k|\omega-\xi|^2/ \alpha} d\theta d\omega\,
		\end{align*}
		We immediately have
		\begin{align*}
				\int_{\mathcal{N}}  e^{-k|\omega-\xi|^2/ \alpha} d\theta d\omega = &2C_2\eta^2 \int_{|\omega-\xi_0|\leq C_2\eta^2} e^{-k|\omega-\xi|^2/ \alpha} d\omega
				\end{align*}
				 since $|\theta-2\pi\xi_0t\mbox{ (mod }2\pi)|\leq C_2\eta^2$. Finally, we use the trivial bound when $\xi=\xi_0$,
				\begin{align*}
				 \int_{|\omega-\xi_0|\leq C_2\eta^2} e^{-k|\omega-\xi|^2/ \alpha} d\omega\leq \frac{\alpha}{k}(1-e^{-kC_2^2\eta^4/\alpha})\,,
				\end{align*}
				which is bounded by $C_6\alpha \eta$ for a constant $C_6>0$ by the assumption that $\eta>\alpha$, and hence we have
		\begin{align*}
		\int_{\mathcal{N}}
				\frac{C_5}{E^k_{\alpha,\xi}(\omega)A^{1/2}_\eta(\theta,\omega)}d\theta\,d\omega\leq \frac{C_7}{A\hat{\hbar}(\xi_0)\xi_0}\eta^4\alpha^{-k/2+1}
				\end{align*} 
				for a constant $C_7>0$.
		We thus finish the proof.
		
\end{proof}

Finally we describe the moments of $Y_{\alpha,\xi,\eta}$ in  both the null and non-null cases.
			
			\begin{thm}[Moments of $Y_{\alpha,\xi,\eta}$, null case]\label{Proposition joint density Y and omega gen333}
			Suppose Assumptions~\ref{assump:noise part}, \ref{assump:nonnull signal} and \ref{assump:gaussian window} hold. Assume $\hbar$ is a smooth function satisfying $\hat{\hbar}=\hat{h}\varphi$ for a symmetric, smooth and bounded function $\varphi$.
			Fix $\xi>0$ and $k \in \NN$. 
			When $k$ is odd, for any $\eta>0$, we have
			\[
			\mathbb{E}Y_{\alpha,\xi,\eta}^k=0\,.
			\]
			When $k$ is even, for $\eta\geq1$, when $\alpha$ is sufficiently small, we have
			\begin{align}
				|\EE Y_{\alpha,\xi,\eta}^k|
				=O( \alpha^{-k/2+1}) \,,
				\end{align}
				where 
				 the implied constant depends on $\varphi$, $\varrho$, $k$ and $\frac{\eta^{k\varrho/2}e^{-4\pi^2\eta^2}}{(1+4\pi^2|\eta-\xi|^2)^{(k+4)/2}}$; 
				 for $\eta<1$, particularly when $\eta$ is close to $0$, we have a simple bound
			\begin{align}
				|\EE Y^k_{\alpha,\xi,\eta}|\leq  \EE |Y_{\alpha,\xi,\eta}|^k\,.
				\end{align}
				\end{thm}
				
				\begin{proof}
We follow the notation and convention used in the proof for Lemma \ref{LEMMA joint density Y and omega gen}. 
By a direct expansion, we have
			\begin{align}
			\EE Y^k_{\alpha,\xi,\eta}
			=&\,Kc_\eta e^{-C_{\xi_0,\eta}}\int_{\mathbb{C}}  E^4_{\alpha,\xi}(\omega) \int_0^{2\pi}e^{ik\theta}\Big[\int_0^\infty
			r^{k+3}\nonumber\\
			&\qquad\times e^{-E^2_{\alpha,\xi}(\omega) A_\eta(\theta,\omega) r^2}e^{2E_{\alpha,\xi}(\omega) B_{\xi_0,\eta}(\theta,\omega) r}
			\,dr\Big]d\theta\,d\omega\nonumber\,.
			\end{align}
			Note that the only difference between $\mathbb{E}Y^k_{\alpha,\xi,\eta}$ and $\mathbb{E}|Y_{\alpha,\xi,\eta}|^k$ in \eqref{Expansion absolute moment of order k} is the phase $e^{ik\theta}$. Therefore, by the same derivation of \eqref{var eq intermediate}, by \eqref{eq:Hermite function of negative order integral representation}, when $k$ is even, $\EE Y^k_{\alpha,\xi,\eta}$ becomes
			\begin{align}
			 &\,Kc_\eta e^{-C_{\xi_0,\eta}}\int_{\mathbb{C}} 
			\frac{1}{E^k_{\alpha,\xi}(\omega)} 
			\int_0^{\pi} \frac{e^{ik\theta}}{A^{2+k/2}_\eta(\theta,\omega)}
			\nonumber\\
			&\qquad\times \Hypergeometric{1}{1}{\frac{k}{2}+2}{\frac{1}{2}}{\frac{B^2_{\xi_0,\eta}(\theta,\omega)}{A_\eta(\theta,\omega)}}
			\,d\theta\,d\omega\label{var eq intermediate2even}\,,
			\end{align}
			where $K:=\frac{\sqrt{\pi}\Gamma(k+4)}{2^{k+3}\Gamma(\frac{k+5}{2})}$;
			when $k$ is odd, the same derivation of \eqref{var eq intermediate} is slightly changed and we have
			\begin{align}
			 &\,K'c_\eta e^{-C_{\xi_0,\eta}}\int_{\mathbb{C}} 
			\frac{1}{E^k_{\alpha,\xi}(\omega)} 
			\int_0^{\pi} \frac{e^{ik\theta}B_{\xi_0,\eta}(\theta,\omega) }{A^{5/2+k/2}_\eta(\theta,\omega)}
			\nonumber\\
			&\qquad\times \Hypergeometric{1}{1}{\frac{k+5}{2}}{\frac{3}{2}}{\frac{B^2_{\xi_0,\eta}(\theta,\omega)}{A_\eta(\theta,\omega)}}
			\,d\theta\,d\omega\label{var eq intermediate2odd}
			\end{align}
			where $K':=\frac{\sqrt{\pi}\Gamma(k+4)}{2^{k+3}\Gamma(\frac{k+4}{2})} $. Indeed, when we evaluate moments, we have an extra $e^{ik\theta}$ inside the integrant of $d\theta$ in \eqref{var eq intermediate}. Since $e^{ik(\theta+\pi)}=(-1)^ke^{ik\theta}$, when $k$ is odd, $e^{ik(\theta+\pi)}=-e^{ik\theta}$. As a result, the format of odd moments different from that of even moments. 

			Since $B_{\xi_0,\eta}(\theta,\omega)=C_{\xi_0,\eta}=0$ in the null case, clearly we have $\EE Y^k_{\alpha,\eta}=0$ for all $\eta>0$ when $k$ is odd. When $k$ is even, 
			\begin{align}
			 \EE Y^k_{\alpha,\xi,\eta}=&\,Kc_\eta \int_{\mathbb{C}} 
			\frac{1}{E^k_{\alpha,\xi}(\omega)} 
			\int_0^{\pi} \frac{e^{ik\theta}}{A^{2+k/2}_\eta(\theta,\omega)}d\theta\,d\omega\,.
			\end{align}
			In this case, {since $e^{ik(\theta+\pi/2)}=(-1)^{k/2}e^{ik\theta}$ when $\theta\in [0,\pi/2]$, we have }
			\begin{equation*}
			\int_0^{\pi} \frac{e^{ik\theta}}{A^{2+k/2}_\eta(\theta,\omega)}d\theta=\int_0^{\pi/2} e^{ik\theta}\Big[\frac{1}{A^{2+k/2}_\eta(\theta,\omega)}+\frac{(-1)^{k/2}}{A^{2+k/2}_\eta(\theta+\pi/2,\omega)}\Big]d\theta\,.
			\end{equation*}
			\medskip
			
			\underline{{\bf When $\eta\geq 1$ and $k$ is even:}} 		
			When $k/2$ is odd, by using \eqref{Expansion Atheta omega in a different form}, we have 
			\begin{equation}
			|A_\eta(\theta,\omega)-\boldsymbol{\omega}^*\overline{{P}^{-1}}\boldsymbol{\omega}|\leq |\boldsymbol{\omega}^\top\overline{\Gamma^{-1}}C\overline{{P}^{-1}}\boldsymbol{\omega}|= O(e^{-4\pi^2\eta^2}\|\boldsymbol{\omega}\|)\,,
			\end{equation}
			which leads to 
			\begin{equation}\label{Proof moments difference of A}
			\Big|\frac{1}{A^{2+k/2}_\eta(\theta,\omega)}-\frac{1}{A^{2+k/2}_\eta(\theta+\pi/2,\omega)}\Big|=O\Big(\frac{\eta^{(k+4)\rho/2} e^{-4\pi^2\eta^2}}{{(1+4\pi^2|\eta-\omega|^2)^{(k+4)/2}}}\Big)
			\end{equation}
			by {the same approximation like that of \eqref{upper and lower bounds of A-k+4/2 integral} and} a direct binomial expansion, and hence
			\begin{align}
			 |\EE Y^k_{\alpha,\xi,\eta}|=O( \alpha^{-k/2+1})\,,
			\end{align}
			where the implied constant depends on $\varphi$, $\varrho$, $k$ and $\frac{\eta^{k\rho/2}e^{-4\pi^2\eta^2}}{(1+4\pi^2|\eta-\xi|^2)^{(k+4)/2}}$.
			When $k/2$ is even, we simply apply $ |\EE Y^k_{\alpha,\xi,\eta}|\leq  \EE |Y_{\alpha,\xi,\eta}|^k$.
			
			\medskip
			\underline{{\bf When $\eta< 1$, particularly when $\eta$ is close to $0$,  and $k$ is even:}} 
		  In this case, we simply applied the rough bound; that is $ |\EE Y^k_{\alpha,\xi,\eta}|\leq  \EE |Y_{\alpha,\xi,\eta}|^k$, and hence the proof.
			\end{proof}
			
			\begin{remark}
			When $\eta\to 0$, again we need to handle the degeneracy of $\underline{\Sigma}_\eta$.
		 Ideally, we would expect that the trick \eqref{Proof moments difference of A} would work and the moments would be ``smaller'' than the absolute moments due to the oscillation. However, if we examine the proof for the absolute moment carefully, the main place we could possibly obtain an improvement with this trick is when $\theta\in I_3\cap [0,\eta^{5/2}]$ and $\phi\in J_1$, since this is the dominant term in the absolute moments and the other bounds are much smaller. In this case, by a direct expansion with the third fact in Lemma \ref{Lemma: c of mu_xi0eta when eta small}, we see that $\frac{1}{A^{2+k/2}_\eta(\theta,\omega)}$ is of order $\frac{\eta^2}{\theta^2/\eta^4+\eta^2+r^2}$ since $\eta^4r^2\cos^2(\theta+\phi)$ dominates $\eta^6r^2\sin^2(\theta+\phi)$, while $\frac{1}{A^{2+k/2}_\eta(\theta+\pi/2,\omega)}$ is of order $\frac{\eta^6}{1+\eta^4r^2\min\{(\theta+\phi-\pi/2)^2,(\theta+\phi-3\pi/2)^2\}+\eta^6 r^2}$, which is much smaller than $\frac{1}{A^{2+k/2}_\eta(\theta,\omega)}$. Thus, we do not obtain a benefit by the cancellation. Since this bound is not essential for our proof, we do not pursue a better bound but simply apply the rough bound and hence the proof.
			\end{remark}

			The final part of this section is estimating the moments of $Y_{\alpha,\xi,\eta}$ in the non-null case. In this case, the signal plays a role in the analysis. While we can always apply the trivial bound $ |\EE Y^k_{\alpha,\xi,\eta}|\leq  \EE |Y_{\alpha,\xi,\eta}|^k$, it is desirable that the moments are in general ``smaller'' than the absolute moments. However, note that in general we have 
			\[
			A_\eta(\theta,\omega)\neq A_\eta(\theta+\pi/2,\omega)
			\mbox{ 
			and 
			}
			B_{\xi_0,\eta}(\theta,\omega)\neq B_{\xi_0,\eta}(\theta+\pi/2,\omega)
			\] 
			for $\theta\in [0,\pi/2]$. Indeed, by \eqref{Expansion Atheta omega in a different form}, we have
			\begin{align*}
			A^{(h)}_{\eta}(\theta,\omega)&=\boldsymbol{\omega}^* \conj{{P}^{-1}} \boldsymbol{\omega} - \Re \Big(e^{2i\theta}\boldsymbol{\omega}^\top {R}^\top\conj{ P^{-1}} \boldsymbol{\omega}\Big)\,\\
			A^{(h)}_{\eta}(\theta+\pi/2,\omega)&=\boldsymbol{\omega}^* \conj{{P}^{-1}} \boldsymbol{\omega} + \Re \Big(e^{2i\theta}\boldsymbol{\omega}^\top {R}^\top\conj{ P^{-1}} \boldsymbol{\omega}\Big)\,\\
			B^{(h)}_{\xi_0,\eta}(\theta,\omega)&=
			\Re\Big[ e^{i\theta}\Big(\mu_{\xi_0,\eta}^{(h)*}-\mu_{\xi_0,\eta}^{(h)\top} R^\top\Big)\conj{ P^{-1}}\boldsymbol{\omega} \Big]\,\nonumber\\
			B^{(h)}_{\xi_0,\eta}(\theta+\pi/2,\omega)&=
			-\Im\Big[ e^{i\theta}\Big(\mu_{\xi_0,\eta}^{(h)*}-\mu_{\xi_0,\eta}^{(h)\top} R^\top\Big)\conj{ P^{-1}}\boldsymbol{\omega} \Big]\,.
			\end{align*}
			Thus, the same trick like \eqref{Proof moments difference of A} may not work in general. 
			Specifically, while it is intuitive that $\mathbb{E} Y_{\alpha,\xi,\eta}=0$, it is not transparent to see it. Below, we show some results that are sufficient for our purpose, while they are not the most general results.
			
			\begin{thm}[Moments of $Y_{\alpha,\xi,\eta}$, non-null case]\label{Proposition joint density Y and omega gen444}
			Suppose Assumptions~\ref{assump:noise part}, \ref{assump:nonnull signal} and \ref{assump:gaussian window} hold. Assume $\hbar$ is a smooth function satisfying $\hat{\hbar}=\hat{h}\varphi$ for a symmetric, smooth and bounded function $\varphi$.
			Fix $\xi>0$ and $k \in \NN$. 
			When $\eta\geq 1$ and $|\xi_0-\eta|\geq 1/2$, for any $k\in \mathbb{N}$, when $\alpha$ is sufficiently small, we have
			\[
			|\mathbb{E} Y_{\alpha,\xi,\eta}^k|=O(e^{-4\pi^2(\xi_0-\eta)^2})\,,
			\]
			where the implied constant depends on $\EE|Y_{\alpha,\xi,\eta}|^k$.
			When $\eta< 1$ or when $\eta\geq 1$ and $|\xi_0-\eta|< 1/2$, when $\alpha$ is sufficiently small, we have the trivial bound
			\begin{align}
				|\EE Y_{\alpha,\xi,\eta}^k|
				\leq \EE|Y_{\alpha,\xi,\eta}|^k\,.
				\end{align}
				\end{thm}

			\begin{proof}
			We follow the notation and convention used in the proof for Lemma \ref{LEMMA joint density Y and omega gen}. 
			By the same expansion, $\mathbb{E} Y_{\alpha,\xi,\eta}^k$ is shown in \eqref{var eq intermediate2even} and \eqref{var eq intermediate2odd} for even and odd $k$ respectively. The approximation technique is the same as the above, so we only sketch the proof by indicating each key steps.
			
		\medskip
			\underline{{\bf When $\eta\geq 1$ and $|\eta-\xi_0|\geq 1$:}} 
			In this case, we have a bound similar to that of the null case. Indeed, by \eqref{Expansion Atheta omega in a different form}, since $\mu_{\xi_0,\eta} =A\hat{h}(\xi_0-\eta)e^{i2\pi \xi_0 t}\begin{bmatrix}1&i2\pi\xi_0\end{bmatrix}^\top$, when $|\eta-\xi_0|$ is sufficiently large, we have $|B_{\xi_0,\eta}(\theta,\omega)|=O(e^{-2\pi^2(\xi_0-\eta)^2}\|\boldsymbol{\omega}\|)$. Therefore, we have
			\begin{equation}
			\frac{B^2_{\xi_0,\eta}(\theta,\omega)}{A_\eta(\theta,\omega)}=O(e^{-4\pi^2(\xi_0-\eta)^2})\,,
			\end{equation}
			which by \cite[p.323 (13.2.13)]{NIST} leads to
			\begin{align}
			\Hypergeometric{1}{1}{\frac{k}{2}+2}{\frac{1}{2}}{\frac{B^2_{\xi_0,\eta}(\theta,\omega)}{A_\eta(\theta,\omega)}}&\,=1+O(e^{-4\pi^2(\xi_0-\eta)^2})\\
			\Hypergeometric{1}{1}{\frac{k+5}{2}}{\frac{3}{2}}{\frac{B^2_{\xi_0,\eta}(\theta,\omega)}{A_\eta(\theta,\omega)}}&\,=1+O(e^{-4\pi^2(\xi_0-\eta)^2})\,.\nonumber
			\end{align}
			By the same argument as that in the null case, we get the claim. Indeed, 
			the moment is controlled by $O(e^{-4\pi^2(\xi_0-\eta)^2})$, and hence the claim.

			\medskip
			\underline{{\bf When $\eta< 1$ or when $\eta\geq 1$ and $|\eta-\xi_0|<1$:}} 
		  In this case, we simply applied the rough bound; that is $ |\EE Y^k_{\alpha,\xi,\eta}|\leq  \EE |Y_{\alpha,\xi,\eta}|^k$, and hence the proof.
		  
		\end{proof}

With the above theorems, we have the following corollary. 

			\begin{corro}[Variance of $Y_{\alpha,\xi,\eta}$]\label{Proposition joint density Y and omega gen COR1}
			Suppose Assumptions~\ref{assump:noise part}, \ref{assump:nonnull signal} and \ref{assump:gaussian window} hold. Assume $\hbar$ is a smooth function satisfying $\hat{\hbar}=\hat{h}\varphi$ for a symmetric, smooth and bounded function $\varphi$.
			When $A=0$, we have
			\begin{enumerate}
			\item for $\eta\geq1$ and $\xi>0$, when $\alpha$ is sufficiently small, we have
			\begin{align*}
				\textup{Var}({Y}_{\alpha,\xi,\eta})
				\asymp c_1 \,,
				\end{align*}
				where 
				 $c_1$ depends on $\varphi$, $\varrho$ and $\frac{\eta^{\varrho}}{(1+4\pi^2|\eta-\xi|^2)^{3}}$; 
				\item for $\eta<1$ and $\xi>0$, when $\alpha$ is sufficiently small, we have
				 \begin{align}
				c_{2,1} \leq \textup{Var}({Y}_{\alpha,\xi,\eta})
				\leq  c_{2,2} \max\Big\{\frac{1}{\sqrt{\alpha\eta}}e^{-\frac{\xi^2}{\alpha}},\, 1\Big\}\,,
				\end{align}
				where 
				 $c_{2,1}$ depends on $\varphi$ and $\varrho$ and $\frac{\eta^{14}}{(1+4\pi^2|\eta-\xi|^2)^{(k+4)/2}}$, and $c_{2,2}$ depends on $\varphi$ and $\varrho$.
\end{enumerate}
When $A>0$, we have
\begin{enumerate} 
				\item For $\eta\geq 1$ so that $|\eta-\xi_0|\geq 1/2$ and $\xi>0$, when $\alpha$ is sufficiently small, we have
			\begin{align*}
				\textup{Var}({Y}_{\alpha,\xi,\eta})\asymp c_3\,,
				\end{align*}
				where 
				$c_3$ depends on $\xi_0$, $A$, $\varphi$, $\varrho$, and $\frac{\eta^{\varrho}}{(1+4\pi^2|\eta-\xi|^2)^{3}}$; 
			\item	
				for $\eta\geq 1$ so that $|\eta-\xi_0|< 1/2$, when $\alpha$ is sufficiently small, we have
			\begin{align*}
				\textup{Var}({Y}_{\alpha,\xi,\eta}) \leq  c_4\,, \nonumber
				\end{align*}
				where 
				  $c_4$ depends on $\varphi$, $\varrho$, $k$ and $A^{k+3}\xi_0^{(2-\rho)(k+3)/2}$;

				\item when $\eta<1$,
				 when $\alpha$ is sufficiently small so that $\alpha<\eta$, we have
			\begin{align*}
				\textup{Var}({Y}_{\alpha,\xi,\eta}) \leq  c_5  \max\Big\{\frac{1}{\sqrt{\alpha\eta}}e^{-\frac{\xi^2}{\alpha}},\, 1,\,\frac{1}{A\hat{\hbar}(\xi_0)\xi_0}\Big\} \,,\nonumber
				\end{align*}
				where 
				  $c_5$ depends on $\varphi$ and $\varrho$.
				 \end{enumerate}
				\end{corro}
%

%
			\section{Proof of covariance-related theorems}
\label{Section: proof of M dependent related arguments}

This section is divided into two parts. The first part is quantifying how the window perturbation impacts the variance of ${Y}^{(\hbar)}_{\alpha,\xi,\eta}$ when the window is $\hbar$. 
The second part is evaluating the covariance of ${Y}^{(\hbar)}_{\alpha,\xi,\eta}$ and ${Y}^{(\hbar)}_{\alpha,\xi,\eta'}$ when $\eta$ and $\eta'$ are far apart and how window perturbation impacts this covariance. 
	Intuitively, suppose we have a window $h_2$ that is ``very close'' to $h_1$. It is intuitive to expect that all quantities regarding SST are also ``very close'' if we replace $h_1$ by $h_2$; that is,
		 we expect similar expectations and 
		\[
			\text{Var}({Y}^{(h_1)}_{\alpha,\xi,\eta})
			\approx
			\text{Var}( {Y}^{(h_2)}_{\alpha,\xi,\eta})
			\approx 
			\text{Cov}({Y}^{(h_1)}_{\alpha,\xi,\eta},{Y}^{(h_2)}_{\alpha,\xi,\eta})
		\]
		 for any $(t,\eta)$, where the covariance of ${Y}^{(h_1)}_{\alpha,\xi,\eta}$ and ${Y}^{(h_2)}_{\alpha,\xi,\eta}$ captures the ``similarity'' of two random variables. The following two lemmas quantify this intuition. 

		\begin{lemma}\label{Lemma: EY and EvY difference}
		Suppose Assumptions~\ref{assump:noise part}, \ref{assump:nonnull signal} and \ref{assump:gaussian window} hold. Take the notations and assumptions in Lemma \ref{Lemma covariance eigenstructure after perturbation}. When $A=0$, $\eta,\xi,\alpha>0$, we have
		\[
				\mathbb{E}Y^{(h_1)}_{\alpha,\xi,\eta}=\mathbb{E}{Y}^{(h_2)}_{\alpha,\xi,\eta}=0\,.
			\]
		When $A>0$, for $\eta,\xi,\alpha>0$, we have
			\[
				|\mathbb{E}Y^{(h_1)}_{\alpha,\xi,\eta}-\mathbb{E}{Y}^{(h_2)}_{\alpha,\xi,\eta}|=O(\epsilon)\,,
			\]
			where the implied constant depends on $\varphi_1$, $\varphi_2$ and $\mathbb{E}|Y^{(h_1)}_{\alpha,\xi,\eta}|$.
		\end{lemma}

		\begin{proof}
		The null case follows the same proof of Theorem \ref{Proposition joint density Y and omega gen333}. We next show the non-null case. To ease the notation, in this proof we denote 
		$Y_{\alpha,\xi,\eta}:=Y^{(h_1)}_{\alpha,\xi,\eta}\mbox{ and }\mathsf Y_{\alpha,\xi,\eta}:=Y^{(h_2)}_{\alpha,\xi,\eta}$ and omit the superscript $h_1$ and $h_2$ describing the dependence on the windows.
		The proof is a direct calculation based on the derivations in the proof of Lemma \ref{LEMMA joint density Y and omega gen}. We follow the same notations there. Since the proof is similar, we provide the main ingredients and skip details. By \eqref{var eq intermediate2odd}, we have $\mathbb E Y_{\alpha,\xi,\eta}$.
			By the same derivation like that of \eqref{var eq intermediate2odd}, we also have $\mathbb E  \mathsf Y_{\alpha,\xi,\eta}$ expanded like
			\begin{align}
			 K'\mathsf c_\eta e^{-\mathsf C_{\xi_0,\eta}}\int_{\mathbb{C}} 
			\frac{s}{E_{\alpha,\xi}(\omega)} 
			\int_0^{\pi} \frac{e^{i\theta}\mathsf B_{\xi_0,\eta}(\theta,\omega) }{\mathsf A^{3}_\eta(\theta,\omega)}
			\Hypergeometric{1}{1}{3}{\frac{3}{2}}{\frac{\mathsf B^2_{\xi_0,\eta}(\theta,\omega)}{4 \mathsf A_\eta(\theta,\omega)}}
			\,d\theta\,d\omega\,,\nonumber
			\end{align}
			where $\mathsf c_\eta$, $\mathsf A_\eta(\theta,\omega)$, $\mathsf B_{\xi_0,\eta}(\theta,\omega)$ and $\mathsf C_{\xi_0,\eta}$ are defined in the same way as \eqref{Expansion:ABC} with the window $h_2$. 
			Thus, $|\mathbb E  Y_{\alpha,\xi,\eta}-\mathbb E \mathsf Y_{\alpha,\xi,\eta}|$ is controlled by bounding $\left|\mathsf c_\eta e^{-\mathsf C_{\xi_0,\eta}}- c_\eta e^{- C_{\xi_0,\eta}}\right|$, $\left|\frac{\mathsf B_{\xi_0,\eta}(\theta,\omega) }{\mathsf A^{3}_\eta(\theta,\omega)}-\frac{ B_{\xi_0,\eta}(\theta,\omega) }{ A^{3}_\eta(\theta,\omega)}\right|$ 
			and $\left|\frac{\mathsf B^2_{\xi_0,\eta}(\theta,\omega) }{\mathsf A_\eta(\theta,\omega)}-\frac{ B^{2}_{\xi_0,\eta}(\theta,\omega) }{ A_\eta(\theta,\omega)}\right|$ and taking the smoothness of $\Hypergeometric{1}{1}{3}{\frac{3}{2}}{x}$ when $x>0$ into account. Note that we have $\frac{d}{dx}\Hypergeometric{1}{1}{3}{\frac{3}{2}}{x}=2\Hypergeometric{1}{1}{4}{\frac{5}{2}}{x}$, $0\leq \frac{ B^{2}_{\xi_0,\eta}(\theta,\omega) }{ A_\eta(\theta,\omega)}\leq 1$ and $\Hypergeometric{1}{1}{4}{\frac{5}{2}}{x}$ is bounded for $x\in [0,1]$, the error term involving perturbing the confluent hypergeometric function is thus well controlled. 
			As a result, an inequality like controlling the absolute moments \eqref{Proof:Bound k abs moments of Y part1} is achieved, and with the same trick like \eqref{Proof:Bound k abs moments of Y part2}, we are done with the proof.
			\end{proof}

		\begin{lemma}\label{Main Lemma 2 for Kernel SST}
			Suppose Assumptions~\ref{assump:noise part}, \ref{assump:nonnull signal} and \ref{assump:gaussian window} hold and follow the notations and assumptions in Lemma \ref{Lemma covariance eigenstructure after perturbation}. Fix $\alpha, \xi, \eta>0$. 
			When $\alpha$ is sufficiently small, we have
			\begin{equation}
				\begin{aligned}
					&|\text{Var}({Y}^{(h_1)}_{\alpha,\xi,\eta})
					-
					\text{Cov}({Y}^{(h_1)}_{\alpha,\xi,\eta}, {Y}^{(h_2)}_{\alpha,\eta})|
					=
					O(\alpha^{-5}\epsilon^3)\\
					&|\text{Cov}({Y}^{(h_1)}_{\alpha,\xi,\eta}, \overline{{Y}^{(h_1)}_{\alpha,\xi,\eta}})
					-
					\text{Cov}({Y}^{(h_1)}_{\alpha,\xi,\eta}, \overline{{Y}^{(h_2)}_{\alpha,\xi,\eta})}|
					=
					O(\alpha^{-5}\epsilon^3)\,,
					\end{aligned}
					\end{equation}
					where the implied constant depends on $\varphi_1$, $\varphi_2$ and $\varrho$, and the dependence on $\eta$ is $O(\eta^{-1/2})$ when $\eta$ is sufficiently small.
		\end{lemma}
		
		Note that when we control $|\text{Var}({Y}^{(h_1)}_{\alpha,\xi,\eta})
					-
					\text{Cov}({Y}^{(h_1)}_{\alpha,\xi,\eta}, {Y}^{(h_2)}_{\alpha,\xi,\eta})|$ and other terms, the error depends on $\alpha^{-5}$ instead of $\alpha^{-1}$. This comes from the fact that we count on the approximation of identity to finish the proof. While this bound might not be optimal, it is sufficient for our purpose.
Lemma \ref{Main Lemma 2 for Kernel SST} is the key for the perturbation argument, particularly the $M$-dependence argument for the asymptotical behavior of SST. At the first glance, the statement is intuitively clear. However, due to the nonlinear structure of the transform, it is not clear how the error term is controlled. Since the error term is critical for the asymptotical analysis of SST, with the above preparations, below we provide a careful analysis.

\begin{proof} 
		To ease the notation, in this proof we denote 
		\[
		Y_{\alpha,\xi,\eta}:=Y^{(h_1)}_{\alpha,\xi,\eta}\mbox{ and }\mathsf Y_{\alpha,\xi,\eta}:=Y^{(h_2)}_{\alpha,\xi,\eta}\,.
		\]
		Since all bounds follow the same arguments, below we only quantify how similar $\text{Var}(Y_{\alpha,\xi,\eta})$ and $\text{Cov}(Y_{\alpha,\xi,\eta}, \mathsf Y_{\alpha,\xi,\eta})$ are when $\epsilon$ is sufficiently small. We start from the joint structure of $Y_{\alpha,\xi,\eta}$ and $\mathsf Y_{\alpha,\xi,\eta}$. 
		Suppose the joint density function of random variables $Y_{\alpha,\xi,\eta}$ and $\mathsf Y_{\alpha,\xi,\eta}$ is $f_{Y_{\alpha,\xi,\eta},\mathsf Y_{\alpha,\xi,\eta}}$. We have 
			\begin{align}
				\text{Cov}(Y_{\alpha,\xi,\eta},\mathsf Y_{\alpha,\xi,\eta})=\iint &f_{Y_{\alpha,\xi,\eta},\mathsf Y_{\alpha,\xi,\eta}}(y_1,y_2)\nonumber\\
				&\times(y_1-\mathbb{E}Y_{\alpha,\xi,\eta})(y_2-\mathbb{E}\mathsf Y_{\alpha,\xi,\eta})^*\textup dy_1\textup dy_2\label{Definition Expansion Cov Y Yv}.
			\end{align}
			\underline{\bf{Step 1}:} We start from evaluating the joint density of $(Y_{\alpha,\xi,\eta},\mathsf Y_{\alpha,\xi,\eta})$ by applying Lemma \ref{Lemma:ChangeVariableJacobianC4};
			that is, take $g_{Y_1Y_2}:\CC^4\setminus\{z_1z_3=0\} \to \CC^4\setminus\{z_1z_3=0\}$ defined in Lemma \ref{Lemma:ChangeVariableJacobianC4} to build the density function of 
			\[
				\B{Z}_{\alpha,\xi,\eta}=\begin{bmatrix}Y_{\alpha,\xi,\eta}&Q_\eta&\mathsf Y_{\alpha,\xi,\eta}&\mathsf Q_{\eta}\end{bmatrix}^\top\in \mathbb{C}^4
			\]
			from the gaussian vector $\B{U}:=\mu^{(h_1, h_2)}_{\xi_0,\eta,\eta}+\B{W}^{(h_1, h_2)}_{\eta,\eta}\in \mathbb{C}^4$. 
			By the change of variable and Lemma \ref{Lemma:ChangeVariableJacobianC4}, we have
			\begin{align}
				f_{\B{Z}_{\alpha,\eta}}(y_1,\omega_1,y_2,\omega_2)=&16\pi^4 |y_1|^2|y_2|^2E^4_{\alpha,\xi}(\omega_1)E^4_{\alpha,\xi}(\omega_2)\label{Equation fZ for covariance Y vY}\\
				&\times f_{\B{U}}(g^{-1}(y_1,\omega_1,y_2,\omega_2))\,\nonumber
			\end{align}
			and hence the density function $f_{Y_{\alpha,\xi,\eta},\mathsf Y_{\alpha,\xi,\eta}}(y_1,y_2)$ by marginalizing out $\omega_1$ and $\omega_2$.
			By Lemma \ref{Lemma covariance eigenstructure after perturbation}, we know how the eigenstructure of $\underline{\Sigma}_{\eta,\eta}^{(h_1, h_2)}$ is related to that of $\underline{\Sigma}_{\eta}^{(h_1)}$. 
			
			\underline{\bf{Step 2}:} We now carry out some change of variables in order to compare the covariance structures associated with $\text{Cov}(Y_{\alpha,\xi,\eta},\mathsf Y_{\alpha,\xi,\eta})$ and $\text{Var}(Y_{\alpha,\xi,\eta})$. First, run a change of variable by mapping $g_{Y_1Y_2}^{-1}(y_1,\omega_1,y_2,\omega_2)-\mu^{(h_1, h_2)}_{\xi_0,\eta,\eta}\in \mathbb{C}^4$ to $(a,b,c,d)\in \mathbb{C}^4$; that is, 
			\begin{align}
			y_1&\,=(a+\mu_1)E_{\alpha,\xi}\Big(\frac{b+\mu_2}{2\pi i(a+\mu_1)}\Big)^{-1}\\
			y_2&\,=(c+\mu_3)E_{\alpha,\xi}\Big(\frac{d+\mu_4}{2\pi i(c+\mu_3)}\Big)^{-1}.\nonumber
			\end{align}
			A direction $\mathbb{CR}$-calculation shows that this transform has the Jacobian $(2\pi)^{-4} |y_1|^{-2}|y_2|^{-2}E^{-4}_{\alpha,\xi}(\omega_1)E^{-4}_{\alpha,\xi}(\omega_2)$. Note that $y_1$ and $y_2$ are not defined at $0$, and $E_{\alpha,\xi}(\omega_1)$ and $E_{\alpha,\xi}(\omega_2)$ are both nonzero. The expansion of $\text{Cov}(Y_{\alpha,\xi,\eta},\mathsf Y_{\alpha,\xi,\eta})$ in \eqref{Definition Expansion Cov Y Yv} therefore becomes
			\begin{align}
				&\iint \Big\{\iint 16\pi^4 |y_1|^2|y_2|^2E^4_{\alpha,\xi}(\omega_1)E^4_{\alpha,\xi}(\omega_2) \nonumber\\
				&\qquad\qquad\times f_{\B{U}_{\xi_0,\eta}}(g_{Y_1Y_2}^{-1}(y_1,\omega_1,y_2,\omega_2))\textup d\omega_1\textup d\omega_2\Big\} \label{Expansion Cov Y vY first}\\
				&\qquad\qquad\times(y_1-\mathbb{E}Y_{\alpha,\xi,\eta})(y_2-\mathbb{E}\mathsf Y_{\alpha,\xi,\eta})^* \textup dy_1\textup dy_2  \,\nonumber\\
				=&\,\frac{1}{\pi^4\sqrt{\det(\underline\Sigma_{\eta,\eta}^{(h_1, h_2)})}}\iiiint e^{-F_4(a,b,c,d)}G(a,b)\tilde{G}(c,d)^*\textup da\textup db \textup dc \textup dd\nonumber\,,
			\end{align}
			where $F_4(a,b,c,d) := \frac{1}{2}\underline {\boldsymbol u}^*  \underline{\Sigma}_{\eta,\eta}^{(h_1, h_2)-1}\underline{\boldsymbol u}$, $\boldsymbol u:=\begin{bmatrix} a&b&c&d \end{bmatrix}$, 
			\begin{equation}
			G(a,b):=(a+\mu_1)E_{\alpha,\xi}^{-1}\Big(\frac{1}{2\pi i}
				\frac{b+\mu_2}{a+\mu_1}\Big)-\mathbb{E} Y_{\alpha,\xi,\eta}
				\end{equation} 
				and 
				\begin{equation}
				\tilde{G}(c,d):=(c+\mu_3)E_{\alpha,\xi}^{-1}\Big(\frac{1}{2\pi i}
				\frac{d+\mu_4}{c+\mu_3}\Big)-\mathbb{E}\mathsf Y_{\alpha,\xi,\eta}\,. 
				\end{equation}
				Since $E_{\alpha,\xi}^{-1}(z)$ behaves like a Gaussian function centered at $\xi$, $G$ is a smooth function over $(\mathbb{C}\backslash\{-\mu_1\})\times \mathbb{C}$. When $a\to -\mu_1$ and $b+\mu_2\neq 0$, $\frac{b+\mu_2}{a+\mu_1}\to \infty$ and $G(a,b)\to 0-\mathbb{E} Y_{\alpha,\xi,\eta}=-\mathbb{E} Y_{\alpha,\xi,\eta}$. When $a\to \infty$, $\frac{b+\mu_2}{a+\mu_1}\to0$ and $G(a,b)\to \infty$ linearly.
				Note that 
			\begin{align}
			&\underline{\boldsymbol{u}}^* \underline{\Sigma}_{\eta,\eta}^{(h_1, h_2)-1}\underline{\boldsymbol{u}}=\underline{\boldsymbol{u}}^*\mathsf P^\top (\mathsf P \underline{\Sigma}_{\eta,\eta}^{(h_1, h_2)-1}\mathsf P^\top)  \mathsf P\underline{\boldsymbol{u}}\\
			=&\,(\mathsf P\underline{\boldsymbol{u}})^* J J^\top (\mathsf P \underline \Sigma_{\eta,\eta}^{(h_1, h_2)-1} \mathsf P^\top) JJ^\top \mathsf P \underline{\boldsymbol{u}} \,,\nonumber
			\end{align}
			where $\mathsf P \underline{\Sigma}_{\eta,\eta}^{(h_1, h_2)-1}\mathsf P^\top$ is used to enhance the eigenstructure relationship between $\underline{\Sigma}_{\eta,\eta}^{(h_1,h_2)}$ and $\underline\Sigma_{\eta}^{(h_1)}$ discussed above and 
			\begin{equation}
			J =\frac{1}{\sqrt{2}} \begin{bmatrix}I_4 & I_4 \\ I_4 & -I_4\end{bmatrix}\in O(8)
			\end{equation}
			is chosen to pair parameters representing ``highly correlated'' random variable pairs, $(Q_{\eta},\,\mathsf Q_{\eta})$ and $(Y_{\alpha,\xi,\eta},\,\mathsf Y_{\alpha,\xi,\eta})$. By denoting $p=\frac{a+c}{\sqrt{2}}$, $q=\frac{b+d}{\sqrt{2}}$, $\zeta=\frac{a-c}{\sqrt{2}}$ and $\upsilon=\frac{b-d}{\sqrt{2}}$, we see that 
			\begin{equation}
			J^\top \mathsf P\underline{\boldsymbol{u}}=\begin{bmatrix}p & q& \bar p & \bar q& \zeta& \upsilon& \bar\zeta& \bar \upsilon \end{bmatrix}^\top.
			\end{equation}
			Clearly, this is a change of variable by a rotation, and after this pairing step, $\text{Cov}(Y_{\alpha,\xi,\eta},\mathsf Y_{\alpha,\xi,\eta})$ becomes
			\begin{align}
				\frac{1}{\pi^4\sqrt{\det(\underline\Sigma_{\eta,\eta}^{(h,\hbar)})}}\iiiint& e^{-\tilde F_4(p,q,\zeta,\upsilon)}G\Big(\frac{p+\zeta}{\sqrt{2}},\frac{q+\nu}{\sqrt{2}}\Big)\tilde G\Big(\frac{p-\zeta}{\sqrt{2}},\frac{q-\nu}{\sqrt{2}}\Big)^*\textup dp\textup dq \textup d\zeta \textup d\upsilon\label{Expansion Cov of Y Yv 2nd}\,,
			\end{align}
			where 
			\begin{equation}
			\tilde F_4(p,q,\zeta,\upsilon) := \frac{1}{2}\begin{bmatrix}\underline{\boldsymbol{Q}}  & \underline{\boldsymbol{\Upsilon}} \end{bmatrix}^* J^\top (\mathsf P \underline{\Sigma}_4^{-1}\mathsf P^\top) J\begin{bmatrix}\underline{\boldsymbol{Q}}  & \underline{\boldsymbol{\Upsilon}} \end{bmatrix}^\top, 
			\end{equation}
			$\boldsymbol Q:=\begin{bmatrix} p&q\end{bmatrix}^\top$ and $\boldsymbol \Upsilon:=\begin{bmatrix} \zeta&\upsilon\end{bmatrix}^\top$.
			To continue, note that by \eqref{expansion:U4}, we have
			\begin{equation}
			U_4^* J=
				\begin{bmatrix}
				\tilde C^* U_2^* & (I+B^*)\Gamma_2^*U_2^*\\
				(I+D^*)U_2^* & C^*\Gamma_2^*U_2^*
				\end{bmatrix}\,.
			\end{equation}
			As a result, we have
						\begin{align}
				\tilde F_4(p,q,\zeta,\upsilon)=&\frac{1}{2}\Big[\|{\Lambda_1'}^{-1/2}(\tilde C^*U_2^*\underline{\boldsymbol Q}+(I+B^*)\Gamma_2^*U_2^*\underline{\boldsymbol{\Upsilon}})\|^2\nonumber\\
				&\quad+\|(2D_2+\Lambda_2')^{-1/2}((I+D^*)U_2^*\underline{\boldsymbol Q}+C^*\Gamma_2^*U_2^*\underline{\boldsymbol{\Upsilon}})\|^2\Big]\label{Expansion F4 in Q Upsilon}\\
				=&\,\frac{1}{2}\Big[\|{\Lambda_1'}^{-1/2}\tilde C^*U_2^*\underline{\boldsymbol Q}\|^2+\|{\Lambda_1'}^{-1/2}(I+B^*)\Gamma_2^*U_2^*\underline{\boldsymbol{\Upsilon}}\|^2\nonumber\\
				&\quad+\|(2D_2+\Lambda_2')^{-1/2}(I+D^*)U_2^*\underline{\boldsymbol Q}\|^2\nonumber\\
				&\quad+\|(2D_2+\Lambda_2')^{-1/2}C^*\Gamma_2^*U_2^*\underline{\boldsymbol{\Upsilon}}\|^2\nonumber\\
				&\quad+\Re(\boldsymbol{Q}^*U_2\tilde C {\Lambda_1'}^{-1}(I+B^*)\Gamma_2^*U_2^*\underline{\boldsymbol{\Upsilon}})\nonumber\\
				&\quad+\Re(\boldsymbol{Q}^*U_2(I+D)(2D_2+\Lambda_2')^{-1}C^*\Gamma_2^*U_2^*\underline{\boldsymbol{\Upsilon}})\Big]\,.\nonumber
			\end{align}

			\underline{\bf{Step 3}:} Finally, we apply the asymptotic expansion. 
			Note that when $\epsilon$ is sufficiently small, $\Lambda'_1$ is sufficiently small and we would expect that an integration against $\zeta$ and $\upsilon$ behaves like an approximation of identity at $\boldsymbol \Upsilon=0$ with a controllable error term. 
			To precisely carry out this step, we view the augmented vector $\underline{\boldsymbol\Upsilon}\in \mathbb{C}^4$ as a real 4-dim variable, and set 
			\begin{equation}
			(I+B^*)\Gamma_2^*U_2^*\underline{\boldsymbol{\Upsilon}}\mapsto \boldsymbol{x}:=[x_1,x_2,x_3,x_4]^\top\in \mathbb{R}^4\label{Proof m-dependence Change of Variable complex to real}
			\end{equation} 
			as a change of variable, which is nothing but a rotation of $\mathbb{R}^4\subset \mathbb{C}^4$ in $\mathbb{C}^4$ with a ``small''
			 dilation perturbation via $(I+B^*)$
			 Note that 
			 \begin{equation}
			 T:=\frac{1}{2}\begin{bmatrix}
			 1 & 0 & 1 & 0\\
			 0 & 1 & 0 & 1 \\
			 -i & 0 & i & 0\\
			 0 & -i & 0 & i
			 \end{bmatrix}
			 \end{equation}
			 maps $\underline{\boldsymbol{\Upsilon}}$ to $(\Re \zeta, \Re\upsilon, \Im \zeta, \Im\upsilon)$, so $(I+B^*)\Gamma_2^*U_2^*\underline{\boldsymbol{\Upsilon}}=(I+B^*)\Gamma_2^*U_2^*T^{-1}(T\underline{\boldsymbol{\Upsilon}})$. By noting that the $|\textup{det}T|$ is $1/4$, we get that the Jacobian of \eqref{Proof m-dependence Change of Variable complex to real} is $\frac{1}{4}\det (I+B)$. 
			 Therefore, the integration against $\textup d\zeta\textup d \upsilon$ over $\mathbb{C}^2$ becomes an integration against $\frac{1}{4}\det(I+B)\textup{d} \boldsymbol x$ over $\mathbb{R}^4$. 
			 Also, $\|{\Lambda_1'}^{-1/2}(I+B^*)\Gamma_2^*U_2^*\underline{\boldsymbol{\Upsilon}}\|^2$ in \eqref{Expansion F4 in Q Upsilon} becomes $\|{\Lambda_1'}^{-1/2}\boldsymbol{x}\|^2$, and $G\Big(\frac{p+\zeta}{\sqrt{2}},\frac{q+\nu}{\sqrt{2}}\Big)$ in \eqref{Expansion Cov of Y Yv 2nd} becomes 
			 \begin{equation}
			 G\Big(\frac{p+e_1^\top U_2\Gamma_2(I+B^*)^{-1}\boldsymbol{x}}{\sqrt{2}},\frac{q+e_2^\top U_2\Gamma_2(I+B^*)^{-1}\boldsymbol{x}}{\sqrt{2}}\Big)\,,\nonumber
			 \end{equation}
			where $e_i\in \mathbb{R}^4$ is a unit vector with $1$ in the $i$-th entry, $i=1,\ldots,4$. The other terms are expanded in the same way and \eqref{Expansion Cov of Y Yv 2nd} becomes
			\begin{align}
			\frac{\det(I+B)}{4\pi^4\sqrt{\det(\underline\Sigma_{\eta,\eta}^{(h_1,h_2)})}}\iint &\, e^{-\frac{1}{2}\big[\|{\Lambda_1'}^{-1/2}\tilde C^*U_2^*\underline{\boldsymbol Q}\|^2+\|(2D_2+\Lambda_2')^{-1/2}(I+D^*)U_2^*\underline{\boldsymbol Q}\|^2\big]}\nonumber\\
			\times&\Big\{\int_{\mathbb{R}^4} e^{-\frac{1}{2}\|{\Lambda_1'}^{-1/2}\boldsymbol{x}\|^2}H_{p,q}(\boldsymbol{x}) d\boldsymbol{x}\Big\}\textup dp\textup dq \,,\nonumber
			\end{align}
			where 
			\begin{align}
			H_{p,q}(\boldsymbol{x}):=&\,e^{-\frac{1}{2}\|(2D_2+\Lambda_2')^{-1/2}C^* (I+B^*)^{-1}\boldsymbol{x}\|^2}e^{-\frac{1}{2}\Re(\boldsymbol{Q}^*U_2\tilde C {\Lambda_1'}^{-1}{\boldsymbol{x}})}\nonumber\\
			&\times e^{-\frac{1}{2}\Re(\boldsymbol{Q}^*U_2(I+D)(2D_2+\Lambda_2')^{-1}C^*(I+B^*)^{-1}\boldsymbol{x})}\\
			&\times G\Big(\frac{p+e_1^\top U_2\Gamma_2(I+B^*)^{-1}\boldsymbol{x}}{\sqrt{2}},\frac{q+e_2^\top U_2\Gamma_2(I+B^*)^{-1}\boldsymbol{x}}{\sqrt{2}}\Big) \nonumber\\
			&\times \tilde G\Big(\frac{p-e_1^\top U_2\Gamma_2(I+B^*)^{-1}\boldsymbol{x}}{\sqrt{2}},\frac{q-e_2^\top U_2\Gamma_2(I+B^*)^{-1}\boldsymbol{x}}{\sqrt{2}}\Big)^*\,\nonumber
			\end{align}
			is a smooth function. 
			Note that $\|{\Lambda_1'}^{-1/2}\boldsymbol{x}\|^2=\boldsymbol{x}^\top {\Lambda_1'}^{-1}\boldsymbol{x}$ is of order $e^{-M^2}\|\boldsymbol{x}\|^2$ by the previous calculation,
			so we can now apply the approximation of identity at $\boldsymbol x=0$. To do so, we check the property of $H_{p,q}(\boldsymbol{x})$.
			Note that $(2D_2+\Lambda_2')^{-1/2}C^* (I+B^*)^{-1}$ is of order $\epsilon^2$ and $U_2(I+D)(2D_2+\Lambda_2')^{-1}C^*(I+B^*)^{-1}$ is of order $\epsilon^2$. Also note that $\tilde C {\Lambda_1'}^{-1}$ is of order $1$ since both $\tilde C$ and $\Lambda_1'$ are of order $\epsilon$ and by \eqref{Expansion Gamma2 and Lambda2'} and \eqref{Expansion tilde C} we have
			\begin{equation}
			\tilde C {\Lambda_1'}^{-1}=\frac{1}{4}D_2^{-1}U_2^*\breve E_2(\breve E_2-2\tilde{E})^{-1}U_2\Gamma_2\,
			\end{equation}
			up to error of order $O(\epsilon^2)$. Also, the $G$ and $\tilde{G}$ functions are smooth and both grow up to $\infty$ linearly when $\|\boldsymbol{x}\|\to \infty$. Therefore, the Hessian of $H$ at $0$ is of order $1$. As a consequence, we have the following approximation of identity:
			\begin{align}
			&\int_{\mathbb{R}^4} e^{-\frac{1}{2}\|{\Lambda_1'}^{-1/2}\boldsymbol{x}\|^2}H_{p,q}(\boldsymbol{x}) d\boldsymbol{x}\\
			=&\,4\pi^2 \sqrt{\det(\Lambda_1')}\Big[H_{p,q}(0)+\frac{1}{2}\text{tr}({\Lambda'_1}\nabla^2H_{p,q}(0))+O(\|\Lambda'_1\|^2)\Big]\,\nonumber\\
			=&\,4\pi^2 \sqrt{\det(\Lambda_1')}H_{p,q}(0)(1+O(\alpha^{-4}e^{-3M^2}))\,,\nonumber
			\end{align}
			where the last equality holds when $\alpha$ is small. Also note that $\nabla^2H_{p,q}(0)$ is bounded for any pairs of $p,q$. 
			As a result, \eqref{Expansion Cov of Y Yv 2nd} becomes
			\begin{align}
				\frac{\det(I+B)(1+O(\alpha^{-4}e^{-3M^2}))}{\pi^2\sqrt{\det(2D_2+\Lambda_2')}}&\iint e^{-\frac{1}{2}\big[\|{\Lambda_1'}^{-1/2}\tilde C^*U_2^*\underline{\boldsymbol Q}\|^2+\|(2D_2+\Lambda_2')^{-1/2}(I+D^*)U_2^*\underline{\boldsymbol Q}\|^2\big] }\nonumber\\
				&\times G\Big(\frac{p}{\sqrt{2}},\frac{q}{\sqrt{2}}\Big)\tilde{G}\Big(\frac{p}{\sqrt{2}},\frac{q}{\sqrt{2}}\Big)^*\textup dp\textup dq\,,\label{Expansion Cov of Y Yv 3rd} 
			\end{align}
			where we use the fact that $\frac{1\sqrt{\det (\Lambda_1')}}{\pi^2\sqrt{\det(\underline\Sigma_{\eta,\eta}^{(h,\hbar)})}}=\frac{1}{\pi^2\sqrt{\det(2D_2+\Lambda_2')}}$.
			Note that $\|{\Lambda_1'}^{-1/2}\tilde C^*U_2^*\underline{\boldsymbol Q}\|^2$ is controlled by $O(\epsilon\|\boldsymbol Q\|^2)$ since $\Lambda_1'$ is of order $\epsilon$ and $\tilde C$ is of the same order. Also note that $\sqrt{\det(2D_2+\Lambda_2')}=4\sqrt{\det(D_2)}+O(\epsilon^2)$.
			On the other hand, by Lemma \ref{Lemma: EY and EvY difference}, we have 
			\[
				\mathbb{E}Y_{\alpha,\xi,\eta}=\mathbb{E}\mathsf Y_{\alpha,\xi,\eta}+O(\epsilon)\,.
			\]
			Also, $\mu_3=\mu_1+O(\epsilon)$ and $\mu_4=\mu_2+O(\epsilon)$. By plugging these approximations into \eqref{Expansion Cov of Y Yv 3rd} and by another asymptotical expansion, we have
			\begin{align}
				\frac{1+O(\alpha^{-4}\epsilon)}{4\pi^2\sqrt{\det(D_2)}}&\iint e^{-\frac{1}{2}\|(2D_2)^{-1/2}U_2^*\underline{\boldsymbol Q}\|^2}\left|G\Big(\frac{p}{\sqrt{2}},\frac{q}{\sqrt{2}}\Big)\right|^2
				\textup dp\textup dq \,\nonumber,
				\end{align}
				which by another change of variable becomes
				\begin{align}
				&\frac{1+O(\alpha^{-4}\epsilon)}{\pi^2\sqrt{\det(D_2)}}\iint e^{-\frac{1}{2}\underline{\boldsymbol Q}^*\underline{\Sigma}^{-1}_2\underline{\boldsymbol Q}}\left|G(p,q)\right|^2
				\textup dp\textup dq \,\nonumber\\
				=&\,\frac{1}{\pi^2\sqrt{\det(D_2)}}\iint e^{-\frac{1}{2}\underline{\boldsymbol Q}^*\underline{\Sigma}^{-1}_2\underline{\boldsymbol Q}}\left|G(p,q)\right|^2
				\textup dp\textup dq+O(\alpha^{-5}\epsilon)
			\end{align}
			since $\|D_2^{-1/2}U_2^*\underline{\boldsymbol Q}\|^2=\underline{\boldsymbol Q}^*\underline{\Sigma}_\eta^{(h)-1}\underline{\boldsymbol Q}$ and $\text{Var}(Y_{\alpha,\xi,\eta})\asymp \alpha^{-1}$ when $\alpha$ is small by Proposition \ref{Proposition joint density Y and omega}. Note that the implied constant depends on $\eta$ and is of order $\eta^{4}$ when $\eta\to 0$ by Proposition \ref{Proposition joint density Y and omega}.
			By comparing this formula with that of $\text{Var}(Y_{\alpha,\xi,\eta})$ shown in Lemma \ref{Lemma:ChangeOfVariableVariance}, we have
			the first claim that 
			\[
				|\text{Cov}(Y_{\alpha,\xi,\eta},\mathsf Y_{\alpha,\xi,\eta})-\text{Var}(Y_{\alpha,\xi,\eta})|=O(\alpha^{-5}\epsilon),\,
			\] 
			where the implied constant depends on $\eta$ and is of order $\eta^{-1/2}$ when $\eta\to 0$.
			By exactly the same argument, we have the second claim that 
			\[
				|\text{Cov}(Y_{\alpha,\xi,\eta},\mathsf Y_{\alpha,\xi,\eta})-\text{Var}(\mathsf Y_{\alpha,\xi,\eta})|=O(\alpha^{-5}\epsilon)\,,
			\]  
			where the implied constant depends on $\eta$ and is of order $\eta^{-1/2}$ when $\eta\to 0$
			\end{proof}

			\begin{lemma}\label{Lemma: Y covariance between two different eta eta'}
			Suppose Assumptions~\ref{assump:noise part}, \ref{assump:nonnull signal} and \ref{assump:gaussian window} hold. Take an even, bounded and smooth function $\varphi_1$ and set a new window $h_1$ so that $\hat{h}_1=\hat{h}\varphi_1$. 
			When $|\eta-\eta'|\geq 1$, we have
					\begin{align}
					&|\text{Cov}({Y}^{(h_1)}_{\alpha,\xi,\eta}, {{Y}^{(h_1)}_{\alpha,\xi,\eta'}})|
					=
					O((\eta-\eta')^2e^{-\pi^2(\eta-\eta')^2})\nonumber\\
					&|\text{Cov}({Y}^{(h_1)}_{\alpha,\xi,\eta}, \overline{{Y}^{(h_1)}_{\alpha,\xi,\eta'}})|
					=
					O((\eta-\eta')^2e^{-\pi^2(\eta-\eta')^2})\,,\nonumber
				\end{align}
			where the implied constants depend on $\varphi$ and $\varrho$.
			Second, there exists $\eta_0>0$ that depends on $\varphi_1$ and $\rho$ so that when
			$|\eta-\eta'|\leq \eta_0$, we have
			\begin{align}
					|\text{Cov}({Y}^{(h_1)}_{\alpha,\xi,\eta}, Y^{(h_1)}_{\alpha,\xi,\eta'})-\textup{Var}({Y}^{(h_1)}_{\alpha,\xi,\eta})| &\leq \frac{1}{10}\textup{Var}({Y}^{(h_1)}_{\alpha,\xi,\eta})\nonumber\\
				|\text{Cov}({Y}^{(h_1)}_{\alpha,\xi,\eta}, \overline{Y^{(h_1)}_{\alpha,\xi,\eta'}})-\text{Cov}({Y}^{(h_1)}_{\alpha,\xi,\eta}, \overline{Y^{(h_1)}_{\alpha,\xi,\eta}})| &\leq \frac{1}{10}|\text{Cov}({Y}^{(h_1)}_{\alpha,\xi,\eta}, \overline{Y^{(h_1)}_{\alpha,\xi,\eta}})| \,.\nonumber
			\end{align}

			Moreover, take another even, smooth and bounded function $\varphi_2$ so that conditions in Lemma \ref{Lemma ABC perturbation after window perturbation} is satisfied. Then we have the window perturbation bounds:
					\begin{equation}
					\begin{aligned}
					&|\text{Cov}({Y}^{(h_1)}_{\alpha,\xi,\eta}, Y^{(h_1)}_{\alpha,\xi,\eta'})
					-
					\text{Cov}({Y}^{(h_1)}_{\alpha,\xi,\eta}, {Y}^{(h_2)}_{\alpha,\xi,\eta'})|
					=
					O(\alpha^{-5}\epsilon^3)\\
					&|\text{Cov}({Y}^{(h_1)}_{\alpha,\xi,\eta}, \overline{Y^{(h_1)}_{\alpha,\xi,\eta'}})
					-
					\text{Cov}({Y}^{(h_1)}_{\alpha,\xi,\eta}, \overline{{Y}^{(h_2)}_{\alpha,\xi,\eta'}})|
					=
					O(\alpha^{-5}\epsilon^3)\,,
				\end{aligned}
			\end{equation}
			where the implied constants depend on $\varphi_1$, $\varphi_2$ and $O(e^{-(\eta-\eta')^2})$ if $\eta$ and $\eta'$ are both away from $0$, $O(\eta^{-1/2}e^{-(\eta-\eta')^2})$ when $\eta$ is sufficiently small, and $O({\eta'}^{-1/2}e^{-(\eta-\eta')^2})$ when $\eta'$ is sufficiently small.
			\end{lemma}
			\begin{proof}
			The proofs follow essentially the same perturbation argument as before, so we skip details but simply provide key steps. 
			
			For the first part, follow the same notations used in Lemma \ref{Appendex:C and P of W for Cov} with the symmetric, bounded and smooth function $\varphi$. 
			To simplify the notation, we omit the superscript depending on $h_1$. 
			Take $\eta,\eta'>0$. When $|\eta-\eta'|\geq 1$, denote
				\begin{align}
					\underline{\Sigma}_{\eta,\eta'}
					=&\,
					\begin{bmatrix}
						{\Gamma}_{\eta} 					& {\Gamma}_{\eta,\eta'} & C_\eta 											& C_{\eta,\eta'}\\
						{\Gamma}_{\eta,\eta'}^* 	& {\Gamma}_{\eta'} 			& C_{\eta,\eta'}^\top 				& C_{\eta'}\\
						\conj{C}_\eta 						& \conj{C}_{\eta,\eta'}	& \conj{\Gamma}_{\eta}				& \conj{\Gamma}_{\eta,\eta'} \\
						C_{\eta,\eta'}^* 					& \conj{C}_{\eta'} 			& \Gamma_{\eta,\eta'}^\top		& \conj{\Gamma}_{\eta'}
					\end{bmatrix}\nonumber,\quad
					\underline{\Sigma}_{\eta,\eta'}^\circ
					=&\,
					\begin{bmatrix}
						{\Gamma}_{\eta} & 0 & C_\eta & 0\\
						0 & {\Gamma}_{\eta'} & 0 & C_{\eta'}\\
						\overline{C}_\eta & 0 & \conj{\Gamma}_{\eta} &  0\\
						0 & \overline{C}_{\eta'} & 0 & \conj{\Gamma}_{\eta'}
					\end{bmatrix}.
				\end{align}
				By a direct expansion, every entry in the matrix
				\begin{align*}
					E
					\vcentcolon=
					\underline{\Sigma}_{\eta,\eta'} -\underline{\Sigma}_{\eta,\eta'}^\circ
					=
					\begin{bmatrix}
						0 											& {\Gamma}_{\eta\eta'} 	& 0 												& C_{\eta\eta'}								\\
						{\Gamma}^*_{\eta\eta'}	& 0 										& C_{\eta\eta'} 						& 0														\\
						0												& \conj{C}_{\eta\eta'} 	& 0 												& \conj{\Gamma}_{\eta\eta'} 	\\
						\conj{C}_{\eta\eta'} 		& 0 										& {\Gamma}^\top_{\eta\eta'} & 0														\\
					\end{bmatrix}
				\end{align*}
				is of order $(\eta-\eta')^2e^{-\pi^2(\eta-\eta')^2}$. Thus $\underline{\Sigma}_{\eta,\eta'}$ is well-approximated by $\underline{\Sigma}_{\eta,\eta'}^\circ$ by the eigenstructure perturbation bound \cite{VanDerAn_TerMorsche:2007} with the error of order $(\eta-\eta')^2e^{-\pi^2(\eta-\eta')^2}$. Note that if we replace the augmented matrix $\underline{\Sigma}_{\eta,\eta'}$ in $\text{Cov}({Y}_{\alpha,\xi,\eta}, Y_{\alpha,\xi,\eta'})$ by $\underline{\Sigma}_{\eta,\eta'}^\circ$ and call the resulting quantity $\text{Cov}_0({Y}_{\alpha,\xi,\eta}, Y_{\alpha,\xi,\eta'})$, by a direct expansion we have $\text{Cov}_0({Y}_{\alpha,\xi,\eta}, Y_{\alpha,\xi,\eta'})=0$. Similar result holds for $\text{Cov}_0({Y}_{\alpha,\xi,\eta}, Y_{\alpha,\xi,\eta'})$.

				Second, when $|\eta-\eta'|$ is sufficiently small, denote
				\begin{align}
					\underline{\Sigma}^\triangle_{\eta,\eta'}
					=&\,
					\begin{bmatrix}
						{\Gamma}_{\eta} 					& {\Gamma}_{\eta} & C_\eta 											& C_{\eta}\\
						{\Gamma}_{\eta}^* 	& {\Gamma}_{\eta} 			& C_{\eta}^\top 				& C_{}\\
						\conj{C}_\eta 						& \conj{C}_{\eta}	& \conj{\Gamma}_{\eta}				& \conj{\Gamma}_{\eta} \\
						C_{\eta}^* 					& \conj{C}_{\eta} 			& \Gamma_{\eta}^\top		& \conj{\Gamma}_{\eta}
					\end{bmatrix}\nonumber\mbox{ and }E'
					=
					\begin{bmatrix}
						0					& {\Gamma}_{\eta,\eta'}-{\Gamma}_\eta & 0 											& C_{\eta,\eta'}-C_\eta\\
						{\Gamma}_{\eta,\eta'}^*- {\Gamma}_{\eta}^*	& {\Gamma}_{\eta'} -{\Gamma}_{\eta} 			& C_{\eta,\eta'}^\top-C_{\eta}^\top 				& C_{\eta'}-C_{\eta}^\top\\
						0	& \conj{C}_{\eta,\eta'}-\conj{C}_{\eta}	& 0				& \conj{\Gamma}_{\eta,\eta'}-\conj{\Gamma}_{\eta} \\
						C_{\eta,\eta'}^* 	-C_\eta^*				& \conj{C}_{\eta'} 	-C_\eta^*	& \Gamma_{\eta,\eta'}^\top-\Gamma_{\eta}^\top		& \conj{\Gamma}_{\eta'}-\conj{\Gamma}_{\eta}
					\end{bmatrix}
				\end{align}
				By a direct expansion, when $|\eta-\eta'|$ is small, all non-zero entries in $E'$ are bounded, and hence by a similar perturbation argument, $|\text{Cov}({Y}^{(h_1)}_{\alpha,\xi,\eta}, Y^{(h_1)}_{\alpha,\xi,\eta'})-\textup{Var}({Y}^{(h_1)}_{\alpha,\xi,\eta})|\leq \frac{1}{10}\textup{Var}({Y}^{(h_1)}_{\alpha,\xi,\eta})$ holds. The other statement holds by the same claim.
				
				For the window perturbation bound, we follow the same proof as that of Lemma \ref{Main Lemma 2 for Kernel SST}. Let $\mathsf P$ to be a permutation matrix defined in \eqref{definition of P permutation first time} so that $\mathsf P \underline{\Sigma}_{\eta,\eta'}^\circ \mathsf P^\top$ is a block-diagonal matrix with each nonzero block depending on only one of $\eta,\eta'$ and $\mathsf P E \mathsf P^\top$ is nonzero in the off-diagonal blocks like that shown in \eqref{Definition E4 error matrix for window perturbation}. The proof proceeds in the same way, and
				we omit details.
			\end{proof}

	\begin{proof}[Proof of Theorem~\ref{thm: Q covariance}]
	The proof follows the same line as the above lemmas, so we omit details.
	\end{proof}
	
\section{Proof of SST-related results}\label{Proof SST related Section4}

We consider an $M$-dependent approximation as follows.

		\begin{defn}[$M$-dependent truncation] \label{def: Main Lemma for Kernel SST}
			Fix $M>0$ and let $\psi \in C^\infty_c$ be an even function, decreasing for $|t|\in[M,2M]$ and satisfying $\psi(t)=1$ when $|t|\leq M$, and $\psi(t)=0$ when $|t|>2M$.
			The $M$-dependent truncation of a non-zero Schwartz function $\hbar$, denoted as $\mathsf{h}$, is defined by $\hat{\mathsf{h}}:=\hat{\hbar}\psi$.
		\end{defn}
		In this section, since $t$, $\xi$ and $f+\Phi$ are fixed, to ease the notation, when there is no danger of confusion, we suppress $t$, $\xi$, and $f+\Phi$ and denote 
		\begin{align*}
		&{Y}_{\alpha,\xi,\eta}:=Y_{f+\Phi}^{(\hbar,\alpha,\xi)}(t,\eta),\,  \Omega_\eta:=\Omega_{f+\Phi}^{(\hbar)}(t,\eta),\,  V_\eta:=V_{f+\Phi}^{(\hbar)}(t,\eta)\\
		&\mathsf{Y}_{\alpha,\xi,\eta}:=Y_{f+\Phi}^{(\mathsf h,\alpha,\xi)}(t,\eta),\, \mathsf \Omega_\eta:=\Omega_{f+\Phi}^{(\mathsf h)}(t,\eta),\, \mathsf V_\eta:=V_{f+\Phi}^{(\mathsf h)}(t,\eta)\,.
				\end{align*}
		The first Lemma shows that $\mathsf{Y}_{\alpha,\xi,\eta}$ is an $M$-dependent random process indexed by $\eta$. 
				\begin{lemma}\label{Main Lemma for Kernel SST}
			{Suppose Assumptions~\ref{assump:noise part}, \ref{assump:nonnull signal} and \ref{assump:gaussian window} hold.} When $|\eta-\eta'|>4M$, the random variables 
			$\mathsf V_{\eta}$ and $\mathsf V_{\eta'}$ (respectively: $\partial_tV_{f+\Phi}^{(\mathsf h)}(t,\eta)$ and $\partial_tV_{f+\Phi}^{(\mathsf h)}(t,\eta')$, {$\mathsf \Omega_{\eta}$ and $\mathsf \Omega_{\eta'}$}, and $\mathsf{Y}_{\alpha,\xi,\eta}$ and $\mathsf{Y}_{\alpha,\xi,\eta'}$) are independent. 
		\end{lemma}
		
			\begin{proof}
				By a direct calculation, the covariance of $\mathsf V_{\eta}$ and $\mathsf V_{\eta'}$ is diagonal and the associated pseudocovariance is zero when $|\eta-\eta'|>4M$.
				By Gaussianity, we conclude that $\mathsf V_{\eta}$ and $\mathsf V_{\eta'}$ are independent. The same argument holds for $\partial_tV_{f+\Phi}^{(\mathsf h)}(t,\eta)$ and and $\partial_tV_{f+\Phi}^{(\mathsf h)}(t,\eta')$. Since {$\mathsf \Omega_{\eta}$ and $\mathsf \Omega_{\eta'}$} are transforms of independent random vectors, they are independent. A similar argument holds for $\mathsf{Y}_{\alpha,\xi,\eta}$ and $\mathsf{Y}_{\alpha,\xi,\eta'}$.
			\end{proof}

			\begin{lemma}\label{Lemma Integration of Falpha for the CLT}
			Suppose Assumptions~\ref{assump:noise part}, \ref{assump:nonnull signal} and \ref{assump:gaussian window} hold and assume $\varrho<5$. 
			Take $\varphi$ to be a symmetric, bounded and smooth function. Take the window $\hbar$ satisfying $\hat{\hbar}=\hat{h}\varphi$. 
			For $\xi>0$, denote a matrix-valued function over $(\eta,\eta')\in \mathbb{R}^+\times\mathbb{R}^+$ by
			\begin{align}
			&F_{\alpha,\xi}(\eta,\eta')\nonumber\\
			:=\frac{1}{2}& \begin{bmatrix}\Re(\text{Cov}({Y}_{\alpha, \xi,\eta},{Y}_{\alpha, \xi,\eta'})+\text{Cov}({Y}_{\alpha,\xi, \eta},\overline{{Y}_{\alpha,\xi, \eta'}})) & \Im\text{Cov}({Y}_{\alpha, \xi,\eta},\overline{{Y}_{\alpha,\xi, \eta'}}) \\
			\Im\text{Cov}({Y}_{\alpha,\xi, \eta},\overline{{Y}_{\alpha, \xi,\eta'}})& \Re(\text{Cov}({Y}_{\alpha,\xi, \eta},{Y}_{\alpha, \xi,\eta'})-\text{Cov}({Y}_{\alpha, \xi,\eta},\overline{{Y}_{\alpha,\xi, \eta'}}))\end{bmatrix}\,.\nonumber
			\end{align}
			Then $F_{\alpha,\xi}$ is continuous and integrable on $\mathbb{R}^+\times \mathbb{R}^+$. 
			Moreover, when $A=0$ (null case) and $\alpha>0$ is sufficiently small, we have
			\begin{align*}
			\int_{\alpha}^\infty\int_{\alpha}^\infty F_{\alpha,\xi}(\eta,\eta')d\eta d\eta' \asymp
			\left\{
			\begin{array}{ll}
				 \xi^{\rho} & \mbox{when }\xi>1\\
				 1 & \mbox{when }\xi\in({\alpha}^{1/4},1)\,,
				 \end{array}\right.
			\end{align*} 
			where the implied constant depends on $\varphi$ and $\varrho$. 
			When $A>0$ (non-null case) and $\alpha>0$ is sufficiently small, 
			\begin{align*}
			\int_{\alpha}^\infty\int_{\alpha}^\infty  F_{\alpha,\xi}(\eta,\eta')d\eta d\eta' \asymp
			\left\{
			\begin{array}{ll}
				\xi^{\rho} & \mbox{when }\xi>1\\
				 1& \mbox{when }\xi\in({\alpha}^{1/4},1)\,,
			 \end{array}\right.
			\end{align*} 
			where the implied constant depends on $A$, $\xi_0$, $\varphi$ and $\varrho$.
			\end{lemma}
\begin{proof}

			By Theorem \ref{Theorem covariance of Yeta and Yetap}, they are both continuous functions over $(\eta,\eta')$. 
			Next we check its integrability. By a simple control $|\text{Cov}(\mathsf{Y}_{\alpha,\xi, \eta},\mathsf{Y}_{\alpha, \xi,\eta'})|\,\vee\,|\text{Cov}(\mathsf{Y}_{\alpha, \xi,\eta},\overline{\mathsf{Y}_{\alpha,\xi, \eta'}})| \leq \sqrt{\text{Var}(\mathsf{Y}_{\alpha,\xi, \eta})\text{Var}(\mathsf{Y}_{\alpha, \xi,\eta'})}$, and the integrability of $\sqrt{\text{Var}(\mathsf{Y}_{\alpha,\xi, \eta})}$ as a function of $\eta$ by Corollary \ref{Proposition joint density Y and omega gen COR1} for any $\alpha,\xi>0$, we get the integrability of $F_{\alpha,\xi}$ over $\mathbb{R}^+\times \mathbb{R}^+$ for any $\alpha,\xi>0$. Note that this integrability is too rough and not sufficient for our application.

			Consider 
			\[
			\mathcal B =\{(\eta,\eta')\in [\alpha,\infty)\times [\alpha,\infty)|\, |\eta-\eta'|\leq 1/2\}\,.
			\]
			We first control $\int_{\mathcal B} F_{\alpha,\xi}(\eta,\eta')d\eta d\eta' $. 
			Consider the case when $A=0$. In this case, for a fixed small $\alpha>0$, when $\eta>\alpha$ and $\xi>{\alpha}^{1/4}$, by Corollary \ref{Proposition joint density Y and omega gen COR1}, we have
			\[
			\text{Var}(\mathsf{Y}_{\alpha,\xi, \eta})\asymp 1\,,
			\] 
			where the implied constant depends on $\varphi$, $\varrho$ and $\frac{\eta^{\varrho}}{(1+4\pi^2|\eta-\xi|^2)^{3}}$. On the other hand, by Lemma \ref{Lemma: Y covariance between two different eta eta'}, $\text{Cov}({Y}_{\alpha, \xi,\eta},{Y}_{\alpha, \xi,\eta'})$ (and $\text{Cov}({Y}_{\alpha,\xi, \eta},\overline{{Y}_{\alpha,\xi, \eta'}})$) behaves like $\text{Var}({Y}_{\alpha, \xi,\eta})$ (and $\text{Cov}({Y}_{\alpha,\xi, \eta},\overline{{Y}_{\alpha,\xi, \eta}})$) when $|\eta'-\eta|\leq \eta_0$.
			Thus, for $i=1,2$,
			\[
			\int_{\mathcal B} e_i^\top F_{\alpha,\xi}(\eta,\eta')e_i d\eta d\eta' \asymp \eta_0\int_\alpha^\infty \text{Var}({Y}_{\alpha, \xi,\eta})d\eta \,,
			\]
			where the implied constant depends on $\varphi$ and $\varrho$. 
			By a direct calculation, we have
			\[
			\int_\alpha^\infty \text{Var}({Y}_{\alpha, \xi,\eta})d\eta\asymp 
						\left\{
			\begin{array}{ll}
				 \xi^{\varrho/2} & \mbox{when }\xi>1\\
				 1 & \mbox{when }\xi\in({\alpha}^{1/4},1)
				 \end{array}\right.,
			\]
			where we use the fact that $\varrho<5$, and the implied constant depends on $\varphi$ and $\varrho$. A similar control holds for $\int_{\mathcal B} e_i^\top F_{\alpha,\xi}(\eta,\eta')e_j d\eta d\eta'$ when $j\neq i$. The control of $\int_{\mathcal B^c} F_{\alpha,\xi}(\eta,\eta')d\eta d\eta'$ depends on Lemma \ref{Lemma: Y covariance between two different eta eta'}, where we see that $F_{\alpha,\xi}(\eta,\eta')$ is controlled by $(\eta-\eta')^2e^{-4\pi^2(\eta-\eta')^2}$ when $|\eta-\eta'|>\eta_0$. Therefore, each entry of $\int_{\mathcal B^c} F_{\alpha,\xi}(\eta,\eta')d\eta d\eta'$ is dominated by the associated entry of $\int_{\mathcal B} F_{\alpha,\xi}(\eta,\eta')d\eta d\eta'$. Thus we get the claim for the null case. When $A>0$, the argument is similar and we omit details. 
						
			\end{proof}

		Note that results in Section \ref{Section: proof of M dependent related arguments} hold with $\hbar$ and $\mathsf h$ with the associated error bounds. Below, we provide two more general theorems that will be combined to prove Theorem \ref{thm:SSQ-CLT} and used for the local bootstrapping proof as well.

		\begin{thm}[CLT for $M$-dependent kernel] \label{thm:SSQ-CLT general}
			{Suppose Assumptions~\ref{assump:noise part}, \ref{assump:nonnull signal} and \ref{assump:gaussian window} hold and assume $\varrho<5$.} Take $\varphi$ to be a symmetric, bounded and smooth function and a window $\mathsf{h}$ satisfying $\hat{\mathsf{h}}=\hat{h}\varphi\psi$, where $\psi$ is from Definition \ref{def: Main Lemma for Kernel SST} with $M=\sqrt{2\log(n)}$. Denote $\mathsf S_{\xi,n}\vcentcolon= \Delta \eta\sum_{l=1}^n  \mathsf Y_{\alpha,\xi,l}$.
			Fix a small $\delta>0$. 
			Take $\beta=\frac{\delta}{4(2+\delta)}$ and
			assume $\Delta \eta=n^{-1/2-\beta}$. 
			Assume $\alpha=\alpha(n)$ so that $\alpha\to 0$ and $n\alpha\to 1$ when $n\to \infty$. For any fixed $\xi>0$, 
			we have 
			\[
			\mathsf S_{\xi,n}-\mathbb{E}\mathsf S_{\xi,n}\to \CC N_1(0,\nu,\wp) 
			\]
			weakly when $n\to \infty$, where $\nu>0$ and $\wp\in \mathbb{C}$ are of order $(1+\xi)^\rho$, where the implied constant depends on $\varrho$ and $\varphi$ when $A=0$ and on $A$, $\xi_0$, $\varrho$ and $\varphi$ when $A>0$.

		\end{thm}

\begin{proof} 
			We follow the notation used in the proof of Lemma \ref{Lemma Integration of Falpha for the CLT}. 			 We combine the Cramer-Wold theorem and the CLT for $M$-dependent random variables \cite{berk1973} to study the asymptotic behavior of 
			 \[
			 \Delta \eta \sum_{l=1}^n[\mathsf{Y}_{\alpha,\xi, l}-\mathbb E\mathsf{Y}_{\alpha,\xi, l}] \,. 
			 \] 
			 Rewrite the complex random variable as $\mathsf{Y}_{\alpha,\xi, l}=\mathsf{R}_{\alpha, \xi,l}+i\mathsf{I}_{\alpha, \xi,l}$, and consider it as a two dimensional real random vector $\begin{bmatrix}\mathsf{R}_{\alpha,\xi, l}&\mathsf{I}_{\alpha,\xi, l} \end{bmatrix}^\top$. To apply the Cramer-Wold theorem, below we check the case $\mathsf{C}^{(\theta)}_{\alpha, \xi,l} := \cos(\theta) \mathsf{R}_{\alpha,\xi,l}+\sin(\theta)\mathsf{I}_{\alpha,\xi,l}$ for a fixed $\theta$, and the other $\theta$ follows the same argument.
			 Clearly, we have 
			 \begin{align}
			&\text{Var}(\cos(\theta) \mathsf{R}_{\alpha,l}+\sin(\theta)\mathsf{I}_{\alpha,l})=\frac{\cos(\theta)^2}{2}\big[\text{Var}\mathsf{Y}_{\alpha, l}+\Re \mathbb{E}(\mathsf{Y}_{\alpha, l}-\mathbb{E}\mathsf{Y}_{\alpha, l})^2\big]\nonumber\\
			&+    \frac{\sin(\theta)^2}{2}\big[\text{Var}\mathsf{Y}_{\alpha, l}-\Re \mathbb{E}(\mathsf{Y}_{\alpha, l}-\mathbb{E}\mathsf{Y}_{\alpha, l})^2\big]       +\cos(\theta)\sin(\theta)\Im \mathbb{E}(\mathsf{Y}_{\alpha, l}-\mathbb{E}\mathsf{Y}_{\alpha, l})^2 \nonumber
			 \end{align}
			 for any $\theta\in [0,2\pi)$, since $\text{Var}\mathsf{R}_{\alpha, l} = \frac{1}{2}\big[\text{Var}\mathsf{Y}_{\alpha, l}+\Re \mathbb{E}(\mathsf{Y}_{\alpha, l}-\mathbb{E}\mathsf{Y}_{\alpha, l})^2\big]$, $\text{Var} \mathsf{I}_{\alpha, l} = \frac{1}{2}\big[\text{Var}\mathsf{Y}_{\alpha, l}-\Re \mathbb{E}(\mathsf{Y}_{\alpha, l}-\mathbb{E}\mathsf{Y}_{\alpha, l})^2\big]$, and $\text{Cov}(\mathsf{R}_{\alpha, l},\mathsf{I}_{\alpha, l}) = \frac{-1}{2}\Im \mathbb{E}(\mathsf{Y}_{\alpha, l}-\mathbb{E}\mathsf{Y}_{\alpha, l})^2$. 

			 For either the null and non-null cases, we need to check the four conditions in the main theorem of \cite{berk1973}. 
			 It is clear that
			\begin{align}
			\lim_{n\to \infty}M^{2+2/\delta}/n=\lim_{n\to \infty} [\sqrt{4\beta\log(n)}]^{2+2/\delta}/n=0\,,
			\end{align} 
			so \cite[Theorem (iv)]{berk1973} holds. The other conditions depend on the case.

			Since we assume $\xi>0$ is fixed, we have $\xi>{\alpha}^{1/4}$ when $n$ is sufficiently large.
			Thus, by Theorems \ref{Proposition joint density Y and omega gen} and Theorem \ref{Proposition joint density Y and omega gen222}, the $k$-th absolute moment of $\mathsf{Y}_{\alpha, \xi,l}$ is of order $\alpha^{-k/2+1}$ in both null and non-null cases, so are the $k$-th absolute moments of $\mathsf{R}_{\alpha,\xi, l}$ and $\mathsf{I}_{\alpha,\xi, l}$ and hence the $k$-th absolute moments of $\mathsf{C}^{(\theta)}_{\alpha,\xi, l}$. 
			By the assumption that $n\alpha\to 1$ when $n\to \infty$, the $2+\delta$ moment of $\alpha^{\frac{\delta}{2(2+\delta)}}\mathsf{Y}_{\alpha, \xi,l}$ is bounded, and hence \cite[Theorem (i)]{berk1973} holds.

			Next, for a given $\theta$, we check the positive finiteness of $\frac{1}{n}\text{Var}\big[\sum_{l=1}^n\alpha^{\frac{\delta}{2(2+\delta)}}\mathsf{C}^{(\theta)}_{\alpha,\xi,l}\big]$. By a direct expansion, we have
			\begin{align}
			\frac{1}{n}\text{Var}\Big[\sum_{l=1}^n\mathsf{C}^{(\theta)}_{\alpha,\xi,l}\Big]=\frac{1}{n}\sum_{l,k=1}^n\text{Cov}(\mathsf{C}^{(\theta)}_{\alpha,\xi,l},\mathsf{C}^{(\theta)}_{\alpha,\xi,k})\,,
			\end{align}
			where 
			\begin{align}
			&\,\text{Cov}(\mathsf{C}^{(\theta)}_{\alpha,\xi,l},\mathsf{C}^{(\theta)}_{\alpha,\xi,k})\nonumber\\
			=&\,\frac{\cos(\theta)^2}{2}\Re(\text{Cov}(\mathsf{Y}_{\alpha, \xi,l},\mathsf{Y}_{\alpha, \xi,k})+\text{Cov}(\mathsf{Y}_{\alpha,\xi, l},\overline{\mathsf{Y}_{\alpha, \xi,k}}))\label{Expansion:cov Yl Yk into theta}\\
			&+\frac{\sin(\theta)^2}{2}\Re(\text{Cov}(\mathsf{Y}_{\alpha,\xi, l},\mathsf{Y}_{\alpha,\xi, k})-\text{Cov}(\mathsf{Y}_{\alpha,\xi, l},\overline{\mathsf{Y}_{\alpha, \xi,k}}))\nonumber\\
			&+\cos(\theta)\sin(\theta)\Im\text{Cov}(\mathsf{Y}_{\alpha,\xi, l},\overline{\mathsf{Y}_{\alpha,\xi, k}})=v_\theta^\top \mathsf{M}_{\alpha,\xi,l,k}v_\theta,\nonumber
			\end{align}
			where 
			\begin{align}
			&\mathsf{M}_{\alpha,\xi,l,k}\nonumber\\
			:=&\,\frac{1}{2}\begin{bmatrix}\Re(\text{Cov}(\mathsf{Y}_{\alpha,\xi, l},\mathsf{Y}_{\alpha, \xi,k})+\text{Cov}(\mathsf{Y}_{\alpha, \xi,l},\overline{\mathsf{Y}_{\alpha,\xi, k}})) & \Im\text{Cov}(\mathsf{Y}_{\alpha,\xi, l},\overline{\mathsf{Y}_{\alpha, \xi,k}}) \\
			\Im\text{Cov}(\mathsf{Y}_{\alpha,\xi, l},\overline{\mathsf{Y}_{\alpha,\xi, k}})& \Re(\text{Cov}(\mathsf{Y}_{\alpha,\xi, l},\mathsf{Y}_{\alpha,\xi, k})-\text{Cov}(\mathsf{Y}_{\alpha, \xi,l},\overline{\mathsf{Y}_{\alpha, \xi,k}}))\nonumber
			\end{bmatrix}
			\end{align}
			and $v_\theta=\begin{bmatrix}\cos(\theta)&\sin(\theta)\end{bmatrix}^\top$. 
			Since $\mathsf{M}_{\alpha,\xi,l,k}=\mathsf{F}_{\alpha,\xi}(\eta_l,\eta_k)$ defined in Lemma \ref{Lemma Integration of Falpha for the CLT} and $\mathsf{F}_{\alpha,\xi}$ is continuous and integrable by Lemma \ref{Lemma Integration of Falpha for the CLT}, we can approximate  
			\begin{align}
			\sum_{l,k=1}^n\text{Cov}(\mathsf{C}^{(\theta)}_{\alpha,\xi,l},\mathsf{C}^{(\theta)}_{\alpha,\xi,k})&=v_\theta^\top\Big[\sum_{l,k=1}^n\mathsf{M}_{\alpha,\xi,l,k}\Big]v_\theta\nonumber
			\end{align}
			by the Riemannian sum
			\begin{align}
			\frac{1}{(\Delta\eta)^2}v_\theta^\top \left[\int_0^{\infty}\int_{\max\{0,\eta-H(n)/2\}}^{\eta+H(n)/2} \mathsf{F}_{\alpha,\xi}(\eta,\eta')\textup d\eta'\textup d\eta\right] v_\theta\label{Equation Expansion Cov summation}\,,
			\end{align}
			where $H(n)=4\sqrt{2\log(n)}<n\Delta \eta=n^{1/2-\beta}$ when $n$ is sufficiently large. 
			Note that $H(n)=4\sqrt{2\log(n)}$ since $\mathsf{M}_{\alpha,\xi,l,k}=0$ when $|l-k|\Delta \eta>4M=4\sqrt{2\log(n)}$. 
			With the above facts, by Lemma \ref{Lemma Integration of Falpha for the CLT}, all entries of the matrix 
			\[
			\mathsf M_n:=\int_0^{\infty}\int_{\max\{0,\eta-H(n)/2\}}^{\eta+H(n)/2} \mathsf{F}_{\alpha,\xi}(\eta,\eta')\textup d\eta'\textup d\eta
			\]
			are finite, independent of $\theta$, and of order $(1+ \xi)^\rho$ when $\alpha$ is sufficiently small. Here, note that we again use the fact that $\xi>\alpha^{1/4}$ when $n$ is sufficiently large. Since $\frac{1}{n(\Delta\eta)^{2}}=n^{2\beta}$ by assumption, we have 
			\begin{align*}
			&\frac{1}{n}\text{Var}\Big[\sum_{l=1}^n\alpha^{\frac{\delta}{2(2+\delta)}}\mathsf{C}^{(\theta)}_{\alpha,\xi,l}\Big]=\alpha^{\frac{\delta}{2(2+\delta)}}\frac{1}{n}\sum_{l, k=1}^n\text{Cov}(\mathsf{C}^{(\theta)}_{\alpha,\xi,l},\mathsf{C}^{(\theta)}_{\alpha,\xi,k})\\
			=&\,\alpha^{\frac{\delta}{2(2+\delta)}} v_\theta^\top\Big[ \frac{1}{n}\sum_{l, k=1}^n\mathsf{M}_{\alpha,\xi,l,k}\Big]v_\theta\asymp \alpha^{\frac{\delta}{2(2+\delta)}} n^{2\beta}
			\end{align*}
			when $n$ is sufficiently large.
			By the assumption that $\beta=\frac{\delta}{4(2+\delta)}$ and $\alpha=\alpha(n)$ so that $n\alpha\to 1$ as $n\to \infty$, 
			we conclude that 
			\begin{align}\label{Definition nu kernel regression}
			\lim_{n\to \infty}\frac{1}{n}\text{Var}\sum_{l=1}^nn^{-\frac{\delta}{2(2+\delta)}}\mathsf{C}^{(\theta)}_{\alpha,l}=v_\theta^\top \mathsf Mv_\theta>0
			\end{align}
			for some positive definite matrix $\mathsf M$ for any $\theta\in[0,2\pi)$. We thus have \cite[Theorem (iii)]{berk1973} for each $\theta$. By a similar calculation, we know that there exists $K^{(\theta)}>0$ so that
			\begin{align}
			\text{Var}\sum_{l=i+1}^jn^{-\frac{\delta}{2(2+\delta)}}\mathsf{C}^{(\theta)}_{\alpha,l}\leq (j-i)K^{(\theta)}
			\end{align}
			for all $1\leq i\leq j\leq n$ and hence \cite[Theorem (ii)]{berk1973} also holds. Finally, recall the following relationship -- if the covariance matrix of the real random vector $\begin{bmatrix}X& Y\end{bmatrix}^\top$ is $\begin{bmatrix}c & d \\ d& e\end{bmatrix}$ for $c,d,e\in \mathbb{R}$, by a direct calculation, the augmented covariance matrix of the complex random variable $X+iY$ is $\begin{bmatrix}c+e & c-e+2id\\ c-e-2id& c+e\end{bmatrix}$. With this fact, the main theorem in \cite{berk1973} and the Cramer-Wold theorem, we deduce the desired result that 
			\begin{equation*}
			\mathsf S_{\xi,n}-\mathbb E\mathsf S_{\xi,n}=\frac{1}{\sqrt{n}}\sum_{l=1}^nn^{-\frac{\delta}{2(2+\delta)}}[\mathsf{Y}_{\alpha,l}-\mathbb E\mathsf{Y}_{\alpha, l}]\to \mathbb{C}{N}_1(0,\nu,\wp)\,,
			\end{equation*} 
			in distribution when $n\to \infty$, where $\nu>0$ and $\wp\in \mathbb{C}$ are both from $\mathsf{M}$, of order $(1+\xi)^\rho$, and the implied constant depends on $\varrho$ and $\psi$ when $A=0$, and on $A$, $\xi_0$, $\varrho$ and $\psi$ when $A>0$.

\end{proof}

\begin{thm}[$M$-dependent as a perturbation] \label{thm:SSQ-CLT general 2}
			Follow the same notation and assumptions in Theorem \ref{thm:SSQ-CLT general}.
			We then have 
			\[
			\mathsf S_{\xi,n}-\mathbb{E}\mathsf S_{\xi,n}\to  S_{\xi,n}-\mathbb{E} S_{\xi,n}
			\]
			in probability when $n\to \infty$.
		\end{thm}
		
		\begin{proof}
		Note that 
					\begin{align*}
			&\mathsf S_{\xi,n}-\mathbb{E}\mathsf S_{\xi,n}=\frac{1}{\sqrt{n}}\sum_{l=1}^n n^{-\frac{\delta}{2(2+\delta)}}[\mathsf{Y}_{\alpha,l}-\mathbb E\mathsf{Y}_{\alpha, l}]\\
			&S_{\xi,n}-\mathbb{E} S_{\xi,n}=\frac{1}{\sqrt{n}}\sum_{l=1}^n n^{-\frac{\delta}{2(2+\delta)}}[{Y}_{\alpha,l}-\mathbb{E}Y_{\alpha, l}]
			\end{align*}
			To control the difference between $ \mathsf S_{\xi,n}$ and $S_{\xi,n}$, 
			by Chebychev's inequality, for $\varepsilon>0$, we have 
			\begin{align}
			&\text{Pr}\Big\{\Big|\frac{1}{\sqrt{n}}\sum_{l=1}^n(\mathsf{Y}_{\alpha,l}-Y_{\alpha,l})\Big|\geq \varepsilon\Big\}\leq\frac{1}{n\varepsilon^2}\text{Var}\sum_{l=1}^n(Y_{\alpha,l}-\mathsf{Y}_{\alpha,l}).
			\end{align}
			By a direct expansion, we have 
			\begin{align}
			&\text{Var}\sum_{l=1}^n(Y_{\alpha,l}-\mathsf{Y}_{\alpha,l})
			= \sum_{l=1}^n[\text{Var}(Y_{\alpha,l})-\text{Cov}(Y_{\alpha,l}, \mathsf{Y}_{\alpha,l})]\nonumber\\
			&\quad+\sum_{l=1}^n[\text{Var}(\mathsf Y_{\alpha,l})-\text{Cov}(\mathsf Y_{\alpha,l}, {Y}_{\alpha,l})]+\sum_{l\neq k}[\text{Cov}(Y_{\alpha,l}, Y_{\alpha,k})-\text{Cov}(\mathsf Y_{\alpha,l}, {Y}_{\alpha,k})]\nonumber\\
			&\quad+\sum_{l\neq k}[\text{Cov}(\mathsf Y_{\alpha,l}, \mathsf Y_{\alpha,k})-\text{Cov}(Y_{\alpha,l}, \mathsf {Y}_{\alpha,k})]\nonumber\,.
			\end{align}
			By Lemma \ref{Main Lemma 2 for Kernel SST} with the error $\epsilon$ set to $e^{-M^2}$, we have
			\[
			\frac{1}{n}\sum_{l=1}^n[\text{Var}(Y_{\alpha,l})-\text{Cov}(Y_{\alpha,l}, \mathsf{Y}_{\alpha,l})]+\frac{1}{n}\sum_{l=1}^n[\text{Var}(\mathsf Y_{\alpha,l})-\text{Cov}(\mathsf Y_{\alpha,l}, {Y}_{\alpha,l})]=O(
			\alpha^{-5}e^{-3M^2}),
			\]
			and by Lemma \ref{Lemma: Y covariance between two different eta eta'}, we have
			\begin{align}
			\frac{1}{n}&\sum_{l\neq k}[\text{Cov}(Y_{\alpha,l}, Y_{\alpha,k})-\text{Cov}(\mathsf Y_{\alpha,l}, {Y}_{\alpha,k})]\nonumber\\
			&+\,\frac{1}{n}\sum_{l\neq k}[\text{Cov}(\mathsf Y_{\alpha,l}, \mathsf Y_{\alpha,k})-\text{Cov}(Y_{\alpha,l}, \mathsf {Y}_{\alpha,k})]=O(\alpha^{-5}e^{-3M^2}).\nonumber
			\end{align}
			We thus have
			\begin{equation}
			\frac{1}{n}\text{Var}\sum_{l=1}^n(Y_{\alpha,l}-\mathsf{Y}_{\alpha,l})=O(\alpha^{-5}e^{-3M^2})\,
			\end{equation}
			when $n\to \infty$. Hence,
			\begin{align}\label{Bound of Y with different windows}
			&\text{Pr}\Big\{\Big|\frac{1}{\sqrt{n}}\sum_{l=1}^n n^{-\frac{\delta}{2(2+\delta)}}[\mathsf{Y}_{\alpha,l}-Y_{\alpha,l}]\Big|\geq \varepsilon\Big\}=O\left(\frac{\alpha^{-5}e^{-3M^2}}{\varepsilon^2n^{\frac{\delta}{2+\delta}}}\right)\,
			\end{align}
			which goes to $0$ when $n\to \infty$ for any $\varepsilon>0$ since $\alpha^{-5}e^{-3M^2}=n^{-1}$ by assumption. On the other hand, we have $|\mathbb E\mathsf{Y}_{\alpha, l}- \mathbb E{Y}_{\alpha, l}|=O(e^{-M^2})$ by Lemma \ref{Lemma: EY and EvY difference} when $n\to \infty$ for all $l$, so 
			\[
			\frac{1}{\sqrt{n}}\sum_{l=1}^nn^{-\frac{\delta}{2(2+\delta)}} [\mathbb E {Y}_{\alpha, l}-\mathbb E\mathsf{Y}_{\alpha, l}]=O(n^{-\frac{3}{2}-\frac{\delta}{2(2+\delta)}})\to 0
			\] 
			when $n\to \infty$. We thus conclude that 
			\[
			\frac{1}{\sqrt{n}}\sum_{l=1}^n n^{-\frac{\delta}{2(2+\delta)}}[\mathsf{Y}_{\alpha,l}-\mathbb E\mathsf{Y}_{\alpha, l}]\to\frac{1}{\sqrt{n}}\sum_{l=1}^n n^{-\frac{\delta}{2(2+\delta)}}[{Y}_{\alpha,l}-\mathbb{E}Y_{\alpha, l}]
			\]
			 in probability when $n\to \infty$.
		\end{proof}
		
		We can now finish the proof of Theorem \ref{thm:SSQ-CLT}.
		\begin{proof}[Proof of Theorem \ref{thm:SSQ-CLT}]
		Take $\psi=1$. By Theorem \ref{thm:SSQ-CLT general}, we obtain the CLT when the kernel is $M$-dependent. By  
		the discrepancy control in Theorem \ref{thm:SSQ-CLT general 2}, we obtain the theorem.
		\end{proof}

\section{Proofs for bootstrapping}\label{Section proof for section 5}

In this section, to ease the notation, when there is no danger of confusion, we suppress $t$ and $f+\Phi$ and denote 
		\begin{align*}
		&{Y}_{\alpha,\xi,l}:=Y_{f+\Phi}^{(\hbar,\alpha,\xi)}(t,\eta_l),\,\mathsf{Y}_{\alpha,\xi,l}:=Y_{f+\Phi}^{(\mathsf h,\alpha,\xi)}(t,\eta_l)\,.
				\end{align*}

The proof of Corollary \ref{Discretization corollary} and Theorem \ref{Bootstrapping corollary} is almost the same as that for Theorem \ref{thm:SSQ-CLT}, with a perturbation argument about the window.

\begin{proof}[Proof of Corollary \ref{Discretization corollary}]

Recall the discretization of STFT at time $0$ and frequency $\eta$ shown in \eqref{Discretization Vf}:
\begin{equation}\label{Discretization tildeVf}
	{\mathbf V}_{\eta,N}=\Phi({\mathbf h}_{\eta,N})\,.
	\end{equation}
where
\[
\widehat{{\mathbf h}}_{\eta,N}(\xi)=\left[\frac{1}{\sqrt{N}}\sum_{j=\lfloor -N/2\rfloor}^{\lfloor N/2\rfloor}h\left(\frac{j}{\sqrt{N}}\right)e^{-i2\pi\frac{j}{\sqrt{N}}(\xi+\eta)}\right]\hat{\psi}(\xi/\sqrt{N})\,.
\]
Note that $\frac{1}{\sqrt{N}}\sum_{j=\lfloor -N/2\rfloor}^{\lfloor N/2\rfloor}h\left(\frac{j}{\sqrt{N}}\right)e^{-i2\pi\frac{j}{\sqrt{N}}(\xi+\eta)}$ is periodic with the periodicity $\sqrt{N}$.
We can view ${\mathbf h}_{\eta,N}$ as a perturbation of the window $h$. By a direct truncation and the Poisson summation formula, when $N$ is sufficiently large so that $\sqrt{N}/4>n^{1/2+\beta}$, where $n$ is the discretization in $\eta$, for $\eta\leq \sqrt{N}/4$ and $|\xi| \leq \sqrt{N}/2$, we have
\begin{align*}
&\frac{1}{\sqrt{N}}\sum_{j=\lfloor -N/2\rfloor}^{\lfloor N/2\rfloor}h\left(\frac{j}{\sqrt{N}}\right)e^{-i2\pi\frac{j}{\sqrt{N}}(\xi+\eta)}\\
=\,&\frac{1}{\sqrt{N}}\sum_{j\in \mathbb{Z}}h\left(\frac{j}{\sqrt{N}}\right)e^{-i2\pi\frac{j}{\sqrt{N}}(\xi+\eta)}+O\left(\frac{8\sqrt{2}}{\pi \sqrt{N}}e^{-N/16}\right)\\
=\,& \hat{h}(\xi+\eta)+\sum_{j\neq 0}\hat{h}\left(\xi+\eta+j\sqrt{N}\right)e^{-i2\pi\frac{j}{\sqrt{N}}(\xi+\eta)}+O\left(\frac{8\sqrt{2}}{\pi \sqrt{N}}e^{-N/16}\right)\\
=\,& \hat{h}(\xi+\eta)+O\left(\frac{8}{\pi^{5/2}N}e^{-\pi^2N/8}\right)+O\left(\frac{8\sqrt{2}}{\pi \sqrt{N}}e^{-N/16}\right)\,.
\end{align*}
On the other hand, note that $\hat{\psi}(\xi)=0$ when $|\xi|>\sqrt{N}/2$ and $\hat{\psi}(\xi)=1$ when $|\xi|\leq \sqrt{N}/4$. Thus, we have $\widehat{{\mathbf h}}_{\eta,N}(\xi)=\hat{h}_{0,\eta}(\xi)\phi_{1,N}(\xi)$ for some bounded, smooth and symmetric function $\phi_{1,N}$.
For a sufficiently large $N$, by denoting $\varphi_1=1$ and $\varphi_2=\phi_{1,N}$ and a direct bound, we obtain
\[
\sup_{0<\eta\leq \sqrt{N}/4}\max_{l,k=1,2}\left\{|\gamma^{[\varphi_1,\varphi_1]}_i(\eta)-\gamma^{[\varphi_l,\varphi_k]}_i(\eta)|\,,|\nu^{[\varphi_1,\varphi_1]}_i(\eta)-\nu^{[\varphi_l,\varphi_k]}_i(\eta)|\right\}=O(1/N)\,.
\]
Thus, we can apply the perturbation argument to compare the SST of $\mathsf{X}$ in the discretized setup and $\Phi$ in the continuous setup, denoted as ${\mathbf Y}_{\alpha,\xi,l,N}$ and ${Y}_{\alpha,\xi,l}$ respectively.
By 
the same argument for \eqref{Bound of Y with different windows}, for any $\varepsilon>0$ and $l$, when $n$ is sufficiently large, we have 
		\begin{align}
			&\text{Pr}\Big\{\Big|\frac{1}{\sqrt{n}}\sum_{l=1}^n n^{-\frac{\delta}{2(2+\delta)}}({\mathbf Y}_{\alpha,\xi,l,N}-{Y}_{\alpha,\xi,l})\Big|\geq \varepsilon\Big\}\leq\frac{\alpha^{-5}N^{-3}}{\varepsilon^2n^{\frac{\delta}{2+\delta}} }\,.
			\end{align}
			Since $\frac{1}{\sqrt{n}}\sum_{l=1}^nn^{-\frac{\delta}{2(2+\delta)}}\mathbf{Y}_{\alpha,\xi,l,N}=\mathbf{S}_{\alpha,\xi,n,N}$ and $\frac{1}{\sqrt{n}}\sum_{l=1}^nn^{-\frac{\delta}{2(2+\delta)}}Y_{\alpha,\xi,l}=S_{\alpha,\xi,n}$, when $N\to\infty$, we have $\mathbf{S}_{\alpha,\xi,n,N}\to S_{\alpha,\xi,n}$ in probability.
					
		\end{proof}

	\begin{proof}[Proof of Theorem \ref{Bootstrapping corollary}]
	
	We follow the notations used in the proof of Corollary \ref{Discretization corollary}.
Note that $\mathsf{X}$ has a short range dependence.
When $N$ is sufficiently large, according to \cite[Theorem 4]{xiao2012covariance},  the covariance structure of $\mathsf{X}$, denoted as ${\Sigma}_N$, can be estimated by the banding covariance approximation approach consistently. Specifically, if we consider the band to be of size $(N/\log(N))^{1/3}$, and denote the estimated banded covariance matrix as $\tilde{\Sigma}_N$ and $\epsilon_{N}:=\tilde{\Sigma}_{N}-\Sigma_{N}$, when $N$ is sufficiently large, we have the operator norm bound
\[
\|\epsilon_{N}\|=O_P\left((\log(N)/N)^{1/3}\right)\,.
\] 
Denote $\tilde{\mathsf X}:=[\tilde {\mathsf X}_{-\lfloor\frac{N}{2}\rfloor},\ldots,\tilde {\mathsf X}_{\lfloor\frac{N}{2}\rfloor}]$ to be a Gaussian process with mean $0$ and $\tilde{\Sigma}_N$ as the covariance structure.
By \cite[Equations (27) and (28)]{xiao2012covariance}, the spectral functions associated with $\mathsf X$ and $\tilde {\mathsf X}$ are uniformly bounded by $O_P\left((\log(N)/N)^{1/3}\right)$. 
When we run STFT on $\tilde{\mathsf X}$  at time $0$ and frequency $\eta$, we get a parallel formula like \eqref{Discretization Vf}:
\begin{equation}\label{Discretization tildeVf}
	\tilde{\mathbf V}_{\eta,N}=\Phi(\tilde{\mathbf h}_{\eta,N})\,.
	\end{equation}
where the difference between $\tilde{\mathbf h}_{\eta,N}$ and ${\mathbf h}_{\eta,N}$ leads to
\[
\sup_{0<\eta\leq \sqrt{N}/4}\max_{l,k=1,2}\left\{|\gamma^{[\varphi_1,\varphi_1]}_i(\eta)-\gamma^{[\varphi_l,\varphi_k]}_i(\eta)|\,,|\nu^{[\varphi_1,\varphi_1]}_i(\eta)-\nu^{[\varphi_l,\varphi_k]}_i(\eta)|\right\}=O\left(\Big(\frac{\log N}{N}\Big)^{1/6}\right)\,.
\]
Thus, we apply the perturbation argument to compare SST's of $\mathsf{X}$ and $\tilde{\mathsf{X}}$. Denote the integrands of SST's for $\mathsf{X}$ and $\tilde{\mathsf{X}}$ as ${\mathbf Y}_{\alpha,\xi,l,N}$ and $\tilde{\mathbf Y}_{\alpha,\xi,l,N}$ respectively. When $N$ is sufficiently large so that the number of discretization $n$ in $\eta$ satisfies $n^{1/2-\beta}<\sqrt{N}/2$,
by a similar argument for Theorems \ref{thm:SSQ-CLT general} and \ref{thm:SSQ-CLT general 2}, for each $\xi_j\in G$, 
			we have
			\begin{align}\label{bootstrapping Y approximation}
			\Big\{\Big|\frac{1}{\sqrt{n}}\sum_{l=1}^nn^{-\frac{\delta}{2(2+\delta)}}({\mathbf Y}_{\alpha, \xi_j,l,N}-\tilde{\mathbf Y}_{\alpha,\xi_j,l,N})\Big|\geq \varepsilon\Big\}\leq\frac{\epsilon_N^{3}}{\varepsilon^2 n^{\frac{\delta}{2+\delta}} \alpha^{5}}\,,
			\end{align}
			where $\epsilon_N=O((\log(N)/N)^{1/6})$.

To finish the proof, note that for any $\varepsilon>0$, by the definition of ${\mathfrak{m}}_{n,N}$
			and 
			$\mathfrak{m}^*_{n,N}$ and the triangular inequality, we have
\begin{align}
&\text{Pr}\left\{|\mathfrak{m}^*_{n,N}-\mathfrak{m}_{n,N}|>\varepsilon\right\}\nonumber\\
=\,&\text{Pr}\left\{\left|\max_{\xi_j\in G} \Big|\frac{1}{\sqrt{n}}\sum_{l=1}^n n^{-\frac{\delta}{2(2+\delta)}} \tilde{\mathbf{Y}}_{\alpha,\xi_j, l,N}\Big|-\max_{\xi_j\in G}\Big|\frac{1}{\sqrt{n}}\sum_{l=1}^n n^{-\frac{\delta}{2(2+\delta)}} \mathbf Y_{\alpha,\xi_j,l,N}\Big|\right|>\varepsilon\right\}\nonumber\\
\leq\,&\text{Pr}\left\{\max_{\xi_j\in G} \Big|\frac{1}{\sqrt{n}}\sum_{l=1}^n n^{-\frac{\delta}{2(2+\delta)}} \tilde{\mathbf{Y}}_{\alpha,\xi_j, l,N}-\frac{1}{\sqrt{n}}\sum_{l=1}^n  n^{-\frac{\delta}{2(2+\delta)}}\mathbf Y_{\alpha,\xi_j,l,N}\Big|>\varepsilon\right\}\nonumber\\
\leq\,			&\sum_{j=1}^{|G|}\text{Pr}\Big\{\Big|\frac{1}{\sqrt{n}}\sum_{l=1}^nn^{-\frac{\delta}{2(2+\delta)}}({\mathbf Y}_{\alpha, \xi_j,l,N}-\tilde{\mathbf Y}_{\alpha,\xi_j,l,N})\Big|> \varepsilon\Big\}\,.\nonumber
			\end{align}
			Since $|G|=\sqrt{N}/\log(N)$, by \eqref{bootstrapping Y approximation}, we have
			\begin{align}
\text{Pr}\left\{|\mathfrak{m}^*_{n,N}-\mathfrak{m}_{n,N}|>\varepsilon\right\}=O\left(\frac{\log(N)^{-1/2}}{\varepsilon^2n^{\frac{\delta}{2+\delta}} \alpha^{5}}\right)\,,
\end{align}
and hence 
\[
\mathfrak{m}^*_{n,N}- \mathfrak{m}_{n,N}\to 0
\]
			in probability when $N\to \infty$ as claimed. 
					
	\end{proof}

	\section{More numerical simulation}\label{Appendix more numerical evidence}
	
	In this section, we provide numerical evidence supporting the developed theorems. Consider $f(t)=A\exp(2\pi i \times 10t)+\Phi$, where $A> 0$, $\xi_0=10$ is the frequency, and $\Phi$ is the standard Gaussian white random process. Set $n=8,192$ and $N=n$. By taking $\beta=0.05$, we realize $f$ with the sampling rate $n^{-1/2+\beta}/2=115.36$Hz and sample $n$ points from $f$ in the time domain. Set $\Delta \eta=n^{-1/2-\beta}$ for the frequency axis discretization. According to the theorem, when $\alpha$ is of order $n^{-1}$, asymptotically SST $\mathbf S_{\alpha,\xi,n,N}$ converges to a normal distribution. 
	We choose $h$ to be the Gaussian window, and choose $\alpha=0.1\times 5^{j-1}/n$, where $j=1,\ldots,5$ and $\xi\in\{2,6,10,14\}$. In the null case, the QQ plots of $1,000$ realizations of $\Re S_\Phi^{(h,\alpha)}(0,\xi)$ against the standard normal distribution with different $\xi$ and $\alpha$ are shown in Figure \ref{SST_QQplot2}. It is clear that when $\alpha$ is sufficiently large, the distribution of $\Re S_\Phi^{(h,\alpha)}(0,\xi)$ gets closer to Gaussian for different $\xi$. The results of $\Im S_\Phi^{(h,\alpha)}(0,\xi)$ and other combinations of $\Re S_\Phi^{(h,\alpha)}(0,\xi)$ and $\Im S_\Phi^{(h,\alpha)}(0,\xi)$ have the same behavior, but not shown here.
	In the non-null case when $A=1$, the QQ plots of $1,000$ realizations of $S_{f+\Phi}^{(h,\alpha)}(0,\xi)$ against the standard normal distribution with different $\xi$ and $\alpha$ are shown in Figure \ref{SST_QQplot2}. Like the results in the null case, it is clear that when $\alpha$ is sufficiently large, the distribution of $\Re S_{f+\Phi}^{(h,\alpha)}(0,\xi)$ is close to Gaussian for different $\xi$, while when $\xi=\xi_0$, the mean is not zero, {as is predicted by Theorem \ref{Proposition joint density Y and omega}}. 
	\begin{figure}[hbt!]
		\centering
		\includegraphics[width=.99\textwidth]{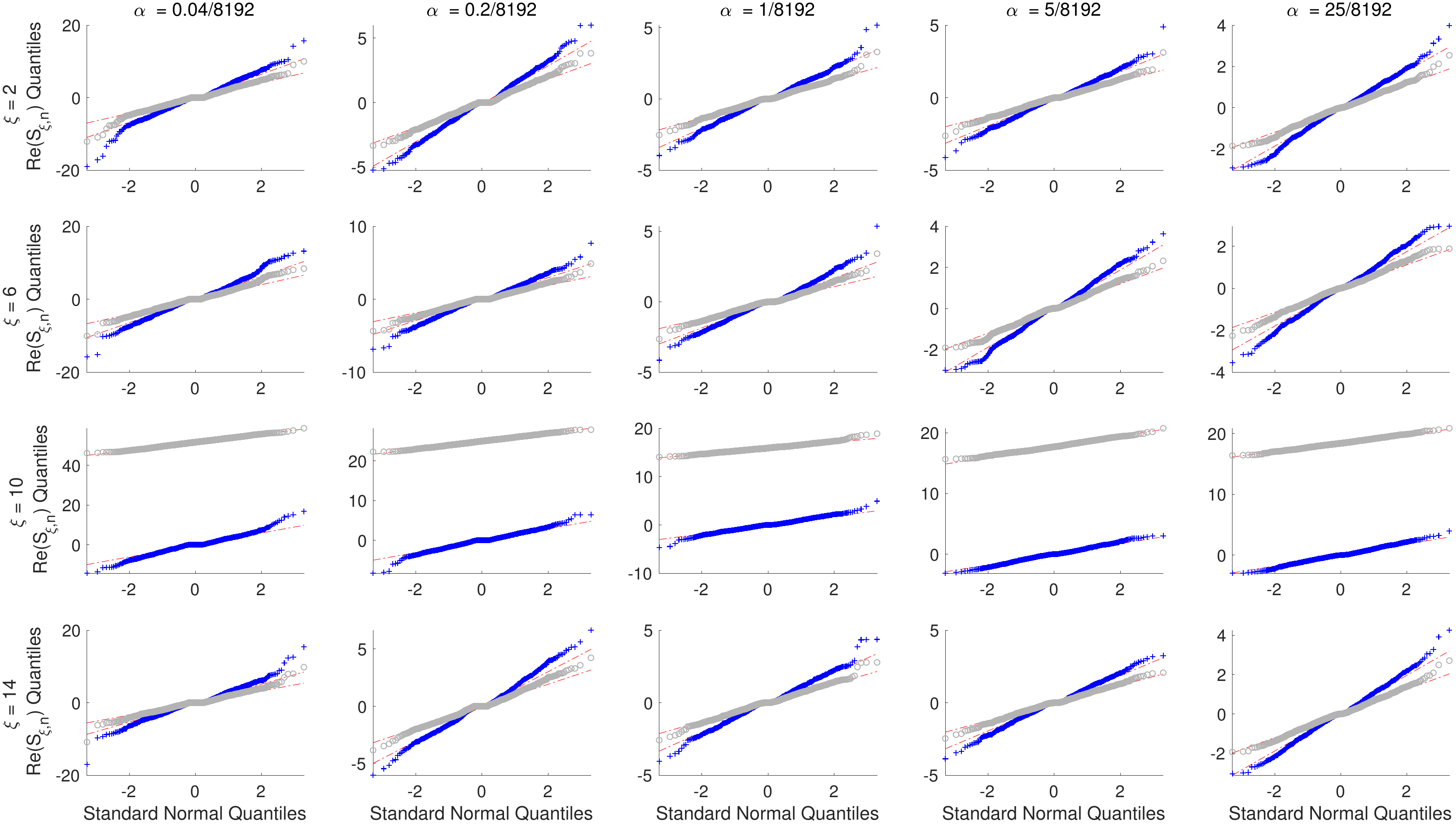}
		\caption{The QQ plots of $S_{A\exp(2\pi i \times 10t)+\Phi}^{(h,\alpha)}(0,\xi)$, where $A=0$ (the null case, shown in blue crosses) and $A=1$ (the non-null case, shown in gray circles), against the standard normal distribution with different $\xi$ and $\alpha$, where $h$ is the Gaussian window. The $x$-axis is the quantiles of the standard normal distribution, and the $y$-axis is the quantiles of $S_{A\exp(2\pi i \times 10t)+\Phi}^{(h,\alpha)}(0,\xi)$ with $1,000$ realizations of $\Phi$.}\label{SST_QQplot2}
	\end{figure}

\end{document}